\documentclass[reqno,11pt]{amsart}

\usepackage{a4wide}
\usepackage{color}
\usepackage{mathrsfs}
\usepackage{mathtools}
\usepackage{tikz-cd}
\usepackage{amsmath}
\usepackage{amssymb}
\usepackage{bbm}
\mathtoolsset{showonlyrefs}

\numberwithin{equation}{section}
\usepackage[colorlinks,citecolor=green,linkcolor=red]{hyperref}
\usepackage{hyphenat}
\usepackage{pgfplots}

\pgfplotsset{compat=1.18}
\usepgfplotslibrary{external}

\usepackage[latin1]{inputenc}
\newcommand{\mrestr}{%
	\,\raisebox{-.127ex}{\reflectbox{\rotatebox[origin=br]{-90}{$\lnot$}}}\,%
}

\newcommand{\N}{\mathbb{N}}
\newcommand{\R}{\mathbb{R}}
\newcommand{\sfd}{{\sf d}}
\renewcommand{\d}{{\mathrm d}}
\newcommand{\restr}[1]{\lower3pt\hbox{\(|_{#1}\)}}

\newcommand{\nchi}{{\raise.3ex\hbox{\(\chi\)}}}
\newcommand{\1}{\mathbbm 1}
\newcommand{\fr}{\penalty-20\null\hfill\(\blacksquare\)}
\newcommand{\X}{{\rm X}}

\newcommand{\mm}{\mathfrak m}

\newtheorem{theorem}{Theorem}[section]
\newtheorem{corollary}[theorem]{Corollary}
\newtheorem{lemma}[theorem]{Lemma}
\newtheorem{proposition}[theorem]{Proposition}
\newtheorem{definition}[theorem]{Definition}
\newtheorem{example}[theorem]{Example}
\newtheorem{remark}[theorem]{Remark}

\title[A Bochner-type integration theory for random normed modules]{A Bochner-type integration theory \\ for random normed modules}
\author{Andrea Kubin}
\address{Department of Mathematics and Statistics,
P.O.\ Box 35 (MaD), FI-40014 University of Jyvaskyla}
\email{andrea.a.kubin@jyu.fi}
\author{Enrico Pasqualetto}
\address{Department of Mathematics and Statistics,
P.O.\ Box 35 (MaD), FI-40014 University of Jyvaskyla}
\email{enrico.e.pasqualetto@jyu.fi}

\begin{document}
\date{\today} 
\begin{abstract}
We develop a measure and integration theory for random normed modules. Given a probability
space $({\rm X},\Sigma,\mathfrak m)$, we introduce and study measures taking values into the
space $L^0(\mathfrak m)$ of $\mathfrak m$-measurable functions quotiented up to
$\mathfrak m$-a.e.\ equality. Moreover, we develop a Bochner-type integration theory with respect
to an $L^0(\mathfrak m)$-valued measure $\mu$, for maps whose target ${\rm M}$ is a complete random
normed module with base $({\rm X},\Sigma,\mathfrak m)$, or equivalently an $L^0(\mathfrak m)$-Banach
$L^0(\mathfrak m)$-module. Inter alia, we prove versions of the Radon--Nikod\'{y}m theorem and of
the Riesz--Markov--Kakutani representation theorem for $L^0(\mathfrak m)$-valued measures.
We also outline several applications of our integration theory: we introduce a notion of martingale
with values in a complete random normed module, we propose a definition of random Radon--Nikod\'{y}m
property and we discuss random sets of finite perimeter.
\end{abstract}
\maketitle
\tableofcontents
\section{Introduction}
\subsection{General overview}\label{s:general_overview}
The classical \emph{Bochner integral} for maps defined on a measurable space \((\Omega,\mathcal A)\) and
with values in a Banach space \(\mathbb B\) is constructed by approximation via \emph{simple maps}.
Letting \(\mu\) be a real-valued measure on \((\Omega,\mathcal A)\) and \(v=\sum_{i=1}^n\1_{A_i}v_i\colon\Omega\to\mathbb B\)
a measurable simple map, we have that the Bochner integral \(\int v\,\d\mu\in\mathbb B\) is defined as \(\sum_{i=1}^n\mu(A_i)v_i\).
Crucially, this definition relies on the fact that \(\mu(A_i)\in\R\) and \(v_i\in\mathbb B\) can be multiplied,
and that the elements \(\mu(A_i)v_i\) can be added, due to the fact that \(\mathbb B\) is a vector space.
Somehow dually, it is possible -- the so-called \emph{Bartle integral} -- to integrate real-valued measurable
maps \(f\colon\Omega\to\R\) with respect to a vector measure \(\nu\colon\mathcal A\to\mathbb B\).
In this case, the integral of a simple function \(f=\sum_{i=1}^n\1_{A_i}\lambda_i\colon\Omega\to\R\)
is given by \(\sum_{i=1}^n\lambda_i\,\nu(A_i)\in\mathbb B\). Here, the scalar \(\lambda_i\in\R\) is multiplied
by the vector \(\nu(A_i)\in\mathbb B\). We refer to \cite{DiestelUhl} for a detailed account of these topics.
\medskip

The goal of this paper is to develop a Bochner-type integration theory for maps taking values in a complete
\emph{random normed module} \(({\rm M},|\cdot|)\) whose base is a probability space \((\X,\Sigma,\mm)\). In short,
\({\rm M}\) is a module over the commutative ring \(L^0(\mm)\) of real-valued \(\mm\)-measurable functions
(considered up to \(\mm\)-a.e.\ equality) that is equipped with a \emph{random norm}
\(|\cdot|\colon{\rm M}\to L^0_+(\mm)\); see Section \ref{s:about_RNM} for historical remarks about this notion
and its applications, and Section \ref{s:RNM} for a brief technical account of the theory of random normed modules.
The basic idea our random-normed-module-valued theory of integration relies on is the following: since the elements
of \({\rm M}\) can be multiplied by functions in \(L^0(\mm)\), it is natural to integrate a map \(v\colon\Omega\to{\rm M}\)
with respect to an \emph{\(L^0\)-valued measure} \(\mu\colon\mathcal A\to L^0(\mm)\), which is a concept that we
will introduce and study in Section \ref{s:L0-val_meas}. By analogy with the classical Bochner integration theory,
we define the integral of a measurable simple map \(v=\sum_{i=1}^n\1_{A_i}v_i\colon\Omega\to{\rm M}\) as
\[
\int v\cdot\d\mu\coloneqq\sum_{i=1}^n\mu(A_i)\cdot v_i\in{\rm M}.
\]
By an approximation procedure, we then define the Lebesgue--Bochner-type space \(L^1_\mu(\Omega;{\rm M})\);
see Section \ref{s:M-val_int}. Before delving into a more detailed description of the contents of this paper,
we provide a brief outline of our main achievements:
\begin{itemize}
\item We thoroughly investigate \(L^0\)-valued measures in Section \ref{s:L0-val_meas}.
\item We develop the Bochner-type integration for random normed modules in Section \ref{s:M-val_int}.
\item We establish several related useful results, for instance a version of the \emph{Radon--Nikod\'{y}m theorem}
for \(L^0\)-valued measures (Section \ref{s:Radon-Nikodym_thm}) and the \emph{random Riesz--Markov--Kakutani
representation theorem} (Section \ref{s:RMK}).
\item We set-up some new tools that involve \(L^0\)-valued measures and
\({\rm M}\)-valued integrable maps. For example, we introduce \emph{\({\rm M}\)-valued martingales}, as well as
\emph{random sets of finite perimeter}; see Section \ref{s:appl}.
\end{itemize}
It is worth pointing out that the theory of random normed modules is in fact an extension of the Banach space theory:
if \(\X\) is a one-point space \(\{o\}\) and \(\mm\) is the Dirac delta \(\delta_o\),
then \(L^0(\delta_o)\) can be canonically identified with the real line \(\R\), and thus the complete random normed
modules with base \((\{o\},\{\varnothing,\{o\}\},\delta_o)\) can be identified with the Banach spaces (exactly with
the same algebraic and metric structures). Let us also highlight that the structure of \(L^0(\mm)\) is typically
more complicated than that of \(\R\), and this affects the kinds of results for random-normed-module-valued measures and
integrals that one might expect to hold, as well as their proofs. For example, the fact that \(L^0(\mm)\) -- differently
from \(\R\) -- is not totally ordered is responsible for the lack of any Hahn-type decomposition for \(L^0\)-valued
measures, as we will discuss in Section \ref{s:L0-val_meas_basic}. Due to this and other reasons, it was not even clear
how to formulate some of the generalisations of classical results for Banach spaces to the random normed module setting,
e.g.\ which are the assumptions in the random Radon--Nikod\'{y}m Theorem \ref{thm:RN}, or what is the natural candidate
for an isometric predual of the space of Radon \(L^0\)-valued measures over a compact Hausdorff space in the random
Riesz--Markov--Kakutani Theorem \ref{thm:RMK}. Moreover, some of the proofs will require new ideas and techniques
if compared to their Banach space counterparts, as we will see throughout the paper.
\subsection{About random normed modules}\label{s:about_RNM}
Random normed modules, which lie at the core of the so-called \emph{random functional analysis}, were developed
by Guo in \cite{guo1989theory,guo1992,guo1993} starting from ideas that date back to the study of probabilistic
metric spaces initiated by Menger, Schweizer and Sklar (see \cite{SS1983}). A central object of study in the
theory of random normed modules is the \emph{random conjugate space}, which was studied e.g.\ in
\cite{Guo1996,GuoYou96,Guo2000,GuoLi05,Guo2008,GuoZeng10}. The larger class of \emph{random locally convex modules}
was then introduced by Guo in \cite{guo1999} and studied e.g.\ in \cite{guo2013,GuoZhangWuYangYuanZeng17}.
In recent years, one of the main topics of research about random normed and locally convex modules is the
\emph{random fixed point theory}, which has been investigated e.g.\ in the works
\cite{GuoWangYangZhang20,GuoZhangWangGuo20,TuMuGuo24,TuMuGuoYangSun25,SunGuoTu25,GuoWangXuYuanChen25}.
See also \cite{GuoTuMuSun26} for a survey on random fixed point theory.
\medskip

Later and independently, Gigli developed in \cite{gigli2018nonsmooth} (see also \cite{gigli2017})
the theory of \emph{\(L^0(\mm)\)-Banach \(L^0(\mm)\)-modules}, taking inspiration from Weaver's work
\cite{Weaver99}, with the aim of building a nonsmooth differential structure for metric measure spaces.
A couple of words on terminology: in \cite{gigli2018nonsmooth,gigli2017} the term `\(L^0(\mm)\)-normed \(L^0(\mm)\)-module'
was used instead, but here (as well as in some previous papers) we prefer to highlight that we assume completeness.
\(L^0(\mm)\)-Banach \(L^0(\mm)\)-modules provide a robust framework for defining different measurable tensor
fields, such as \(1\)-forms and vector fields, in the nonsmooth setting. Although complete random normed modules
and \(L^0(\mm)\)-Banach \(L^0(\mm)\)-modules are fully equivalent concepts, the latter have a somewhat different
interpretation, coming from differential geometry: their elements are thought of as the measurable sections of
some measurable bundle whose fibers are Banach spaces. The representation of \(L^0(\mm)\)-Banach \(L^0(\mm)\)-modules
as section spaces of (various kinds of) measurable Banach bundles was studied in \cite{LP19,DMLP25,GLP},
and the relation between measurable Banach bundles and their section spaces was investigated further in \cite{LPV22,CLP}.
The class of \(L^0(\mm)\)-Banach \(L^0(\mm)\)-modules was examined from a categorical perspective in \cite{Pas24}.
\medskip

Furthermore, other strictly related notions have been studied in the literature. For instance:
\begin{itemize}
\item The notion of \emph{randomly normed \(L^0\)-module}, which is consistent with Guo's theory
of random normed modules, was developed by Haydon, Levy and Raynaud in \cite{HLR91} as 
a tool for studying ultraproducts of Lebesgue--Bochner spaces.
\item Filipovi\'{c}, Kupper and Vogelpoth \cite{FilKupVog09} introduced the \emph{locally \(L^0\)-convex modules}
in \cite{FilKupVog09}, as a framework for studying conditional convex risk measures. The locally \(L^0\)-convex
topology on a random locally convex module was investigated in \cite{WuGuo15,Zapata17}, and its relation with
the topology usually considered on them -- the so-called \emph{\((\varepsilon,\lambda)\)-topology} --
was then understood in \cite{Guo10}. These two topologies and their mutual connections 
have been widely used in mathematical finance (see e.g.\ \cite{Guo-2011,GuoZhaoZeng14,GuoZhaoZeng15,Guo24}).
\item In the works \cite{gigli2018nonsmooth,gigli2017} the focus was on \emph{\(L^p(\mm)\)-Banach \(L^\infty(\mm)\)-modules}
of exponent \(p\in[1,\infty]\), which are normed modules that satisfy an integrability condition.
Similar structures have been considered in the framework of Dirichlet forms (see e.g.\ \cite{IonescuRogersTeplayev12,HinzRocknerTeplyaev13}).
\item In \cite{CGP} the second named author -- together with Debin and Gigli -- introduced and studied the class
of \emph{\(L^0({\rm Cap})\)-Banach \(L^0({\rm Cap})\)-modules}, where \({\rm Cap}\) denotes the
\emph{Sobolev \(2\)-capacity} (which is an outer measure, but not a \(\sigma\)-additive measure)
over a given metric measure space. This specific variant of normed module turned out to be of fundamental
importance in understanding the fine differential structure of \emph{\({\sf RCD}\) spaces},
which are metric measure spaces with synthetic lower Ricci curvature and upper dimension bounds;
see e.g.\ \cite{BruePasqualettoSemola22,BreGig23}.
\item \emph{Normed \(A\)-modules} in the sense of
\cite{CerreiaVioglioKupperMaccheroniMarinacciVogelpoth16,CerreiaVioglioMaccheroniMarinacci17}, where \(A\)
is a suitable \(f\)-algebra, were studied in connection with mathematical models in finance.
\item The axiomatic theory of normed modules developed in \cite{luvcic2024axiomatic} covers
essentially all the normed-module-type structures that are mentioned above.
\end{itemize}
\subsection{Contents of the paper}
We now describe in details definitions and results that we will encounter in this paper.
Fix a probability space \((\X,\Sigma,\mm)\) and a measurable space \((\Omega,\mathcal A)\).
\subsubsection*{\(L^0\)-valued and \({\rm M}\)-valued measures}
In Definition \ref{def:L0_meas} we introduce our notion of \emph{\(L^0\)-valued measure}
\(\mu\colon\mathcal A\to L^0(\mm)\), whose characterising property is the \(\sigma\)-additivity condition
\[
\mu\bigg(\bigcup_{n\in\N}A_n\bigg)=\sum_{n\in\N}\mu(A_n)\quad\text{ for every sequence }(A_n)_{n\in\N}\subseteq\mathcal A
\text{ of pairwise disjoint sets,}
\]
where sums are intended in a suitable sense. The space of those \(L^0\)-valued measures \(\mu\)
\emph{of bounded variation}, which we define in Definition \ref{def:L0_meas_bv}, is a complete random normed module
\[
(\mathcal M(\Omega;L^0(\mm)),|\cdot|_{\rm TV})
\]
with base \((\X,\Sigma,\mm)\); see Lemma \ref{lem:L0-val_meas_RNM}. A distinguished class of \(L^0\)-valued measures
is the one of those \(\mu\in\mathcal M(\Omega;L^0(\mm))\) that can be \emph{foliated}, which means -- roughly speaking
-- that there exists a measurable selection \(\X\ni x\mapsto\mu_x\) of real-valued measures on \((\Omega,\mathcal A)\)
such that \(\mu(A)\) is the \(\mm\)-a.e.\ equivalence class of \(\X\ni x\mapsto\mu_x(A)\) for every \(A\in\mathcal A\);
see Definition \ref{def:foliation_meas}.
\medskip

In the special case where \(\Omega=(\Omega,\tau)\) is a Hausdorff topological space and \(\mathcal A\)
is the Borel \(\sigma\)-algebra \(\mathscr B(\Omega)\) of \(\Omega\), we call \emph{Borel \(L^0\)-valued measures}
the elements of \(\mathcal M(\Omega;L^0(\mm))\). By taking inspiration from the classical measure theory, we also
propose a notion of \emph{Radon \(L^0\)-valued measure} in Definition \ref{def:RadonL0_meas}. Radon \(L^0\)-valued
measures form the complete random normed submodule \((\mathfrak M(\Omega;L^0(\mm)),|\cdot|_{\rm TV})\) of
\(\mathcal M(\Omega;L^0(\mm))\). In Section \ref{s:Radon_L0-val_meas} we prove the basic properties of Radon
\(L^0\)-valued measures. We also study \emph{outer \(L^0\)-valued measures} and \emph{Baire \(L^0\)-valued measures}
in Sections \ref{s:outer_L0-val_meas} and \ref{s:Baire_L0-val_meas}, respectively, as auxiliary tools for the
construction of Radon \(L^0\)-valued measures in the proof of the random Riesz--Markov--Kakutani theorem.
\medskip

More generally, in Section \ref{s:M-val_meas} we discuss \emph{\(\rm M\)-valued measures}, where
\(({\rm M},|\cdot|)\) is a complete random normed module with base \((\X,\Sigma,\mm)\). One aspect
of the theory of \(\rm M\)-valued measures that is not an easy adaptation of its \(L^0\)-valued counterpart
is the construction of a foliation. Even the definition of the latter is non-trivial, since it requires
that the module \({\rm M}\) under consideration is the section space of some measurable Banach bundle;
see Section \ref{s:M-val_foliat}.
\subsubsection*{Module-valued Bochner integration}
In Section \ref{s:M-val_int} we develop our Bochner integration theory for maps \(v\colon\Omega\to{\rm M}\)
with respect to a given \(L^0\)-valued measure \(\mu\in\mathcal M(\Omega;L^0(\mm))\). As we mentioned in
Section \ref{s:general_overview}, we first define the integral of all measurable simple maps \(v\) and then
we obtain the \emph{module-valued Lebesgue--Bochner space} \(L^1_\mu(\Omega;{\rm M})\) via a completion
procedure; see Definition \ref{def:M-val_Leb-Boch}. It is not clear to us whether all the elements of
\(L^1_\mu(\Omega;{\rm M})\) -- which are defined through an abstract procedure -- can be represented
as (equivalence classes of) measurable maps from \(\Omega\) to \({\rm M}\). We prove in Section
\ref{s:ptwse_descr_L_1} that this is possible at least in the case where \(\mu\) can be foliated.
We also study the spaces \(L^p_\mu(\Omega;{\rm M})\) for \(p=\infty\) and \(1<p<\infty\) in Sections
\ref{s:Linfty_to_M} and \ref{s:Lp_to_M}, respectively, while in Section \ref{s:L1_tensor_prod} we characterise
\(L^1_\mu(\Omega;{\rm M})\) as a projective tensor product of random normed modules (developed by the
second named author in \cite{Pas23}), thus extending a well-known result for classical Lebesgue--Bochner spaces.
\medskip

An important result concerning \(L^0\)-valued measures and \(L^0\)-valued integrals that we achieve
is a version of the \emph{Radon--Nikod\'{y}m theorem} for random normed modules; see Theorem \ref{thm:RN}.
The result states that, given two \(L^0\)-valued measures of bounded variation \(\mu\) and \(\nu\) on
\((\Omega,\mathcal A)\) such that \(\mu\) satisfies an absolute-continuity-type condition with respect to \(\nu\),
then \(\mu\) can be represented in the integral form
\[
\mu(A)=\int_A\delta\cdot\d\nu\quad\text{ for every }A\in\mathcal A,
\]
for some `random Radon--Nikod\'{y}m derivative' \(\delta\in L^1_\nu(\Omega;L^0(\mm))\). More precisely, the
absolute continuity assumption of Theorem \ref{thm:RN} is that \(\hat\mu\ll\hat\nu\), where \(\hat\mu\) and
\(\hat\nu\) are real-valued measures on the product space \((\Omega\times\X,\mathcal A\otimes\Sigma)\) that
are canonically induced by \(\mu\) and \(\nu\), respectively, as in Proposition \ref{prop:def_hat_mu}. We do
not know whether the random Radon--Nikod\'{y}m theorem is valid under the (possibly strictly) less stringent and
seemingly more natural assumption that \(\mu\ll\nu\), by which we mean that \(\mu(N)=0\in L^0(\mm)\) whenever
\(N\in\Sigma\) satisfies \(\nu(N)=0\in L^0(\mm)\).
\subsubsection*{The random Riesz--Markov--Kakutani theorem}
A considerable part of this paper -- namely, the whole Section \ref{s:RMK} -- is devoted to the
study of the \emph{random Riesz--Markov--Kakutani theorem}, which consists of
the identification of a random normed module whose random conjugate space is isometrically isomorphic
to the space \(\mathfrak M(\Omega;L^0(\mm))\) of all Radon \(L^0\)-valued measures over a given
compact Hausdorff space \((\Omega,\tau)\). The classical Riesz--Markov--Kakutani theorem states
that the Banach space dual of \(C(\Omega)\) is the Banach space of real-valued Radon measures
\(\mathfrak M(\Omega)\). In the more general context of random normed modules, there are several
different random normed modules that can be regarded as generalisations of the space \(C(\Omega)\)
(see \cite[Section 3.4]{Pas23}), thus it was not a priori clear which could be a good candidate
for a predual of \(\mathfrak M(\Omega;L^0(\mm))\). We have identified in the random normed module
\({\rm UC}_{\rm ord}(\Omega;L^0(\mm))\) of all \emph{uniformly-order-continuous} maps \(f\colon\Omega\to L^0(\mm)\)
(see Definition \ref{def:UC_ord}) an isometric predual of the space \(\mathfrak M(\Omega;L^0(\mm))\). Namely,
we will prove in Theorem \ref{thm:RMK} that
\[
{\rm UC}_{\rm ord}(\Omega;L^0(\mm))^*\text{ and }\mathfrak M(\Omega;L^0(\mm))\text{ are isometrically isomorphic}
\]
for every compact Hausdorff space \((\Omega,\tau)\). In Section \ref{s:UC_ord} we study the space
\({\rm UC}_{\rm ord}(\Omega;{\rm M})\), previously introduced by the second named author
in \cite[Definition 3.13]{Pas23}, in its own right. Random normed modules are equipped with two
distinguished uniform structures: the one associated to the \emph{\((\varepsilon,\lambda)\)-topology}
\(\mathcal T_{\varepsilon,\lambda}\) introduced by Guo \cite{guo1989theory}  and considered also by Gigli
\cite{gigli2018nonsmooth}, and the \emph{locally \(L^0\)-convex topology} \(\mathcal T_c\)
introduced by Filipovi\'{c}, Kupper and Vogelpoth \cite{FilKupVog09}; see Remark \ref{rmk:two_unif_struct}.
Nevertheless, we show that -- interestingly -- the maps \(f\colon\Omega\to L^0(\mm)\) belonging to
\({\rm UC}_{\rm ord}(\Omega;L^0(\mm))\) are (in general) neither those that are uniformly continuous with
respect to the \((\varepsilon,\lambda)\)-uniformity (Example \ref{ex:UC_1}) nor those that are uniformly
continuous with respect to the locally \(L^0\)-convex uniformity (Example \ref{ex:UC_2}). In fact, it is
not even clear whether the uniform order continuity is the uniform continuity with respect to
some uniformity on \(L^0(\mm)\).
\medskip

In Section \ref{s:aux_RMK} we collect many auxiliary technical statements, while in Section \ref{s:proof_RMK}
we provide two different proofs of the random Riesz--Markov--Kakutani theorem. In the particular case where
\((\Omega,\sfd)\) is a compact metric space, we first present an argument that relies on the von Neumann theory of
liftings, where we essentially reduce the proof to the classical Riesz--Markov--Kakutani theorem for Banach spaces.
The same arguments also allow us to show that every Radon \(L^0\)-valued measure on \((\Omega,\sfd)\) can be
foliated (Corollary \ref{corolmainthm}), along with an alternative, `pointwise' characterisation of
\(\mathfrak M(\Omega;L^0(\mm))\) (Theorem \ref{thm:charact_foliation_cpt_metric}). We then present another
proof for the general case where \((\Omega,\tau)\) is a compact Hausdorff space. Shortly said, this second
argument is based on a reduction to the case where \((\Omega,\tau)\) is extremally disconnected, which can
be more easily established. This technique is adapted from Hartig's paper \cite{Hartig83}, where the argument
is attributed to Garling.
\subsubsection*{Applications}
In Section \ref{s:appl} we set-up some potential applications of our measure and integration theory for
random normed modules. More specifically:
\begin{itemize}
\item In Section \ref{s:random_martingale} we introduce \emph{random-normed-module-valued martingales},
see Definition \ref{def:RNM_martingale}. To show their well-posedness, we first need to prove the existence of
\emph{conditional expectations} for maps with values in a random normed module; see Theorem \ref{thm:E_M}.
\item In Section \ref{s:rRNP} we propose a definition of \emph{random Radon--Nikod\'{y}m property},
see Definition \ref{def:RNP}. We also discuss some possible reformulations of this condition, related to
module-valued martingales and to the differentiability of \emph{random Lipschitz curves}.
\item In Section \ref{s:random_FP_sets} we introduce the notion of \emph{random set of finite perimeter}
(Definition \ref{def:random_set_FP}) and we discuss some of its basic properties.
\end{itemize}
\subsection{Further research directions}
In addition to the applications that we have outlined above, we would like to discuss a few other potential
research directions. First, it would be interesting to compare our notions and results with others that are
available in the literature, for example with \emph{martingale-valued measures} (see e.g.\ \cite{Applebaum}),
with the theory of \emph{local vector measures} developed by Brena and Gigli in \cite{BreGig}, or with the
integration theory for module-valued maps developed by the second named author together with Caputo and
Gigli in \cite[Section 3.3]{CGP}. We expect that the last-mentioned concept of integration is consistent
with the one that we study in this paper; more precisely, the space \(L^1([0,1],\mathscr H)\) introduced
in \cite[Definition 3.15(i)]{CGP} -- for \(\mathscr H\) a separable Hilbertian random normed module
\(\mathscr H\) with base a metric measure space -- is likely to correspond to \(L^1_\mu([0,1];\mathscr H)\),
where as \(\mu\) we choose the `constant' \(L^0\)-valued measure induced by the Lebesgue measure
(as in Remark \ref{rmk:lambda^c}).
\medskip

Another possible research direction, which is one of the main motivations behind our newly-developed
integration theory, is to make further progress in \(L^0\)-convex analysis. In the paper \cite{GuoWangTang23},
Guo, Wang and Tang proved a Krein--Milman theorem for random locally convex modules.
A further step could be to obtain a version of the
\emph{Choquet(--Bishop--de Leeuw) theorem} \cite[Chapter 4]{Phelps03} for random locally convex modules
(or at least for random normed modules) by leveraging our random Lebesgue--Bochner space machinery
(or a suitable variant of it), and to explore its consequences.
\subsection{Main notation}
We provide a list of non-standard symbols that appear in this paper.
\medskip

\halign{#\quad&#\hfil\cr
\((\X,\bar\Sigma_\mm,\bar\mm)\) & completion of a probability space \((\X,\Sigma,\mm)\); Section \ref{s:meas_theory}.\cr
\(\1_E^\mm\) & \(\mm\)-a.e.\ equivalence class of the characteristic function \(\1_E\) of \(E\);
\eqref{eq:1_E_mm}.\cr
\({\rm M}^*\) & random conjugate space of a random normed module \({\rm M}\); Section \ref{s:RNM}.\cr
\(\bar{\rm M}\) & \(L^0\)-completion of a random normed module \({\rm M}\); Theorem \ref{thm:L0-complet}.\cr
\(\mathcal P_f(S)\) & family of all finite subsets of a set \(S\); Section \ref{s:RNM}.\cr
\(\lambda^c\) & \(L^0\)-valued measure induced by a real-valued measure \(\lambda\); Remark \ref{rmk:lambda^c}.\cr
\({|\cdot|_{\rm TV}}\) & \(L^0\)-total variation \(L^0\)-norm; \eqref{eq:def_|mu|_TV}.\cr
\(\mathcal M(\Omega;L^0(\mm))\) & space of \(L^0\)-valued measures of bounded variation; Definition \ref{def:L0_meas_bv}.\cr
\(\mathcal M_+(\Omega;L^0(\mm))\) & non-negative \(L^0\)-valued measures of bounded variation; Definition \ref{def:L0_meas_bv}.\cr
\(\mu\mrestr A\) & restriction of \(\mu\in\mathcal M(\Omega;L^0(\mm))\) to a measurable set \(A\); \eqref{eq:def_restr_L0-meas}.\cr
\(\varphi_\#\mu\) & pushforward of \(\mu\in\mathcal M(\Omega;L^0(\mm))\) with respect to \(\varphi\);
Definition \ref{def:pushforward_meas}.\cr
\(\mathcal P(\Omega;L^0(\mm))\) & space of probability \(L^0\)-valued measures; Definition \ref{def:prob_L0-meas}.\cr
\([\mu]\) & measure on \(\Omega\) induced by \(\mu\in\mathcal P(\Omega;L^0(\mm))\); \eqref{eq:def_[mu]}.\cr
\(\hat\mu\) & measure on \(\Omega\times\X\) induced by \(\mu\in\mathcal P(\Omega;L^0(\mm))\);
Proposition \ref{prop:def_hat_mu}.\cr
\(\mathcal N_\mu\) & collection of all \(\mu\)-null sets; Definition \ref{def:mu-null}.\cr
\(\mu^+\) & positive part of \(\mu\in\mathcal M(\Omega;L^0(\mm))\); \eqref{eq:pos_neg_parts}.\cr
\(\mu^-\) & negative part of \(\mu\in\mathcal M(\Omega;L^0(\mm))\); \eqref{eq:pos_neg_parts}.\cr
\(\mathfrak M(\Omega;L^0(\mm))\) & space of Radon \(L^0\)-valued measures; Definition \ref{def:RadonL0_meas}.\cr
\(\mathfrak M_+(\Omega;L^0(\mm))\) & non-negative Radon \(L^0\)-valued measures; Definition \ref{def:RadonL0_meas}.\cr
\(\Gamma(\nu)\) & Carath\'{e}odory \(\sigma\)-algebra of \(\nu\); Theorem \ref{thm:Carath}.\cr
\(\mathscr B_0(\Omega)\) & Baire \(\sigma\)-algebra of a compact Hausdorff space \((\Omega,\tau)\); Section \ref{s:Baire_L0-val_meas}.\cr
\(\mathcal M_+(\Omega,\mathscr B_0(\Omega);L^0(\mm))\) & space of Baire \(L^0\)-valued measures; \eqref{eq:def_space_Bair_L0-meas}.\cr
\(S(\Omega;{\rm M})\) & space of module-valued simple maps; Definition \ref{defMappesempliciM}.\cr
\(S_\mu(\Omega;{\rm M})\) & space of \(\mu\)-a.e.\ equivalence classes of elements of \(S(\Omega;{\rm M})\); \eqref{eq:def_S_mu}.\cr
\({\int v\cdot\d\mu}\) & \(\mu\)-integral of \(v\colon\Omega\to{\rm M}\); Definition \ref{def:int-mod-simp-map}.\cr
\(S^1_\mu(\Omega;{\rm M})\) & space of \(\mu\)-integrable module-valued simple maps; Definition \ref{def:int-mod-simp-map}.\cr
\((L^1_\mu(\Omega;{\rm M}),{|\cdot|^{1,\mu}})\) & module-valued Lebesgue--Bochner space; Definition \ref{def:M-val_Leb-Boch}.\cr
\(S_{\mu,f}(\Omega;{\rm M})\) & finite-range module-valued simple maps; \eqref{funsemplicisomsufinit}.\cr
\(L^1_\mu(\Omega;L^0_+(\mm))\) & space of non-negative elements of \(L^1_\mu(\Omega;L^0(\mm))\); \eqref{eq:non-neg_L1}.\cr
\(\mathcal L^0(\Omega;{\rm M})\) & space of measurable maps \(v\colon\Omega\to{\rm M}\); Section \ref{s:Linfty_to_M}.\cr
\(L^0_\mu(\Omega;{\rm M})\) & quotient of \(\mathcal L^0(\Omega;{\rm M})\) up to \(\mu\)-a.e.\ equality; \eqref{eq:def_L0_mu}.\cr
\(S^\infty_\mu(\Omega;{\rm M})\) & space of essentially bounded module-valued simple maps; \eqref{eq:def_Sinfty_mu}.\cr
\((L^\infty_\mu(\Omega;{\rm M}),{|\cdot|^{\infty,\mu}})\) & space of essentially bounded maps \(v\colon\Omega\to{\rm M}\); \eqref{defLinfmuOmM}.\cr
\((L^p_\mu(\Omega;{\rm M}),{|\cdot|^{p,\mu}})\) & space of \(p\)-integrable maps \(v\colon\Omega\to{\rm M}\); \eqref{eq:def_Lp_mu}.\cr
\(({\rm M}\hat\otimes_\pi{\rm N},{|\cdot|_\pi})\) & projective tensor product of \({\rm M}\) and \({\rm N}\); Section \ref{s:L1_tensor_prod}.\cr
\({\hat b\colon{\rm M}\hat\otimes_\pi{\rm N}\to{\rm Q}}\) & \(L^0\)-linearisation of \(b\colon{\rm M}\times{\rm N}\to{\rm Q}\); Theorem \ref{thm:univ_prop_proj_tens}.\cr
\(\tilde{\mathcal L}^1_\mu(\Omega;\mathcal L^0(\bar\Sigma_\mm))\) & `pointwise' \(\mu\)-integrable
maps \(\Omega\to\mathcal L^0(\bar\Sigma_\mm)\); Definition \ref{def:tildeL1}.\cr
\((\tilde L^1_\mu(\Omega;L^0(\mm)),{|\cdot|^{1,\mu}_\sim})\) & `pointwise' Lebesgue-Bochner space; Section \ref{s:ptwse_descr_L_1}.\cr
\(\tilde L^1_\mu(\Omega;L^0_+(\mm))\) & space of non-negative elements of \(\tilde L^1_\mu(\Omega;L^0(\mm))\); \eqref{eq:non-neg_tildeL1}.\cr
\(\tilde{\mathcal L}^p_\mu(\Omega;\mathcal L^0(\bar\Sigma_\mm))\) & space of `pointwise' \(p\)-integrable maps \(\Omega\to\mathcal L^0(\bar\Sigma_\mm)\);
\eqref{eq:def_tildecalLp}.\cr
\((\tilde L^p_\mu(\Omega;L^0(\mm)),{|\cdot|^{p,\mu}_\sim})\) & quotient of \(\tilde{\mathcal L}^p_\mu(\Omega;\mathcal L^0(\bar\Sigma_\mm))\)
up to \(\mu\)-a.e.\ equality; \eqref{eq:def_tildeLp} and \eqref{eq:def_tildeLp_norm}.\cr
\({\rm Var}(v;\mathcal U)\) & variation of a map \(v\colon\Omega\to{\rm M}\) on an entourage \(\mathcal U\); Definition \ref{def:UC_ord}.\cr
\(({\rm UC}_{\rm ord}(\Omega;{\rm M}),{|\cdot|^\infty})\) & space of uniformly-order-continuous maps \(v\colon\Omega\to{\rm M}\); Definition \ref{def:UC_ord}.\cr
\({\rm UC}_{{\rm ord},\mu}(\Omega;{\rm M})\) & quotient of \({\rm UC}_{\rm ord}(\Omega;{\rm M})\) up to \(\mu\)-a.e.\ equality; Section \ref{s:UC_ord}.\cr
\(\mathscr L_1\) & restriction of the Lebesgue measure to \([0,1]\); Example \ref{ex:UC_1}.\cr
\({\rm UC}_{\rm ord}(\Omega;L^0(\mm))^*_+\) & positive operators \(L\colon{\rm UC}_{\rm ord}(\Omega;L^0(\mm))\to L^0(\mm)\); Definition \ref{def:posit_operators}.\cr
\(L^0(\mm;\mathbb B)\) & \(L^0\)-Lebesgue--Bochner space from \((\X,\Sigma,\mm)\) to \(\mathbb B\); Section \ref{s:aux_RMK}.\cr
\({\rm I}_\Omega\) & the Riesz--Markov--Kakutani isometric isomorphism; Theorem \ref{thm:RMK}.\cr
\(L^0_{w^*}(\mm;\mathbb B^*)\) & space of weakly\(^*\) \(\mm\)-measurable maps from \(\X\) to \(\mathbb B^*\);
Section \ref{s:proof_RMK}.\cr
\((\mathcal M(\Omega;{\rm M}),{|\cdot|_{\rm TV}})\) & space of \({\rm M}\)-valued measures of bounded variation; Section \ref{s:M-val_basic}.\cr
\(\mathfrak M(\Omega;{\rm M})\) & space of Radon \({\rm M}\)-valued measures; Section \ref{s:M-val_basic}.\cr
\(\bar\Gamma(E)\) & measurable sections of a measurable Banach bundle \(E\); Section \ref{s:mBb}.\cr
\(\Gamma(E)\) & section space of a measurable Banach bundle \(E\); Section \ref{s:mBb}.\cr
\(({\rm M}_x,{\|\cdot\|_x})\) & fibers of a countably-generated random normed module \({\rm M}\); \eqref{eq:fibers_of_M}.\cr
\(\mathbb E(\cdot|\mathcal B)\) & module-valued conditional expectation; Theorem \ref{thm:E_M}.\cr
\(P(\nchi)\) & random perimeter of \(\nchi\); Definition \ref{def:random_set_FP}.\cr
}
\section{Preliminaries}
\subsection{Measure theory}\label{s:meas_theory}
Let \((\X,\Sigma,\mm)\) be a given probability space. We denote by \((\X,\bar\Sigma_\mm,\bar\mm)\) the probability
space that is obtained by taking the completion of \((\X,\Sigma,\mm)\). Let us define
\[\begin{split}
\bar{\mathcal L}^0(\Sigma)&\coloneqq\big\{f\colon\X\to[-\infty,+\infty]\;\big|\;f\text{ is }\Sigma\text{-measurable}\big\},\\
\mathcal L^0(\Sigma)&\coloneqq\big\{f\in\bar{\mathcal L}^0(\Sigma)\;\big|\;f(x)\in\R\text{ for every }x\in\X\big\},\\
\mathcal L^\infty(\Sigma)&\coloneqq\big\{f\in\mathcal L^0(\Sigma)\;\big|\;f\text{ is bounded}\big\}.
\end{split}\]
By identifying two functions \(f,g\in\bar{\mathcal L}^0(\Sigma)\) if they agree \(\mm\)-a.e., we obtain an equivalence relation
whose quotient space we denote by \(\bar L^0(\mm)\). The equivalence class of \(f\in\bar{\mathcal L}^0(\Sigma)\)
will be indicated with \([f]_\mm\in\bar L^0(\mm)\). We then define the spaces \(L^0(\mm),L^\infty(\mm)\subseteq\bar L^0(\mm)\) as
\[
L^0(\mm)\coloneqq\big\{[f]_\mm\;\big|\;f\in\mathcal L^0(\Sigma)\big\},\qquad
L^\infty(\mm)\coloneqq\big\{[f]_\mm\;\big|\;f\in\mathcal L^\infty(\Sigma)\big\}.
\]
We denote by \(L^p(\mm)\), for \(p\in[1,\infty)\), the space of all those \(f\in L^0(\mm)\) that are \(p\)-integrable
with respect to \(\mm\). For any \(E\in\Sigma\), we denote by \(\1_E\in\mathcal L^\infty(\Sigma)\) its characteristic
function and we write
\begin{equation}\label{eq:1_E_mm}
\1_E^\mm\coloneqq[\1_E]_\mm\in L^\infty(\mm).
\end{equation}
The function spaces introduced above carry different, mutually compatible structures:
\begin{itemize}
\item \textsc{Order structure.} The space \(\bar{\mathcal L}^0(\Sigma)\) is a lattice with respect to
the partial order \(\leq\) given by the pointwise inequality between functions, while
\(\bar L^0(\mm)\) is a lattice with respect to the \(\mm\)-a.e.\ inequality order relation, which
we still denote by \(\leq\). Likewise, \(\mathcal L^0(\Sigma)\), \(\mathcal L^\infty(\Sigma)\) and
\(L^0(\mm)\), \(L^p(\mm)\) for \(p\in[1,\infty]\) are lattices with respect to the partial orders induced
by \(\bar{\mathcal L}^0(\Sigma)\) and \(\bar L^0(\mm)\), respectively.
\item \textsc{Algebraic structure.} The spaces \(\mathcal L^0(\Sigma)\), \(\mathcal L^\infty(\Sigma)\),
\(L^0(\mm)\) and \(L^\infty(\mm)\) are commutative algebras, while \(L^p(\mm)\) for \(p\in[1,\infty)\)
is a vector space, with respect to the usual pointwise operations. All these spaces are Riesz spaces
with respect to the above order structures. We denote by \(L^0_+(\mm)\) the \emph{positive cone} of \(L^0(\mm)\), i.e.
\[
L^0_+(\mm)\coloneqq\big\{f\in L^0(\mm)\;\big|\;f\geq 0\big\}.
\]
Likewise, one defines \(L^\infty_+(\mm)\), \(\mathcal L^0_+(\Sigma)\) and \(\mathcal L^\infty_+(\Sigma)\).
\item \textsc{Topological structure.} We endow the space \(L^0(\mm)\) with the complete distance
\[
\sfd_{L^0(\mm)}(f,g)\coloneqq\int|f-g|\wedge 1\,\d\mm\quad\text{ for every }f,g\in L^0(\mm).
\]
The distance \(\sfd_{L^0(\mm)}\) metrises the convergence in \(\mm\)-measure, and the
\(\mm\)-a.e.\ convergence up to a subsequence. I.e., given any \((f_n)_{n\in\N}\subseteq L^0(\mm)\)
and \(f\in L^0(\mm)\), the following conditions are equivalent:
\begin{itemize}
\item[\(\rm i)\)] \(\sfd_{L^0(\mm)}(f_n,f)\to 0\) as \(n\to\infty\).
\item[\(\rm ii)\)] \(\mm(\{|f_n-f|>\varepsilon\})\to 0\) as \(n\to\infty\) for every \(\varepsilon>0\).
\item[\(\rm iii)\)] Any subsequence \((f_{n_i})_{i\in\N}\) has a subsequence
\((f_{n_{i_j}})_{j\in\N}\) such that \(f_{n_{i_j}}(x)\to f(x)\) as \(j\to\infty\) for \(\mm\)-a.e.\ \(x\in\X\).
\end{itemize}
The space \(L^0(\mm)\) is a topological algebra with respect to the topology induced by the distance \(\sfd_{L^0(\mm)}\).
Moreover, for any \(p\in[1,\infty]\), the space \(L^p(\mm)\) is a Banach space with respect
to the \(L^p\)-norm \(\|\cdot\|_{L^p(\mm)}\).
\end{itemize}

The lattice \(\bar L^0(\mm)\) is \emph{complete}, meaning that all its subsets \(S=\{f_i\}_{i\in I}\)
have both a supremum \(\bigvee S=\bigvee_{i\in I}f_i\in\bar L^0(\mm)\) and an infimum
\(\bigwedge S=\bigwedge_{i\in I}f_i\in\bar L^0(\mm)\). As a consequence, one has that
\(L^0(\mm)\) and \(L^\infty(\mm)\) are \emph{conditionally-complete} lattices, i.e.\ all
their subsets that are bounded above (resp.\ below) have a supremum (resp.\ an infimum).
Furthermore, the space \(\bar L^0(\mm)\) (thus also \(L^0(\mm)\) and \(L^\infty(\mm)\)) has the
\emph{countable sup} (resp.\ \emph{inf}) \emph{property}: if \(\{f_i\}_{i\in I}\) has a supremum
(resp.\ an infimum), there is a countable set \(C\subseteq I\) such that \(\bigvee_{i\in I}f_i=\bigvee_{i\in C}f_i\)
(resp.\ \(\bigwedge_{i\in I}f_i=\bigwedge_{i\in C}f_i\)).
\begin{remark}{\rm
The spaces \(\bar L^0(\mm)\) and \(\bar L^0(\bar\mm)\) can be canonically identified.
Similarly, \(L^0(\mm)\) can be identified with \(L^0(\bar\mm)\), and \(L^\infty(\mm)\) can be identified
with \(L^\infty(\bar\mm)\). Therefore, without loss of generality, we shall occasionally consider \(\bar\mm\)
in place of \(\mm\).
\fr}\end{remark}
\begin{remark}{\rm
Given any \(\sigma\)-finite measure \(\mathfrak n\) on a measurable space \((\X,\Sigma)\), one can easily
construct a probability measure \(\tilde{\mathfrak n}\) on \((\X,\Sigma)\) such that
\(\mathfrak n\ll\tilde{\mathfrak n}\leq C\mathfrak n\) for some constant \(C>0\). It is also clear that
\(\bar L^0(\mathfrak n)\) can be canonically identified with \(\bar L^0(\tilde{\mathfrak n})\); similarly
for \(L^0\) and \(L^\infty\). Accordingly, many of the results that we will present are valid also in the more
general case where \((\X,\Sigma,\mm)\) is a \(\sigma\)-finite measure space, but it is not restrictive
to assume that \(\mm\) is a probability measure, which we do for the sake of simplicity.
\fr}\end{remark}
\begin{remark}\label{rmk:suff_cond_ord-bdd}{\rm
Let \(\{f_i\}_{i\in I}\subseteq\mathcal L^0(\Sigma)\) be a (possibly uncountable) family of functions such that
\(\sup_{i\in I}|f_i(x)|<+\infty\) for every \(x\in\X\). Then we claim that \(\{[f_i]_\mm:i\in I\}\) is an order-bounded
subset of \(L^0(\mm)\); this is not obvious, since the function \(\X\ni x\mapsto\sup_{i\in I}|f_i(x)|\)
is not necessarily measurable. To prove it, take a countable set \(C\subseteq I\) such that
\(\bigvee_{i\in C}[|f_i|]_\mm=\bigvee_{i\in I}[|f_i|]_\mm\), where suprema are intended in \(\bar L^0(\mm)\).
Then the function \(g\in\mathcal L^0_+(\Sigma)\), which we define as
\[
g(x)\coloneqq\sup_{i\in C}|f_i(x)|\quad\text{ for every }x\in\X,
\]
satisfies \(|[f_i]_\mm|\leq\bigvee_{j\in I}[|f_j|]_\mm=\bigvee_{j\in C}[|f_j|]_\mm=[g]_\mm\) for every \(i\in I\),
proving the claim.
\fr}\end{remark}

In the proof of Theorem \ref{thm:RMK}, we will make use of the \emph{von Neumann lifting theory}, which we
are going to remind. The \emph{von Neumann--Maharam theorem} (see \cite[Theorem 341K]{fre200})
ensures the existence of a linear map \(\ell\colon L^\infty(\bar\mm)\to\mathcal L^\infty(\bar\Sigma_\mm)\) that satisfies
the following properties:
\[\begin{split}
[\ell(f)]_{\bar\mm}=f&\quad\text{ for every }f\in L^\infty(\bar\mm),\\
\ell(fg)=\ell(f)\ell(g)&\quad\text{ for every }f,g\in L^\infty(\bar\mm),\\
\ell(\lambda\1_\X^{\bar\mm})=\lambda\1_\X&\quad\text{ for every }\lambda\in\R,\\
\ell(f)\in\mathcal L^\infty_+(\bar\Sigma_\mm)&\quad\text{ for every }f\in L^\infty_+(\bar\mm),\\
|\ell(f)|=\ell(|f|)&\quad\text{ for every }f\in L^\infty(\bar\mm),\\
\sup_{x\in\X}|\ell(f)(x)|=\|f\|_{L^\infty(\bar\mm)}&\quad\text{ for every }f\in L^\infty(\bar\mm).
\end{split}\]
The map \(\ell\) is called a \emph{von Neumann lifting} of \(\bar\mm\). The existence of von Neumann liftings requires
a rather strong form of the Axiom of Choice and the completeness of the measure \(\bar\mm\) under consideration. We use
the same symbol \(\ell\) to denote the map \(\ell\colon\bar\Sigma_\mm\to\bar\Sigma_\mm\) defined as follows: given any
\(E\in\bar\Sigma_\mm\), we denote by \(\ell(E)\in\bar\Sigma_\mm\) the unique subset of \(\X\) satisfying
\[
\1_{\ell(E)}=\ell(\1_E^{\bar\mm}).
\]

Now, let \((\Omega,\tau)\) be a topological space. We denote by \(\mathcal M(\Omega)\) the space of all finite signed Borel measures
on \((\Omega,\tau)\), while \(\mathcal M_+(\Omega)\coloneqq\{\mu\in\mathcal M(\Omega):\mu(B)\geq 0\text{ for every }B\in\mathscr B(\Omega)\}\),
where \(\mathscr B(\Omega)\) stands for the Borel \(\sigma\)-algebra of \((\Omega,\tau)\). For any \(\mu\in\mathcal M(\Omega)\), we denote
by \(|\mu|\in\mathcal M_+(\Omega)\) its total variation measure. Recall that \(\mathcal M(\Omega)\) is a Banach space with respect to the
total variation norm
\[
\|\mu\|_{\rm TV}\coloneqq|\mu|(\Omega)\quad\text{ for every }\mu\in\mathcal M(\Omega).
\]
Assuming in addition that \((\Omega,\tau)\) is Hausdorff, we denote by \(\mathfrak M(\Omega)\) the space of finite signed Radon measures
on \((\Omega,\tau)\). Accordingly, we set \(\mathfrak M_+(\Omega)\coloneqq\mathfrak M(\Omega)\cap\mathcal M_+(\Omega)\). It holds that
\[
(\mathfrak M(\Omega),\|\cdot\|_{\rm TV})\quad\text{ is a Banach space,}
\]
as \(\mathfrak M(\Omega)\) is a closed vector subspace of \(\mathcal M(\Omega)\). The \emph{Riesz--Markov--Kakutani theorem} states that if
\((\Omega,\tau)\) is a compact Hausdorff space, then \(\mathfrak M(\Omega)\) is isometrically isomorphic to the dual of the Banach space
\(C(\Omega)\), which consists of all real-valued continuous functions defined on \((\Omega,\tau)\) and is equipped with the supremum norm
\(\|f\|_{C(\Omega)}\coloneqq\sup_{p\in\Omega}|f(p)|\). The canonical duality pairing between \(\mathfrak M(\Omega)\) and \(C(\Omega)\) is
given by
\[
\mathfrak M(\Omega)\times C(\Omega)\ni(\mu,f)\mapsto\int f\,\d\mu\in\R.
\]
\subsection{Random normed modules}\label{s:RNM}
Throughout this section, we assume that
\[
(\X,\Sigma,\mm)\quad\text{ is a probability space.}
\]
Given that \(L^0(\mm)\) is a commutative ring, we can consider modules \({\rm M}\) over \(L^0(\mm)\).
We denote by \(f\cdot v\in{\rm M}\) the multiplication of \(f\in L^0(\mm)\) and \(v\in{\rm M}\).
Let us now recall the definition of random normed module with base \((\X,\Sigma,\mm)\):
\begin{definition}[Random normed module]\label{def:random_normed_module}
We say that \(({\rm M},|\cdot|)\) is a \textbf{random normed module} with base \((\X,\Sigma,\mm)\) if
\({\rm M}\) is a module over \(L^0(\mm)\) and \(|\cdot|\colon{\rm M}\to L^0_+(\mm)\) is a map such that
\[\begin{split}
|0|=0&,\\
|v|\neq 0&\quad\text{ for every }v\in{\rm M}\setminus\{0\},\\
|v+w|\leq|v|+|w|&\quad\text{ for every }v,w\in{\rm M},\\
|f\cdot v|=|f||v|&\quad\text{ for every }f\in L^0(\mm)\text{ and }v\in{\rm M}.
\end{split}\]
We call \(|\cdot|\) an \textbf{\(L^0\)-norm} on \({\rm M}\). Moreover, we define the distance \(\sfd_{\rm M}\) on \({\rm M}\) as
\begin{equation}\label{eq:def_d_M}
\sfd_{\rm M}(v,w)\coloneqq\int|v-w|\wedge 1\,\d\mm=\sfd_{L^0(\mm)}(|v-w|,0)\quad\text{ for every }v,w\in{\rm M}.
\end{equation}
If \(({\rm M},\sfd_{\rm M})\) is complete, then we say that \(({\rm M},|\cdot|)\) is a \textbf{complete random normed module.}
\end{definition}
\begin{remark}\label{remarkainounif}{\rm
Given a random normed module \(({\rm M},|\cdot|)\) with base \((\X,\Sigma,\mm)\), we have
that the sets \(\{\mathcal U(\varepsilon,\lambda):\varepsilon>0,\lambda\in(0,1)\}\), which are given by
\[
\mathcal U(\varepsilon,\lambda)\coloneqq\big\{(v,w)\in{\rm M}\times{\rm M}\;\big|\;\mm(\{|v-w|\geq\varepsilon\})<\lambda\big\},
\]
form a fundamental system of entourages of a uniformity on \({\rm M}\), which we denote by \(\Phi_{\rm M}\).
The topology on \({\rm M}\) that is induced by \(\Phi_{\rm M}\) is typically called the \textbf{\((\varepsilon,\lambda)\)-topology}
and denoted by \(\mathcal T_{\varepsilon,\lambda}\). The distance \(\sfd_{\rm M}\) defined in \eqref{eq:def_d_M}
induces the topology \(\mathcal T_{\varepsilon,\lambda}\), the completeness of the uniform space \(({\rm M},\Phi_{\rm M})\) is equivalent
to the completeness of \(({\rm M},\sfd_{\rm M})\), and the uniformity \(\Phi_{\rm M}\) is generated by
\[
\Big\{\big\{(v,w)\in{\rm M}\times{\rm M}\;\big|\;\sfd_{\rm M}(v,w)<r\big\}\;\Big|\;r>0\Big\}.
\]
Consequently, Definition \ref{def:random_normed_module} is consistent both with the concept of
complete random normed module introduced by Guo in \cite{guo1989theory,guo1992} and with Gigli's
notion of \(L^0(\mm)\)-Banach \(L^0(\mm)\)-module \cite{gigli2018nonsmooth,gigli2017} (therein called
just `\(L^0(\mm)\)-normed \(L^0(\mm)\)-module').
\fr}\end{remark}

The space \(L^0(\mm)\) itself is an example of complete random normed module. Given two random normed modules
\(({\rm M},|\cdot|)\), \(({\rm N},|\cdot|)\) with base \((\X,\Sigma,\mm)\) and a map \(\varphi\colon{\rm M}\to{\rm N}\),
the following conditions are known to be equivalent:
\begin{itemize}
\item[\(\rm i)\)] \(\varphi\) is \(L^0(\mm)\)-linear (i.e.\ a homomorphism of \(L^0(\mm)\)-modules), and it is continuous
as a map from \(({\rm M},\sfd_{\rm M})\) to \(({\rm N},\sfd_{\rm N})\).
\item[\(\rm ii)\)] \(\varphi\) is linear and there exists \(g\in L^0_+(\mm)\) such that
\[
|\varphi(v)|\leq g|v|\quad\text{ for every }v\in{\rm M}.
\]
\end{itemize}
Each map \(\varphi\colon{\rm M}\to{\rm N}\) satisfying the above conditions is called a
\textbf{homomorphism of random normed modules}. We define \(|\varphi|\in L^0_+(\mm)\) as
\begin{equation}\label{eq:ptwse_norm_hom}
|\varphi|\coloneqq\bigwedge\big\{g\in L^0_+(\mm)\;\big|\;|\varphi(v)|\leq g|v|\text{ for every }v\in{\rm M}\big\}.
\end{equation}
We say that \(\varphi\) is an \textbf{isometric isomorphism
of random normed modules} if it is an isomorphism of \(L^0(\mm)\)-modules (i.e.\ an \(L^0(\mm)\)-linear bijection)
with \(|\varphi(v)|=|v|\) for all \(v\in{\rm M}\).
\begin{proposition}\label{prop:ext_result}
Let \(({\rm M},|\cdot|)\) and \(({\rm N},|\cdot|)\) be complete random normed modules with base \((\X,\Sigma,\mm)\).
Let \({\rm V}\) be a vector subspace of \({\rm M}\) that \textbf{generates} \({\rm M}\) in the sense of random normed modules, i.e.
\begin{equation}\label{eq:def_generates}
{\rm cl}_{\rm M}\bigg(\bigg\{\sum_{i=1}^n f_i\cdot v_i\;\bigg|\;n\in\N,\,(f_i)_{i=1}^n\subseteq L^0(\mm),
\,(v_i)_{i=1}^n\subseteq{\rm V}\bigg\}\bigg)={\rm M}.
\end{equation}
Let \(\varphi\colon{\rm V}\to{\rm N}\) be a linear operator. Assume that for some \(g\in L^0_+(\mm)\) it holds that
\[
|\varphi(v)|\leq g|v|\quad\text{ for every }v\in{\rm V}.
\]
Then there is a unique homomorphism of random normed modules \(\bar\varphi\colon{\rm M}\to{\rm N}\)
that extends \(\varphi\). Moreover, \(|\bar\varphi|\leq g\). If in addition \(|\varphi(v)|=g|v|\) for
all \(v\in{\rm V}\), then \(|\bar\varphi(v)|=g|v|\) for all \(v\in{\rm M}\).
\end{proposition}

Using the fact that \({\rm V}\) is a vector space, it is easy to check that \eqref{eq:def_generates} is equivalent to
\[
{\rm cl}_{\rm M}\bigg(\bigg\{\sum_{i=1}^n\1_{E_i}^\mm\cdot v_i\;\bigg|\;n\in\N,\,(E_i)_{i=1}^n\subseteq\Sigma,
\,(v_i)_{i=1}^n\subseteq{\rm V}\bigg\}\bigg)={\rm M}.
\]
The \textbf{random conjugate space} \({\rm M}^*\) of \({\rm M}\)
is defined as the space of all homomorphisms of random normed modules from \({\rm M}\) to \(L^0(\mm)\).
Then \({\rm M}^*\) is a complete random normed module with base \((\X,\Sigma,\mm)\) if endowed with the
natural pointwise operations and the \(L^0\)-norm defined in \eqref{eq:ptwse_norm_hom}.
\begin{theorem}[\(L^0\)-completion]\label{thm:L0-complet}
Let \(({\rm M},|\cdot|)\) be a random normed module with base \((\X,\Sigma,\mm)\). Then there exist a
complete random normed module \((\bar{\rm M},|\cdot|)\), which we call the \textbf{\(L^0\)-completion}
or \textbf{random-normed-module completion} of \({\rm M}\), and a homomorphism of random normed modules
\(\iota\colon{\rm M}\to\bar{\rm M}\) such that \(|\iota(v)|=|v|\) for every \(v\in{\rm M}\) and \(\iota({\rm M})\)
is dense in \(\bar{\rm M}\).

Moreover, the pair \((\bar{\rm M},\iota)\) is unique up to a unique isomorphism, in the following sense: given any
pair \(({\rm N},j)\) having the same properties as \((\bar{\rm M},\iota)\), there exists a unique isometric isomorphism
of random normed modules \(\Phi\colon\bar{\rm M}\to{\rm N}\) such that \(j=\Phi\circ\iota\).
\end{theorem}

In the sequel, we implicitly assume that \({\rm M}\) is an \(L^0(\mm)\)-submodule of its \(L^0\)-completion
\(\bar{\rm M}\) and that the map \(\iota\colon{\rm M}\to\bar{\rm M}\) is the inclusion.
\begin{corollary}\label{cor:ext_hom_mod}
Let \(({\rm M},|\cdot|)\) be a random normed module and \(({\rm N},|\cdot|)\) a complete random normed module.
Let \(\varphi\colon{\rm M}\to{\rm N}\) be a homomorphism of random normed modules. Then there exists a unique homomorphism
of random normed modules \(\bar\varphi\colon\bar{\rm M}\to{\rm N}\) such that \(\bar\varphi|_{\rm M}=\varphi\). Moreover, it holds that
\[
|\bar\varphi|=|\varphi|.
\]
\end{corollary}
\begin{proof}
The space \({\rm M}\) is a dense vector subspace of \(\bar{\rm M}\), thus a fortiori it generates \(\bar{\rm M}\) in
the sense of normed modules. Moreover, the inclusion map \(\iota\colon{\rm M}\to\bar{\rm M}\) is a linear operator
satisfying \(|\iota(v)|=|v|\) for every \(v\in{\rm M}\). Therefore, the statement follows from Proposition \ref{prop:ext_result}.
\end{proof}

Next, we recall the concept of a summable (countable) family of elements of a random normed module,
following \cite[Definition 3.5]{Pas23}. We adopt the notation \(\mathcal P_f(S)\) to denote
the collection of all finite subsets of a given set \(S\).
\begin{definition}[Summable sequence]\label{defconvserie}
Let \(({\rm M},|\cdot|)\) be a random normed module with base \((\X,\Sigma,\mm)\). Then we say that a countable family
\(\{v_n:n\in\N\}\subseteq{\rm M}\) is \textbf{summable}, with \textbf{sum} \(v\in{\rm M}\), provided
\begin{equation}\label{eq:sum_in_Ban_mod}
\bigvee_{F\in\mathcal P_f(\{n+1,n+2,\ldots\})}\bigg|v-\sum_{i\in\{1,\ldots,n\}\cup F}v_i\bigg|\to 0\quad\text{ in }L^0(\mm)\text{ as }n\to\infty.
\end{equation}
The sum \(v\) of \(\{v_n:n\in\N\}\), which is uniquely determined by \eqref{eq:sum_in_Ban_mod}, will be denoted by \(\sum_{n\in\N}v_n\).
\end{definition}
\begin{remark}\label{rmk:fact_about_summability}{\rm
Some observations regarding summable sequences are in order:
\begin{itemize}
\item[\(\rm i)\)] If \(v_n=0\) for all but finitely many \(n\in\N\), then \(\{v_n:n\in\N\}\) is summable and \(\sum_{n\in\N}v_n\) coincides with the algebraic sum
in \({\rm M}\) of the non-zero elements of \(\{v_n:n\in\N\}\).
\item[\(\rm ii)\)] If \(\{v_n:n\in\N\}\) is summable, then for any permutation \(\sigma\colon\N\to\N\) of \(\N\) we have that
\[
\sum_{n=1}^N v_{\sigma(n)}\to\sum_{n\in\N}v_n\quad\text{ in }{\rm M}\text{ as }N\to\infty.
\]
In particular, \(\{v_{\sigma(n)}:n\in\N\}\) is summable and \(\sum_{n\in\N}v_{\sigma(n)}=\sum_{n\in\N}v_n\).
\item[\(\rm iii)\)] Recall that \((\mathcal P_f(\N),\subseteq)\) is a directed set. Define \(s_F\coloneqq\sum_{n\in F}v_n\in{\rm M}\) for all \(F\in\mathcal P_f(\N)\).
Then it can be readily checked that \(\{v_n:n\in\N\}\subseteq{\rm M}\) is summable with sum \(v\in{\rm M}\) if and only if the net \((s_F)_{F\in\mathcal P_f(\N)}\)
converges to \(v\) in \({\rm M}\).
\item[\(\rm iv)\)] As it was observed in \cite[Section 3.2]{Pas23}, a sequence \(\{f_n:n\in\N\}\subseteq L^0_+(\mm)\) is summable in \(L^0(\mm)\)
if and only if \(\big\{\sum_{i=1}^n f_i:n\in\N\big\}\) is order bounded in \(L^0(\mm)\), and in this case
\[
\sum_{n\in\N}f_n=\bigvee_{n\in\N}\sum_{i=1}^n f_i\in L^0_+(\mm).
\]
\item[\(\rm v)\)] \textsc{Cauchy summability criterion \cite[Proposition 3.6]{Pas23}.} If \({\rm M}\) is complete,
then a given sequence \(\{v_n:n\in\N\}\subseteq{\rm M}\) is summable if and only if 
\[
\bigvee_{F\in\mathcal P_f(\{n,n+1,\ldots\})}\bigg|\sum_{n\in F}v_n\bigg|\to 0\quad\text{ in }L^0(\mm)\text{ as }n\to\infty.
\]
\item[\(\rm vi)\)] If \({\rm M}\) is a complete random normed module and a sequence \(\{v_n:n\in\N\}\subseteq{\rm M}\) is chosen
so that \(\{|v_n|:n\in\N\}\subseteq L^0(\mm)\) is summable, then \(\{v_n:n\in\N\}\) is summable in \({\rm M}\) and
\[
\bigg|\sum_{n\in\N}v_n\bigg|\leq\sum_{n\in\N}|v_n|.
\]
\end{itemize}
The validity of vi) follows from the elementary inequality
\[
\bigvee_{F\in\mathcal P_f(\{n,n+1,\ldots\})}\bigg|\sum_{n\in F}v_n\bigg|\leq\bigvee_{F\in\mathcal P_f(\{n,n+1,\ldots\})}\sum_{n\in F}|v_n|,
\]
the Cauchy summability criterion v) and \cite[Remark 3.7]{Pas23}.
\fr}\end{remark}
\section{\texorpdfstring{\(L^0\)}{L0}-valued measures}\label{s:L0-val_meas}
We assume throughout this section that \((\X,\Sigma,\mm)\) is a probability space. Moreover, by \((\Omega,\mathcal A)\) we mean an arbitrary measurable space,
until further notice.
\subsection{Definitions and basic properties}\label{s:L0-val_meas_basic}
Let us introduce the following concept:
\begin{definition}[\(L^0\)-valued measure]\label{def:L0_meas}
We say that \(\mu\colon\mathcal A\to L^0(\mm)\) is an \textbf{\(L^0\)-valued measure} provided:
\begin{itemize}
\item[\(\rm i)\)] \(\mu(\varnothing)=0\).
\item[\(\rm ii)\)] If \((A_n)_{n\in\N}\subseteq\mathcal A\) are pairwise disjoint, then \(\{|\mu(A_n)|:n\in\N\}\)
is summable and
\[
\mu\bigg(\bigcup_{n\in\N}A_n\bigg)=\sum_{n\in\N}\mu(A_n).
\]
\end{itemize}
If in addition \(\mu(A)\geq 0\) for every \(A\in\mathcal A\), then we say that \(\mu\) is \textbf{non-negative}.
\end{definition}

In the above definition, the summability of the sequence \(\{\mu(A_n):n\in\N\}\subseteq L^0(\mm)\) is guaranteed by Remark \ref{rmk:fact_about_summability} vi).
\begin{remark}\label{rmk:lambda^c}{\rm
If \(\mu(A)\) is (\(\mm\)-a.e.\ equal to) some constant function \(\lambda_\mu(A)\1_\X\) for every \(A\in\mathcal A\),
then it is clear that \(\lambda_\mu\colon\mathcal A\to\R\) is a finite signed measure on \((\Omega,\mathcal A)\).
Conversely, if \(\lambda\colon\mathcal A\to\R\) is a finite signed measure, then by defining
\[
\lambda^c(A)\coloneqq\lambda(A)\1_\X^\mm\in L^0(\mm)\quad\text{ for every }A\in\mathcal A
\]
we obtain an \(L^0\)-valued measure \(\lambda^c\colon\mathcal A\to L^0(\mm)\). This clarifies
how the class of \(L^0\)-valued measures generalises the one of finite signed measures.
\fr}\end{remark}
\begin{definition}[\(L^0\)-valued measure of bounded variation]\label{def:L0_meas_bv}
Let \(\mu\colon\mathcal A\to L^0(\mm)\) be an \(L^0\)-valued measure. For any set \(A\in\mathcal A\),
we define the \textbf{\(L^0\)-total variation} of \(\mu\) on \(A\) as
\begin{equation}\label{eq:def_|mu|}
|\mu|(A)\coloneqq\bigvee\bigg\{\sum_{n=1}^\infty|\mu(A_n)|\;\bigg|\;(A_n)_{n\in\N}\subseteq\mathcal A\text{ partition of }A\bigg\}\in\bar L^0_+(\mm).
\end{equation}
Then we say that the \(L^0\)-valued measure \(\mu\) is \textbf{of bounded variation} provided it satisfies
\begin{equation}\label{eq:def_|mu|_TV}
|\mu|_{\rm TV}\coloneqq|\mu|(\Omega)\in L^0_+(\mm).
\end{equation}
We denote by \(\mathcal M(\Omega;L^0(\mm))\) the space of all \(L^0\)-valued measures \(\mu\colon\mathcal A\to L^0(\mm)\)
of bounded variation and by \(\mathcal M_+(\Omega;L^0(\mm))\) the space of all those \(\mu\in\mathcal M(\Omega;L^0(\mm))\)
that are non-negative.
\end{definition}
\begin{remark}\label{rmk:cont_above_L0_meas}{\rm
Several properties of real-valued measures can be generalised to \(L^0\)-valued measures. For example, given \(\mu\in\mathcal M_+(\Omega;L^0(\mm))\)
and \((A_n)_{n\in\N},(B_n)_{n\in\N}\subseteq\mathcal A\), the following properties hold:
\begin{itemize}
\item[\(\rm i)\)] If \(A_n\subseteq A_{n+1}\) for every \(n\in\N\), then \(\mu(A_n)\to\mu\big(\bigcup_{m\in\N}A_m\big)\) in \(L^0(\mm)\) as \(n\to\infty\).
\item[\(\rm ii)\)] If \(B_{n+1}\subseteq B_n\) for every \(n\in\N\), then \(\mu(B_n)\to\mu\big(\bigcap_{m\in\N}B_m\big)\) in \(L^0(\mm)\) as \(n\to\infty\).
\end{itemize}
To prove i), note that \(A_1,A_2\setminus A_1,A_3\setminus A_2,\ldots\) are pairwise disjoint sets with union \(\bigcup_{n\in\N}A_n\), thus
\[
\mu(A_n)=\mu(A_1)+\sum_{i=2}^n\mu(A_n\setminus A_{n-1})\to\mu\bigg(\bigcup_{m\in\N}A_m\bigg)\quad\text{ in }L^0(\mm)\text{ as }n\to\infty.
\]
To prove ii), note that \(\Omega\setminus B_n\subseteq\Omega\setminus B_{n+1}\) for every \(n\in\N\), thus i) yields
\[
\mu(B_n)=\mu(\Omega)-\mu(\Omega\setminus B_n)\to\mu(\Omega)-\mu\bigg(\bigcup_{m\in\N}\Omega\setminus B_m\bigg)=\mu\bigg(\bigcap_{m\in\N}B_m\bigg)
\]
in \(L^0(\mm)\) as \(n\to\infty\). Consequently, the claim is proved.
\fr}\end{remark}
It is easy to check that \(\mathcal M(\Omega;L^0(\mm))\) is an \(L^0(\mm)\)-module with respect to these operations:
\[\begin{split}
(\mu+\nu)(A)\coloneqq\mu(A)+\nu(A)&\quad\text{ for every }\mu,\nu\in\mathcal M(\Omega;L^0(\mm))\text{ and }A\in\mathcal A,\\
(f\cdot\mu)(A)\coloneqq f\cdot\mu(A)&\quad\text{ for every }f\in L^0(\mm),\,\mu\in\mathcal M(\Omega;L^0(\mm))\text{ and }A\in\mathcal A.
\end{split}\]
Furthermore, it is straightforward to show that
\[
|\mu|\in\mathcal M_+(\Omega;L^0(\mm))\quad\text{ for every }\mu\in\mathcal M(\Omega;L^0(\mm)),
\]
where \(|\mu|\colon\mathcal A\to L^0_+(\mm)\) is given by \eqref{eq:def_|mu|}.
\begin{lemma}\label{lem:L0-val_meas_RNM}
The space \((\mathcal M(\Omega;L^0(\mm)),|\cdot|_{\rm TV})\) is a complete random normed module.
\end{lemma}
\begin{proof}
We check only completeness. Fix a Cauchy sequence \((\mu_n)_{n\in\N}\) in \((\mathcal M(\Omega;L^0(\mm)),|\cdot|_{\rm TV})\).
For any \(A\in\mathcal A\) we have that
\[
\sfd_{L^0(\mm)}(\mu_n(A),\mu_m(A))\leq\sfd_{L^0(\mm)}(|\mu_n-\mu_m|(A),0)\leq\sfd_{L^0(\mm)}(|\mu_n-\mu_m|_{\rm TV},0)\to 0
\]
as \(n,m\to\infty\), which shows that \((\mu_n(A))_{n\in\N}\subseteq L^0(\mm)\) is Cauchy. Since \(L^0(\mm)\) is complete, we can then define
\[
\mu(A)\coloneqq\lim_{n\to\infty}\mu_n(A)\in L^0(\mm)\quad\text{ for every }A\in\mathcal A.
\]
We claim that \(\mu\in\mathcal M(\Omega;L^0(\mm))\) and \(|\mu_n-\mu|_{\rm TV}\to 0\) in \(L^0(\mm)\). Trivially,
\(\mu\) satisfies Definition \ref{def:L0_meas} i). To prove that \(\mu\) satisfies Definition \ref{def:L0_meas} ii),
we first note that for any \(A\in\mathcal A\) it holds
\[
\sfd_{L^0(\mm)}(|\mu_n|(A),|\mu_m|(A))\leq\sfd_{L^0(\mm)}(|\mu_n-\mu_m|_{\rm TV},0)\quad\text{ for every }n,m\in\N,
\]
thus the limit \(g(A)\coloneqq\lim_n|\mu_n|(A)\in L^0_+(\mm)\) exists. Observe that \(g(A_1)+\ldots+g(A_k)\leq g(\Omega)\)
whenever \(A_1,\ldots,A_k\in\mathcal A\) are pairwise disjoint sets. If \((A_j)_{j\in\N}\subseteq\mathcal A\) are pairwise
disjoint, then
\[\begin{split}
\sum_{j\in\N}|\mu(A_j)|&\leq\sum_{j\in\N}\lim_{n\to\infty}|\mu_n(A_j)|
=\lim_{N\to\infty}\sum_{j=1}^N\lim_{n\to\infty}|\mu_n(A_j)|
\leq\lim_{N\to\infty}\sum_{j=1}^N\lim_{n\to\infty}|\mu_n|(A_j)\\
&=\lim_{N\to\infty}\sum_{j=1}^N g(A_j)\leq g(\Omega),
\end{split}\]
which implies that \(\{|\mu|(A_j):j\in\N\}\subseteq L^0(\mm)\) is summable. Also, the element
\(\sum_{j\in\N}\mu(A_j)\), whose existence is guaranteed by Remark \ref{rmk:fact_about_summability} vi), satisfies
\[
\bigg|\sum_{j\in\N}\mu_n(A_j)-\sum_{j\in\N}\mu(A_j)\bigg|\leq\lim_{N\to\infty}\sum_{j=1}^N|\mu_n-\mu|(A_j)
\leq|\mu_n-\mu|_{\rm TV}\to 0\quad\text{ in }L^0(\mm)
\]
as \(n\to\infty\), which implies that
\[
\mu\bigg(\bigcup_{j\in\N}A_j\bigg)=\lim_{n\to\infty}\mu_n\bigg(\bigcup_{j\in\N}A_j\bigg)
=\lim_{n\to\infty}\sum_{j\in\N}\mu_n(A_j)=\sum_{j\in\N}\mu(A_j)\quad\text{ in }L^0(\mm).
\]
Hence, we have shown that \(\mu\) is an \(L^0\)-valued measure. Next, we observe that for any given partition
\((A_j)_{j\in\N}\subseteq\mathcal A\) of \(\Omega\) we can estimate
\[
\sum_{j=1}^\infty|\mu(A_j)|=\lim_{N\to\infty}\lim_{n\to\infty}\sum_{j=1}^N|\mu_n(A_j)|
\leq\lim_{N\to\infty}\lim_{n\to\infty}|\mu_n|_{\rm TV}=g(\Omega)\in L^0_+(\mm),
\]
which means that \(\mu\) is of bounded variation and thus \(\mu\in\mathcal M(\Omega;L^0(\mm))\).
Finally, given \(n\in\N\) and a partition \((A_j)_{j\in\N}\subseteq\mathcal A\) of \(\Omega\), we have that
\[\begin{split}
\sum_{j\in\N}|\mu_n(A_j)-\mu(A_j)|&=\sup_{N\in\N}\sum_{j=1}^N|\mu_n(A_j)-\mu(A_j)|
=\sup_{N\in\N}\lim_{m\to\infty}\sum_{j=1}^N|\mu_n(A_j)-\mu_m(A_j)|\\
&\leq\varliminf_{m\to\infty}|\mu_n-\mu_m|_{\rm TV},
\end{split}\]
which gives \(|\mu_n-\mu|_{\rm TV}\leq\varliminf_m|\mu_n-\mu_m|_{\rm TV}\). Therefore,
an application of Fatou's lemma yields
\[\begin{split}
\varlimsup_{n\to\infty}\sfd_{L^0(\mm)}(|\mu_n-\mu|_{\rm TV},0)&=\varlimsup_{n\to\infty}\int|\mu_n-\mu|_{\rm TV}\wedge 1\,\d\mm
\leq\varlimsup_{n\to\infty}\int\varliminf_{m\to\infty}|\mu_n-\mu_m|_{\rm TV}\wedge 1\,\d\mm\\
&\leq\varlimsup_{n\to\infty}\varliminf_{m\to\infty}\sfd_{L^0(\mm)}(|\mu_n-\mu_m|_{\rm TV},0)=0,
\end{split}\]
proving that \(|\mu_n-\mu|_{\rm TV}\to 0\) in \(L^0(\mm)\). This completes the proof.
\end{proof}

It is clear that every non-negative \(L^0\)-valued measure \(\mu\colon\mathcal A\to L^0(\mm)\) is of bounded variation
and \(|\mu|=\mu\). Given \(\mu\in\mathcal M(\Omega;L^0(\mm))\) and \(A\in\mathcal A\), we define the restriction
\(\mu\mrestr A\) of \(\mu\) to \(A\) as
\begin{equation}\label{eq:def_restr_L0-meas}
(\mu\mrestr A)(B)\coloneqq\mu(A\cap B)\quad\text{ for every }B\in\mathcal A.
\end{equation}
It is easy to check that \(\mu\mrestr A\in\mathcal M(\Omega;L^0(\mm))\) and \(|\mu\mrestr A|=|\mu|\mrestr A\).
\begin{definition}[Pushforward of an \(L^0\)-valued measure]\label{def:pushforward_meas}
Let \((\tilde\Omega,\tilde{\mathcal A})\) be a measurable space and \(\varphi\colon\Omega\to\tilde\Omega\)
a measurable map. Then we define the \textbf{pushforward} operator
\[
\varphi_\#\colon\mathcal M(\Omega;L^0(\mm))\to\mathcal M(\tilde\Omega;L^0(\mm))
\]
as follows: given any \(\mu\in\mathcal M(\Omega;L^0(\mm))\), we define \(\varphi_\#\mu\in\mathcal M(\tilde\Omega;L^0(\mm))\) as
\[
\varphi_\#\mu(\tilde A)\coloneqq\mu(\varphi^{-1}(\tilde A))\in L^0(\mm)\quad\text{ for every }\tilde A\in\tilde{\mathcal A}.
\]
\end{definition}

Let us check that \(\varphi_\#\) is well defined. Since \((\varphi^{-1}(\tilde A_n))_{n\in\N}\subseteq\mathcal A\)
are pairwise disjoint whenever \((\tilde A_n)_{n\in\N}\subseteq\tilde{\mathcal A}\) are pairwise disjoint, it is easy to
show that \(\varphi_\#\mu\) is an \(L^0\)-valued measure. Moreover, given any \(\tilde A\in\tilde{\mathcal A}\) and a partition
\((\tilde A_n)_{n\in\N}\subseteq\tilde{\mathcal A}\) of \(\tilde A\), we can estimate
\[
\sum_{n=1}^\infty|\varphi_\#\mu(\tilde A_n)|=\sum_{n=1}^\infty|\mu(\varphi^{-1}(\tilde A_n))|\leq|\mu|(\varphi^{-1}(\tilde A))
=\varphi_\#|\mu|(\tilde A),
\]
whence -- by taking the supremum over \((A_n)_{n\in\N}\) -- it follows that \(|\varphi_\#\mu|(\tilde A)\leq\varphi_\#|\mu|(\tilde A)\).
Moreover, we have \(\varphi_\#|\mu|(\tilde A)=\varphi_\#\mu^+(\tilde A)+\varphi_\#\mu^-(\tilde A)\leq 2|\varphi_\#\mu|(\tilde A)\).
This proves that
\begin{equation}\label{eq:ineq_pushforward_as_meas}
|\varphi_\#\mu|\leq\varphi_\#|\mu|\leq 2|\varphi_\#\mu|\quad\text{ for every }\mu\in\mathcal M(\Omega;L^0(\mm)),
\end{equation}
thus in particular \(|\varphi_\#\mu|(\tilde\Omega)\leq\varphi_\#|\mu|(\tilde\Omega)=|\mu|(\varphi^{-1}(\tilde\Omega))=|\mu|(\Omega)\),
so that
\begin{equation}\label{eq:ineq_pushforward}
|\varphi_\#\mu|_{\rm TV}\leq|\mu|_{\rm TV}\quad\text{ for every }\mu\in\mathcal M(\Omega;L^0(\mm)),
\end{equation}
which implies that \(\varphi_\#\mu\in\mathcal M(\tilde\Omega;L^0(\mm))\) for every \(\mu\in\mathcal M(\Omega;L^0(\mm))\).
It is also easy to check that \(\varphi_\#\colon\mathcal M(\Omega;L^0(\mm))\to\mathcal M(\tilde\Omega;L^0(\mm))\)
is a homomorphism of random normed modules.
\medskip

A distinguished class of \(L^0\)-valued measures is the following one:
\begin{definition}[Probability \(L^0\)-valued measure]\label{def:prob_L0-meas}
We say that \(\mu\in\mathcal M_+(\Omega;L^0(\mm))\) is a \textbf{probability \(L^0\)-valued measure} if
\[
\mu(\Omega)=\1_\X^\mm.
\]
We denote by \(\mathcal P(\Omega;L^0(\mm))\) the space of all probability \(L^0\)-valued measures \(\mu\colon\mathcal A\to L^0_+(\mm)\).
\end{definition}

A probability \(L^0\)-valued measure \(\mu\) is canonically associated with two probability measures
\([\mu]\) and \(\hat\mu\), defined on \((\Omega,\mathcal A)\) and \((\Omega\times\X,\mathcal A\otimes\Sigma)\), respectively. The former is given by
\begin{equation}\label{eq:def_[mu]}
[\mu](A)\coloneqq\int\mu(A)\,\d\mm\quad\text{ for every }A\in\mathcal A.
\end{equation}
The fact that \([\mu]\) is countably additive can be checked by using the monotone convergence theorem. Since \([\mu](\varnothing)=\int\mu(\varnothing)\,\d\mm=0\)
and \([\mu](\Omega)=\int\mu(\Omega)\,\d\mm=\mm(\X)=1\), we have that \([\mu]\in\mathcal P(\Omega)\). As for the latter measure \(\hat\mu\), its existence
is guaranteed by the following result:
\begin{proposition}\label{prop:def_hat_mu}
Let \(\mu\in\mathcal P(\Omega;L^0(\mm))\) be given. Then there exists a unique probability measure \(\hat\mu\colon\mathcal A\otimes\Sigma\to[0,1]\) such that
\begin{equation}\label{eq:def_hat_mu}
\hat\mu(A\times E)=\int_E\mu(A)\,\d\mm\quad\text{ for every }A\in\mathcal A\text{ and }E\in\Sigma.
\end{equation}
\end{proposition}
\begin{proof}
By a measurable rectangle in \(\Omega\times\X\) we mean a set of the form \(A\times E\), with \(A\in\mathcal A\) and \(E\in\Sigma\). We denote by \(\mathcal R\) the algebra of all finite unions of pairwise-disjoint measurable rectangles in \(\Omega\times\X\). Since \(\mathcal A\otimes\Sigma\) is the \(\sigma\)-algebra generated by measurable
rectangles, we know a fortiori that \(\mathcal A\otimes\Sigma\) is the \(\sigma\)-algebra generated by \(\mathcal R\). Now, we define \(\tilde\mu\colon\mathcal R\to[0,1]\) as
\[
\tilde\mu\bigg(\bigsqcup_{i=1}^n A_i\times E_i\bigg)\coloneqq\sum_{i=1}^n\int_{E_i}\mu(A_i)\,\d\mm\quad\text{ for every }\bigsqcup_{i=1}^n A_i\times E_i\in\mathcal R.
\]
It is easy to check that \(\tilde\mu\) is well posed (meaning that it does not depend on the specific way of expressing an element of \(\mathcal R\)) and
\(\tilde\mu\big(\bigcup_{n\in\N}U_n\big)=\sum_{n=1}^\infty\tilde\mu(U_n)\) holds whenever \((U_n)_{n\in\N}\subseteq\mathcal R\) are pairwise-disjoint sets satisfying
\(\bigcup_{n\in\N}U_n\in\mathcal R\). By applying the Hahn--Kolmogorov theorem, we thus conclude that there exists a unique measure
\(\hat\mu\colon\mathcal A\otimes\Sigma\to[0,1]\) extending \(\tilde\mu\). Note also that \(\hat\mu\) is the unique measure on \((\Omega\times\X,\mathcal A\otimes\Sigma)\)
satisfying \eqref{eq:def_hat_mu}. Moreover, we have
\[
\hat\mu(\Omega\times\X)=\int_\X\mu(\Omega)\,\d\mm=\mm(\X)=1,
\]
so that \(\hat\mu\in\mathcal P(\Omega\times\X)\). This completes the proof.
\end{proof}

It is straighforward to check that if \(\lambda\colon\mathcal A\to[0,1]\) is a given probability measure,
then its associated probability \(L^0\)-valued measure \(\lambda^c\colon\mathcal A\to L^0_+(\mm)\)
as in Remark \ref{rmk:lambda^c} satisfies
\begin{equation}\label{eq:hat_lambda^c}
[\lambda^c]=\lambda,\qquad\hat{\lambda^c}=\lambda\otimes\mm.
\end{equation}
\begin{lemma}\label{lem:rect_dense_L1_hat_mu}
Fix any \(\mu\in\mathcal P(\Omega;L^0(\mm))\). Then the \(\R\)-linear span of
\(\{\1_{A\times E}^{\hat\mu}:A\in\mathcal A,E\in\Sigma\}\) is dense in \(L^1(\hat\mu)\).
\end{lemma}
\begin{proof}
Let us denote \({\rm R}\coloneqq\{A\times E:A\in\mathcal A,E\in\Sigma\}\) and
\[
{\rm S}\coloneqq\bigg\{\sum_{i=1}^n\lambda_i\1_{R_i}^{\hat\mu}\;\bigg|\;n\in\N,\,(\lambda_i)_{i=1}^n\subseteq\R,\,
(R_i)_{i=1}^n\subseteq{\rm R}\bigg\}.
\]
Thanks to a truncation argument, it suffices to check that for any \(F\in L^\infty(\hat\mu)\) there exists a
sequence \((F_k)_{k\in\N}\subseteq{\rm S}\) such that \(F_k\to F\) in \(L^1(\hat\mu)\). To this aim, we define
\[
{\rm H}\coloneqq\big\{\bar F\in\mathcal L^\infty(\mathcal A\otimes\Sigma)\;\big|\;\text{there exists }
(F_k)_{k\in\N}\subseteq{\rm S}\text{ such that }F_k\to[\bar F]_{\hat\mu}\text{ in }L^1(\hat\mu)\big\}.
\]
Clearly, we have that \(\1_R\in{\rm H}\) for every \(R\in{\rm R}\) and \({\rm H}\) is a vector subspace
of \(\mathcal L^\infty(\mathcal A\otimes\Sigma)\). Moreover, if \((\bar F_j)_{j\in\N}\subseteq{\rm H}\) is a
sequence of non-negative functions such that \(\bar F_j\nearrow\bar F\) for some bounded function \(\bar F\colon\Omega\times\X\to[0,+\infty)\),
then \(\bar F\in\mathcal L^\infty(\mathcal A\otimes\Sigma)\) and \([\bar F_j]_{\hat\mu}\to[\bar F]_{\hat\mu}\)
in \(L^1(\hat\mu)\) by the monotone convergence theorem; since each \([\bar F_j]_{\hat\mu}\) can be approximated
in \(L^1(\hat\mu)\) by elements of \({\rm S}\), by a diagonalisation argument we conclude that \(\bar F\in{\rm H}\).
Note also that \({\rm R}\) is a \(\pi\)-system (i.e.\ a non-empty family of sets that is closed under finite unions)
containing \(\Omega\times\X\). The monotone class theorem ensures that
\({\rm H}=\mathcal L^\infty(\mathcal A\otimes\Sigma)\), whence the statement follows.
\end{proof}

The relation between \([\mu]\) and \(\hat\mu\) is illustrated by the following identity:
\begin{equation}\label{eq:relation_hat_mu_[mu]}
(\pi_\Omega)_\#\hat\mu=[\mu],
\end{equation}
where \(\pi_\Omega\colon\Omega\times\X\to\Omega\) denotes the projection map \(\pi_\Omega(p,x)\coloneqq p\).
Indeed, we have
\[
(\pi_\Omega)_\#\hat\mu(A)=\hat\mu(\pi_\Omega^{-1}(A))=\hat\mu(A\times\X)=\int\mu(A)\,\d\mm=[\mu](A)
\quad\text{ for every }A\in\mathcal A,
\]
thus proving \eqref{eq:relation_hat_mu_[mu]}.
\medskip

More generally, for any non-negative \(L^0\)-valued measure \(\mu\in\mathcal M_+(\Omega;L^0(\mm))\), one can define
\begin{equation}\label{def[mu](A)}
   [\mu](A)\coloneqq\sum_{\substack{k\in\N:\\\mu(P_k)\neq 0}}\frac{\int\mu(A\cap P_k)\,\d\mm}{2^k\int\mu(P_k)\,\d\mm}\quad\text{ for every }A\in\mathcal A, 
\end{equation}
where the partition \((P_k)_{k\in\N}\subseteq\Sigma\) is chosen so that \(k-1\leq\mu(\Omega)\leq k\) holds \(\mm\)-a.e.\ on \(P_k\) for every \(k\in\N\).
Note that \([\mu]\in\mathcal P(\Omega)\). Moreover, one can define \(\hat\mu\in\mathcal P(\Omega\times\X)\) as the unique measure such that
\[
\hat\mu(A\times E)=\sum_{\substack{k\in\N:\\ \mu(P_k)\neq 0}}\frac{\int_E\mu(A\cap P_k)\,\d\mm}{2^k\int\mu(P_k)\,\d\mm}\quad\text{ for every }A\in\mathcal A\text{ and }E\in\Sigma.
\]
These definitions are consistent with the above ones for \(\mu\in\mathcal P(\Omega;L^0(\mm))\), and \((\pi_\Omega)_\#\hat\mu=[\mu]\).
\begin{definition}[\(\mu\)-null set]\label{def:mu-null}
Fix \(\mu\in\mathcal M(\Omega;L^0(\mm))\). Then we say that \(N\in\mathcal A\) is \textbf{\(\mu\)-null} if
\[
|\mu|(N)=0\in L^0(\mm).
\]
We denote by \(\mathcal N_\mu\subseteq\mathcal A\) the collection of all \(\mu\)-null sets.
\end{definition}

Note that \(\mathcal N_\mu\) is a \(\sigma\)-ideal, i.e.\ a family of sets containing the empty
set that is closed under taking subsets and countable unions. It is immediate to check that,
for any \(N\in\mathcal A\), one has
\[
N\in\mathcal N_\mu\quad\Longleftrightarrow\quad[|\mu|](N)=0,
\]
thus in particular \(\mathcal N_\mu\) coincides with the \(\sigma\)-ideal of null sets of some real-valued
finite measure.
\begin{example}\label{ex:Rn-valued_meas}{\rm
When \(\mm\) is purely atomic and consists of \(n\)-many atoms, with \(n\in\N\), the space \(L^0(\mm)\)
can be identified with \(\R^n\) in a natural way. It can be also readily checked that \(\mathcal M(\Omega;L^0(\mm))\)
can be identified with the space of all vector measures \(\mu\colon\mathcal A\to\R^n\).
\fr}\end{example}

Due to the fact that \((L^0(\mm),\leq)\) is not totally ordered (unless \(\mm\) consists of a unique atom), it is not reasonable
to expect the existence of some sort of Hahn decomposition of an \(L^0\)-valued measure \(\mu\in\mathcal M(\Omega;L^0(\mm))\);
this assertion is supported also by Example \ref{ex:Rn-valued_meas}. Still, it is possible to define the
\textbf{positive part} \(\mu^+\) and \textbf{negative part} \(\mu^-\) of \(\mu\in\mathcal M(\Omega;L^0(\mm))\) as
\begin{equation}\label{eq:pos_neg_parts}
\mu^+\coloneqq\frac{|\mu|+\mu}{2}\in\mathcal M_+(\Omega;L^0(\mm)),
\qquad\mu^-\coloneqq\frac{|\mu|-\mu}{2}\in\mathcal M_+(\Omega;L^0(\mm)),
\end{equation}
respectively.
\begin{definition}[Absolute continuity of \(L^0\)-valued measures]\label{def:ac_L0-meas}
Let \(\mu,\nu\in\mathcal M_+(\Omega;L^0(\mm))\) be given. Then we say that
\(\mu\) is \textbf{absolutely continuous} with respect to \(\nu\), and we write
\[
\mu\ll\nu,
\]
provided it holds that \(\mathcal N_\nu\subseteq\mathcal N_\mu\) (i.e.\ \(\mu(N)=0\) for every
\(N\in\mathcal A\) such that \(\nu(N)=0\)).
\end{definition}

It can be easily checked that, given any \(\mu,\nu\in\mathcal M_+(\Omega;L^0(\mm))\), it holds that
\[
\mu\ll\nu\quad\Longleftrightarrow\quad[\mu]\ll[\nu].
\]
\begin{remark}\label{rmk:implication_ac}{\rm
We claim that for any \(\mu,\nu\in\mathcal M_+(\Omega;L^0(\mm))\) it holds that
\[
\hat\mu\ll\hat\nu\quad\Longrightarrow\quad\mu\ll\nu.
\]
Indeed, if \(N\in\mathcal A\) with \(\nu(N)=0\) is given, then we have that \(\hat\nu(N\times\X)=\int\nu(N)\,\d\mm=0\),
thus \(\int\mu(N)\,\d\mm=\hat\mu(N\times\X)=0\) and accordingly \(\mu(N)=0\). We do not know
whether also the converse implication holds.
\fr}\end{remark}
\subsection{The lattice of non-negative \texorpdfstring{\(L^0\)}{L0}-valued measures}
The space \(\mathcal M_+(\Omega;L^0(\mm))\) is a poset if endowed with the ensuing partial order:
given \(\mu,\nu\in\mathcal M_+(\Omega;L^0(\mm))\), we declare that
\[
\mu\leq\nu\quad\Longleftrightarrow\quad\mu(A)\leq\nu(A)\;\text{ for every }A\in\mathcal A.
\]
It is straightforward to check that \((\mathcal M_+(\Omega;L^0(\mm)),\leq)\) is a conditionally-complete lattice.
More specifically, given a subset \(\mathcal S\) of \(\mathcal M_+(\Omega;L^0(\mm))\) for which there exists \(\nu\in\mathcal M_+(\Omega;L^0(\mm))\)
satisfying \(\mu\leq\nu\) for every \(\mu\in\mathcal S\), we have that
\[
\bigg(\bigvee\mathcal S\bigg)(A)=\bigg(\bigvee_{\mu\in\mathcal S}\mu\bigg)(A)=\bigvee\bigg\{\sum_{n=1}^\infty\mu_n(A_n)\;\bigg|\;(\mu_n)_{n\in\N}\subseteq\mathcal S,\,
(A_n)_{n\in\N}\subseteq\mathcal A\text{ partition of }A\bigg\}
\]
for all \(A\in\mathcal A\). Moreover, if \(\mathcal S\) is an arbitrary subset of \(\mathcal M_+(\Omega;L^0(\mm))\), then for any \(A\in\mathcal A\) we have that
\[
\bigg(\bigwedge\mathcal S\bigg)(A)=\bigg(\bigwedge_{\mu\in\mathcal S}\mu\bigg)(A)=\bigwedge\bigg\{\sum_{n=1}^\infty\mu_n(A_n)\;\bigg|\;(\mu_n)_{n\in\N}\subseteq\mathcal S,\,
(A_n)_{n\in\N}\subseteq\mathcal A\text{ partition of }A\bigg\}.
\]

In particular, given any \(p\in\mathcal M_+(\Omega;L^0(\mm))\) and letting \(S_p\coloneqq\{\mu\in\mathcal M_+(\Omega;L^0(\mm)):\mu\leq p\}\), we have that
\((S_p,\leq)\) is a \emph{Heyting algebra}, since for any \(\mu_1,\mu_2\in S_p\) there exists the supremum of the set \(\{\nu\in S_p:\mu_1\wedge\nu\leq\mu_2\}\) in \((S_p,\leq)\).
\begin{remark}\label{ossmintotvar}{\rm
Given any \(\mu\in\mathcal M(\Omega;L^0(\mm))\), its associated total variation \(L^0\)-valued measure \(|\mu|\) can be equivalently characterised as follows:
\begin{equation}\label{eq:equiv_tv_meas_mu}
|\mu|=\bigwedge\big\{\sigma\in\mathcal M_+(\Omega;L^0(\mm))\;\big|\;|\mu(A)|\leq\sigma(A)\text{ for every }A\in\mathcal A\big\}.
\end{equation}
Let us prove it. On the one hand, the inequality \(\geq\) in \eqref{eq:equiv_tv_meas_mu} follows from the observation that
\(|\mu(A)|\leq|\mu|(A)\) for every \(A\in\mathcal A\). On the other hand, for any \(\sigma\) as above and \(A\in\mathcal A\) we have
\[\begin{split}
|\mu|(A)&=\bigvee\bigg\{\sum_{n=1}^\infty|\mu(A_n)|\;\bigg|\;(A_n)_{n\in\N}\subseteq\mathcal A\text{ partition of }A\bigg\}\\
&\leq\bigvee\bigg\{\sum_{n=1}^\infty\sigma(A_n)\;\bigg|\;(A_n)_{n\in\N}\subseteq\mathcal A\text{ partition of }A\bigg\}=\sigma(A),
\end{split}\]
whence the validity of the inequality \(\leq\) in \eqref{eq:equiv_tv_meas_mu} follows.
\fr}\end{remark}
\subsection{Foliation of an \texorpdfstring{\(L^0\)}{L0}-valued measure}
In this section, we study those \(L^0\)-valued measures that can be `foliated' with respect to the random parameter \(x\in\X\).
\begin{lemma}\label{lem:foliation}
Let \((\mu_x)_{x\in\X}\subseteq\mathcal M(\Omega)\) be an \textbf{\(\mm\)-measurable collection of measures},
with which we mean that the function \(\X\ni x\mapsto\mu_x(A)\in\R\) is \(\mm\)-measurable for every \(A\in\mathcal A\).
We define the map \(\mu\colon\mathcal A\to L^0(\mm)\) as
\[
\mu(A)\coloneqq[\mu_\cdot(A)]_\mm\in L^0(\mm)\quad\text{ for every }A\in\mathcal A.
\]
Then it holds that \(\mu\in\mathcal M(\Omega;L^0(\mm) )\). In particular, we have that \(\mu\) is of bounded variation.
\end{lemma}
\begin{proof}
Since \(\mu_x(\varnothing)=0\) for every \(x\in\X\), we have that \(\mu(\varnothing)=0\in L^0(\mm)\). Given pairwise disjoint sets
\((A_n)_{n\in\N}\subseteq\mathcal A\), we have that \(\sum_{n=1}^\infty|\mu_x(A_n)|<+\infty\) for every \(x\in\X\), thus in particular the \(\mm\)-a.e.\ equivalence
class \(G\) of \(\X\ni x\mapsto\sum_{n=1}^\infty|\mu_x(A_n)|\) belongs to \(L^0_+(\mm)\). Since \(\sum_{i=1}^n|\mu(A_i)|\leq G\) for all \(n\in\N\), the set
\(\big\{\sum_{i=1}^n|\mu(A_i)|:n\in\N\big\}\) is order bounded in \(L^0(\mm)\) and thus \(\{|\mu(A_n)|:n\in\N\}\) is summable by Remark \ref{rmk:fact_about_summability} iv).
Moreover, we have that \(\sum_{n\in\N}|\mu(A_n)|=\bigvee_{n\in\N}\sum_{i=1}^n|\mu(A_i)|\) is the \(\mm\)-a.e.\ equivalence class of the function
\(\X\ni x\mapsto\sum_{n=1}^\infty|\mu_x(A_n)|\), so that \(\{\mu(A_n):n\in\N\}\subseteq L^0(\mm)\) is summable by Remark \ref{rmk:fact_about_summability} vi).
Letting \(A\coloneqq\bigcup_{n\in\N}A_n\) for brevity, we have that
\[
\sum_{i=1}^n\mu_x(A_i)\to\sum_{i=1}^\infty\mu_x(A_i)=\mu_x(A)\quad\text{ as }n\to\infty
\]
for every \(x\in\X\), whence it follows that \(\sum_{i=1}^n\mu(A_i)\to\mu(A)\) in \(L^0(\mm)\) as \(n\to\infty\). Taking Remark \ref{rmk:fact_about_summability} ii) into account,
we deduce that \(\mu(A)=\sum_{n\in\N}\mu(A_n)\), proving that \(\mu\) is an \(L^0\)-valued measure. To conclude, it remains to check that \(\mu\) is of bounded variation. Note that for any \(x\in\X\) we have
\[
\sup_{(A_n)_n}\sum_{n=1}^\infty|\mu_x(A_n)|\leq\sup_{(A_n)_n}\sum_{n=1}^\infty|\mu_x|(A_n)=|\mu_x|(\Omega),
\]
where the supremum is taken over all partitions \((A_n)_{n\in\N}\subseteq\mathcal A\) of \(\Omega\). Recalling Remark
\ref{rmk:suff_cond_ord-bdd} and using the fact that \(\X\ni x\mapsto\sum_{n=1}^\infty|\mu_x(A_n)|\) is an
\(\mm\)-a.e.\ representative of \(\sum_{n=1}^\infty|\mu(A_n)|\), we deduce that
\(\big\{\sum_{n=1}^\infty|\mu(A_n)|:(A_n)_{n\in\N}\subseteq\mathcal A\text{ partition of }\Omega\big\}\)
is an order-bounded subset of \(L^0(\mm)\), so that accordingly \(|\mu|_{\rm TV}\in L^0_+(\mm)\). Hence,
\(\mu\) is of bounded variation.
\end{proof}
\begin{definition}[Foliation of an \(L^0\)-valued measure]\label{def:foliation_meas}
We say that a given \(\mu\in\mathcal M(\Omega;L^0(\mm))\) can be \textbf{foliated} if it can be expressed as in Lemma
\ref{lem:foliation} for some \(\mm\)-measurable collection of measures \((\mu_x)_{x\in\X}\subseteq\mathcal M(\Omega)\),
which we call a \textbf{foliation} of \(\mu\).
\end{definition}

The foliation of an \(L^0\)-valued measure is not necessarily unique. In the following proposition we provide
a sufficient condition for the (essential) uniqueness of the foliation.
\begin{proposition}[Uniqueness of the foliation]\label{prop:unique_foliation}
Assume that \(\mathcal A\) is countably generated and \((\mu_x)_{x\in\X},(\nu_x)_{x\in\X}\subseteq\mathcal M(\Omega)\)
are foliations of the same \(L^0\)-valued measure \(\mu\in\mathcal M(\Omega;L^0(\mm))\). Then it holds that
\[
\mu_x=\nu_x\quad\text{ for }\mm\text{-a.e.\ }x\in\X.
\]
\end{proposition}
\begin{proof}
Fix a countable set \({\rm C}\subseteq\mathcal A\) that generates the \(\sigma\)-algebra \(\mathcal A\), and denote
\begin{equation}\label{eq:def_pi-system_P}
{\rm P}\coloneqq\{\Omega\}\cup\big\{A_1\cap\ldots\cap A_n\;\big|\;n\in\N,\,A_1,\ldots,A_n\in{\rm C}\big\}\subseteq\mathcal A.
\end{equation}
Then \({\rm P}\) is a countable \(\pi\)-system that generates the \(\sigma\)-algebra \(\mathcal A\).
Now, for any \(x\in\X\) we define
\[
{\rm D}_x\coloneqq\big\{A\in\mathcal A\;\big|\;\mu_x(A)=\nu_x(A)\big\}.
\]
Given any \(A,B\in{\rm D}_x\) with \(A\subseteq B\), we have
\[
\mu_x(B\setminus A)=\mu_x(B)-\mu_x(A)=\nu_x(B)-\nu_x(A)=\nu_x(B\setminus A),
\]
so that \(B\setminus A\in{\rm D}_x\). Moreover, given any sequence \((A_n)_{n\in\N}\subseteq{\rm D}_x\) with
\(A_1\subseteq A_2\subseteq A_3\subseteq\ldots\), we have
\(\mu_x\big(\bigcup_{n\in\N}A_n\big)=\sup_{n\in\N}\mu_x(A_n)=\sup_{n\in\N}\nu_x(A_n)=\nu_x\big(\bigcup_{n\in\N}A_n\big)\),
thus \(\bigcup_{n\in\N}A_n\in{\rm D}_x\). For any \(A\in{\rm P}\), we have that \(\mu_x(A)=\mu(A)(x)=\nu(A)(x)=\nu_x(A)\)
for \(\mm\)-a.e.\ \(x\in\X\). Since \({\rm P}\) is a countable family, we deduce that \(\mm(N)=0\), where we set
\[
N\coloneqq\big\{x\in\X\;\big|\;\mu_x(A)\neq\nu_x(A)\text{ for some }A\in{\rm P}\big\}.
\]
Fix \(x\in\X\setminus N\). Then we have that \(\mu_x(\Omega)=\nu_x(\Omega)\), thus accordingly \({\rm D}_x\) is a Dynkin
system, and \({\rm P}\subseteq{\rm D}_x\). Thanks to the Sierpi\'{n}ski--Dynkin \(\pi\)-\(\lambda\) theorem, we conclude
that \({\rm D}_x=\mathcal A\) (i.e.\ that \(\mu_x=\nu_x\)) for every \(x\in\X\setminus N\), which yields the statement.
\end{proof}
\begin{remark}\label{rmk:integr_mu_x}{\rm
Assume \((\mu_x)_{x\in\X}\subseteq\mathcal M(\Omega)\) is a foliation of a given \(\mu\in\mathcal M(\Omega;L^0(\mm))\) such that
\((|\mu_x|)_{x\in\X}\subseteq\mathcal M_+(\Omega)\) is an \(\mm\)-measurable collection. Let \(\nu\in\mathcal M_+(\Omega;L^0(\mm))\)
denote the \(L^0\)-valued measure whose foliation is \((|\mu_x|)_{x\in\X}\). Then we claim that \(|\mu|\leq\nu\).
Indeed, for every \(A\in\mathcal A\) and \(x\in\X\) we have that \(|\mu_x(A)|\leq|\mu_x|(A)\), so that \(|\mu(A)|\leq\nu(A)\).
Recalling Remark \ref{ossmintotvar}, we thus conclude that \(|\mu|\leq\nu\), as we claimed.
\fr}\end{remark}

We expect that, in general, the fact that \((\mu_x)_{x\in\X}\subseteq\mathcal M(\Omega)\) is an \(\mm\)-measurable collection
of measures does not imply that \((|\mu_x|)_{x\in\X}\subseteq\mathcal M_+(\Omega)\) is an \(\mm\)-measurable collection.
However, in Corollary \ref{cor:meas_|mu_x|} we will provide a sufficient condition for the validity of this implication.
\begin{lemma}\label{lem:meas_by_mon_class_thm}
Let \(\mu\in\mathcal P(\Omega;L^0(\mm))\) have a foliation \((\mu_x)_{x\in\X}\subseteq\mathcal P(\Omega)\).
Fix \(F\in\mathcal L^0_+(\mathcal A\otimes\bar\Sigma_\mm)\).
Then the function \(\X\ni x\mapsto\int F(\cdot,x)\,\d\mu_x\in[0,+\infty]\) is \(\mm\)-measurable and
\begin{equation}\label{eq:meas_by_mon_class_thm_cl}
\int F\,\d\hat\mu=\int\!\!\!\int F(\cdot,x)\,\d\mu_x\,\d\mm(x).
\end{equation}
\end{lemma}
\begin{proof}
Let us denote \({\rm R}\coloneqq\{A\times E:A\in\mathcal A,E\in\bar\Sigma_\mm\}\) and
\[
{\rm H}\coloneqq\bigg\{F\in\mathcal L^\infty(\mathcal A\otimes\bar\Sigma_\mm)\;\bigg|\;\int F(p,\cdot)\,\d\mu_\cdot(p)
\in\mathcal L^\infty(\bar\Sigma_\mm),\,\int F\,\d\hat\mu=\int\!\!\!\int F(\cdot,x)\,\d\mu_x\,\d\mm(x)\bigg\}.
\]
Note that \({\rm R}\) is a \(\pi\)-system containing the set \(\Omega\times\X\). It is easy to check that \({\rm H}\)
is a vector subspace of \(\mathcal L^\infty(\mathcal A\otimes\bar\Sigma_\mm)\) and, using the monotone convergence theorem,
that \(F\in{\rm H}\) whenever \(F\colon\Omega\times\X\to[0,+\infty)\) is a bounded function such that
\(F_n\nearrow F\) for some \((F_n)_n\subseteq{\rm H}\) with \(F_n\geq 0\). Given \(A\times E\in{\rm R}\),
we have that the function \(\X\ni x\mapsto\int\1_{A\times E}(p,x)\,\d\mu_x(p)=\1_E(x)\mu_x(A)\) is \(\mm\)-measurable and
\[
\int\1_{A\times E}\,\d\hat\mu=\hat\mu(A\times E)=\int_E\mu(A)\,\d\mm=\int_E\mu_x(A)\,\d\mm(x)
=\int\!\!\!\int\1_{A\times E}(\cdot,x)\,\d\mu_x\,\d\mm(x)
\]
by \eqref{eq:def_hat_mu}, so that \(\1_{A\times E}\in{\rm H}\). Hence, the monotone class theorem yields
\({\rm H}=\mathcal L^\infty(\mathcal A\otimes\bar\Sigma_\mm)\). Finally, given any function
\(F\in\mathcal L^0_+(\mathcal A\otimes\bar\Sigma_\mm)\), it holds that \(F\wedge n\in{\rm H}\) for every \(n\in\N\)
and \(F\wedge n\nearrow F\) as \(n\to\infty\), thus the monotone convergence theorem implies that
the function \(\X\ni x\mapsto\int F(\cdot,x)\,\d\mu_x\in[0,+\infty]\) is \(\mm\)-measurable and that
\eqref{eq:meas_by_mon_class_thm_cl} holds. This completes the proof of the statement.
\end{proof}
\begin{remark}{\rm
Assume that \(\mu\in\mathcal M_+(\Omega;L^0(\mm))\) and \(\nu\in\mathcal M_+(\Omega;L^0(\mm))\) have foliations
\((\mu_x)_{x\in\X}\subseteq\mathcal M_+(\Omega)\) and \((\nu_x)_{x\in\X}\subseteq\mathcal M_+(\Omega)\),
respectively. Assume also \(\mu(\Omega),\nu(\Omega)\in L^\infty(\mm)\). Then we claim that
\begin{equation}\label{eq:impl_ac_2}
\mu_x\ll\nu_x\text{ for }\mm\text{-a.e.\ }x\in\X\quad\Longrightarrow\quad\hat\mu\ll\hat\nu.
\end{equation}
Indeed, if \(N\in\mathcal A\otimes\Sigma\) and \(\hat\nu(N)=0\), then 
\(\int\nu_x(\{p\in\Omega:(p,x)\in N\})\,\d\mm(x)=0\) by Lemma \ref{lem:meas_by_mon_class_thm}, which gives
\(\nu_x(\{p\in\Omega:(p,x)\in N\})=0\), so that \(\mu_x(\{p\in\Omega:(p,x)\in N\})=0\), for \(\mm\)-a.e.\ \(x\in\X\).
Accordingly, we have that \(\hat\mu(N)=\int\mu_x(\{p\in\Omega:(p,x)\in N\})\,\d\mm(x)=0\), thus proving
\eqref{eq:impl_ac_2}. We do not know whether the converse implication holds.
\fr}\end{remark}
\subsection{Radon \texorpdfstring{\(L^0\)}{L0}-valued measures}\label{s:Radon_L0-val_meas}
Let \((\Omega,\tau)\) be a topological space. We make the implicit assumption that we equip
\(\Omega\) with the Borel \(\sigma\)-algebra \(\mathcal A=\mathscr B(\Omega)\) induced by \((\Omega,\tau)\).
In this case, we refer to the elements of \(\mathcal M(\Omega;L^0(\mm))\) as the \textbf{Borel \(L^0\)-valued measures}
on \(\Omega\).
\begin{definition}[Radon \(L^0\)-valued measure]\label{def:RadonL0_meas}
Let \((\Omega,\tau)\) be a Hausdorff space. Then we say that a given Borel \(L^0\)-valued measure \(\mu\) on \(\Omega\)
is \textbf{Radon} provided the following conditions hold:
\begin{align}
\label{eq:inner_reg_on_open}
|\mu|(U)=\bigvee\big\{|\mu|(K)\;\big|\;K\subseteq\Omega\text{ compact, }K\subseteq U\big\}&\quad\text{ for every }U\in\tau,\\
\label{eq:outer_reg}
|\mu|(B)=\bigwedge\big\{|\mu|(U)\;\big|\;U\in\tau,\,B\subseteq U\big\}&\quad\text{ for every }B\in\mathscr B(\Omega).
\end{align}
We refer to \eqref{eq:inner_reg_on_open} and \eqref{eq:outer_reg} as the \textbf{inner regularity on open sets}
and \textbf{outer regularity}, respectively. Moreover, we denote by \(\mathfrak M(\Omega;L^0(\mm))\) the space
of all Radon \(L^0\)-valued measures \(\mu\colon\mathscr B(\Omega)\to L^0(\mm)\) and we define
\(\mathfrak M_+(\Omega;L^0(\mm))\coloneqq\mathfrak M(\Omega;L^0(\mm))\cap\mathcal M_+(\Omega;L^0(\mm))\).
\end{definition}

We have chosen the above axiomatisation because it is the most natural generalisation of the usual notion of
a (real-valued) Radon measure. Observe that \(L^0\)-valued measures are by definition `randomly finite'
(in the sense that \(|\mu|(A)\leq|\mu|_{\rm TV}\in L^0_+(\mm)\) for all \(A\in\mathscr B(\Omega)\)), thus the
`finiteness on compact sets' (i.e.\ \(|\mu|(K)\in L^0_+(\mm)\) for every compact set \(K\subseteq\Omega\)) is
automatically guaranteed. In the next result, we show that the inner regularity property can be extended to all Borel sets:
\begin{lemma}\label{lem:inner_reg_self-improv}
Let \((\Omega,\tau)\) be a Hausdorff space. Let \(\mu\in\mathcal M(\Omega;L^0(\mm))\) be given. Then we have that \(\mu\in\mathfrak M(\Omega;L^0(\mm))\) if and only if
\(\mu\) is \textbf{inner regular (on Borel sets)}, i.e.
\begin{equation}\label{eq:inner_reg}
|\mu|(B)=\bigvee\big\{|\mu|(K)\;\big|\;K\subseteq\Omega\text{ compact, }K\subseteq B\big\}
\quad\text{ for every }B\in\mathscr B(\Omega).
\end{equation}
\end{lemma}
\begin{proof}
For the `if' implication, observe that the inner regularity implies (by passing to the complement)
the outer regularity. Let us then focus on the `only if' implication. Assume \(\mu\) is Radon.
Fix \(B\in\mathscr B(\Omega)\) and \(\varepsilon>0\). Using \eqref{eq:outer_reg} and the fact that finite intersections of open
sets are open, we can find sets \(U,V\in\tau\) such that \(B\subseteq U\) with \(\sfd_{L^0(\mm)}(|\mu|(U\setminus B),0)\leq\varepsilon\)
and \(\Omega\setminus B\subseteq V\) with \(\sfd_{L^0(\mm)}(|\mu|(V\cap B),0)\leq\varepsilon\). Letting \(C\) denote the closed
set \(\Omega\setminus V\), we thus have that \(C\subseteq B\) and \(\sfd_{L^0(\mm)}(|\mu|(B\setminus C),0)\leq\varepsilon\).
Next, using \eqref{eq:inner_reg_on_open} and the fact that finite unions of compact sets are compact, we can find a compact
set \(K'\subseteq U\) such that \(\sfd_{L^0(\mm)}(|\mu|(U\setminus K'),0)\leq\varepsilon\). Therefore, we have that
\(K\coloneqq C\cap K'\) is a compact subset of \(B\) satisfying \(\sfd_{L^0(\mm)}(|\mu|(B),|\mu|(K))\leq 3\varepsilon\).
\end{proof}

It is straightforward to check that \(\mathfrak M(\Omega;L^0(\mm))\) is an \(L^0(\mm)\)-submodule of \(\mathcal M(\Omega;L^0(\mm))\).
We also claim that
\begin{equation}\label{MesurL0mcompleto}
   \big(\mathfrak M(\Omega;L^0(\mm)),|\cdot|_{\rm TV}\big)\quad\text{ is a complete random normed module.} 
\end{equation}
To prove it, we show that \(\mathfrak M(\Omega;L^0(\mm))\) is closed in \(\mathcal M(\Omega;L^0(\mm))\).
Fix \((\mu_n)_{n\in\N}\subseteq\mathfrak M(\Omega;L^0(\mm))\) and \(\mu\in\mathcal M(\Omega;L^0(\mm))\) such
that \(|\mu_n-\mu|_{\rm TV}\to 0\) in \(L^0(\mm)\) as \(n\to\infty\). Given any \(\varepsilon>0\), we take
\(n_\varepsilon\in\N\) such that \(\sfd_{L^0(\mm)}(|\mu_{n_\varepsilon}-\mu|_{\rm TV},0)\leq\varepsilon\).
Now, fix any \(B\in\mathscr B(\Omega)\). Lemma \ref{lem:inner_reg_self-improv} guarantees the existence of a compact
set \(K\subseteq B\) such that \(\sfd_{L^0(\mm)}(|\mu_{n_\varepsilon}|(B),|\mu_{n_\varepsilon}|(K))\leq\varepsilon\).
Hence, we have that
\[\begin{split}
&\sfd_{L^0(\mm)}(|\mu|(B),|\mu|(K))\\
\leq\,&\sfd_{L^0(\mm)}(|\mu|(B),|\mu_{n_\varepsilon}|(B))+\sfd_{L^0(\mm)}(|\mu_{n_\varepsilon}|(B),|\mu_{n_\varepsilon}|(K))
+\sfd_{L^0(\mm)}(|\mu_{n_\varepsilon}|(K),|\mu|(K))\\
\leq\,&2\,\sfd_{L^0(\mm)}(|\mu_{n_\varepsilon}-\mu|_{\rm TV},0)+\varepsilon\leq 3\varepsilon.
\end{split}\]
This proves that \(\mu\) is inner regular on Borel sets, which implies that \(\mu\in\mathfrak M(\Omega;L^0(\mm))\).
\begin{proposition}\label{prop:suff_cond_Radon}
Let \((\Omega,\tau)\) be a compact Hausdorff space. Assume that \(\mu\in\mathcal M(\Omega;L^0(\mm))\)
is inner regular on open sets, i.e.\ it satisfies \eqref{eq:inner_reg_on_open}. Then it holds that
\(\mu\in\mathfrak M(\Omega;L^0(\mm))\).
\end{proposition}
\begin{proof}
We denote by \(\tilde{\mathcal F}\) the collection of all \(B\in\mathscr B(\Omega)\) that satisfy
the identity in \eqref{eq:inner_reg}. We claim that \(\tilde{\mathcal F}\) is closed under countable unions
and intersections. To prove it, fix \((B_n)_{n\in\N}\subseteq\tilde{\mathcal F}\).
\begin{itemize}
\item Let \(U\coloneqq\bigcup_{n\in\N}B_n\) and fix \(\varepsilon>0\). Then \(\sfd_{L^0(\mm)}(|\mu|(U\setminus(B_1\cup\ldots\cup B_{\bar n})),0)\leq\varepsilon/2\) for some \(\bar n\in\N\).
For any \(i=1,\ldots,\bar n\), we can find a compact set \(K_i\subseteq B_i\) such that
\(\sfd_{L^0(\mm)}(|\mu|(B_i\setminus K_i),0)\leq\varepsilon/(2\bar n)\). Then \(K\coloneqq K_1\cup\ldots\cup K_{\bar n}\)
is a compact subset of \(B_1\cup\ldots\cup B_{\bar n}\subseteq U\) such that
\[
\sfd_{L^0(\mm)}(|\mu|(U\setminus K),0)\leq\frac{\varepsilon}{2}+\sum_{i=1}^{\bar n}\sfd_{L^0(\mm)}(|\mu|(B_i\setminus K_i),0)\leq\varepsilon.
\]
By the arbitrariness of \(\varepsilon>0\), we deduce that \(U\in\tilde{\mathcal F}\).
\item Let \(E\coloneqq\bigcap_{n\in\N}B_n\) and fix \(\varepsilon>0\). Recursively, we can find compact sets
\(K_1\subseteq B_1\) and \(K_{n+1}\subseteq B_{n+1}\cap K_n\) for \(n\in\N\) such that \(\sfd_{L^0(\mm)}(|\mu|(B_1\setminus K_1),0)\leq\varepsilon/2\) and
\[
\sfd_{L^0(\mm)}(|\mu|((B_{n+1}\cap K_n)\setminus K_{n+1}),0)\leq\frac{\varepsilon}{2^{n+1}}\quad\text{ for every }n\in\N.
\]
Since \(E\setminus K\subseteq(B_1\setminus K_1)\cup\bigcup_{n\in\N}((B_{n+1}\cap K_n)\setminus K_{n+1})\), the set \(K\coloneqq\bigcap_{n\in\N}K_n\subseteq E\) satisfies
\[\begin{split}
&\sfd_{L^0(\mm)}(|\mu|(E\setminus K),0)\\
\leq\,&\sfd_{L^0(\mm)}(|\mu|(B_1\setminus K_1),0)+\sum_{n\in\N}\sfd_{L^0(\mm)}(|\mu|((B_{n+1}\cap K_n)\setminus K_{n+1}),0)\leq\varepsilon.
\end{split}\]
Since \(K\) is compact, by the arbitrariness of \(\varepsilon>0\) we deduce that \(E\in\tilde{\mathcal F}\).
\end{itemize}
All in all, we have shown that \(\tilde{\mathcal F}\) is closed under countable unions and intersections. Therefore,
\[
\mathcal F\coloneqq\big\{B\in\tilde{\mathcal F}\;\big|\;\Omega\setminus B\in\tilde{\mathcal F}\big\}\quad\text{ is a }\sigma\text{-algebra.}
\]
Finally, observe that \(\tau\subseteq\mathcal F\), since any open set \(U\in\tau\) belongs to \(\tilde{\mathcal F}\) by assumption,
whereas its complement \(\Omega\setminus U\) is compact and thus trivially satisfies the identity in \eqref{eq:inner_reg}. Hence, we conclude that
\(\mathcal F=\mathscr B(\Omega)\), which means that \(\mu\) is inner regular on all Borel sets. Since a subset of \(\Omega\) is compact if and only
if its complement is open, inner regularity and outer regularity are equivalent, thus accordingly \(\mu\) is a Radon \(L^0\)-valued measure.
\end{proof}
The proof of the following result is adapted from \cite[Proposition 3.3.44]{HKST:15}:
\begin{lemma}\label{lem:enough_outer_reg}
Let \((\Omega,\sfd)\) be a complete separable metric space. Assume that \(\mu\in\mathcal M(\Omega;L^0(\mm))\) is outer regular,
i.e.\ that \(\mu\) satisfies \eqref{eq:outer_reg}. Then it holds that \(\mu\in\mathfrak M(\Omega;L^0(\mm))\).
\end{lemma}
\begin{proof}
It suffices to prove the statement in the case where \(\mu\in\mathcal M_+(\Omega;L^0(\mm))\). Our goal is to show that
\(\mu\) is inner regular. Fix any \(B\in\mathscr B(\Omega)\) and \(\varepsilon>0\). The outer regularity of \(\mu\) applied
to \(\Omega\setminus B\) provides us with a closed set \(C\subseteq\Omega\) such that \(C\subseteq B\) and
\(\sfd_{L^0(\mm)}(\mu(B),\mu(C))\leq\varepsilon\). Since \((\Omega,\sfd)\) is separable, for any \(n\in\N\) we can find
a countable collection \(\{B^n_i:i\in\N\}\) of closed balls in \(\Omega\) with center in \(C\), of radius \(1/n\) and such
that \(C\subseteq\bigcup_{i\in\N}B^n_i\). Remark \ref{rmk:cont_above_L0_meas} i) ensures that for any \(n\in\N\)
there exists \(i_n\in\N\) such that \(\sfd_{L^0(\mm)}(\mu(C\setminus(B^n_1\cup\ldots\cup B^n_{i_n})),0)\leq\varepsilon 2^{-n}\).
Now, define \(C_n\coloneqq B^n_1\cup\ldots\cup B^n_{i_n}\) for all \(n\in\N\) and \(K\coloneqq C\cap\bigcap_{n\in\N}C_n\subseteq C\).
The closed set \(K\) is totally bounded, thus the completeness of \((\Omega,\sfd)\)
ensures that \(K\) is compact. Moreover,
\[\begin{split}
\sfd_{L^0(\mm)}(\mu(C),\mu(K))&=\lim_{n\to\infty}\sfd_{L^0(\mm)}\big(\mu(C),\mu(C\cap C_1\cap\ldots\cap C_n)\big)\\
&\leq\varliminf_{n\to\infty}\sum_{j=1}^n\sfd_{L^0(\mm)}(\mu(C\setminus C_j),0)
\leq\sum_{j=1}^\infty\varepsilon 2^{-j}=\varepsilon.
\end{split}\]
It follows that \(\sfd_{L^0(\mm)}(\mu(B),\mu(K))\leq\sfd_{L^0(\mm)}(\mu(B),\mu(C))+\sfd_{L^0(\mm)}(\mu(C),\mu(K))\leq 2\varepsilon\).
Thanks to the arbitrariness of \(B\) and \(\varepsilon\), we conclude that \(\mu\) is inner regular, completing the proof.
\end{proof}

The proof of the next result follows along the lines of \cite[Proposition 3.3.37]{HKST:15}:
\begin{corollary}\label{cor:every_mu_Radon_Polish}
Let \((\Omega,\sfd)\) be a complete separable metric space. Then it holds that
\[
\mathcal M(\Omega;L^0(\mm))=\mathfrak M(\Omega;L^0(\mm)).
\]
\end{corollary}
\begin{proof}
It suffices to prove that any given \(\mu\in\mathcal M_+(\Omega;L^0(\mm))\) is Radon. To this aim, it is enough to check
that \(\mu\) is outer regular, as it is guaranteed by Lemma \ref{lem:enough_outer_reg}. We denote by \(\mathcal F\) the collection
of all those Borel subsets \(B\) of \(\Omega\) that satisfy the following two conditions:
\begin{equation}\label{eq:every_mu_Radon_Polish_aux}
\mu(B)=\bigwedge\big\{\mu(U)\;\big|\;U\in\tau,\,B\subseteq U\big\},\qquad
\mu(\Omega\setminus B)=\bigwedge\big\{\mu(V)\;\big|\;V\in\tau,\,\Omega\setminus B\subseteq V\big\}.
\end{equation}
If \(B\in\tau\), then the first identity in \eqref{eq:every_mu_Radon_Polish_aux} is clearly satisfied, while the second
one can be proved as follows: letting \(U_n\coloneqq\{p\in\Omega:\sfd(p,\Omega\setminus B)<1/n\}\in\tau\) for every \(n\in\N\),
we have that \(U_{n+1}\subseteq U_n\) for every \(n\in\N\) and \(\Omega\setminus B=\bigcap_{n\in\N}U_n\),
so that \(\mu(U_n)\to\mu(\Omega\setminus B)\) in \(L^0(\mm)\) by Remark \ref{rmk:cont_above_L0_meas} ii) and
thus the second condition in \eqref{eq:every_mu_Radon_Polish_aux} holds. Therefore, we have shown that \(\tau\subseteq\mathcal F\).
Finally, by slightly adapting the arguments in the proof of Proposition \ref{prop:suff_cond_Radon}, one can show that
\(\mathcal F\) is a \(\sigma\)-algebra. Consequently, we have that \(\mathcal F=\mathscr B(\Omega)\), thus \(\mu\) is outer regular
and the statement is achieved.
\end{proof}
\begin{remark}\label{rmk:pushforward_Radon}{\rm
Let \((\Omega,\tau)\), \((\tilde\Omega,\tilde\tau)\) be Hausdorff spaces and \(\varphi\colon\Omega\to\tilde\Omega\)
a continuous map. Then
\[
\varphi_\#\mu\in\mathfrak M(\tilde\Omega;L^0(\mm))\quad\text{ for every }\mu\in\mathfrak M(\Omega;L^0(\mm)).
\]
To prove it, fix \(\tilde B\in\mathscr B(\tilde\Omega)\) and \(\varepsilon>0\). The inner regularity of \(\mu\) 
ensures the existence of a compact set \(K\subseteq\varphi^{-1}(\tilde B)\) with
\(\sfd_{L^0(\mm)}(|\mu|(\varphi^{-1}(\tilde B)),|\mu|(K))\leq\varepsilon\). Letting
\[
\tilde K\coloneqq\varphi(K)\subseteq\varphi(\varphi^{-1}(\tilde B))\subseteq\tilde B,
\]
we have that \(\tilde K\) is compact by the continuity of \(\varphi\), and \eqref{eq:ineq_pushforward_as_meas} implies that
\[\begin{split}
\sfd_{L^0(\mm)}(|\varphi_\#\mu|(\tilde B),|\varphi_\#\mu|(\tilde K))&=\sfd_{L^0(\mm)}(|\varphi_\#\mu|(\tilde B\setminus\tilde K),0)
\leq\sfd_{L^0(\mm)}(|\mu|(\varphi^{-1}(\tilde B\setminus\tilde K)),0)\\
&=\sfd_{L^0(\mm)}(|\mu|(\varphi^{-1}(\tilde B)\setminus\varphi^{-1}(\varphi(K))),0)\\
&\leq\sfd_{L^0(\mm)}(|\mu|(\varphi^{-1}(\tilde B)\setminus K),0)=\sfd_{L^0(\mm)}(|\mu|(\varphi^{-1}(\tilde B)),|\mu|(K))\leq\varepsilon.
\end{split}\]
This shows that \(\varphi_\#\mu\) is inner regular, thus \(\varphi_\#\mu\in\mathfrak M(\tilde\Omega;L^0(\mm))\)
by Lemma \ref{lem:inner_reg_self-improv}.
\fr}\end{remark}

We say that a family \(\mathcal F\) of subsets of a non-empty set \(\Omega\) is \textbf{directed} provided the following condition holds:
given any \(A,B\in\mathcal F\), there exists \(C\in\mathcal F\) such that \(A\cup B\subseteq C\).
\begin{lemma}\label{lem:tau-add}
Let \((\Omega,\tau)\) be a compact Hausdorff space. Let \(\mu\in\mathfrak M_+(\Omega;L^0(\mm))\) be given.
Then \(\mu\) is \textbf{\(\tau\)-additive}, which means that the following property holds:
if \(\{U_i:i\in I\}\subseteq\tau\) is a directed family of open sets, then
\[
\mu\bigg(\bigcup_{i\in I}U_i\bigg)=\bigvee_{i\in I}\mu(U_i).
\]
\end{lemma}
\begin{proof}
Fix \(\varepsilon>0\). By the inner regularity of \(\mu\), there is a compact subset \(K\) of
\(U\coloneqq\bigcup_{i\in I}U_i\) such that \(\sfd_{L^0(\mm)}(\mu(U),\mu(K))\leq\varepsilon\). As \(\{U_i:i\in I\}\)
is a directed open cover of \(K\), we can find \(i_0\in I\) such that \(K\subseteq U_{i_0}\). In particular, we have that
\[
\sfd_{L^0(\mm)}(\mu(U),\mu(U_{i_0}))\leq\sfd_{L^0(\mm)}(\mu(U),\mu(K))\leq\varepsilon.
\]
In view of the arbitrariness of \(\varepsilon\), the statement follows.
\end{proof}
\begin{corollary}\label{cor:agree_on_good_basis}
Let \((\Omega,\tau)\) be a Hausdorff space and \(\mu_1,\mu_2\in\mathfrak M_+(\Omega;L^0(\mm))\).
Assume that \(\mathfrak B\) is a basis for \(\tau\) that is closed under finite intersections and such that
\(\mu_1(U)=\mu_2(U)\) for every \(U\in\mathfrak B\). Then it holds that \(\mu_1=\mu_2\).
\end{corollary}
\begin{proof}
Fix any \(U\in\tau\). Since \(\mathfrak B\) is a basis for \(\tau\), we have that the collection \(\{U_i:i\in I\}\)
of all finite unions of elements of \(\mathfrak B\) contained in \(U\) is a directed family of sets satisfying
\(U=\bigcup_{i\in I}U_i\). Now, fix any \(i\in I\). Say that \(U_i=V_1\cup\ldots\cup V_n\), with \(V_1,\ldots,V_n\in\mathfrak B\).
For any \(F\subseteq\{1,\ldots,n\}\), we denote
\[
V_F\coloneqq\bigg(\bigcap_{j\in F}V_j\bigg)\setminus\bigg(\bigcup_{k\in\{1,\ldots,n\}\setminus F}V_k\bigg).
\]
Note that \(U_i\) is the disjoint union of the sets \(V_F\), as \(F\) runs over all subsets of \(\{1,\ldots,n\}\).
By an induction argument over \(n-{\rm card}(F)\), based on the fact that \(\mathfrak B\) is closed under finite intersections,
one can readily prove that \(\mu_1(V_F)=\mu_2(V_F)\) for every \(F\subseteq\{1,\ldots,n\}\), thus also \(\mu_1(U_i)=\mu_2(U_i)\).
By virtue of Lemma \ref{lem:tau-add}, we deduce that
\[
\mu_1(U)=\bigvee_{i\in I}\mu_1(U_i)=\bigvee_{i\in I}\mu_2(U_i)=\mu_2(U).
\]
This shows that \(\mu_1\) and \(\mu_2\) agree on all open sets, so that the outer regularity of \(\mu_1\) and \(\mu_2\)
finally allows us to conclude that \(\mu_1=\mu_2\), proving the statement.
\end{proof}
\subsection{Outer \texorpdfstring{\(L^0\)}{L0}-valued measures}\label{s:outer_L0-val_meas}
By mimicking the standard definition of outer measure, we propose the following notion of outer \(L^0\)-valued measure.
\begin{definition}[Outer \(L^0\)-valued measure]\label{def:outer_meas}
We say that a given map \(\nu\colon 2^\Omega\to L^0_+(\mm)\) is an \textbf{outer \(L^0\)-valued measure} if:
\begin{itemize}
\item[\(\rm i)\)] \(\nu(\varnothing)=0\).
\item[\(\rm ii)\)] \(\nu(A)\leq\nu(B)\) for every \(A,B\in 2^\Omega\) with \(A\subseteq B\).
\item[\(\rm iii)\)] \(\nu\big(\bigcup_{n\in\N}A_n\big)\leq\sum_{n\in\N}\nu(A_n)\) for every \((A_n)_{n\in\N}\subseteq 2^\Omega\).
\end{itemize}
\end{definition}

The next result provides us with a general method to construct an outer \(L^0\)-valued measure:
\begin{proposition}\label{prop:induced_outer_meas}
Let \(\mathcal C\subseteq 2^\Omega\) be a family of sets containing \(\varnothing\) and \(\Omega\).
Let \(\tau\colon\mathcal C\to L^0_+(\mm)\) be a set-function such that \(\tau(\varnothing)\coloneqq 0\).
Let us define \(\tau^*\colon 2^\Omega\to L^0_+(\mm)\) as
\[
\tau^*(S)\coloneqq\bigwedge\bigg\{\sum_{n=1}^\infty\tau(A_n)\;\bigg|\;(A_n)_{n\in\N}\subseteq\mathcal C,
\,S\subseteq\bigcup_{n\in\N}A_n\bigg\}\quad\text{ for every }S\in 2^\Omega.
\]
Then \(\tau^*\) is an outer \(L^0\)-valued measure.
\end{proposition}
\begin{proof}
Clearly, \(\tau^*\) verifies items i) and ii) of Definition \ref{def:outer_meas}, thus let us focus on the verification
of iii). Fix an arbitrary sequence \((A_n)_{n\in\N}\subseteq 2^\Omega\). For any \(\varepsilon>0\) and \(n\in\N\),
we find a sequence \((A_n^k)_{k\in\N}\subseteq\mathcal C\) such that \(A_n\subseteq\bigcup_{k\in\N}A_n^k\) and
\(\sfd_{L^0(\mm)}\big(\sum_{k=1}^\infty\tau(A_n^k),\tau^*(A_n)\big)\leq\varepsilon 2^{-n}\). Then
\[\begin{split}
&\sfd_{L^0(\mm)}\bigg(\bigg(\tau^*\Big(\bigcup_{n\in\N}A_n\Big)-\sum_{n=1}^\infty\tau^*(A_n)\bigg)^+,0\bigg)\\
\leq\,&\sfd_{L^0(\mm)}\bigg(\bigg(\tau^*\Big(\bigcup_{n\in\N}A_n\Big)-\sum_{n,k=1}^\infty\tau(A_n^k)\bigg)^+,0\bigg)
+\sfd_{L^0(\mm)}\bigg(\bigg(\sum_{n=1}^\infty\sum_{k=1}^\infty\tau(A_n^k)-\sum_{n=1}^\infty\tau^*(A_n)\bigg)^+,0\bigg)\\
\leq\,&\sfd_{L^0(\mm)}\bigg(\sum_{n=1}^\infty\bigg|\sum_{k=1}^\infty\tau(A_n^k)-\tau^*(A_n)\bigg|,0\bigg)
\leq\sum_{n=1}^\infty\sfd_{L^0(\mm)}\bigg(\sum_{k=1}^\infty\tau(A_n^k),\tau^*(A_n)\bigg)\leq\varepsilon,
\end{split}\]
as a consequence of the fact that \(\bigcup_{n\in\N}A_n\subseteq\bigcup_{n,k\in\N}A_n^k\). By the arbitrariness of
\(\varepsilon\), we conclude that \(\tau^*\big(\bigcup_{n\in\N}A_n\big)\leq\sum_{n=1}^\infty\tau^*(A_n)\), thus
completing the proof of the statement.
\end{proof}

We omit to prove the following result, which can be obtained by repeating almost verbatim the standard proof of the corresponding
result for outer measures \cite[Theorems 1.1 and 1.2]{evans2015measure}.
\begin{theorem}[Carath\'{e}odory measurable sets]\label{thm:Carath}
Let \(\nu\colon 2^\Omega\to L^0_+(\mm)\) be an outer \(L^0\)-valued measure. Then we define the family
\(\Gamma(\nu)\subseteq 2^\Omega\) of all \textbf{Carath\'{e}odory \(\nu\)-measurable sets} as
\[
\Gamma(\nu)\coloneqq\big\{A\in 2^\Omega\;\big|\;\nu(S\cap A)+\nu(S\setminus A)\leq\nu(S)\text{ for every }S\in 2^\Omega\big\}.
\]
Then \(\Gamma(\nu)\) is a \(\sigma\)-algebra on \(\Omega\) and \(\nu|_{\Gamma(\nu)}\colon\Gamma(\nu)\to L^0_+(\mm)\)
is a non-negative \(L^0\)-valued measure.
\end{theorem}
\subsection{Baire \texorpdfstring{\(L^0\)}{L0}-valued measures}\label{s:Baire_L0-val_meas}
Let \((\Omega,\tau)\) be a compact Hausdorff space. We denote by \(\mathscr B_0(\Omega)\) the \emph{Baire \(\sigma\)-algebra}
of \((\Omega,\tau)\), which is -- by definition -- the \(\sigma\)-algebra generated by all compact \(G_\delta\) subsets of \(\Omega\);
recall that any countable intersection of open subsets is called a \emph{\(G_\delta\) set}.
Note that \(\mathscr B_0(\Omega)\subseteq\mathscr B(\Omega)\) and that the inclusion can be strict (unless e.g.\ \(\tau\) is metrisable).
In the next result, we remind a useful equivalent characterisation of \emph{Baire sets} (i.e.\ of the elements of the Baire \(\sigma\)-algebra).
The \emph{Hilbert cube} \(\mathbb H\coloneqq[0,1]^\N\) will be endowed with the distance \(\delta\colon\mathbb H\times\mathbb H\to[0,1]\) given by
\[
\delta\big((t_n)_{n\in\N},(s_n)_{n\in\N}\big)\coloneqq\sum_{n\in\N}\frac{|t_n-s_n|}{2^n}\quad\text{ for every }(t_n)_{n\in\N},(s_n)_{n\in\N}\in\mathbb H,
\]
which induces the product topology. In particular, \((\mathbb H,\delta)\) is compact by Tychonoff's theorem. Observe that \(\mathbb H\) is
homeomorphic to the product \({\mathbb H}^\N\) of countably many copies of itself.
\begin{lemma}\label{lem:char_Bar}
Let \((\Omega,\tau)\) be a compact Hausdorff space. Then \(A\subseteq\Omega\) is a Baire set if and only if there exist
a continuous map \(\varphi\colon\Omega\to\mathbb H\) and a Borel set \(B\subseteq\mathbb H\) such that \(A=\varphi^{-1}(B)\).
\end{lemma}
\begin{proof}
For the reader's usefulness, we give a proof of this statement. To prove it, we aim to show that
\(\mathfrak S=\mathscr B_0(\Omega)\), where we define
\[
\mathfrak S\coloneqq\big\{\varphi^{-1}(B)\;\big|\;\varphi\colon\Omega\to\mathbb H\text{ continuous, }B\subseteq\mathbb H\text{ Borel}\big\}.
\]
Clearly, \(\Omega\in\mathfrak S\) and \(\mathfrak S\) is closed under complementation. We claim that \(\mathfrak S\) is also closed under
countable intersections. To prove it, fix a sequence \((A_n)_{n\in\N}\subseteq\mathfrak S\), say \(A_n=\varphi_n^{-1}(B_n)\). Let
\[
\varphi(p)\coloneqq(\varphi_n(p))_{n\in\N}\in{\mathbb H}^\N\quad\text{ for every }p\in\Omega.
\]
Note that \(\varphi\colon\Omega\to{\mathbb H}^\N\) is a continuous map and that \(\prod_{n\in\N}B_n\subseteq{\mathbb H}^\N\) is Borel. Since
\[
\bigcap_{n\in\N}A_n=\bigcap_{n\in\N}\varphi_n^{-1}(B_n)=\varphi^{-1}\bigg(\prod_{n\in\N}B_n\bigg)
\]
and \({\mathbb H}^\N\) is homeomorphic to \(\mathbb H\), we deduce that \(\bigcap_{n\in\N}A_n\in\mathfrak S\). All in all, we have shown
that \(\mathfrak S\) is a \(\sigma\)-algebra on \(\Omega\). Next, we claim that each compact \(G_\delta\) set \(A\subseteq\Omega\) belongs
to \(\mathfrak S\). To prove it, write \(A\) as \(\bigcap_{n\in\N}U_n\) for \((U_n)_{n\in\N}\subseteq\tau\). Fix \(n\in\N\).
Since \((\X,\tau)\) is normal, and \(A\) and \(\Omega\setminus U_n\) are disjoint closed sets, we can find a continuous function
\(\eta_n\colon\Omega\to[0,1]\) such that \(\eta_n=0\) on \(A\) and \(\eta_n=1\) on \(\Omega\setminus U_n\). In particular, we have that the
continuous map \(\Omega\ni p\mapsto\eta(p)\coloneqq(\eta_n(p))_{n\in\N}\in{\mathbb H}^\N\) satisfies
\(A=\eta^{-1}(\{(0,0,\ldots)\})\), thus accordingly \(A\in\mathfrak S\), as we claimed. Therefore,
we have that \(\mathscr B_0(\Omega)\subseteq\mathfrak S\).

Let us now prove the converse inclusion. Let \(\varphi\colon\Omega\to\mathbb H\) be a continuous map. We claim that the \(\sigma\)-algebra
\(\mathscr B(\mathbb H)\cap\varphi_*\mathscr B_0(\Omega)=\{B\in\mathscr B(\mathbb H):\varphi^{-1}(B)\in\mathscr B_0(\Omega)\}\)
contains all closed subsets of \(\mathbb H\). Indeed, if \(C\subseteq\mathbb H\) is a given closed set, then we have that
\[
\varphi^{-1}(C)=\varphi^{-1}\bigg(\bigcap_{n\in\N}\bigg\{t\in\mathbb H\;\bigg|\;\delta(t,C)<\frac{1}{n}\bigg\}\bigg)
=\bigcap_{n\in\N}\varphi^{-1}\bigg(\bigg\{t\in\mathbb H\;\bigg|\;\delta(t,C)<\frac{1}{n}\bigg\}\bigg)
\]
belongs to \(\mathscr B_0(\Omega)\), since the continuity of \(\varphi\) ensures that \(\varphi^{-1}(C)\) is closed
(thus, compact) and that each set \(\varphi^{-1}(\{t\in\mathbb H:\delta(t,C)<1/n\})\) is open (thus, \(\varphi^{-1}(C)\)
is a \(G_\delta\) set). Hence, we have shown that \(C\in\mathscr B(\mathbb H)\cap\varphi_*\mathscr B_0(\Omega)\). Given
that \(\mathscr B(\mathbb H)\) is generated by closed sets, we deduce that
\(\mathscr B(\mathbb H)\subseteq\varphi_*\mathscr B_0(\Omega)\) for every continuous map
\(\varphi\colon\Omega\to\mathbb H\), so that \(\mathfrak S\subseteq\mathscr B_0(\Omega)\).
\end{proof}

Let us now get back to \(L^0\)-valued measures. Given a compact Hausdorff space \((\Omega,\tau)\), we denote by
\begin{equation}\label{eq:def_space_Bair_L0-meas}
\mathcal M_+(\Omega,\mathscr B_0(\Omega);L^0(\mm))
\end{equation}
the collection of all non-negative \(L^0\)-valued measures of bounded variation defined on
\(\mathscr B_0(\Omega)\), which we refer to as \textbf{non-negative Baire \(L^0\)-valued measures}.
\begin{proposition}\label{prop:Baire_meas_is_reg}
Let \((\Omega,\tau)\) be a compact Hausdorff space.
Let \(\mu\in\mathcal M_+(\Omega,\mathscr B_0(\Omega);L^0(\mm))\) be given. Then for any \(A\in\mathscr B_0(\Omega)\) we have that
\begin{equation}\label{eq:Baire_meas_is_reg_cl1}
\mu(A)=\bigvee\Big\{\mu\big(f^{-1}(\{0\})\big)\;\Big|\;f\colon\Omega\to[0,1]\text{ continuous, }f^{-1}(\{0\})\subseteq A\Big\}.
\end{equation}
In particular, we have that \(\mu\) is \emph{regular}, meaning that for any \(A\in\mathscr B_0(\Omega)\) it holds that
\begin{equation}\begin{split}\label{eq:Baire_meas_is_reg_cl2}
\mu(A)&=\bigvee\big\{\mu(K)\;\big|\;K\in\mathscr B_0(\Omega)\text{ compact, }K\subseteq A\big\}\\
&=\bigwedge\big\{\mu(U)\;\big|\;U\in\tau\cap\mathscr B_0(\Omega),\,A\subseteq U\big\}.
\end{split}\end{equation}
\end{proposition}
\begin{proof}
Fix a continuous map \(\varphi\colon\Omega\to\mathbb H\) and a Borel set \(B\subseteq\mathbb H\) such that
\(A=\varphi^{-1}(B)\), whose existence is guaranteed by Lemma \ref{lem:char_Bar}. Let us then consider the
pushforward \(L^0\)-valued measure \(\varphi_\#\mu\in\mathcal M_+(\mathbb H;L^0(\mm))\). We know from Corollary
\ref{cor:every_mu_Radon_Polish} that \(\varphi_\#\mu\) is Radon, thus for any \(\varepsilon>0\) we can find a
compact set \(K_\varepsilon\subseteq\mathbb H\)
such that \(K_\varepsilon\subseteq B\) and \(\sfd_{L^0(\mm)}(\varphi_\#\mu(B),\varphi_\#\mu(K_\varepsilon))\leq\varepsilon\).
Take a continuous function \(\tilde f_\varepsilon\colon\mathbb H\to[0,1]\) with \(\{\tilde f_\varepsilon=0\}=K_\varepsilon\).
Then \(f_\varepsilon\coloneqq\tilde f_\varepsilon\circ\varphi\colon\Omega\to[0,1]\) is a continuous function satisfying
\(f_\varepsilon^{-1}(\{0\})=\varphi^{-1}(K_\varepsilon)\subseteq\varphi^{-1}(B)=A\) and accordingly
\[
\sfd_{L^0(\mm)}\big(\mu(A),\mu\big(f_\varepsilon^{-1}(\{0\})\big)\big)
=\sfd_{L^0(\mm)}\big(\varphi_\#\mu(B),\varphi_\#\mu(K_\varepsilon)\big)\leq\varepsilon.
\]
Thanks to the arbitrariness of \(\varepsilon>0\), we deduce that \eqref{eq:Baire_meas_is_reg_cl1} is satisfied.
Moreover, each set \(f_\varepsilon^{-1}(\{0\})\) is compact (and is a Baire set), whence the first identity in
\eqref{eq:Baire_meas_is_reg_cl2} follows. The second identity then follows by applying the first one to the complement of \(A\).
\end{proof}

The following result can be proved by arguing as in Lemma \ref{lem:tau-add}, by taking also Proposition \ref{prop:Baire_meas_is_reg}
into account.
\begin{lemma}\label{lem:cont_below_wrt_net}
Let \((\Omega,\tau)\) be a compact Hausdorff space. Fix any \(\mu\in\mathcal M_+(\Omega,\mathscr B_0(\Omega);L^0(\mm))\).
Let \(\{U_i:i\in I\}\subseteq\tau\cap\mathscr B_0(\Omega)\) be a directed family of sets such that
\(\bigcup_{i\in I}U_i\in\mathscr B_0(\Omega)\). Then
\[
\mu\bigg(\bigcup_{i\in I}U_i\bigg)=\bigvee_{i\in I}\mu(U_i).
\]
\end{lemma}
\begin{remark}\label{rmk:open_union_of_funct_open}{\rm
We remind that if \((\Omega,\tau)\) is a completely-regular Hausdorff space, then the collection of all
\emph{functionally-open subsets} of \(\Omega\) (i.e.\ of all those sets of the form \(f^{-1}(U)\), for some
\(f\colon\Omega\to[0,1]\) continuous and \(U\subseteq[0,1]\) open) is a basis for the topology \(\tau\).
If in addition we assume that \((\Omega,\tau)\) is compact, then each functionally-open subset of \(\Omega\)
is a Baire set (by Lemma \ref{lem:char_Bar}), thus in particular
\[
\text{each }U\in\tau\text{ is the union of some directed family }\{U_i:i\in I\}\subseteq\tau\cap\mathscr B_0(\Omega)
\]
(e.g.\ by taking all finite unions of those functionally-open sets that are contained in \(U\)).
\fr}\end{remark}

To prove the following result, we adapt the proof of \cite[Theorem 7.3.2]{Bog:07} to our setting.
\begin{theorem}[Extension of Baire \(L^0\)-valued measures]\label{thm:ext_Baire_to_Radon}
Let \((\Omega,\tau)\) be a compact Hausdorff space. Fix any \(\mu\in\mathcal M_+(\Omega,\mathscr B_0(\Omega);L^0(\mm))\).
Let us define \(\bar\mu\colon\mathscr B(\Omega)\to L^0_+(\mm)\) as
\[
\bar\mu(B)\coloneqq\bigwedge_{B\subseteq U\in\tau}\bigvee_{U\supseteq A\in\mathscr B_0(\Omega)}\mu(A)
\quad\text{ for every }B\in\mathscr B(\Omega).
\]
Then \(\bar\mu\in\mathfrak M_+(\Omega;L^0(\mm))\). Also, \(\bar\mu\) is the unique Radon \(L^0\)-valued measure on \(\Omega\)
extending \(\mu\).
\end{theorem}
\begin{proof}
First of all, we define the auxiliary set-function \(\mu_*\colon\tau\to L^0_+(\mm)\) as
\[
\mu_*(U)\coloneqq\bigvee\big\{\mu(A)\;\big|\;A\in\mathscr B_0(\Omega),\,A\subseteq U\big\}\quad\text{ for every }U\in\tau,
\]
so \(\bar\mu(B)=\bigwedge\{\mu_*(U):B\subseteq U\in\tau\}\) for every \(B\in\mathscr B(\Omega)\).
We split the proof into many steps.
\smallskip

\textbf{Step 1.} We claim that
\begin{equation}\label{eq:ext_Baire_step1}
\mu_*\bigg(\bigcup_{i\in I}U_i\bigg)=\bigvee_{i\in I}\mu(U_i)\quad\text{ whenever }\{U_i:i\in I\}\subseteq\tau\cap\mathscr B_0(\Omega)
\text{ is directed.}
\end{equation}
Denote \(U\coloneqq\bigcup_{i\in I}U_i\) for brevity. Clearly, we have that \(\bigvee_{i\in I}\mu(U_i)\leq\mu_*(U)\).
Let us prove the converse inequality. The inner regularity of \(\mu\) (cf.\ Proposition \ref{prop:Baire_meas_is_reg})
and the countable sup property yield the existence of a sequence \((K_n)_{n\in\N}\) of compact \(G_\delta\) sets in \(U\)
such that \(\bigvee_{n\in\N}\mu(K_n)=\mu_*(U)\). Since each set \(K_1\cup\ldots\cup K_n\subseteq U\) is compact and \(\{U_i:i\in\N\}\)
is a directed family, we can find indices \(\{i_n:n\in\N\}\subseteq I\) such that \(K_n\subseteq U_{i_n}\) for all \(n\in\N\).
Hence, we have that \(\mu_*(U)=\bigvee_{n\in\N}\mu(U_{i_n})\leq\bigvee_{i\in I}\mu(U_i)\). This proves \eqref{eq:ext_Baire_step1}.
\smallskip

\textbf{Step 2.} We claim that
\begin{equation}\label{eq:ext_Baire_step2}
\mu_*\bigg(\bigcup_{i\in I}U_i\bigg)=\bigvee_{i\in I}\mu_*(U_i)\quad\text{ whenever }\{U_i:i\in I\}\subseteq\tau\text{ is directed.}
\end{equation}
Denote \(U\coloneqq\bigcup_{i\in I}U_i\) for brevity. Clearly, we have that \(\bigvee_{i\in I}\mu_*(U_i)\leq\mu_*(U)\).
To prove the converse inequality, consider the family 
\[
\mathcal W\coloneqq\big\{W\in\tau\cap\mathscr B_0(\Omega)\;\big|\;\text{there exists }i(W)\in I\text{ such that }W\subseteq U_{i(W)}\big\}.
\]
Given any set \(V\in\tau\cap\mathscr B_0(\Omega)\) with \(V\subseteq U\), we have that \(\{V\cap W:W\in\mathcal W\}\subseteq\tau\)
is a directed family of sets whose union is \(V\), thus accordingly \eqref{eq:ext_Baire_step1} ensures that
\[
\mu(V)=\mu_*(V)=\bigvee_{W\in\mathcal W}\mu(V\cap W)\leq\bigvee_{W\in\mathcal W}\mu_*(U_{i(W)})\leq\bigvee_{i\in I}\mu_*(U_i).
\]
Applying again \eqref{eq:ext_Baire_step1}, we then conclude that
\[
\mu_*(U)=\bigvee\big\{\mu(V)\;\big|\;V\in\tau\cap\mathscr B_0(\Omega),\,V\subseteq U\big\}\leq\bigvee_{i\in I}\mu_*(U_i).
\]
This proves \eqref{eq:ext_Baire_step2}.
\smallskip

\textbf{Step 3.} We claim that
\begin{subequations}\begin{align}
\label{eq:ext_Baire_step3a}
\mu_*(U\cup V)\leq\mu_*(U)+\mu_*(V)&\quad\text{ for every }U,V\in\tau,\\
\label{eq:ext_Baire_step3b}
\mu_*(U\cup V)=\mu_*(U)+\mu_*(V)&\quad\text{ for every }U,V\in\tau\text{ with }U\cap V=\varnothing.
\end{align}\end{subequations}
Remark \ref{rmk:open_union_of_funct_open} provides us with two directed families
\(\{U_i:i\in I\},\{V_j:j\in J\}\subseteq\tau\cap\mathscr B_0(\Omega)\) such that \(U=\bigcup_{i\in I}U_i\)
and \(V=\bigcup_{j\in J}V_j\). In particular, \(\{U_i\cup V_j:(i,j)\in I\times J\}\) is a directed family of sets
with union \(U\cup V\). As \(\mu(U_i\cup V_j)\leq\mu(U_i)+\mu(V_j)\leq\mu_*(U)+\mu_*(V)\) for all \((i,j)\in I\times J\),
it follows from \eqref{eq:ext_Baire_step1} that \(\mu_*(U\cup V)\leq\mu_*(U)+\mu_*(V)\), proving \eqref{eq:ext_Baire_step3a}.
Moreover, if we also assume that \(U\cap V=\varnothing\), then \(U_i\cap V_j=\varnothing\) for every \((i,j)\in I\times J\),
thus \eqref{eq:ext_Baire_step1} guarantees that
\[
\mu_*(U)+\mu_*(V)=\bigvee_{i\in I}\mu(U_i)+\bigvee_{j\in J}\mu(V_j)=\bigvee_{(i,j)\in I\times J}\mu(U_i\cup V_j)=\mu_*(U\cup V),
\]
proving \eqref{eq:ext_Baire_step3b}. 
\smallskip

\textbf{Step 4.} Let us define the set-function \(\nu\colon 2^\Omega\to L^0_+(\mm)\) as
\[
\nu(S)\coloneqq\bigwedge\big\{\mu_*(U)\;\big|\;S\subseteq U\in\tau\big\}\quad\text{ for every }S\in 2^\Omega.
\]
We now prove that \(\nu\) is an outer \(L^0\)-valued measure. Trivially,
\(\nu(\varnothing)=\mu_*(\varnothing)=\mu(\varnothing)=0\) and \(\nu(S)\leq\nu(T)\) whenever \(S\subseteq T\subseteq\Omega\),
thus it suffices to check that \(\nu\) is \(\sigma\)-subadditive. To this aim, fix a sequence \((S_n)_{n\in\N}\) of subsets
of \(\Omega\). Fix any \(k\in\N\). For any \(n\in\N\) we can find a set \(U^k_n\in\tau\) such that \(S_n\subseteq U^k_n\)
and \(\sfd_{L^0(\mm)}(\mu_*(U^k_n),\nu(S_n))\leq(k 2^n)^{-1}\). In particular, the element
\(r_k\coloneqq\sum_{n\in\N}\mu_*(U^k_n)-\nu(S_n)\in L^0_+(\mm)\) satisfies \(\sfd_{L^0(\mm)}(r_k,0)\leq 1/k\).
Letting \(\mathcal F_k\) be the family of all finite unions of elements of \(\{U^k_n:n\in\N\}\), we have that
\(\mathcal F_k\) is a directed family of sets such that \(\bigcup\mathcal F_k=\bigcup_{n\in\N}U^k_n\). Note also that
\eqref{eq:ext_Baire_step3a} implies that \(\mu_*(U)\leq\sum_{n\in\N}\mu_*(U^k_n)\) for every \(U\in\mathcal F_k\).
Therefore, taking also \eqref{eq:ext_Baire_step2} and the fact that \(\bigcup_{n\in\N}S_n\subseteq\bigcup_{n\in\N}U^k_n\)
into account, we obtain that
\[
\nu\bigg(\bigcup_{n\in\N}S_n\bigg)\leq\mu_*\bigg(\bigcup_{n\in\N}U^k_n\bigg)=\mu_*\bigg(\bigcup_{U\in\mathcal F_k}U\bigg)
=\bigvee_{U\in\mathcal F_k}\mu_*(U)\leq\sum_{n\in\N}\mu_*(U^k_n)\leq r_k+\sum_{n\in\N}\nu(S_n).
\]
Letting \(k\to\infty\), we get \(\nu\big(\bigcup_{n\in\N}S_n\big)\leq\sum_{n\in\N}\nu(S_n)\),
thus \(\nu\) is an outer \(L^0\)-valued measure.
\smallskip

\textbf{Step 5.} We claim that
\[
\mathscr B(\Omega)\subseteq\Gamma(\nu).
\]
Given that \(\Gamma(\nu)\) is a \(\sigma\)-algebra by Theorem \ref{thm:Carath}, it suffices to check that
\(U\in\Gamma(\nu)\) for every \(U\in\tau\). To this aim, we need to show that
\[
\nu(S\cap U)+\nu(S\setminus U)\leq\nu(S)\quad\text{ for every }S\in 2^\Omega.
\]
By the complete regularity of \((\Omega,\tau)\), we find a directed family of sets \(\{U_i:i\in I\}\subseteq\tau\)
with \(\bigcup_{i\in I}U_i=U\) such that the closure \(C_i\) of \(U_i\) is contained in \(U\). Given any \(\varepsilon>0\),
we find \(V_\varepsilon\in\tau\) such that \(S\subseteq V_\varepsilon\) and
\(\sfd_{L^0(\mm)}(\mu_*(V_\varepsilon),\nu(S))\leq\varepsilon\). Since \(V_\varepsilon\setminus C_i\in\tau\) and
\((V_\varepsilon\cap U_i)\cap(V_\varepsilon\setminus C_i)=\varnothing\) for every \(i\in I\), it follows from
\eqref{eq:ext_Baire_step3b} that
\[\begin{split}
\mu_*(V_\varepsilon\cap U_i)+\nu(V_\varepsilon\setminus U)&\leq\mu_*(V_\varepsilon\cap U_i)+\mu_*(V_\varepsilon\setminus C_i)
=\mu_*(V_\varepsilon\cap(C_i\setminus U_i))\leq\mu_*(V_\varepsilon)\\
&\leq\nu(S)+r_\varepsilon
\end{split}\]
for every \(i\in I\), where we set \(r_\varepsilon\coloneqq\mu_*(V_\varepsilon)-\nu(S)\) for brevity.
Passing to the supremum with respect to \(i\in I\) and using \eqref{eq:ext_Baire_step2}, we deduce that
\[
\nu(S\cap U)+\nu(S\setminus U)\leq\nu(V_\varepsilon\cap U)+\nu(V_\varepsilon\setminus U)
=\bigvee_{i\in I}\mu_*(V_\varepsilon\cap U_i)+\nu(V_\varepsilon\setminus U)\leq\nu(S)+r_\varepsilon.
\]
Since \(\sfd_{L^0(\mm)}(r_\varepsilon,0)\leq\varepsilon\), by letting \(\varepsilon\to 0\) we conclude that
\(\nu(S\cap U)+\nu(S\setminus U)\leq\nu(S)\).
\smallskip

\textbf{Step 6.} The last part of the statement of Theorem \ref{thm:Carath} implies that
\[
\bar\mu=\nu|_{\mathscr B(\Omega)}\in\mathcal M_+(\Omega;L^0(\mm)).
\]
For any \(B\in\mathscr B(\Omega)\) we have \(\bar\mu(B)=\bigwedge\{\mu_*(U):U\in\tau,\,B\subseteq U\}
=\bigwedge\{\bar\mu(U):U\in\tau,\,B\subseteq U\}\), so that \(\bar\mu\) is outer regular. Since every
closed subset of \(\Omega\) is compact, this is equivalent to saying that \(\bar\mu\) is inner regular, thus
\(\bar\mu\in\mathfrak M_+(\Omega;L^0(\mm))\) by Lemma \ref{lem:inner_reg_self-improv}. Finally, since
\[
\bar\mu(U)=\nu(U)=\mu_*(U)=\mu(U)\quad\text{ for every functionally-open set }U\subseteq\Omega,
\]
and the collection of all functionally-open subsets of \(\Omega\) is a basis for \(\tau\) (Remark \ref{rmk:open_union_of_funct_open})
that is closed under finite intersections, \(\bar\mu\) is uniquely determined by \(\mu\) thanks
to Corollary \ref{cor:agree_on_good_basis}.
\end{proof}
\section{Module-valued Bochner integration}\label{s:M-val_int}
Let \((\X,\Sigma,\mm)\) be a probability space. Let \(({\rm M},|\cdot|)\) be a complete random normed module
with base \((\X,\Sigma,\mm)\). Unless otherwise specified, by \((\Omega,\mathcal A)\) we mean an arbitrary measurable space.
\subsection{Definitions and basic properties}\label{s:L1_to_M}
Given any map \(v\colon\Omega\to{\rm M}\), we denote by
\[
v_p\coloneqq v(p)\in{\rm M}
\]
the value of \(v\) at the point \(p\in\Omega\).
\begin{definition}[Module-valued simple map]\label{defMappesempliciM}
We say that a given map \(v\colon\Omega\to{\rm M}\) is a \textbf{module-valued simple map} if there exist an at most countable
partition \((A_i)_{i\in I}\subseteq\mathcal A\) of \(\Omega\) and a set of elements \((\bar v^i)_{i\in I}\subseteq{\rm M}\)
such that 
\[
v_p=\bar v^i\quad\text{ for every }i\in I\text{ and }p\in A_i.
\]
We write \(v=\sum_{i\in I}\1_{A_i}\bar v^i\). We denote by \(S(\Omega;{\rm M})\) the space of all module-valued simple maps.
\end{definition}

If \(\bar v^{i_0}=0\) for some \(i_0\in I\), then we occasionally write
\(\sum_{i\in I\setminus\{i_0\}}\1_{A_i}\bar v^i\) instead of \(\sum_{i\in I}\1_{A_i}\bar v^i\).
It will be convenient to consider module-valued simple maps up to \(\mu\)-null sets, for some given \(L^0\)-valued
measure \(\mu\in\mathcal M(\Omega;L^0(\mm))\): for any \(v,w\in S(\Omega;{\rm M})\), we declare that \(v\sim_\mu w\)
if and only if \(\{p\in\Omega:v_p\neq w_p\}\in\mathcal N_\mu\). Observe that if
\(\sum_{i\in I}\1_{A_i}\bar v^i,\sum_{j\in J}\1_{B_j}\bar w^j\in S(\Omega;{\rm M})\) are written in such a way that
\((\bar v^i)_{i\in I}\), \((\bar w^j)_{j\in J}\) are collections of pairwise distinct elements, then it is clear
that \(\sum_{i\in I}\1_{A_i}\bar v^i\sim_\mu\sum_{j\in J}\1_{B_j}\bar w^j\) if and only if there exists a bijective map
\(\phi\colon I\to J\) such that \(A_i\Delta B_{\phi(i)}\in\mathcal N_\mu\) and \(\bar v^i=\bar w^{\phi(i)}\)
for every \(i\in I\). Let us define the space \(S_\mu(\Omega;{\rm M})\) as
\begin{equation}\label{eq:def_S_mu}
S_\mu(\Omega;{\rm M})\coloneqq S(\Omega;{\rm M})/\sim_\mu.
\end{equation}
We denote by \(\sum_{i\in I}\1_{A_i}^\mu\bar v^i\in S_\mu(\Omega;{\rm M})\) the equivalence class of
\(\sum_{i\in I}\1_{A_i}\bar v^i\in S(\Omega;{\rm M})\). Note that
\[
S_\mu(\Omega;{\rm M})=S_{|\mu|}(\Omega;{\rm M})\quad\text{ for every }\mu\in\mathcal M(\Omega;L^0(\mm)).
\]
The space \(S_\mu(\Omega;{\rm M})\) is a module over \(L^0(\mm)\) with respect to the following operations:
\[\begin{split}
\sum_{i\in I}\1_{A_i}^\mu\bar v^i+\sum_{j\in J}\1_{B_j}^\mu\bar w^j&\coloneqq
\sum_{(i,j)\in I\times J}\1_{A_i\cap B_j}^\mu(\bar v^i+\bar w^j),\\
f\cdot\sum_{i\in I}\1_{A_i}^\mu\bar v^i&\coloneqq\sum_{i\in I}\1_{A_i}^\mu(f\cdot\bar v^i)
\end{split}\]
for every \(\sum_{i\in I}\1_{A_i}^\mu\bar v^i,\sum_{j\in J}\1_{B_j}^\mu\bar w^j\in S_\mu(\Omega;{\rm M})\) and \(f\in L^0(\mm)\).
Likewise, \(S(\Omega;{\rm M})\) is a module over \(L^0(\mm)\) with respect to the operations defined in an analogous way.
\begin{definition}[Integral of a module-valued simple map]\label{def:int-mod-simp-map}
Fix any \(\mu\in\mathcal M(\Omega;L^0(\mm))\). Then we say that a map \(v=\sum_{i\in I}\1_{A_i}^\mu\bar v^i\in S_\mu(\Omega;{\rm M})\)
is \textbf{\(\mu\)-integrable} if \(\{|\mu|(A_i)|\bar v^i|:i\in I\}\subseteq L^0(\mm)\) is summable. Moreover, we define the \textbf{\(\mu\)-integral}
of \(v\) as
\[
\int v\cdot\d\mu\coloneqq\sum_{i\in I}\mu(A_i)\cdot\bar v^i\in{\rm M}.
\]
We denote by \(S_\mu^1(\Omega;{\rm M})\) the space of all \(\mu\)-integrable module-valued simple maps from \(\Omega\) to \({\rm M}\).
\end{definition}

Given any \(f\in S_\mu^1(\Omega;L^0(\mm))\), we write \(\int f\,\d\mu\) instead of \(\int f\cdot\,\d\mu\).
Some comments about the above-defined notions are in order:
\begin{itemize}
\item The summability of \(\{|\mu|(A_i)|\bar v^i|:i\in I\}\subseteq L^0(\mm)\) implies that \(\{\mu(A_i)\cdot\bar v^i:i\in I\}\subseteq{\rm M}\)
is summable (by Remark \ref{rmk:fact_about_summability} vi)) and it can be readily checked that \(\sum_{i\in I}\mu(A_i)\cdot\bar v^i\in{\rm M}\) is independent
of the specific way in which \(\sum_{i\in I}\1_{A_i}^\mu\bar v^i\) is written, thus the above definition of \(\int v\cdot\d\mu\) is well posed. Moreover, we have
\begin{equation}\label{eq:main_ineq_Boch_int_simple}
\bigg|\int v\cdot\d\mu\bigg|\leq\sum_{i\in I}|\mu|(A_i)|\bar v^i|\quad\text{ for every }v=\sum_{i\in I}\1_{A_i}^\mu\bar v^i\in S_\mu^1(\Omega;{\rm M}).
\end{equation}
\item The space \(S_\mu^1(\Omega;{\rm M})\) is an \(L^0(\mm)\)-submodule of \(S_\mu(\Omega;{\rm M})\) and \(S_\mu^1(\Omega;{\rm M})=S_{|\mu|}^1(\Omega;{\rm M})\).
\item The operator \(S_\mu^1(\Omega;{\rm M})\ni v\mapsto\int v\cdot\d\mu\in{\rm M}\) is \(L^0(\mm)\)-linear.
\item We denote
\[
|v|\coloneqq\sum_{i\in I}\1_{A_i}^\mu|\bar v^i|\in S_\mu(\Omega;L^0(\mm))
\quad\text{ for every }v=\sum_{i\in I}\1_{A_i}^\mu\bar v^i\in S_\mu(\Omega;{\rm M}).
\]
It can be easily checked that \(v\in S_\mu^1(\Omega;{\rm M})\) if and only if \(|v|\in S_\mu^1(\Omega;L^0(\mm))\).
\item Consequently, it makes sense to define \(|\cdot|^{1,\mu}\colon S_\mu^1(\Omega;{\rm M})\to L^0_+(\mm) \) as
\[
|v|^{1,\mu}\coloneqq\int|v|\,\d|\mu|\quad\text{ for every }v\in S_\mu^1(\Omega;{\rm M}).
\]
It is easy to check that \((S_\mu^1(\Omega;{\rm M}),|\cdot|^{1,\mu})\) is a random normed module with base
\((\X,\Sigma,\mm)\). Moreover, the inequality in \eqref{eq:main_ineq_Boch_int_simple} can be rewritten as
\begin{equation}\label{eq:main_ineq_Boch_int_simple_bis}
\bigg|\int v\cdot\d\mu\bigg|\leq|v|^{1,\mu}\quad\text{ for every }v\in S_\mu^1(\Omega;{\rm M}).
\end{equation}
\end{itemize}
\begin{definition}[Module-valued Lebesgue--Bochner space]\label{def:M-val_Leb-Boch}
Let \(\mu\in\mathcal M(\Omega;L^0(\mm))\) be given. Then we define the \textbf{module-valued Lebesgue--Bochner space}
\[
(L_\mu^1(\Omega;{\rm M}),|\cdot|^{1,\mu})
\]
as the random-normed-module completion of \((S_\mu^1(\Omega;{\rm M}),|\cdot|^{1,\mu})\).
\end{definition}

The operator \(S_\mu^1(\Omega;{\rm M})\ni v\mapsto\int v\cdot\,\d\mu\in{\rm M}\) is \(L^0(\mm)\)-linear and satisfies
\eqref{eq:main_ineq_Boch_int_simple_bis}, thus (by Corollary \ref{cor:ext_hom_mod}) it can be uniquely extended to a
homomorphism of random normed modules
\[
L_\mu^1(\Omega;{\rm M})\ni v\mapsto\int v\cdot\d\mu\in{\rm M}
\]
satisfying \(\big|\int v\cdot\d\mu\big|\leq|v|^{1,\mu}\) for every \(v\in L_\mu^1(\Omega;{\rm M})\).
\begin{remark}\label{rmk:density_finite_simple_maps}{\rm
We denote by \(S_{\mu,f}(\Omega;{\rm M})\) the set of \emph{finite-range} module-valued simple maps, i.e.
\begin{equation}\label{funsemplicisomsufinit}
    S_{\mu,f}(\Omega;{\rm M})\coloneqq\bigg\{\sum_{i\in F}\1_{A_i}^\mu\bar v^i\in S_\mu(\Omega;{\rm M})
\;\bigg|\;F\text{ is a finite set}\bigg\}.
\end{equation}
We claim that \(S_{\mu,f}(\Omega;{\rm M})\) is dense in \(L_\mu^1(\Omega;{\rm M})\). The claim follows by showing that
 \(S_{\mu,f}(\Omega;{\rm M})\) is dense in \((S_\mu(\Omega;{\rm M}),|\cdot|^{1,\mu})\). To this aim,
fix any element \(v=\sum_{i=1}^\infty\1_{A_i}^\mu\bar v^i\in S_\mu(\Omega;{\rm M})\). Then
\[
S_{\mu,f}(\Omega;{\rm M})\ni\sum_{i=1}^n\1_{A_i}^\mu\bar v^i\to v
\quad\text{ with respect to }|\cdot|^{1,\mu}\text{ as }n\to\infty,
\]
thus proving the claim.
\fr}\end{remark}

We define the subset \(L^1_\mu(\Omega;L^0_+(\mm))\) of \(L^1_\mu(\Omega;L^0(\mm))\) as
\begin{equation}\label{eq:non-neg_L1}
L^1_\mu(\Omega;L^0_+(\mm))\coloneqq{\rm cl}_{L^1_\mu(\Omega;L^0(\mm))}\big(S_{\mu,f}(\Omega;L^0_+(\mm))\big),
\end{equation}
where the subset \(S_{\mu,f}(\Omega;L^0_+(\mm))\) of \(S_{\mu,f}(\Omega;L^0(\mm))\) is given by
\[
S_{\mu,f}(\Omega;L^0_+(\mm))\coloneqq\bigg\{\sum_{i\in F}\1_{A_i}^\mu\bar f^i\in S_{\mu,f}(\Omega;L^0(\mm))\;
\bigg|\;\bar f^i\in L^0_+(\mm)\text{ for every }i\in F\bigg\}.
\]

Perhaps surprisingly, it is currently unclear to us whether (all) the elements of \(L_\mu^1(\Omega;{\rm M})\) -- which are
defined via an abstract completion procedure -- are given by \(\mu\)-a.e.\ equivalence classes of maps from \(\Omega\)
to \({\rm M}\); however, as we will see in Proposition \ref{prop:L_infty_in_L_1}, we can at least identify a subspace of
\(L^1_\mu(\Omega;{\rm M})\)
whose elements can be represented by maps. Furthermore, when \(\mu\) can be foliated, the space
\(L^1_\mu(\Omega;{\rm M})\) can be described as a space of (equivalence classes of) maps; cf.\ Section \ref{s:ptwse_descr_L_1}.
\begin{lemma}\label{lem:full_formula_pushforward}
Let \((\tilde\Omega,\tilde{\mathcal A})\) be a measurable space and \(\varphi\colon\Omega\to\tilde\Omega\)
a measurable map. Then there exists a unique homomorphism of random normed modules
\begin{equation}\label{eq:pushforward_L1}
\varphi^*\colon L^1_{\varphi_\#\mu}(\tilde\Omega;{\rm M})\to L^1_\mu(\Omega;{\rm M})
\end{equation}
such that \(\varphi^*(\1_{\tilde A}\bar v)=\1_{\varphi^{-1}(\tilde A)}\bar v\) for all \(\tilde A\in\tilde{\mathcal A}\)
and \(\bar v\in{\rm M}\). Moreover, \(|\varphi^*v|^{1,\mu}\leq 2|v|^{1,\varphi_\#\mu}\) and
\begin{equation}\label{eq:full_formula_pushforward}
\int\varphi^*v\cdot\d\mu=\int v\cdot\d(\varphi_\#\mu)\quad\text{ for every }v\in L^1_{\varphi_\#\mu}(\tilde\Omega;{\rm M}).
\end{equation}
\end{lemma}
\begin{proof}
First, we define \(\varphi^*\colon S_{\varphi_\#\mu,f}(\tilde\Omega;{\rm M})\to S_{\mu,f}(\Omega;{\rm M})\) as
\[
\varphi^*v\coloneqq\sum_{i\in F}\1_{\varphi^{-1}(\tilde A_i)}^\mu\bar v^i\quad\text{ for every }
v=\sum_{i\in F}\1_{\tilde A_i}^{\varphi_\#\mu}\bar v^i\in S_{\varphi_\#\mu,f}(\tilde\Omega;{\rm M}).
\]
The well-posedness of the definition of \(\varphi^*\) stems from the estimates
\[
\bigg|\sum_{i\in F}\1_{\varphi^{-1}(\tilde A_i)}^\mu\bar v^i\bigg|^{1,\mu}=\sum_{i\in F}\varphi_\#|\mu|(\tilde A_i)|\bar v^i|
\leq 2\sum_{i\in F}|\varphi_\#\mu|(\tilde A_i)|\bar v^i|=2\bigg|\sum_{i\in F}\1_{\tilde A_i}^{\varphi_\#\mu}\bar v^i\bigg|^{1,\varphi_\#\mu},
\]
where the inequality follows from \eqref{eq:ineq_pushforward_as_meas}. Hence, by the density of
\(S_{\varphi_\#\mu,f}(\tilde\Omega;{\rm M})\) in \(L^1_{\varphi_\#\mu}(\tilde\Omega;{\rm M})\),
\(\varphi^*\) can be uniquely extended to a homomorphism as in \eqref{eq:pushforward_L1} satisfying
\(|\varphi^*v|^{1,\mu}\leq 2|v|^{1,\varphi_\#\mu}\) for every \(v\in L^1_{\varphi_\#\mu}(\tilde\Omega;{\rm M})\).
Finally, for any \(v=\sum_{i\in F}\1_{\tilde A_i}^{\varphi_\#\mu}\bar v^i\in S_{\varphi_\#\mu,f}(\tilde\Omega;{\rm M})\)
we can compute
\[
\int\varphi^*v\cdot\,\d\mu=\sum_{i\in F}\mu(\varphi^{-1}(\tilde A_i))\cdot\bar v^i=
\sum_{i\in F}\varphi_\#\mu(\tilde A_i)\cdot\bar v^i=\int v\cdot\d(\varphi_\#\mu),
\]
which gives \eqref{eq:full_formula_pushforward} for every \(v\in S_{\varphi_\#\mu,f}(\tilde\Omega;{\rm M})\).
The general case follows by approximation.
\end{proof}
\subsection{The space \texorpdfstring{\(L^\infty_\mu(\Omega;{\rm M})\)}{Linftymu}}\label{s:Linfty_to_M}
We denote by \(\mathcal L^0(\Omega;{\rm M})\) the collection of all measurable maps
\[
v\colon(\Omega,\mathcal A)\to{\rm M},
\]
where the target space \({\rm M}\) is equipped with the Borel \(\sigma\)-algebra induced by the distance \(\sfd_{\rm M}\). Clearly,
\(S(\Omega;{\rm M})\subseteq\mathcal L^0(\Omega;{\rm M})\). To any \(v\colon\Omega\to{\rm M}\) we associate the
map \(|v|\colon\Omega\to L^0_+(\mm)\) given by
\[
|v|_p\coloneqq|v_p|\in L^0_+(\mm)\quad\text{ for every }p\in\Omega.
\]
Since the \(L^0\)-norm \(|\cdot|\colon({\rm M},\sfd_{\rm M})\to(L^0(\mm),\sfd_{L^0(\mm)})\) is \(1\)-Lipschitz, we have in particular that
\[
|v|\in\mathcal L^0(\Omega;L^0(\mm))\quad\text{ for every }v\in\mathcal L^0(\Omega;{\rm M}).
\]
Given \(v,w\in\mathcal L^0(\Omega;{\rm M})\), we declare that \(v\sim_\mu w\) if and only if \(\{p\in\Omega:v_p\neq w_p\}\in\mathcal N_\mu\),
thus obtaining an equivalence relation \(\sim_\mu\) whose quotient space we denote by
\begin{equation}\label{eq:def_L0_mu}
L^0_\mu(\Omega;{\rm M})\coloneqq\mathcal L^0(\Omega;{\rm M})/\sim_\mu.
\end{equation}
Note that \(S_\mu(\Omega;{\rm M})\subseteq L^0_\mu(\Omega;{\rm M})\). For any \(v\in L^0_\mu(\Omega;{\rm M})\), we denote by
\(|v|\in L^0_\mu(\Omega;L^0(\mm))\) the equivalence class of \(|\bar v|\in\mathcal L^0(\Omega;L^0(\mm))\), where \(\bar v\in\mathcal L^0(\Omega;{\rm M})\)
is any representative of \(v\).
\medskip

Given any map \(v\in\mathcal L^0(\Omega;{\rm M})\), we define \(|v|^{\infty,\mu}\in\bar L^0_+(\mm)\) as
\begin{equation}\label{defNormaLinfmu} |v|^{\infty,\mu}\coloneqq\bigwedge_{N\in\mathcal N_\mu}\bigvee_{p\in\Omega\setminus N}|v_p|.   
\end{equation}
Clearly, \(|v|^{\infty,\mu}\) is invariant under \(\mu\)-a.e.\ modifications of \(v\), meaning that \(|v|^{\infty,\mu}=|w|^{\infty,\mu}\) when
\(v,w\in\mathcal L^0(\Omega;{\rm M})\) satisfy \(v\sim_\mu w\). Accordingly, we can unambiguously write \(|v|^{\infty,\mu}\in\bar L^0_+(\mm)\) for any
\(v\in L^0_\mu(\Omega;{\rm M})\). Clearly, we have \(\big||v|\big|^{\infty,\mu}=|v|^{\infty,\mu}\) for every \(v\in L^0_\mu(\Omega;{\rm M})\).
\begin{lemma}\label{lem:big_space_bdd_maps}
The space \(\big(\{v\in L^0_\mu(\Omega;{\rm M}):|v|^{\infty,\mu}\in L^0_+(\mm)\},|\cdot|^{\infty,\mu}\big)\)
is a complete random normed module with base \((\X,\Sigma,\mm)\).
\end{lemma}
\begin{proof}
We check only completeness. Fix any \((v^n)_{n\in\N}\subseteq L^0_\mu(\Omega;{\rm M})\) with
\(|v^n|^{\infty,\mu}\in L^0_+(\mm)\) for every \(n\in\N\) and \(|v^n-v^m|^{\infty,\mu}\to 0\) in \(L^0(\mm)\)
as \(n,m\to\infty\). By the definition \eqref{defNormaLinfmu} of \(|\cdot|^{\infty,\mu}\), after having fixed
\((\bar v^n)_{n\in\N}\subseteq\mathcal L^0(\Omega;{\rm M})\) so that \(\bar v^n\) is a representative of \(v^n\)
for every \(n\in\N\), we can find a set \(N\in\mathcal N_\mu\) such that
\(\bigvee_{p\in\Omega\setminus N}|\bar v^n_p-\bar v^m_p|\to 0\) in \(L^0(\mm)\) as \(n,m\to\infty\). In particular,
for any \(p\in\Omega\setminus N\) the sequence \((\bar v^n_p)_{n\in\N}\subseteq{\rm M}\) is Cauchy, thus it makes
sense to define \(\bar v\colon\Omega\to{\rm M}\) as
\[
\bar v_p\coloneqq\left\{\begin{array}{ll}
\lim_n\bar v^n_p\\
0
\end{array}\quad\begin{array}{ll}
\text{ for every }p\in\Omega\setminus N,\\
\text{ for every }p\in N.
\end{array}\right.
\]
Note that \(\bar v\in\mathcal L^0(\Omega;{\rm M})\) by construction. Up to a non-relabelled subsequence,
we can also assume that
\begin{equation}\label{eq:aux_infty_mu_complete}
\int\sum_{n\in\N}g_n\wedge 1\,\d\mm=\sum_{n\in\N}\int g_n\wedge 1\,\d\mm<+\infty,\quad
\text{ where we set }g_n\coloneqq\bigvee_{p\in\Omega\setminus N}|\bar v^n_p-\bar v^{n+1}_p|.
\end{equation}
Now, fix any point \(p\in\Omega\setminus N\). We can extract a subsequence \((n_k)_{k\in\N}\), depending on \(p\), such that
\(n_1=1\) and \(|\bar v^{n_k}_p-\bar v_p|\to 0\) in the \(\mm\)-a.e.\ sense as \(k\to\infty\). Then by letting \(k\to\infty\)
in the estimates
\[
|\bar v_p|\leq|\bar v^1_p|+\sum_{i=1}^{k-1}|\bar v^{n_i}_p-\bar v^{n_{i+1}}_p|+|\bar v^{n_k}_p-\bar v_p|
\leq|\bar v^1|^{\infty,\mu}+\sum_{n\in\N}g_n+|\bar v^{n_k}_p-\bar v_p|
\]
we obtain that \(|\bar v_p|\leq|\bar v^1|^{\infty,\mu}+\sum_{n\in\N}g_n\), thus
\(|\bar v|^{\infty,\mu}\leq|\bar v^1|^{\infty,\mu}+\sum_{n\in\N}g_n\). Since
\[
\sum_{n\in\N}g_n(x)\wedge 1<+\infty\quad\text{ for }\mm\text{-a.e.\ }x\in\X
\]
by \eqref{eq:aux_infty_mu_complete}, we can find a partition \((E_j)_{j\in\N}\subseteq\Sigma\)
of \(\X\) such that \(\sup_{n>j}g_n(x)\leq 1\) holds for \(\mm\)-a.e.\ \(x\in E_j\). Denoting
\(g\coloneqq\sum_{n\in\N}g_n\wedge 1\in L^1(\mm)\) for brevity, we deduce that
\[
|\bar v|^{\infty,\mu}\leq|\bar v^1|^{\infty,\mu}+\sum_{j\in\N}\1_{E_j}^\mm(g_1+\ldots+g_j+g)\in L^0_+(\mm),
\]
in particular \(|\bar v|^{\infty,\mu}\in L^0_+(\mm)\). Finally, for any \(n,m\in\N\) with \(n<m\) and \(p\in\Omega\setminus N\),
we have
\[
|\bar v_p-\bar v^n_p|\leq|\bar v_p-\bar v^m_p|+\sum_{i=n}^{m-1}|\bar v^{i+1}_p-\bar v^i_p|
\leq|\bar v_p-\bar v^m_p|+\sum_{i\geq n}g_i,
\]
which implies that \(|\bar v_p-\bar v^n_p|\leq\sum_{i\geq n}g_i\) and thus \(|\bar v-\bar v^n|^{\infty,\mu}\leq\sum_{i\geq n}g_i\).
Arguing as above, one can show that \(\sum_{i\geq n}g_i\to 0\) in \(L^0(\mm)\) as \(n\to\infty\), whence it follows
that \(|\bar v-\bar v^n|^{\infty,\mu}\to 0\) in \(L^0(\mm)\). Letting \(v\in L^0_\mu(\Omega;{\rm M})\) be the equivalence
class of \(\bar v\), we conclude that \(|v-v^n|^{\infty,\mu}\to 0\) in \(L^\infty(\mm)\). This yields the sought completeness.
\end{proof}

Let us next define
\begin{equation}\label{eq:def_Sinfty_mu}
S_\mu^\infty(\Omega;{\rm M})\coloneqq\big\{v\in S_\mu(\Omega;{\rm M})\;\big|\;|v|^{\infty,\mu}\in L^0_+(\mm)\big\}.
\end{equation}
Note that \(|v|^{\infty,\mu}=\bigvee\{|\bar v_i|:i\in I,A_i\notin\mathcal N_\mu\}\) for every \(v=\sum_{i\in I}\1_{A_i}^\mu\bar v_i\in S_\mu(\Omega;{\rm M})\),
thus accordingly \(v\in S_\mu^\infty(\Omega;{\rm M})\) if and only if \(\bigvee\{|\bar v_i|:i\in I,A_i\notin\mathcal N_\mu\}\in L^0_+(\mm)\).
Now, we define
\begin{equation}\label{defLinfmuOmM}
    (L_\mu^\infty(\Omega;{\rm M}),|\cdot|^{\infty,\mu})
\end{equation}
as the closure of \(S_\mu^\infty(\Omega;{\rm M})\) in the random normed module considered in
Lemma \ref{lem:big_space_bdd_maps}.
In particular, \(L_\mu^\infty(\Omega;{\rm M})\) is a complete random normed module.
Note that \(S_{\mu,f}(\Omega;{\rm M})\subseteq S_\mu^\infty(\Omega;{\rm M})\).
\begin{proposition}\label{prop:L_infty_in_L_1}
Let \(\mu\in\mathcal M(\Omega;L^0(\mm))\) be given. Then there exists a unique homomorphism of random normed modules
\[
\iota\colon L^\infty_\mu(\Omega;{\rm M})\to L^1_\mu(\Omega;{\rm M})
\]
such that \(\iota(v)=v\) for every \(v\in S_\mu^\infty(\Omega;{\rm M})\). Moreover, it holds that
\begin{equation}\label{eq:ineq_infty_1}
|\iota(v)|^{1,\mu}\leq|\mu|_{\rm TV}|v|^{\infty,\mu}\quad\text{ for every }v\in L^\infty_\mu(\Omega;{\rm M}).
\end{equation}
\end{proposition}
\begin{proof}
Note that for any given element \(v=\sum_{i\in I}\1_{A_i}^\mu\bar v_i\in S_\mu^\infty(\Omega;{\rm M})\) we can estimate
\[
|v|^{1,\mu}=\int|v|\,\d|\mu|=\sum_{i\in I}|\mu|(A_i)|\bar v_i|\leq|v|^{\infty,\mu}\sum_{i\in I}|\mu|(A_i)=|\mu|_{\rm TV}|v|^{\infty,\mu}.
\]
Thanks to the density of \(S_\mu^\infty(\Omega;{\rm M})\) in \(L^\infty_\mu(\Omega;{\rm M})\), the statement follows.
\end{proof}

With an abuse of notation, for any \(v\in L^\infty_\mu(\Omega;{\rm M})\) we write
\(\int v\cdot\,\d\mu\) in place of \(\int\iota(v)\cdot\d\mu\).
\begin{proposition}\label{prop:mult_Linfty_L1}
Fix any \(\mu\in\mathcal M(\Omega;L^0(\mm))\). Then there exists a unique \(L^0(\mm)\)-bilinear map
\[
L^\infty_\mu(\Omega;L^0(\mm))\times L^1_\mu(\Omega;{\rm M})\ni(g,v)\mapsto g\cdot v\in L^1_\mu(\Omega;{\rm M})
\]
such that \(|g\cdot v|^{1,\mu}\leq|g|^{\infty,\mu}|v|^{1,\mu}\) for every
\((g,v)\in L^\infty_\mu(\Omega;L^0(\mm))\times L^1_\mu(\Omega;{\rm M})\) and
\begin{equation}\label{eq:def_mult_Linfty_L1}
\bigg(\sum_{i\in I}\1_{A_i}^\mu\bar g^i\bigg)\cdot\bigg(\sum_{j\in J}\1_{B_j}^\mu\bar v^j\bigg)
=\sum_{(i,j)\in I\times J}\1_{A_i\cap B_j}^\mu(\bar g^i\cdot\bar v^j)\in S_\mu^1(\Omega;{\rm M})
\end{equation}
for every \(\sum_{i\in I}\1_{A_i}^\mu\bar g^i\in S_\mu^\infty(\Omega;L^0(\mm))\) and \(\sum_{j\in J}\1_{B_j}^\mu\bar v^j\in S_\mu^1(\Omega;{\rm M})\).
\end{proposition}
\begin{proof}
For \(g=\sum_{i\in I}\1_{A_i}^\mu\bar g^i\in S_\mu^\infty(\Omega;L^0(\mm))\) and \(v=\sum_{j\in J}\1_{B_j}^\mu\bar v^j\in S_\mu^1(\Omega;{\rm M})\),
we can estimate
\[\begin{split}
\bigg|\sum_{(i,j)\in I\times J}\1_{A_i\cap B_j}^\mu(\bar g^i\cdot\bar v^j)\bigg|^{1,\mu}
&=\sum_{(i,j)\in I\times J}|\mu|(A_i\cap B_j)|\bar g^i\cdot\bar v^j|
=\sum_{(i,j)\in I\times J}|\mu|(A_i\cap B_j)|\bar g^i||\bar v^j|\\
&\leq|g|^{\infty,\mu}\sum_{(i,j)\in I\times J}|\mu|(A_i\cap B_j)|\bar v^j|
=|g|^{\infty,\mu}|v|^{1,\mu}.
\end{split}\]
This shows that the definition in \eqref{eq:def_mult_Linfty_L1} is well posed and that the resulting map
\begin{equation}\label{eq:mult_Linfty_L1_aux}
S_\mu^\infty(\Omega;L^0(\mm))\times S_\mu^1(\Omega;{\rm M})\ni(g,v)\mapsto g\cdot v\in S_\mu^1(\Omega;{\rm M}),
\end{equation}
which is \(L^0(\mm)\)-bilinear by construction, satisfies \(|g\cdot v|^{1,\mu}\leq|g|^{\infty,\mu}|v|^{1,\mu}\)
for all \(v\in S_\mu^1(\Omega;{\rm M})\) and \(g\in S_\mu^\infty(\Omega;L^0(\mm))\). Since \(S_\mu^\infty(\Omega;L^0(\mm))\)
and \(S_\mu^1(\Omega;{\rm M})\) are dense in \(L^\infty_\mu(\Omega;L^0(\mm))\) and \(L^1_\mu(\Omega;{\rm M})\),
respectively, we conclude that the map in \eqref{eq:mult_Linfty_L1_aux} can be uniquely extended to a map
\(L^\infty_\mu(\Omega;L^0(\mm))\times L^1_\mu(\Omega;{\rm M})\ni(g,v)\mapsto g\cdot v\in L^1_\mu(\Omega;{\rm M})\)
having the sought properties.
\end{proof}

When \({\rm M}=L^0(\mm)\) and \((f,g)\in L^\infty_\mu(\Omega;L^0(\mm))\times L^1_\mu(\Omega;L^0(\mm))\), we write
\(fg\in L^1_\mu(\Omega;L^0(\mm))\) instead of \(f\cdot g\). In view of Proposition \ref{prop:mult_Linfty_L1}, it also
makes sense to define
\[
\int_A v\cdot\d\mu\coloneqq\int((\1_A^\mu\1_\X^\mm)\cdot v)\cdot\d\mu\in{\rm M}\quad\text{ for every }A\in\mathcal A
\text{ and }v\in L^1_\mu(\Omega;{\rm M}).
\]
\begin{lemma}\label{lem:formula_norm_1_mu}
Let \(\mu\in\mathcal M_+(\Omega;L^0(\mm))\) be given. Then for any \(g\in L^1_\mu(\Omega;L^0(\mm))\) it holds that
\begin{equation}\label{eq:formula_norm_1_mu}
|g|^{1,\mu}=\bigvee\bigg\{\int fg\,\d\mu\;\bigg|\;f\in S_{\mu,f}(\Omega;L^0(\mm)),\,|f|^{\infty,\mu}\leq\1_\X^\mm\bigg\}.
\end{equation}
\end{lemma}
\begin{proof}
Given any element \(f\in L^\infty_\mu(\Omega;L^0(\mm))\) with \(|f|^{\infty,\mu}\leq\1_\X^\mm\), we can estimate
\[
\int fg\,\d\mu\leq\int|fg|\,\d\mu=|fg|^{1,\mu}\leq|f|^{\infty,\mu}|g|^{1,\mu}\leq|g|^{1,\mu},
\]
which implies the inequality \(\geq\) in \eqref{eq:formula_norm_1_mu}. To prove the converse inequality, fix \(\varepsilon>0\)
and take an element \(h=\sum_{i\in F}\1_{A_i}^\mu\bar h^i\in S_{\mu,f}(\Omega;L^0(\mm))\) such that
\(\sfd_{L^0(\mm)}(|g-h|^{1,\mu},0)\leq\varepsilon\). Let us then define
\[
f\coloneqq\sum_{i\in F}\1_{A_i}^\mu{\rm sgn}(\bar h^i)\in S_{\mu,f}(\Omega;L^0(\mm)),
\]
where \({\rm sgn}\colon\R\to\{-1,0,1\}\) is defined as \({\rm sgn}(t)\coloneqq-1\) if \(t<0\),
\({\rm sgn}(t)\coloneqq 1\) if \(t>0\), and \({\rm sgn}(0)\coloneqq 0\). Since
\(fh=\sum_{i\in F}\1_{A_i}^\mu{\rm sgn}(\bar h^i)\bar h^i=\sum_{i\in F}\1_{A_i}^\mu|\bar h^i|\) by \eqref{eq:def_mult_Linfty_L1},
we have \(\int fh\,\d\mu=|h|^{1,\mu}\). Hence,
\[\begin{split}
\sfd_{L^0(\mm)}\bigg(|g|^{1,\mu},\int fg\,\d\mu\bigg)&\leq\sfd_{L^0(\mm)}(|g|^{1,\mu},|h|^{1,\mu})
+\sfd_{L^0(\mm)}\bigg(\int fh\,\d\mu,\int fg\,\d\mu\bigg)\\
&\leq\sfd_{L^0(\mm)}(|g-h|^{1,\mu},0)+\sfd_{L^0(\mm)}\bigg(\int f(h-g)\,\d\mu,0\bigg)\\
&\leq\varepsilon+\sfd_{L^0(\mm)}(|f(h-g)|^{1,\mu},0)
\leq\varepsilon+\sfd_{L^0(\mm)}(|f|^{\infty,\mu}|h-g|^{1,\mu},0)\\
&\leq\varepsilon+\sfd_{L^0(\mm)}(|h-g|^{1,\mu},0)\leq 2\varepsilon.
\end{split}\]
Thanks to the arbitrariness of \(\varepsilon>0\), we conclude that \eqref{eq:formula_norm_1_mu} holds.
\end{proof}
\subsection{The spaces \texorpdfstring{\(L^p_\mu(\Omega;{\rm M})\)}{Lpmu} for \texorpdfstring{\(1<p<\infty\)}{1<p<infty}}\label{s:Lp_to_M}
Given any \(p\in(1,\infty)\) and \(\mu\in\mathcal P(\Omega;L^0(\mm))\), we define the map
\(|\cdot|^{p,\mu}\colon S_{\mu,f}(\Omega;{\rm M})\to L^0_+(\mm)\) as
\[
|v|^{p,\mu}\coloneqq\bigg(\sum_{i\in F}|\mu|(A_i)|\bar v^i|^p\bigg)^{1/p}\in L^0_+(\mm)
\quad\text{ for every }v=\sum_{i\in F}\1_{A_i}^\mu\bar v^i\in S_{\mu,f}(\Omega;{\rm M}).
\]
It can be readily checked that \((S_{\mu,f}(\Omega;{\rm M}),|\cdot|^{p,\mu})\) is a random normed module.
We denote by
\begin{equation}\label{eq:def_Lp_mu}
(L^p_\mu(\Omega;{\rm M}),|\cdot|^{p,\mu})
\end{equation}
its random-normed-module completion.
\begin{remark}\label{rmk:def_theta_p}{\rm
Given any \(v=\sum_{i\in F}\1_{A_i}^\mu\bar v^i\in S_{\mu,f}(\Omega;{\rm M})\), we can estimate
\[\begin{split}
|v|^{1,\mu}&=\sum_{i\in F}|\mu|(A_i)|\bar v^i|=\sum_{i\in F}|\mu|(A_i)^{(p-1)/p}(|\mu|(A_i)|\bar v^i|^p)^{1/p}\\
&\leq\bigg(\sum_{i\in F}|\mu|(A_i)\bigg)^{(p-1)/p}\bigg(\sum_{i\in F}|\mu|(A_i)|\bar v^i|^p\bigg)^{1/p}
=|\mu|(\Omega)^{(p-1)/p}|v|^{p,\mu}=|v|^{p,\mu}.
\end{split}\]
Therefore, the identity map \((S_{\mu,f}(\Omega;{\rm M}),|\cdot|^{p,\mu})\to(S_{\mu,f}(\Omega;{\rm M}),|\cdot|^{1,\mu})\)
can be uniquely extended to a homomorphism of random normed modules
\(\theta_p\colon L^p_\mu(\Omega;{\rm M})\to L^1_\mu(\Omega;{\rm M})\) satisfying
\begin{equation}\label{eq:ineq_theta_p}
|\theta_p(v)|^{1,\mu}\leq|v|^{p,\mu}\quad\text{ for every }v\in L^p_\mu(\Omega;{\rm M}),
\end{equation}
as it follows from Corollary \ref{cor:ext_hom_mod}. We point out that similar arguments would give analogous
homomorphisms \(L^p_\mu(\Omega;{\rm M})\to L^q_\mu(\Omega;{\rm M})\) for any choice of \(1<q\leq p<\infty\).
\fr}\end{remark}

Let us now focus on the case \(p=2\). First, we recall that a random normed module \(({\rm H},|\cdot|)\)
with base \((\X,\Sigma,\mm)\) is said to be \textbf{Hilbertian} provided it is complete and it satisfies
\[
|v+w|^2+|v-w|^2=2|v|^2+2|w|^2\quad\text{ for every }v,w\in{\rm H}.
\]
See \cite{guo1999} and \cite[Definition 1.2.20]{gigli2018nonsmooth}.
The space \(L^0(\mm)\) is, clearly, an example of Hilbertian random normed module.
The \textbf{\(L^0\)-scalar product} \(\langle\cdot,\cdot\rangle\colon{\rm H}\times{\rm H}\to L^0(\mm)\)
is the \(L^0(\mm)\)-bilinear operator defined as
\[
\langle v,w\rangle\coloneqq\frac{|v+w|^2-|v|^2-|w|^2}{2}\in L^0(\mm)\quad\text{ for every }v,w\in{\rm H}.
\]
For any closed \(L^0(\mm)\)-submodule \({\rm K}\) of \({\rm H}\), there exists a unique homomorphism of
random normed modules \({\rm p}\colon{\rm H}\to{\rm K}\), called the \textbf{orthogonal projection}
onto \({\rm K}\), such that
\begin{equation}\label{eq:char_orth_proj}
\langle v-{\rm p}(v),w\rangle=0\quad\text{ for every }v\in{\rm H}\text{ and }w\in{\rm K}.
\end{equation}
Moreover, it holds that \({\rm p}\circ{\rm p}={\rm p}\) and \(|{\rm p}(v)|\leq|v|\) for every \(v\in{\rm H}\).
We also recall the random version of the Riesz representation theorem:
the operator sending an element \(v\in{\rm H}\) to
\[
{\rm H}\ni w\mapsto\langle v,w\rangle\in L^0(\mm)
\]
is an isometric isomorphism of random normed modules between \({\rm H}\) and \({\rm H}^*\); in particular,
\({\rm H}^*\) is Hilbertian. See \cite{GuoYou96} and \cite[Theorem 1.2.24]{gigli2018nonsmooth}.
\begin{remark}{\rm
If \(({\rm H},|\cdot|)\) is a Hilbertian random normed module with base \((\X,\Sigma,\mm)\), then
\[
L^2_\mu(\Omega;{\rm H})\quad\text{ is a Hilbertian random normed module.}
\]
Indeed, for any \(v=\sum_{i\in F}\1_{A_i}^\mu\bar v^i\in S_{\mu,f}(\Omega;{\rm H})\) and
\(w=\sum_{j\in G}\1_{B_j}^\mu\bar w^j\in S_{\mu,f}(\Omega;{\rm H})\), we can compute
\[\begin{split}
&\big(|v+w|^{2,\mu}\big)^2+\big(|v-w|^{2,\mu}\big)^2\\
=\,&\sum_{(i,j)\in F\times G}|\mu|(A_i\cap B_j)\big(|\bar v^i+\bar w^j|^2+|\bar v^i-\bar w^j|^2\big)
=\sum_{(i,j)\in F\times G}|\mu|(A_i\cap B_j)\big(2|\bar v^i|^2+2|\bar w^j|^2\big)\\
=\,&2\sum_{i\in F}|\mu|(A_i)|\bar v^i|^2+2\sum_{j\in G}|\mu|(B_j)|\bar w^j|^2
=2\big(|v|^{2,\mu}\big)^2+2\big(|w|^{2,\mu}\big)^2.
\end{split}\]
Since \(S_{\mu,f}(\Omega;{\rm H})\) is dense in \(L^2_\mu(\Omega;{\rm H})\), we deduce by approximation that
\[
\big(|v+w|^{2,\mu}\big)^2+\big(|v-w|^{2,\mu}\big)^2=2\big(|v|^{2,\mu}\big)^2+2\big(|w|^{2,\mu}\big)^2
\quad\text{ for every }v,w\in L^2_\mu(\Omega;{\rm H}),
\]
thus accordingly \(L^2_\mu(\Omega;{\rm H})\) is a Hilbertian random normed module, as we claimed.
\fr}\end{remark}

We make another easy observation: if \(\mu\in\mathcal M_+(\Omega;L^0(\mm))\) and \(h,g\in S_{\mu,f}(\Omega;L^0(\mm))\), then
\begin{equation}\label{eq:compat_mult_Lp}
\langle h,g\rangle=\int hg\,\d\mu,
\end{equation}
where \(\langle h,g\rangle\in L^0(\mm)\) denotes the \(L^0\)-scalar product in \(L^2_\mu(\Omega;L^0(\mm))\),
while \(hg\in L^1_\mu(\Omega;L^0(\mm))\) is given by Proposition \ref{prop:mult_Linfty_L1}.
Indeed, writing \(h=\sum_{i\in F}\1_{A_i}^\mu\bar h^i\) and \(g=\sum_{j\in G}\1_{B_j}^\mu\bar g^j\), we can compute
\[\begin{split}
\langle h,g\rangle&=\frac{\big(|h+g|^{2,\mu}\big)^2-\big(|h|^{2,\mu}\big)^2-\big(|g|^{2,\mu}\big)^2}{2}\\
&=\frac{1}{2}\bigg(\sum_{(i,j)\in F\times G}\mu(A_i\cap B_j)|\bar h^i+\bar g^j|^2-\sum_{i\in F}\mu(A_i)|\bar h^i|^2
-\sum_{j\in G}\mu(B_j)|\bar g^j|^2\bigg)\\
&=\frac{1}{2}\sum_{(i,j)\in F\times G}\mu(A_i\cap B_j)\big(|\bar h^i+\bar g^j|^2-|\bar h^i|^2-|\bar g^j|^2\big)\\
&=\sum_{(i,j)\in F\times G}\mu(A_i\cap B_j)(\bar h^i\bar g^j)=\int hg\,\d\mu,
\end{split}\]
which gives \eqref{eq:compat_mult_Lp}.
\subsection{\texorpdfstring{\(L^1_\mu(\Omega;{\rm M})\)}{L1mu(Omega;M)} as a projective tensor product}\label{s:L1_tensor_prod}
It is well known that, given a finite signed measure \(\nu\) on \((\Omega,\mathcal A)\) and a Banach space
\(\mathbb B\), the Lebesgue--Bochner space \(L^1_\nu(\Omega;\mathbb B)\) can be identified with the projective
tensor product of Banach spaces \(L^1(\nu)\hat\otimes_\pi\mathbb B\); see e.g.\ \cite[Section 2.3]{Ryan02}.
As we shall prove below, this result can be extended to random normed modules.
\begin{remark}\label{rmk:double_simple_dense_L1}{\rm
For any \(\mu\in\mathcal M(\Omega;L^0(\mm))\), the \(\R\)-linear span of
\(\{\1_A^\mu\1_E^\mm:A\in\mathcal A,E\in\Sigma\}\) is dense in \(L^1_\mu(\Omega;L^0(\mm))\).
Indeed, given any \(A\in\mathcal A\) and \(\bar f\in L^0(\mm)\), we can find a sequence \((\bar f^k)_{k\in\N}\subseteq L^0(\mm)\)
of simple functions \(\bar f^k=\sum_{i=1}^{n_k}\lambda^k_i\1_{E^k_i}^\mm\) such that \(\bar f^k\to\bar f\) in \(L^0(\mm)\),
thus accordingly
\[
|\1_A^\mu\bar f^k-\1_A^\mu\bar f|^{1,\mu}=|\mu|(A)|\bar f^k-\bar f|\to 0\quad\text{ in }L^0(\mm)\text{ as }k\to\infty.
\]
Since the \(\R\)-linear span of the elements of \(L^1_\mu(\Omega;L^0(\mm))\) of the form \(\1_A^\mu\bar f\) is
\(S_{\mu,f}(\Omega;L^0(\mm))\), which is dense in \(L^1_\mu(\Omega;L^0(\mm))\) by Remark \ref{rmk:density_finite_simple_maps},
we have proven the claim.
\fr}\end{remark}

Let us now recall the notion of projective tensor product of complete random normed modules, which was
introduced in \cite{Pas23}. Given complete random normed modules \(({\rm M},|\cdot|)\) and \(({\rm N},|\cdot|)\)
with base \((\X,\Sigma,\mm)\), we denote by \({\rm M}\otimes{\rm N}\) their algebraic tensor product as modules
over \(L^0(\mm)\). We remind that each tensor \(\alpha\in{\rm M}\otimes{\rm N}\) is a finite sum of elementary
tensors \(v\otimes w\), with \(v\in{\rm M}\) and \(w\in{\rm N}\). In particular, it makes sense to define the map
\(|\cdot|_\pi\colon{\rm M}\otimes{\rm N}\to L^0_+(\mm)\) as
\[
|\alpha|_\pi\coloneqq\bigwedge\bigg\{\sum_{i=1}^n|v_i||w_i|\;\bigg|\;n\in\N,\,(v_i)_{i=1}^n\subseteq{\rm M},\,
(w_i)_{i=1}^n\subseteq{\rm N},\,\alpha=\sum_{i=1}^n v_i\otimes w_i\bigg\}
\]
for every \(\alpha\in{\rm M}\otimes{\rm N}\).
It was shown in \cite[Theorem 4.1]{Pas23} that \(({\rm M}\otimes{\rm N},|\cdot|_\pi)\) is a random normed module.
Consistently with \cite[Definition 4.2]{Pas23}, we then define the \textbf{projective tensor product}
\(({\rm M}\hat\otimes_\pi{\rm N},|\cdot|_\pi)\) of \({\rm M}\) and \({\rm N}\) as the \(L^0\)-completion of
\(({\rm M}\otimes{\rm N},|\cdot|_\pi)\). The following result was proved in \cite[Theorem 4.10]{Pas23}:
\begin{theorem}[Universal property of the projective tensor product]\label{thm:univ_prop_proj_tens}
Let \(({\rm M},|\cdot|)\), \(({\rm N},|\cdot|)\) and \(({\rm Q},|\cdot|)\) be complete random normed modules with base
\((\X,\Sigma,\mm)\). Let \(b\colon{\rm M}\times{\rm N}\to{\rm Q}\) be an \(L^0(\mm)\)-bilinear map for which there exists
\(g\in L^0_+(\mm)\) such that
\[
|b(v,w)|\leq g|v||w|\quad\text{ for every }(v,w)\in{\rm M}\times{\rm N}.
\]
Then there exists a unique homomorphism of random normed modules \(\hat b\colon{\rm M}\hat\otimes_\pi{\rm N}\to{\rm Q}\)
such that the diagram
\[\begin{tikzcd}
{\rm M}\times{\rm N} \arrow[r,"b"] \arrow[d,swap,"\otimes"] & {\rm Q} \\
{\rm M}\hat\otimes_\pi{\rm N} \arrow[ur,swap,"\hat b"] &
\end{tikzcd}\]
commutes. Moreover, it holds that \(|\hat b|\leq g\). We say that \(\hat b\) is the \textbf{\(L^0\)-linearisation} of \(b\).
\end{theorem}

We are now in a position to obtain an alternative description of \(L^1_\mu(\Omega;{\rm M})\): 
\begin{theorem}\label{thm:L1_M_as_tens_prod}
Let \(\mu\in\mathcal M(\Omega;L^0(\mm))\) be given. Then there exists a unique isometric isomorphism
of random normed modules
\[
J\colon L^1_\mu(\Omega;L^0(\mm))\hat\otimes_\pi{\rm M}\to L^1_\mu(\Omega;{\rm M})
\]
such that \(J((\1_A^\mu\1_E^\mm)\otimes v)=\1_A^\mu(\1_E^\mm\cdot v)\) for every \(A\in\mathcal A\),
\(E\in\Sigma\) and \(v\in{\rm M}\).
\end{theorem}
\begin{proof}
We define the subspace \({\rm V}\) of \(L^1_\mu(\Omega;L^0(\mm))\) as
\[
{\rm V}\coloneqq\bigg\{\sum_{i=1}^n\sum_{j=1}^k\lambda_{ij}\1_{A_i}^\mu\1_{E_j}^\mm\;\bigg|\;
(\lambda_{ij})_{i,j}\subseteq\R,\,(A_i)_{i=1}^n\subseteq\mathcal A\text{ and }
(E_j)_{j=1}^k\subseteq\Sigma\text{ pairwise disjoint}\bigg\}.
\]
We then define the map \(\tilde B\colon{\rm V}\times{\rm M}\to L^1_\mu(\Omega;{\rm M})\) as
\[
\tilde B\bigg(\sum_{i=1}^n\sum_{j=1}^k\lambda_{ij}\1_{A_i}^\mu\1_{E_j}^\mm,v\bigg)
\coloneqq\sum_{i=1}^n\sum_{j=1}^k\lambda_{ij}\1_{A_i}^\mu(\1_{E_j}^\mm\cdot v)
\]
for every \(\big(\sum_{i=1}^n\sum_{j=1}^k\lambda_{ij}\1_{A_i}^\mu\1_{E_j}^\mm,v\big)\in{\rm V}\times{\rm M}\).
From the validity of the identities
\[\begin{split}
\bigg|\sum_{i=1}^n\sum_{j=1}^k\lambda_{ij}\1_{A_i}^\mu(\1_{E_j}^\mm\cdot v)\bigg|^{1,\mu}
&=\sum_{i=1}^n|\mu|(A_i)\bigg|\sum_{j=1}^k\lambda_{ij}\1_{E_j}^\mm\cdot v\bigg|
=|v|\sum_{i=1}^n\sum_{j=1}^k|\lambda_{ij}||\mu|(A_i)\1_{E_j}^\mm\\
&=|v|\sum_{i=1}^n|\mu|(A_i)\bigg|\sum_{j=1}^k\lambda_{ij}\1_{E_j}^\mm\bigg|
=\bigg|\sum_{i=1}^n\sum_{j=1}^k\lambda_{ij}\1_{A_i}^\mu\1_{E_j}^\mm\bigg|^{1,\mu}|v|,
\end{split}\]
it follows that \(\tilde B\) is well defined and satisfies \(|\tilde B(f,v)|^{1,\mu}=|f|^{1,\mu}|v|\)
for all \((f,v)\in{\rm V}\times{\rm M}\), thus it is \(L^0(\mm)\)-bilinear.
Since \({\rm V}\) coincides with the \(\R\)-linear span of \(\{\1_A^\mu\1_E^\mm:A\in\mathcal A,E\in\Sigma\}\),
which is dense in \(L^1_\mu(\Omega;L^0(\mm))\) by Remark \ref{rmk:double_simple_dense_L1}, we deduce that
\({\rm V}\times{\rm M}\) is dense in \(L^1_\mu(\Omega;L^0(\mm))\times{\rm M}\), thus \(\tilde B\) can be uniquely
extended to an \(L^0(\mm)\)-bilinear map \(B\colon L^1_\mu(\Omega;L^0(\mm))\times{\rm M}\to L^1_\mu(\Omega;{\rm M})\)
satisfying \(|B(f,v)|^{1,\mu}=|f|^{1,\mu}|v|\) for every \((f,v)\in L^1_\mu(\Omega;L^0(\mm))\times{\rm M}\).
Consequently, there exists a unique homomorphism of random normed modules
\(J\colon L^1_\mu(\Omega;L^0(\mm))\hat\otimes_\pi{\rm M}\to L^1_\mu(\Omega;{\rm M})\) such that
\(J((\1_A^\mu\1_E^\mm)\otimes v)=B(\1_A^\mu\1_E^\mm,v)=\1_A^\mu(\1_E^\mm\cdot v)\) for every
\((A,E,v)\in\mathcal A\times\Sigma\times{\rm M}\) and \(|J(\alpha)|^{1,\mu}\leq|\alpha|_\pi\)
for all \(\alpha\in L^1_\mu(\Omega;L^0(\mm))\hat\otimes_\pi{\rm M}\). Finally, given
\(v=\sum_{i\in F}\1_{A_i}^\mu\bar v^i\in S_{\mu,f}(\Omega;{\rm M})\), we have that
\(J\big(\sum_{i\in F}(\1_{A_i}^\mu\1_\X^\mm)\otimes\bar v^i\big)=\sum_{i\in F}J((\1_{A_i}^\mu\1_\X^\mm)\otimes\bar v^i)=v\) and
\[
\bigg|\sum_{i\in F}(\1_{A_i}^\mu\1_\X^\mm)\otimes\bar v^i\bigg|_\pi\leq
\sum_{i\in F}|\1_{A_i}^\mu\1_\X^\mm|^{1,\mu}|\bar v^i|=\sum_{i\in F}|\mu|(A_i)|\bar v^i|=|v|^{1,\mu}
=\bigg|J\bigg(\sum_{i\in F}(\1_{A_i}^\mu\1_\X^\mm)\otimes\bar v^i\bigg)\bigg|^{1,\mu}.
\]
Since \(S_{\mu,f}(\Omega;{\rm M})\) is dense in \(L^1_\mu(\Omega;{\rm M})\) by Remark \ref{rmk:density_finite_simple_maps},
we conclude that the homomorphism \(J\) is surjective and \(|J(\alpha)|^{1,\mu}=|\alpha|_\pi\) for all
\(\alpha\in L^1_\mu(\Omega;L^0(\mm))\hat\otimes_\pi{\rm M}\). This yields the statement.
\end{proof}

In Section \ref{s:random_martingale} we will also need the projective tensor product of homomorphisms, which we
are going to remind. Let \({\rm M}\), \(\tilde{\rm M}\), \({\rm N}\) and \(\tilde{\rm N}\) be complete random
normed modules with base \((\X,\Sigma,\mm)\). Let \(\varphi\colon{\rm M}\to\tilde{\rm M}\) and
\(\psi\colon{\rm N}\to\tilde{\rm N}\) be homomorphisms of random normed modules. Then there is a unique
homomorphism of random normed modules \(\varphi\otimes_\pi\psi\colon{\rm M}\hat\otimes_\pi{\rm N}
\to\tilde{\rm M}\hat\otimes_\pi\tilde{\rm N}\) such that
\[
(\varphi\otimes_\pi\psi)(v\otimes w)=\varphi(v)\otimes\psi(w)\quad\text{ for every }(v,w)\in{\rm M}\times{\rm N}.
\]
Moreover, it holds that \(|\varphi\otimes_\pi\psi|=|\varphi||\psi|\). See \cite[Proposition 4.3]{Pas23} for
a proof of these claims.
\subsection{Pointwise description of \texorpdfstring{\(L^1_\mu(\Omega;L^0(\mm))\)}{L1mu} when \texorpdfstring{\(\mu\)}{mu} can be foliated}\label{s:ptwse_descr_L_1}
Let us assume that
\[
\mu\in\mathcal P(\Omega;L^0(\mm))\quad\text{ has a foliation }(\mu_x)_{x\in\X}\subseteq\mathcal P(\Omega).
\]
In this setting, we can show that the elements of \(L^1_\mu(\Omega;L^0(\mm))\) can be described `pointwise',
in the sense that we are going to discuss. To this aim, we introduce the following definition:
\begin{definition}[The space \(\tilde L^1_\mu(\Omega;L^0(\mm))\)]\label{def:tildeL1}
Let us define \(\tilde{\mathcal L}^1_\mu(\Omega;\mathcal L^0(\bar\Sigma_\mm))\subseteq\mathcal L^0(\mathcal A\otimes\bar\Sigma_\mm)\) as
\[
\tilde{\mathcal L}^1_\mu(\Omega;\mathcal L^0(\bar\Sigma_\mm))\coloneqq\big\{F\in\mathcal L^0(\mathcal A\otimes\bar\Sigma_\mm)\;\big|\;
F(\cdot,x)\in\mathcal L^1(\mu_x)\text{ for }\mm\text{-a.e.\ }x\in\X\big\}.
\]
Given any \(F,\tilde F\in\tilde{\mathcal L}^1_\mu(\Omega;\mathcal L^0(\bar\Sigma_\mm))\), we declare that \(F\sim\tilde F\) if and only if
\[
[F(\cdot,x)]_{\mu_x}=[\tilde F(\cdot,x)]_{\mu_x}\quad\text{ for }\mm\text{-a.e.\ }x\in\X.
\]
Then \(\sim\) is an equivalence relation on \(\tilde{\mathcal L}^1_\mu(\Omega;\mathcal L^0(\bar\Sigma_\mm))\), whose quotient space we denote by
\[
\tilde L^1_\mu(\Omega;L^0(\mm))\coloneqq\tilde{\mathcal L}^1_\mu(\Omega;\mathcal L^0(\bar\Sigma_\mm))/\sim.
\]
\end{definition}

We denote by \(\pi\colon\tilde{\mathcal L}^1_\mu(\Omega;\mathcal L^0(\bar\Sigma_\mm))\to\tilde L^1_\mu(\Omega;L^0(\mm))\) the canonical
projection map. It is straightforward to check that \(\tilde L^1_\mu(\Omega;L^0(\mm))\) is a module over \(L^0(\mm)\)
with respect to the following operations:
\[\begin{split}
F_1+F_2\coloneqq\pi(\bar F_1+\bar F_2)&\quad\text{ for every }F_1,F_2\in\tilde L^1_\mu(\Omega;L^0(\mm)),\\
f\cdot F\coloneqq\pi(\bar f\cdot\bar F)&\quad\text{ for every }f\in L^0(\mm)\text{ and }F\in\tilde L^1_\mu(\Omega;L^0(\mm)),
\end{split}\]
where \(\bar F_1,\bar F_2,\bar F\in\tilde{\mathcal L}^1_\mu(\Omega;\mathcal L^0(\bar\Sigma_\mm))\) and \(\bar f\in\mathcal L^0(\bar\Sigma_\mm)\)
are chosen so that \(\pi(\bar F_1)=F_1\), \(\pi(\bar F_2)=F_2\), \(\pi(\bar F)=F\) and \([\bar f]_\mm=f\), and we set
\[
(\bar F_1+\bar F_2)(p,x)\coloneqq\bar F_1(p,x)+\bar F_2(p,x),\quad(\bar f\cdot\bar F)(p,x)\coloneqq\bar f(x)\bar F(p,x)
\quad\text{ for all }(p,x)\in\Omega\times\X.
\]
It can be readily checked that the above definitions are well posed, meaning that they do not depend on the chosen
representatives. We also define \(|\cdot|^{1,\mu}_\sim\colon\tilde L^1_\mu(\Omega;L^0(\mm))\to L^0_+(\mm)\) as
\[
|F|^{1,\mu}_\sim\coloneqq[|\bar F|^{1,\mu}_\sim]_\mm,\quad\text{ where we set }
\X\ni x\mapsto|\bar F|^{1,\mu}_\sim(x)\coloneqq\int|\bar F(\cdot,x)|\,\d\mu_x,
\]
for every \(F\in\tilde L^1_\mu(\Omega;L^0(\mm))\) and \(\bar F\in\tilde{\mathcal L}^1_\mu(\Omega;\mathcal L^0(\bar\Sigma_\mm))\)
with \(\pi(\bar F)=F\). This definition is meaningful thanks to Lemma \ref{lem:meas_by_mon_class_thm},
and clearly independent of the chosen representative \(\bar F\).
\begin{lemma}\label{lem:tilde_L1_complete}
The space \(\big(\tilde L^1_\mu(\Omega;L^0(\mm)),|\cdot|^{1,\mu}_\sim\big)\) is a complete random normed module.
\end{lemma}
\begin{proof}
It is straightforward to verify that \((\tilde L^1_\mu(\Omega;L^0(\mm)),|\cdot|^{1,\mu}_\sim)\) is a random normed module,
thus we focus on the proof of its completeness. Let \((F_n)_{n\in\N}\subseteq\tilde L^1_\mu(\Omega;L^0(\mm))\)
be a Cauchy sequence. Without loss of generality, we can assume that \(\sfd_{L^0(\mm)}(|F_{n+1}-F_n|^{1,\mu}_\sim,0)\leq 2^{-n}\)
for every \(n\in\N\). For any \(n\in\N\), take \(\bar F_n\in\tilde{\mathcal L}^1_\mu(\Omega;\mathcal L^0(\bar\Sigma_\mm))\) so
that \(\pi(\bar F_n)=F_n\). We define the function \(\bar F\in\mathcal L^0(\mathcal A\otimes\bar\Sigma_\mm)\) as
\[
\bar F(p,x)\coloneqq\1_{\{-\infty<\tilde F<+\infty\}}(p,x)\tilde F(p,x)\quad\text{ for every }(p,x)\in\Omega\times\X,
\]
where the auxiliary function \(\tilde F\in\bar{\mathcal L}^0(\mathcal A\otimes\bar\Sigma_\mm)\) is given by
\[
\tilde F(p,x)\coloneqq\varlimsup_{n\to\infty}\bar F_n(p,x)\quad\text{ for every }(p,x)\in\Omega\times\X.
\]
Applying the monotone convergence theorem and Jensen's inequality, we obtain that
\[\begin{split}
&\int\!\!\!\int\sum_{n\in\N}\big|\bar F_{n+1}(\cdot,x)-\bar F_n(\cdot,x)\big|\wedge 1\,\d\mu_x\,\d\mm(x)\\
=\,&\int\sum_{n\in\N}\bigg(\int\big|\bar F_{n+1}(\cdot,x)-\bar F_n(\cdot,x)\big|\wedge 1\,\d\mu_x\bigg)\,\d\mm(x)\\
\leq\,&\int\sum_{n\in\N}\bigg(\int\big|\bar F_{n+1}(\cdot,x)-\bar F_n(\cdot,x)\big|\,\d\mu_x\bigg)\wedge 1\,\d\mm(x)\\
=\,&\sum_{n\in\N}\int\bigg(\int\big|\bar F_{n+1}(\cdot,x)-\bar F_n(\cdot,x)\big|\,\d\mu_x\bigg)\wedge 1\,\d\mm(x)\\
=\,&\sum_{n\in\N}\int|F_{n+1}-F_n|^{1,\mu}_\sim\wedge 1\,\d\mm=\sum_{n\in\N}\sfd_{L^0(\mm)}(|F_{n+1}-F_n|^{1,\mu}_\sim,0)\leq 1.
\end{split}\]
In particular, for \(\mm\)-a.e.\ \(x\in\X\) we have that
\[
\sum_{n\in\N}\big|\bar F_{n+1}(p,x)-\bar F_n(p,x)\big|\wedge 1<+\infty\quad\text{ for }\mu_x\text{-a.e.\ }p\in\Omega,
\]
which implies \(\bar F(p,x)=\lim_n\bar F_n(p,x)\) for \(\mu_x\)-a.e.\ \(p\in\Omega\). The same estimates show also that
\[
\sum_{n\in\N}\bigg(\int\big|\bar F_{n+1}(\cdot,x)-\bar F_n(\cdot,x)\big|\,\d\mu_x\bigg)\wedge 1<+\infty
\quad\text{ for }\mm\text{-a.e.\ }x\in\X,
\]
thus accordingly \(\bar F(\cdot,x)\in\mathcal L^1(\mu_x)\) and \(\lim_n\big\|[\bar F_n(\cdot,x)-\bar F(\cdot,x)]_{\mu_x}\big\|_{L^1(\mu_x)}=0\)
for \(\mm\)-a.e.\ \(x\in\X\). In particular, it holds that \(\bar F\in\tilde{\mathcal L}^1_\mu(\Omega;\mathcal L^0(\bar\Sigma_\mm))\)
and \(\lim_n|F_n-F|^{1,\mu}_\sim(x)=0\) for \(\mm\)-a.e.\ \(x\in\X\), where we set
\(F\coloneqq\pi(\bar F)\in\tilde L^1_\mu(\Omega;L^0(\mm))\). This implies that
\(\sfd_{L^0(\mm)}(|F_n-F|^{1,\mu}_\sim,0)\to 0\) as \(n\to\infty\), thus showing that
\(\big(\tilde L^1_\mu(\Omega;L^0(\mm)),|\cdot|^{1,\mu}_\sim\big)\) is a complete random normed module.
\end{proof}
\begin{theorem}[The identification \(L^1_\mu(\Omega;L^0(\mm))=\tilde L^1_\mu(\Omega;L^0(\mm))\)]\label{thm:L1_tilde_L1}
There exists a unique isometric isomorphism of random normed modules
\[
I\colon L^1_\mu(\Omega;L^0(\mm))\to\tilde L^1_\mu(\Omega;L^0(\mm))
\]
such that \(I(\1_A^\mu\1_\X^\mm)=\pi(\bar I(\1_A^\mu\1_\X^\mm))\) for all \(A\in\mathcal A\),
where \(\bar I(\1_A^\mu\1_\X^\mm)\in\tilde{\mathcal L}^1_\mu(\Omega;\mathcal L^0(\bar\Sigma_\mm))\) is given by
\[
\bar I(\1_A^\mu\1_\X^\mm)(p,x)\coloneqq\1_A(p)\quad\text{ for every }(p,x)\in\Omega\times\X.
\]
\end{theorem}
\begin{proof}
By \(L^0(\mm)\)-linearity, we are forced to set \(I\coloneqq\pi\circ\bar I\colon S_{\mu,f}(\Omega;L^0(\mm))\to\tilde L^1_\mu(\Omega;L^0(\mm))\), where
\[
\bar I(f)(p,x)\coloneqq\sum_{i\in F}\1_{A_i}(p)\bar f^i(x)
\]
for every \(f=\sum_{i\in F}\1_{A_i}^\mu[\bar f^i]_\mm\in S_{\mu,f}(\Omega;L^0(\mm))\) and \((p,x)\in\Omega\times\X\).
Note that
\[
|I(f)|^{1,\mu}_\sim(x)=\int\bigg|\sum_{i\in F}\1_{A_i}(p)\bar f^i(x)\bigg|\,\d\mu_x(p)
\leq\sum_{i\in F}|\bar f^i(x)|\mu_x(A_i)\quad\text{ for }\mm\text{-a.e.\ }x\in\X,
\]
thus \(|I(f)|^{1,\mu}_\sim\leq\sum_{i\in F}\mu(A_i)|[\bar f^i]_\mm|=|f|^{1,\mu}\). Hence, \(I\) can be uniquely extended
to a homomorphism of random normed modules \(I\colon L^1_\mu(\Omega;L^0(\mm))\to\tilde L^1_\mu(\Omega;L^0(\mm))\)
with \(|I(f)|^{1,\mu}_\sim\leq|f|^{1,\mu}\) for every \(f\in L^1_\mu(\Omega;L^0(\mm))\). Let ${\rm V}$ be the \(\R\)-linear span of the set \(\{\1_A^\mu\1_E^\mm:A\in\mathcal A,E\in\Sigma\}\).
We recall from Remark \ref{rmk:double_simple_dense_L1} that ${\rm V}$ is dense in \(L^1_\mu(\Omega;L^0(\mm))\). Fix any element \(f\in{\rm V}\). We can express \(f\) as
\(\sum_{i\in F}\sum_{j=1}^k\lambda_{i,j}\1_{A_j}^\mu\1_{E_i}^\mm\), where \((E_i)_{i\in F}\subseteq\Sigma\)
is a partition of \(\X\), \((A_j)_{j=1}^k\subseteq\mathcal A\) is a partition of \(\Omega\) and
\(\{\lambda_{i,j}:i\in F,j=1,\ldots,k\}\subseteq\R\). Then
\[
|I(f)|^{1,\mu}_\sim=\sum_{i\in F}\1_{E_i}^\mm\sum_{j=1}^k|\lambda_{i,j}|\mu(A_j)
=\sum_{j=1}^k\mu(A_j)\bigg|\sum_{i\in F}\lambda_{i,j}\1_{E_i}^\mm\bigg|=\int|f|\,\d\mu=|f|^{1,\mu}.
\]
By approximation, this implies that \(|I(f)|^{1,\mu}_\sim=|f|^{1,\mu}\) for every \(f\in L^1_\mu(\Omega;L^0(\mm))\).
To conclude, it remains to check that the homomorphism \(I\colon L^1_\mu(\Omega;L^0(\mm))\to\tilde L^1_\mu(\Omega;L^0(\mm))\)
is surjective. To this aim, fix any \(\bar F\in\tilde{\mathcal L}^1_\mu(\Omega;\mathcal L^0(\bar\Sigma_\mm))\)
with \(\bar F\geq 0\). For any \(n\in\N\), we consider the truncated function
\[
\bar F_n\coloneqq\bar F\wedge n\in\mathcal L^\infty(\mathcal A\otimes\bar\Sigma_\mm)
\subseteq\tilde{\mathcal L}^1_\mu(\Omega;\mathcal L^0(\bar\Sigma_\mm)).
\]
Fix \(\varepsilon>0\). Since \(\bar F_n(p,x)\nearrow \bar F(p,x)\) for every \((p,x)\in\Omega\times\X\), using the monotone
convergence theorem we obtain that
\(|\bar F_n-\bar F|^{1,p}_\sim(x)=\int|\bar F_n(\cdot,x)-\bar F(\cdot,x)|\,\d\mu_x\to 0\) as \(n\to\infty\)
for every \(x\in\X\). In particular, letting \(F\coloneqq\pi(\bar F)\) and \(F_n\coloneqq\pi(\bar F_n)\), we have
\(\sfd_{L^0(\mm)}(|F_{n_\varepsilon}-F|^{1,\mu}_\sim,0)\leq\varepsilon/2\) for some \(n_\varepsilon\in\N\).
Next, by applying Lemma \ref{lem:rect_dense_L1_hat_mu} we obtain a sequence \((\bar S_k)_{k\in\N}\subseteq \mathcal L^\infty(\mathcal A\otimes\bar\Sigma_\mm)\)
such that \(\|S_k-F_{n_\varepsilon}\|_{L^1(\hat\mu)}\leq 2^{-k}\) for all \(k\in\N\), where \(S_k\coloneqq[\bar S_k]_{\hat\mu}\)
and each \(\bar S_k\) is of the form
\[
\bar S_k=\sum_{i=1}^{i_k}\sum_{j=1}^{j_k}\lambda^k_{i,j}\1_{A^k_i\times E^k_j}
\]
with \(i_k,j_k\in\N\) and \(\{\lambda^k_{i,j}:i=1,\ldots,i_k,j=1,\ldots,j_k\}\subseteq[0,n_\varepsilon]\), while
\((A^k_i)_{i=1}^{i_k}\subseteq\mathcal A\) is a partition of \(\Omega\) and \((E^k_j)_{j=1}^{j_k}\subseteq\bar\Sigma_\mm\)
is a partition of \(\X\). By virtue of the monotone convergence theorem and Lemma \ref{lem:meas_by_mon_class_thm}, we have that
\[\begin{split}
\int\bigg(\sum_{k\in\N}\int\big|\bar S_k(\cdot,x)-\bar F_{n_\varepsilon}(\cdot,x)\big|\,\d\mu_x\bigg)\,\d\mm(x)
&=\sum_{k\in\N}\int\!\!\!\int\big|\bar S_k(\cdot,x)-\bar F_{n_\varepsilon}(\cdot,x)\big|\,\d\mu_x\,\d\mm(x)\\
&=\sum_{k\in\N}\int\big|\bar S_k(p,x)-\bar F_{n_\varepsilon}(p,x)\big|\,\d\hat\mu(p,x)\\
&=\sum_{k\in\N}\|S_k-F_{n_\varepsilon}\|_{L^1(\hat\mu)}\leq 1.
\end{split}\]
Hence, \(\sum_{k\in\N}\int|\bar S_k(\cdot,x)-\bar F_{n_\varepsilon}(\cdot,x)|\,\d\mu_x<+\infty\)
for \(\mm\)-a.e.\ \(x\in\X\), so that \(|\pi(S_k)-F_{n_\varepsilon}|^{1,\mu}_\sim(x)\to 0\) as \(k\to\infty\) for
\(\mm\)-a.e.\ \(x\in\X\) and thus \(\sfd_{L^0(\mm)}(|\pi(S_{k_\varepsilon})-F_{n_\varepsilon}|^{1,\mu}_\sim,0)\leq\varepsilon/2\)
for some \(k_\varepsilon\in\N\). Letting
\[
f\coloneqq\sum_{i=1}^{i_{k_\varepsilon}}\1_{A^{k_\varepsilon}_i}^\mu\bigg(\sum_{j=1}^{j_{k_\varepsilon}}
\lambda^{k_\varepsilon}_{i,j}\1_{E^{k_\varepsilon}_j}^\mm\bigg)\in S_{\mu,f}(\Omega;L^0(\mm)),
\]
we finally have \(I(f)=\pi(S_{k_\varepsilon})\) and accordingly \(\sfd_{L^0(\mm)}(|I(f)-F|^{1,\mu}_\sim,0)\leq\varepsilon\).
All in all, we have shown that if \(\bar F\in\tilde{\mathcal L}^1_\mu(\Omega;\mathcal L^0(\bar\Sigma_\mm))\) and \(\bar F\geq 0\),
then \(\pi(\bar F)\) belongs to the closure of \(I\big(S_{\mu,f}(\Omega;L^0(\mm))\big)\) in \(\tilde L^1_\mu(\Omega;L^0(\mm))\).
Since \(I\) preserves the \(L^0\)-norm, \(I\big(L^1_\mu(\Omega;L^0(\mm))\big)\) is closed in
\(\tilde L^1_\mu(\Omega;L^0(\mm))\) and thus
\[
{\rm cl}_{\tilde L^1_\mu(\Omega;L^0(\mm))}\big(I\big(S_{\mu,f}(\Omega;L^0(\mm))\big)\big)\subseteq I\big(L^1_\mu(\Omega;L^0(\mm))\big).
\]
To conclude, just observe that if \(\bar F\in\tilde{\mathcal L}^1_\mu(\Omega;\mathcal L^0(\bar\Sigma_\mm))\), then both
\(\bar F^+\coloneqq\bar F\vee 0\) and \(\bar F^-\coloneqq(-\bar F)^+\) belong to \(\tilde{\mathcal L}^1_\mu(\Omega;\mathcal L^0(\bar\Sigma_\mm))\), thus accordingly
\[
\pi(\bar F)=\pi(\bar F^+)-\pi(\bar F^-)\in I\big(L^1_\mu(\Omega;L^0(\mm))\big)-I\big(L^1_\mu(\Omega;L^0(\mm))\big)
=I\big(L^1_\mu(\Omega;L^0(\mm))\big).
\]
This shows that \(I\colon L^1_\mu(\Omega;L^0(\mm))\to\tilde L^1_\mu(\Omega;L^0(\mm))\) is surjective, thus completing the proof.
\end{proof}

In light of Theorem \ref{thm:L1_tilde_L1}, we can unambiguously write
\[
\int F\,\d\mu\coloneqq\int I^{-1}(F)\,\d\mu\in L^0(\mm)\quad\text{ for every }F\in\tilde L^1_\mu(\Omega;L^0(\mm)).
\]
Similarly, we write \(\int_A F\,\d\mu\coloneqq\int_A I^{-1}(F)\,\d\mu\) for every \(A\in\mathcal A\)
and \(F\in\tilde L^1_\mu(\Omega;L^0(\mm))\). Note also that, given \(F\in\tilde L^1_\mu(\Omega;L^0(\mm))\),
it makes sense to write `\(F(\cdot,x)\in L^1(\mu_x)\) for \(\mm\)-a.e.\ \(x\in\X\)'.
\begin{proposition}\label{prop:ptwse_formula_int_of_tilde_L1}
Let \(F\in\tilde L^1_\mu(\Omega;L^0(\mm))\) be given. Then it holds that
\begin{equation}\label{eq:ptwse_formula_int_of_tilde_L1_cl}
\bigg(\int F\,\d\mu\bigg)(x)=\int F(\cdot,x)\,\d\mu_x\quad\text{ for }\mm\text{-a.e.\ }x\in\X.
\end{equation}
\end{proposition}
\begin{proof}
If \(f=\sum_{i\in F}\1_{A_i}^\mu\bar f^i\in S_{\mu,f}(\Omega;L^0(\mm))\) is given, then for \(\mm\)-a.e.\ \(x\in\X\) we have that
\[\begin{split}
\bigg(\int I(f)\,\d\mu\bigg)(x)&=\bigg(\int f\,\d\mu\bigg)(x)=\sum_{i\in F}\mu(A_i)(x)\bar f^i(x)
=\sum_{i\in F}\mu_x(A_i)\bar f^i(x)\\
&=\sum_{i\in F}\bar f^i(x)\int\1_{A_i}\,\d\mu_x
=\sum_{i\in F}\int I(\1_{A_i}^\mu\bar f^i)(\cdot,x)\,\d\mu_x=\int I(f)(\cdot,x)\,\d\mu_x,
\end{split}\]
proving that \eqref{eq:ptwse_formula_int_of_tilde_L1_cl} holds for every \(F\in I\big(S_{\mu,f}(\Omega;L^0(\mm))\big)\).
Now, fix any \(F\in\tilde L^1_\mu(\Omega;L^0(\mm))\) and take a sequence \((F_n)_{n\in\N}\subseteq I\big(S_{\mu,f}(\Omega;L^0(\mm))\big)\)
such that \(F_n\to F\) in the topology of \(\tilde L^1_\mu(\Omega;L^0(\mm))\). Up to passing to a subsequence,
we can also assume that \(|F_n-F|^{1,\mu}_\sim\to 0\) holds \(\mm\)-a.e.\ on \(\X\). Then
\[
\bigg|\int F_n\,\d\mu-\int F\,\d\mu\bigg|\leq|F_n-F|^{1,\mu}_\sim\to 0\quad\text{ in the }\mm\text{-a.e.\ sense as }n\to\infty,
\]
as well as
\[
\bigg|\int F_n(\cdot,x)\,\d\mu_x-\int F(\cdot,x)\,\d\mu_x\bigg|\leq\int|F_n(\cdot,x)-F(\cdot,x)|\,\d\mu_x
=|F_n-F|^{1,\mu}_\sim(x)\to 0
\]
for \(\mm\)-a.e.\ \(x\in\X\) as \(n\to\infty\). Taking into account the first part of the proof, we conclude that
\[
\bigg(\int F\,\d\mu\bigg)(x)=\lim_{n\to\infty}\bigg(\int F_n\,\d\mu\bigg)(x)=\lim_{n\to\infty}\int F_n(\cdot,x)\,\d\mu_x
=\int F(\cdot,x)\,\d\mu_x
\]
for \(\mm\)-a.e.\ \(x\in\X\), thus obtaining the statement.
\end{proof}

We define the subset \(\tilde L^1_\mu(\Omega;L^0_+(\mm))\) of \(\tilde L^1_\mu(\Omega;L^0(\mm))\) as
\begin{equation}\label{eq:non-neg_tildeL1}
\tilde L^1_\mu(\Omega;L^0_+(\mm))\coloneqq\big\{F\in\tilde L^1_\mu(\Omega;L^0(\mm))\;\big|\;
F(\cdot,x)\geq 0\text{ for }\mm\text{-a.e.\ }x\in\X\big\}.
\end{equation}
It can be readily checked that
\[
\tilde L^1_\mu(\Omega;L^0_+(\mm))=\big\{\pi(\bar F)\;\big|\;
\bar F\in\tilde{\mathcal L}^1_\mu(\Omega;\mathcal L^0(\bar\Sigma_\mm))\cap\mathcal L^0_+(\mathcal A\otimes\bar\Sigma_\mm)\big\}
=I\big(L^1_\mu(\Omega;L^0_+(\mm))\big),
\]
where \(I\colon L^1_\mu(\Omega;L^0(\mm))\to\tilde L^1_\mu(\Omega;L^0(\mm))\) is the isometric isomorphism
given by Theorem \ref{thm:L1_tilde_L1}.
\begin{corollary}\label{cor:conseq_uniq_tilde_L1}
Assume that \(\mathcal A\) is countably generated and that \(f,g\in L^1_\mu(\Omega;L^0(\mm))\) satisfy
\[
\int_A f\,\d\mu=\int_A g\,\d\mu\quad\text{ for every }A\in\mathcal A.
\]
Then it holds that \(f=g\).
\end{corollary}
\begin{proof}
Take \(\bar F,\bar G\in\tilde{\mathcal L}^1_\mu(\Omega;\mathcal L^0(\bar\Sigma_\mm))\) with \(\pi(\bar F)=I(f)\)
and \(\pi(\bar G)=I(g)\). Fix a countable family \({\rm C}\subseteq\mathcal A\) that generates \(\mathcal A\), and
define the \(\pi\)-system \({\rm P}\) as in \eqref{eq:def_pi-system_P}. Using the assumptions, the fact that \({\rm P}\)
is countable and Proposition \ref{prop:ptwse_formula_int_of_tilde_L1}, we find a set \(N\in\bar\Sigma_\mm\) with
\(\mm(N)=0\) such that
\[
\int_A\bar F(\cdot,x)\,\d\mu_x=\int_A\bar G(\cdot,x)\,\d\mu_x\quad\text{ for every }x\in\X\setminus N\text{ and }A\in{\rm P}.
\]
Arguing as in the proof of Proposition \ref{prop:unique_foliation} (and using of the Sierpi\'{n}ski--Dynkin
\(\pi\)-\(\lambda\) theorem), we deduce that \(\bar F(\cdot,x)\mu_x=\bar G(\cdot,x)\mu_x\) for every \(x\in\X\setminus N\).
This implies that \(|\bar F-\bar G|^{1,\mu}_\sim(x)=0\) for \(\mm\)-a.e.\ \(x\in\X\), so that accordingly
\(I(f)=I(g)\) and thus \(f=g\), proving the statement.
\end{proof}
\subsection{The Radon--Nikod\'{y}m theorem for \texorpdfstring{\(L^0\)}{L0}-valued measures}\label{s:Radon-Nikodym_thm}
Let us assume that
\[\begin{split}
\mu\in\mathcal P(\Omega;L^0(\mm))&\quad\text{ has a foliation }(\mu_x)_{x\in\X}\subseteq\mathcal P(\Omega),\\
\nu\in\mathcal P(\Omega;L^0(\mm))&\quad\text{ has a foliation }(\nu_x)_{x\in\X}\subseteq\mathcal P(\Omega).
\end{split}\]
Given any exponent \(p\in(1,\infty)\), we define
\begin{equation}\label{eq:def_tildecalLp}
\tilde{\mathcal L}^p_\mu(\Omega;\mathcal L^0(\bar\Sigma_\mm))\coloneqq\bigg\{F\in\mathcal L^0(\mathcal A\otimes\bar\Sigma_\mm)\;
\bigg|\;F(\cdot,x)\in\mathcal L^p(\mu_x)\text{ for }\mm\text{-a.e.\ }x\in\X\bigg\}.
\end{equation}
For any \(F\in\tilde{\mathcal L}^p_\mu(\Omega;\mathcal L^0(\bar\Sigma_\mm))\), we define \(|F|^{p,\mu}_\sim\in\mathcal L^0_+(\bar\Sigma_\mm)\) as
\[
|F|^{p,\mu}_\sim(x)\coloneqq\bigg(\int|F(\cdot,x)|^p\,\d\mu_x\bigg)^{1/p}\quad\text{ for every }x\in\X.
\]
Applying H\"{o}lder's inequality, we obtain that \(\tilde{\mathcal L}^p_\mu(\Omega;\mathcal L^0(\bar\Sigma_\mm))\subseteq
\tilde{\mathcal L}^1_\mu(\Omega;\mathcal L^0(\bar\Sigma_\mm))\) and
\begin{equation}\label{eq:ineq_p_mu_sim}
|F|^{1,\mu}_\sim\leq|F|^{p,\mu}_\sim\quad\text{ for every }F\in\tilde{\mathcal L}^p_\mu(\Omega;\mathcal L^0(\bar\Sigma_\mm)).
\end{equation}
We then define the space \(\tilde L^p_\mu(\Omega;L^0(\mm))\subseteq\tilde L^1_\mu(\Omega;L^0(\mm))\) as
\begin{equation}\label{eq:def_tildeLp}
\tilde L^p_\mu(\Omega;L^0(\mm))\coloneqq\big\{\pi(F)\;\big|\;F\in\tilde{\mathcal L}^p_\mu(\Omega;\mathcal L^0(\bar\Sigma_\mm))\big\},
\end{equation}
as well as
\begin{equation}\label{eq:def_tildeLp_norm}
|F|^{p,\mu}_\sim\coloneqq[|\bar F|^{p,\mu}_\sim]_\mm\quad\text{ for every }F=\pi(\bar F)\in\tilde L^p_\mu(\Omega;L^0(\mm)).
\end{equation}
By adapting the proof of Lemma \ref{lem:tilde_L1_complete}, or by combining the completeness of
\(\tilde L^1_\mu(\Omega;L^0(\mm))\) with \eqref{eq:ineq_p_mu_sim}, it is possible to prove that
\[
\big(\tilde L^p_\mu(\Omega;L^0(\mm)),|\cdot|^{p,\mu}_\sim\big)\quad\text{ is a complete random normed module.}
\]
It can be readily checked that \((\tilde L^2_\mu(\Omega;L^0(\mm)),|\cdot|^{2,\mu}_\sim)\)
is a Hilbertian random normed module. Moreover, if \(F=\pi(\bar F)\in\tilde L^2_\mu(\Omega;L^0(\mm))\) and
\(G=\pi(\bar G)\in\tilde L^2_\mu(\Omega;L^0(\mm))\) are given, then it holds that
\(FG\coloneqq\pi(\bar F\bar G)\in\tilde L^1_\mu(\Omega;L^0(\mm))\), and the \(L^0\)-scalar product of \(F\)
and \(G\) in \(\tilde L^2_\mu(\Omega;L^0(\mm))\) is given by
\begin{equation}\label{eq:scalar_prod_tilde_L2}
\langle F,G\rangle=\int FG\,\d\mu\in L^0(\mm).
\end{equation}
\begin{theorem}[The random Radon--Nikod\'{y}m theorem]\label{thm:RN}
Assume that \(\hat\mu\ll\hat\nu\). Then there exists an element \(\delta\in L^1_\nu(\Omega;L^0_+(\mm))\) such that
\begin{equation}\label{eq:RN_claim}
\mu(A)=\int_A\delta\,\d\nu\quad\text{ for every }A\in\mathcal A.
\end{equation}
If in addition the \(\sigma\)-algebra \(\mathcal A\) is countably generated, then the element \(\delta\) is unique
and we call it the \textbf{Radon--Nikod\'{y}m derivative} of \(\mu\) with respect to \(\nu\).
\end{theorem}
\begin{proof}
First of all, we define the operator \(T\colon\tilde L^2_{\mu+\nu}(\Omega;L^0(\mm))\to L^0(\mm)\) as
\[
T(F)\coloneqq\int F\,\d\mu\quad\text{ for every }F\in\tilde L^2_{\mu+\nu}(\Omega;L^0(\mm)).
\]
Note that \(T\) is \(L^0(\mm)\)-linear and \(|T(F)|\leq\int|F|\,\d\mu\leq|F|^{1,\mu+\nu}_\sim\)
for every \(F\in\tilde L^2_{\mu+\nu}(\Omega;L^0(\mm))\), so that \(T\in\tilde L^2_{\mu+\nu}(\Omega;L^0(\mm))^*\). Applying the
random Riesz representation theorem and recalling \eqref{eq:scalar_prod_tilde_L2}, we deduce that there exists a unique element
\(H\in\tilde L^2_{\mu+\nu}(\Omega;L^0(\mm))\) such that
\[
T(F)=\int FH\,\d(\mu+\nu)\quad\text{ for every }F\in\tilde L^2_{\mu+\nu}(\Omega;L^0(\mm)).
\]
It follows that
\begin{equation}\label{eq:char_RN_pre-deriv}
\int F(1-H)\,\d\mu=\int FH\,\d\nu\quad\text{ for every }F\in\tilde L^2_{\mu+\nu}(\Omega;L^0(\mm)).
\end{equation}
Next, fix a representative \(\bar H\in\tilde{\mathcal L}^2_{\mu+\nu}(\Omega;\mathcal L^0(\bar\Sigma_\mm))\) of \(H\).
We consider the sets \(N,B,U\in\mathcal A\otimes\bar\Sigma_\mm\), which we define as
\[\begin{split}
N&\coloneqq\big\{(p,x)\in\Omega\times\X\;\big|\;\bar H(p,x)<0\big\},\\
B&\coloneqq\big\{(p,x)\in\Omega\times\X\;\big|\;\bar H(p,x)>1\big\},\\
U&\coloneqq\big\{(p,x)\in\Omega\times\X\;\big|\;\bar H(p,x)=1\big\}.
\end{split}\]
We denote \(N_x\coloneqq\{p\in\Omega:(p,x)\in N\}\in\mathcal A\) for every \(x\in\X\), and similarly for \(B_x\) and \(U_x\).
Choosing \(F\coloneqq\pi(\1_N)\) in \eqref{eq:char_RN_pre-deriv} and applying Proposition \ref{prop:ptwse_formula_int_of_tilde_L1},
we obtain that for \(\mm\)-a.e.\ \(x\in\X\) it holds that
\[
\int_{\{\bar H(\cdot,x)<0\}}\bar H(\cdot,x)\,\d\nu_x=\int\1_{N_x}\bar H(\cdot,x)\,\d\nu_x
=\int\1_N(\cdot,x)(1-\bar H(\cdot,x))\,\d\mu_x\geq 0,
\]
which implies that \(\nu_x(N_x)=0\) for \(\mm\)-a.e.\ \(x\in\X\). Therefore, Lemma \ref{lem:meas_by_mon_class_thm} ensures that
\[
\hat\nu(N)=\int\1_N\,\d\hat\nu=\int\!\!\!\int\1_N(\cdot,x)\,\d\nu_x\,\d\mm(x)=\int\nu_x(N_x)\,\d\mm(x)=0,
\]
thus \(\hat\mu(N)=0\). Similar arguments give \(\hat\nu(U)=0\) (thus \(\hat\mu(U)=0\)) and \(\hat\mu(B)=0\). Hence,
\[
0\leq\bar H<1\quad\text{ holds }\hat\mu\text{-a.e.\ on }\Omega\times\X.
\]
In other words, \(L\coloneqq\bigsqcup_{n\in\N}L^n\) satisfies \(\hat\mu((\Omega\times\X)\setminus L)=0\), where we define
\[
L^n\coloneqq\bigg\{(p,x)\in\Omega\times\X\;\bigg|\;1-\frac{1}{n}\leq\bar H(p,x)<1-\frac{1}{n+1}\bigg\}\quad\text{ for every }n\in\N.
\]
We denote \(L^n_x\coloneqq\{p\in\Omega:(p,x)\in L^n\}\) and \(L_x\coloneqq\{p\in\Omega:(p,x)\in L\}\) for every \(x\in\X\);
observe that \(L_x=\bigsqcup_{n\in\N}L^n_x\). We also define \(\bar D\coloneqq\bar H\sum_{n\in\N}\bar D_n\colon\Omega\times\X\to[0,+\infty]\), where
\[
\bar D_n\coloneqq\1_{L^n}\frac{1}{1-\bar H}\in\mathcal L^\infty_+(\mathcal A\otimes\bar\Sigma_\mm)\quad\text{ for every }n\in\N.
\]
For any \(A\in\mathcal A\), we can plug \(F=\pi(\1_{A\times\X}\bar D_n)\) in \eqref{eq:char_RN_pre-deriv},
thus (by Proposition \ref{prop:ptwse_formula_int_of_tilde_L1}) we obtain
\[
\mu_x(A\cap L^n_x)=\int_A\bar D_n(\cdot,x)(1-\bar H(\cdot,x))\,\d\mu_x=\int_A\bar D_n(\cdot,x)\bar H(\cdot,x)\,\d\nu_x
\quad\text{ for }\mm\text{-a.e.\ }x\in\X.
\]
Summing over \(n\in\N\) and applying the monotone convergence theorem, we deduce that
\begin{equation}\label{eq:RN_aux_form}
\mu_x(A\cap L_x)=\sum_{n\in\N}\int_A\bar D_n(\cdot,x)\bar H(\cdot,x)\,\d\nu_x
=\int_A\bar D(\cdot,x)\,\d\nu_x\quad\text{ for }\mm\text{-a.e.\ }x\in\X.
\end{equation}
Choosing \(A=\Omega\), we infer that \(\bar D\in\tilde{\mathcal L}^1_\nu(\Omega;\mathcal L^0(\bar\Sigma_\mm))\).
Moreover, it follows from the identities
\[
\int\mu_x(A\cap L_x)\,\d\mm(x)=\hat\mu((A\times\X)\cap L)=\hat\mu(A\times\X)=\int\mu_x(A)\,\d\mm(x)
\]
that \(\mu_x(A\cap L_x)=\mu_x(A)\) for \(\mm\)-a.e.\ \(x\in\X\), which -- in combination with \eqref{eq:RN_aux_form} -- implies that
\[
\mu(A)(x)=\mu_x(A\cap L_x)=\int_A\bar D(\cdot,x)\,\d\nu_x=\bigg(\int_A\pi(\bar D)\,\d\nu\bigg)(x)
\quad\text{ for }\mm\text{-a.e.\ }x\in\X.
\]
Letting \(\delta\coloneqq I^{-1}(\pi(\bar D))\in L^1_\nu(\Omega;L^0_+(\mm))\), where
\(I\colon L^1_\nu(\Omega;L^0(\mm))\to\tilde L^1_\nu(\Omega;L^0(\mm))\) is the isometric isomorphism given by
Theorem \ref{thm:L1_tilde_L1}, we conclude that \eqref{eq:RN_claim} holds. Finally, the uniqueness part of the
statement follows from Corollary \ref{cor:conseq_uniq_tilde_L1}.
\end{proof}

It is not clear to us whether Theorem \ref{thm:RN} holds under the weaker assumption that \(\mu\ll\nu\).
\section{The random Riesz--Markov--Kakutani theorem}\label{s:RMK}
In this section, we prove a version of the \emph{Riesz--Markov--Kakutani theorem} for random normed modules
(see Theorem \ref{thm:RMK}).
Namely, given a compact Hausdorff space \(\Omega=(\Omega,\tau)\), we show that an isometric predual (in the sense
of random normed modules) of \(\mathfrak M(\Omega;L^0(\mm))\) is the space of all \emph{uniformly-order-continuous maps}
from \(\Omega\) to \(L^0(\mm)\), as well as several related results; as before, \((\X,\Sigma,\mm)\) is a given probability
space. We begin with several auxiliary definitions and results.
\subsection{Uniformly-order-continuous maps}\label{s:UC_ord}
Let \(({\rm M},|\cdot|)\) be a complete random normed module with base \((\X,\Sigma,\mm)\).
We say that \(v\colon\Omega\to{\rm M}\) is \textbf{order bounded} if \(\{|v_p|:p\in\Omega\}\subseteq L^0_+(\mm)\)
is order bounded, or equivalently if it holds that \(|v|^\infty\in L^0_+(\mm)\), where we set
\[
|v|^\infty\coloneqq\bigvee_{p\in\Omega}|v_p|\in\bar L^0_+(\mm).
\]
Note that, given any \(\mu\in\mathcal M(\Omega;L^0(\mm))\), we have that
\[
|v|^{\infty,\mu}=\bigwedge\big\{|\bar v|^\infty\;\big|\;\bar v\in\mathcal L^0(\Omega;{\rm M})\text{ is a representative of }v\big\}
\quad\text{ for every }v\in L^0_\mu(\Omega;{\rm M}).
\]
Moreover, by adapting the arguments in the proof of Lemma \ref{lem:big_space_bdd_maps}, it can be readily checked that
\begin{equation}\label{eq:big_space_bdd_maps_non-quot}
\big(\{v\in\mathcal L^0(\Omega;{\rm M}):|v|^\infty\in\mathcal L^\infty_+(\mm)\},|\cdot|^\infty\big)\quad\text{ is a complete random normed module.}
\end{equation}
The following definition is taken from \cite[Definition 3.13]{Pas23}:
\begin{definition}[Uniformly-order-continuous map]\label{def:UC_ord}
Let \((\Omega,\Phi)\) be a uniform space, \(v\colon\Omega\to{\rm M}\) an order-bounded map and
\(\mathcal U\in\Phi\). Then we define the \textbf{variation} of \(v\) on \(\mathcal U\) as
\[
{\rm Var}(v;\mathcal U)\coloneqq\bigvee_{(p,q)\in\mathcal U}|v_p-v_q|\in L^0_+(\mm).
\]
Then we say that \(v\colon\Omega\to{\rm M}\) is \textbf{uniformly order continuous} if it satisfies
\[
\bigwedge_{\mathcal U\in\Phi}{\rm Var}(v;\mathcal U)=0.
\]
We denote by \({\rm UC}_{\rm ord}(\Omega;{\rm M})\) the space of all uniformly-order-continuous
maps from \(\Omega\) to \({\rm M}\).
\end{definition}

It was proved in \cite[Lemma 3.14]{Pas23} that
\[
({\rm UC}_{\rm ord}(\Omega;{\rm M}),|\cdot|^\infty)\quad\text{ is a complete random normed module.}
\]
Note that \({\rm UC}_{\rm ord}(\Omega;{\rm M})\) is an \(L^0(\mm)\)-submodule of the random normed module
that we have considered in \eqref{eq:big_space_bdd_maps_non-quot}. As it follows from \cite[Eq.\ (3.13)]{Pas23}, it holds that
\begin{equation}\label{eq:conseq_UC_d_M}
v\colon(\Omega,\Phi)\to({\rm M},\sfd_{\rm M})\quad\text{ is uniformly continuous for every }v\in{\rm UC}_{\rm ord}(\Omega;{\rm M}),
\end{equation}
so that \({\rm UC}_{\rm ord}(\Omega;{\rm M})\subseteq\mathcal L^0(\Omega;{\rm M})\). It then makes
sense to define \({\rm UC}_{{\rm ord},\mu}(\Omega;{\rm M})\subseteq L^0_\mu(\Omega;{\rm M})\) as the space of all equivalence classes of elements of
\({\rm UC}_{\rm ord}(\Omega;{\rm M})\) with respect to \(\sim_\mu\).
\begin{remark}\label{rmk:two_unif_struct}{\rm
Besides the uniform structure \(\Phi_{\rm M}\) associated to the \((\varepsilon,\lambda)\)-topology
\(\mathcal T_{\varepsilon,\lambda}\) (cf.\ Remark \ref{remarkainounif}),
another uniform structure on random normed modules that has been studied -- starting from Filipovi\'{c}--Kupper--Vogelpoth's paper \cite{FilKupVog09} -- is the one inducing
the \emph{locally \(L^0\)-convex topology} \(\mathcal T_c\), which we are going to recall. Letting
\[
L^0_{++}(\mm)\coloneqq\big\{\epsilon\in L^0_+(\mm)\;\big|\;\epsilon(x)>0\text{ for }\mm\text{-a.e.\ }x\in\X\big\},
\]
we denote by \(\Phi_{\rm M}^c\) the uniformity on \({\rm M}\) that is generated by \(\{\mathcal V(\epsilon):\epsilon\in L^0_{++}(\mm)\}\), where
\[
\mathcal V(\epsilon)\coloneqq\big\{(v,w)\in{\rm M}\times{\rm M}\;\big|\;|v-w|<\epsilon\big\}\quad\text{ for every }\epsilon\in L^0_{++}(\mm).
\]
The topology \(\mathcal T_c\) is typically finer than the \((\varepsilon,\lambda)\)-topology.
Nonetheless, is was proved by Guo in \cite{Guo10} that a random locally convex module
with the \emph{countable concatenation property} (in particular, a random normed module) is
\(\mathcal T_{\varepsilon,\lambda}\)-complete if and only if it is \(\mathcal T_c\)-complete.
\fr}\end{remark}

Next, we compare different notions of uniform continuity for maps from \(\Omega\) to \({\rm M}\):
\begin{proposition}\label{prop:impl_UC}
Let \((\Omega,\Phi)\) be a uniform space. Let \(v\colon\Omega\to{\rm M}\) be an order-bounded map. Then:
\begin{itemize}
\item[\(\rm i)\)] If \(v\in{\rm UC}_{\rm ord}(\Omega;{\rm M})\), then \(v\colon(\Omega,\Phi)\to({\rm M},\Phi_{\rm M})\)
is uniformly continuous.
\item[\(\rm ii)\)] If \(v\colon(\Omega,\Phi)\to({\rm M},\Phi_{\rm M}^c)\) is uniformly continuous, then
\(v\in{\rm UC}_{\rm ord}(\Omega;{\rm M})\).
\end{itemize}
\end{proposition}
\begin{proof}
Taking Remark \ref{remarkainounif} into account, we have that \(v\colon(\Omega,\Phi)\to({\rm M},\Phi_{\rm M})\) is uniformly continuous if and only if
\(v\colon(\Omega,\Phi)\to({\rm M},\sfd_{\rm M})\) is uniformly continuous, thus accordingly i) follows from \eqref{eq:conseq_UC_d_M}.
Let us now prove ii). Assume \(v\colon(\Omega,\Phi)\to({\rm M},\Phi_{\rm M}^c)\) is uniformly continuous and fix \(\varepsilon>0\).
Since \(\varepsilon\1_\X^\mm\in L^0_{++}(\mm)\), there is \(\mathcal U_\varepsilon\in\Phi\) such that \(\{(v_p,v_q):(p,q)\in\mathcal U_\varepsilon\}\subseteq\mathcal V(\varepsilon\1_\X^\mm)\),
which implies that \({\rm Var}(v;\mathcal U_\varepsilon)=\bigvee_{(p,q)\in\mathcal U_\varepsilon}|v_p-v_q|\leq\varepsilon\1_\X^\mm\). Therefore, we conclude that
\[
\bigwedge_{\mathcal U\in\Phi}{\rm Var}(v;\mathcal U)\leq\bigwedge_{\varepsilon>0}{\rm Var}(v;\mathcal U_\varepsilon)\leq\bigwedge_{\varepsilon>0}\varepsilon\1_\X^\mm=0,
\]
so that \(\bigwedge_{\mathcal U\in\Phi}{\rm Var}(v;\mathcal U)=0\), which shows that \(v\in{\rm UC}_{\rm ord}(\Omega;{\rm M})\). This proves ii).
\end{proof}

Interestingly, the uniform order continuity is in general strictly stronger than the uniform continuity with respect to \(\Phi_{\rm M}\),
but also strictly weaker than the uniform continuity with respect to \(\Phi_{\rm M}^c\), as it follows from the two examples that we present below.
In fact, we do not know whether the uniform order continuity can be characterised as an actual uniform continuity with respect to some uniformity
on \({\rm M}\) that is `intermediate' between \(\Phi_{\rm M}\) and \(\Phi_{\rm M}^c\).
\begin{example}\label{ex:UC_1}{\rm
We now construct an example of an order-bounded uniformly-continuous map
\(f\colon([0,1],\sfd_{\rm e})\to(L^0(\mathscr L_1),\Phi_{L^0(\mathscr L_1)})\),
where \([0,1]\) is equipped with the Euclidean distance \(\sfd_{\rm e}\) and \(\mathscr L_1\) denotes the Lebesgue measure
restricted to \([0,1]\), that does not belong to \({\rm UC}_{\rm ord}([0,1];L^0(\mathscr L_1))\).

Let \(Q=[a,a+2\ell]\times[b,b+2\ell]\) be a subcube of \([0,1]^2\). We define \(\Lambda_Q\colon[0,1]^2\to[0,1]\) as
\[
\Lambda_Q(p,x)\coloneqq\left\{\begin{array}{lll}
(p-a)/\ell\\
1-(p-a-\ell)/\ell\\
0\\
\end{array}\quad\begin{array}{lll}
\text{ if }(p,x)\in(a,a+\ell]\times(b,b+2\ell),\\
\text{ if }(p,x)\in(a+\ell,a+2\ell)\times(b,b+2\ell),\\
\text{ otherwise.}
\end{array}\right.
\]
Moreover, we subdivide any such cube \(Q\) into four subcubes in the following way:
\[\begin{split}
Q_1\coloneqq[a,a+\ell]\times[b,b+\ell],&\qquad Q_2\coloneqq[a+\ell,a+2\ell]\times[b,b+\ell],\\
Q_3\coloneqq[a+\ell,a+2\ell]\times[b+\ell,b+2\ell],&\qquad Q_4\coloneqq[a,a+\ell]\times[b+\ell,b+2\ell].
\end{split}\]
Finally, letting \(\mathcal Q_n\coloneqq\{Q_{i_1\ldots i_n}:(i_1,\ldots,i_n)\in\{2,4\}^{n-1}\times\{1,3\}\}\)
for every \(n\in\N\), we define
\[
\Lambda\coloneqq\sum_{n\in\N}\Lambda^n\colon[0,1]^2\to[0,1],
\quad\text{ where we set }\Lambda^n\coloneqq\sum_{Q\in\mathcal Q_n}\Lambda_Q.
\]
\begin{center}\begin{tikzpicture}
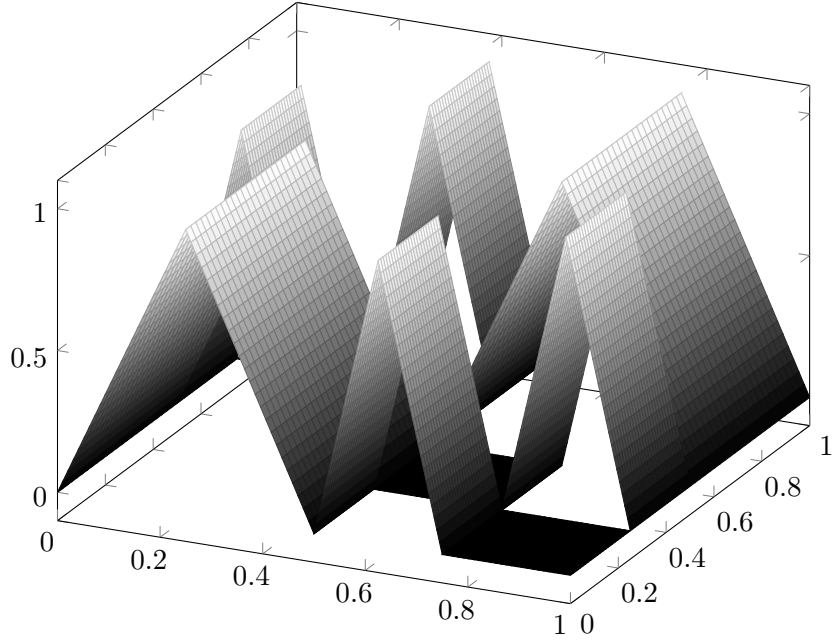
  \begin{axis}[title={\emph{Graph of \(\Lambda^1+\Lambda^2\):}},colormap/blackwhite,width=10cm,height=8cm,scale only axis]
    \addplot3[domain=0:0.125,y domain=0.75:1,surf]{0};  
    \addplot3[domain=0:0.125,y domain=0.5:0.75,surf]{8*x};
    \addplot3[domain=0.125:0.25,y domain=0.5:0.75,surf]{2-8*x};
    \addplot3[domain=0.25:0.375,y domain=0.75:1,surf]{-2+8*x};
    \addplot3[domain=0.375:0.5,y domain=0.75:1,surf]{4-8*x};
    \addplot3[domain=0.25:0.5,y domain=0.5:0.75,surf]{0};  
    
    \addplot3[domain=0:0.25,y domain=0:0.5,surf]{4*x};
    \addplot3[domain=0.25:0.5,y domain=0:0.5,surf]{2-4*x};
    \addplot3[domain=0.5:0.75,y domain=0.5:1,surf]{-2+4*x};
    \addplot3[domain=0.75:1,y domain=0.5:1,surf]{4-4*x};
    
    \addplot3[domain=0.5:0.75,y domain=0.25:0.5,surf]{0};  
    \addplot3[domain=0.5:0.625,y domain=0:0.25,surf]{-4+8*x};
    \addplot3[domain=0.625:0.75,y domain=0:0.25,surf]{6-8*x};
    \addplot3[domain=0.75:0.875,y domain=0.25:0.5,surf]{-6+8*x};
    \addplot3[domain=0.875:1,y domain=0.25:0.5,surf]{8-8*x};
    \addplot3[domain=0.75:1,y domain=0:0.25,surf]{0};  
  \end{axis}
\end{tikzpicture}\end{center}
\medskip

We then define the order-bounded map \(f\colon[0,1]\to L^0(\mathscr L_1)\) as
\[
f_p\coloneqq[\Lambda(p,\cdot)]_{\mathscr L_1}\in L^0(\mathscr L_1)\quad\text{ for every }p\in[0,1].
\]
We claim that \(f\colon([0,1],\sfd_{\rm e})\to(L^0(\mathscr L_1),\Phi_{L^0(\mathscr L_1)})\) is uniformly
continuous. To prove it, let us denote \(G_n\coloneqq\bigcup_{i=1}^n\bigcup_{Q\in\mathcal Q_i}Q\) for all \(n\in\N\).
Note that if \(0<r\leq 2^{-n}\) and \(a\in[0,1-r]\) are given, then there exists a closed set \(P_{n,r,a}\subseteq[0,1]\)
with \(\mathscr L_1(P_{n,r,a})=1-2^{-n+1}\) and \([a,a+r]\times P_{n,r,a}\subseteq G_n\). Now, fix any \(\varepsilon>0\)
and \(\lambda\in(0,1)\). Choose first \(n\in\N\) so that \(2^{-n+1}<\lambda\), and then \(r\in(0,2^{-n})\) so that
\(2^{n+1}r<\varepsilon\). Note that the function \(\Lambda(\cdot,x)|_{G_n^x}\) is \(2^{n+1}\)-Lipschitz for every \(x\in[0,1]\),
where we set \(G_n^x\coloneqq\{p\in[0,1]:(p,x)\in G_n\}\). Hence, for every \(a\in[0,1-r]\) we can estimate
\[\begin{split}
\sup_{p,q\in[a,a+r]}\mathscr L_1(\{|f_p-f_q|\geq\varepsilon\})&=1-\inf_{p,q\in[a,a+r]}\mathscr L_1(\{|f_p-f_q|<\varepsilon\})\\
&\leq 1-\mathscr L_1(P_{n,r,a})=\frac{1}{2^{n-1}}<\lambda.
\end{split}\]
This gives \(\{(f_p,f_q):p,q\in[0,1],|p-q|\leq r\}\subseteq\mathcal U(\varepsilon,\lambda)\), thus
\(f\colon([0,1],\sfd_{\rm e})\to(L^0(\mathscr L_1),\Phi_{L^0(\mathscr L_1)})\) is uniformly continuous
by the arbitrariness of \(\varepsilon\), \(\lambda\). We claim that
\(f\notin{\rm UC}_{\rm ord}([0,1];L^0(\mathscr L_1))\). For any \(n\in\N\) and \(i=1,\ldots,2^n\), we have
\(\bigvee_{p,q\in[(i-1)2^{-n},i2^{-n}]}|f_p-f_q|\geq\1_{[1-i 2^{-n},1-(i-1)2^{-n}]}^{\mathscr L_1}\).
Letting \(B_r(\cdot)\coloneqq\{(p,q)\in[0,1]^2:|p-q|<r\}\) for every \(r\in(0,1)\), we thus have that
\[
{\rm Var}(f;B_{2^{-n}}(\cdot))\geq\bigvee_{i=1}^{2^n}\1_{[1-i 2^{-n},1-(i-1)2^{-n}]}^{\mathscr L_1}=\1_{[0,1]}^{\mathscr L_1},
\]
yielding \(\bigwedge_{r>0}{\rm Var}(f;B_r(\cdot))=\1_{[0,1]}^{\mathscr L_1}\neq 0\). Hence, \(f\) is not uniformly order continuous.
\fr}\end{example}
\begin{example}\label{ex:UC_2}{\rm
Next, we construct an example of a map \(g\in{\rm UC}_{\rm ord}([0,1];L^0(\mathscr L_1))\) such that
\(g\colon([0,1],\sfd_{\rm e})\to(L^0(\mathscr L_1),\Phi_{L^0(\mathscr L_1)}^c)\) is not uniformly continuous,
with \(\sfd_{\rm e}\), \(\mathscr L_1\) as in Example \ref{ex:UC_1}.
Borrowing the notation from Example \ref{ex:UC_1}, we define the function \(\Theta\colon[0,1]^2\to[0,1]\) as
\[
\Theta\coloneqq\Lambda_{Q_1}+\Lambda_{Q_{31}}+\Lambda_{Q_{331}}+\Lambda_{Q_{3331}}+\ldots
\]
and the order-bounded map \(g\colon[0,1]\to L^0(\mathscr L_1)\) as \(g_p\coloneqq[\Theta(p,\cdot)]_{\mathscr L_1}\)
for every \(p\in[0,1]\).
\medskip

\begin{center}\begin{tikzpicture}
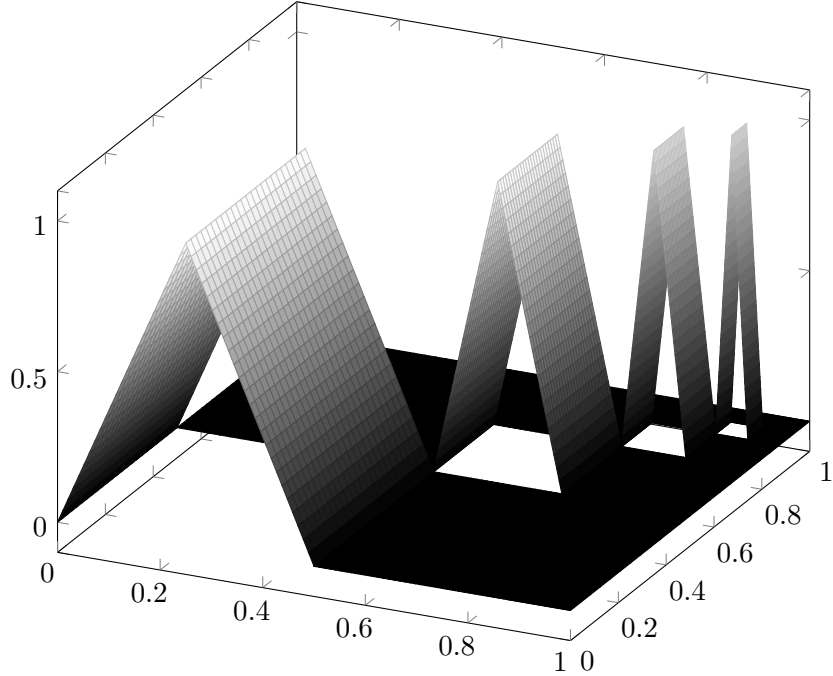
  \begin{axis}[title={\emph{Graph of \(\Lambda_{Q_1}+\Lambda_{Q_{31}}+\Lambda_{Q_{331}}+\Lambda_{Q_{3331}}\):}},
  colormap/blackwhite,width=10cm,height=8.5cm,scale only axis]
    \addplot3[domain=0:1/2,y domain=1/2:1,surf]{0};
    \addplot3[domain=0:1/4,y domain=0:1/2,surf]{4*x};
    \addplot3[domain=1/4:1/2,y domain=0:1/2,surf]{2-4*x};

    \addplot3[domain=1/2:3/4,y domain=3/4:1,surf]{0};
    \addplot3[domain=1/2:5/8,y domain=1/2:3/4,surf]{-4+8*x};
    \addplot3[domain=5/8:3/4,y domain=1/2:3/4,surf]{6-8*x};

    \addplot3[domain=3/4:7/8,y domain=7/8:1,surf]{0};
    \addplot3[domain=3/4:13/16,y domain=3/4:7/8,surf]{-12+16*x};
    \addplot3[domain=13/16:7/8,y domain=3/4:7/8,surf]{14-16*x};

    \addplot3[domain=7/8:15/16,y domain=15/16:1,surf]{0};
    \addplot3[domain=7/8:29/32,y domain=7/8:15/16,surf]{-28+32*x};
    \addplot3[domain=29/32:15/16,y domain=7/8:15/16,surf]{30-32*x};

    \addplot3[domain=1/2:3/4,y domain=0:1/2,surf]{0};
    \addplot3[domain=3/4:7/8,y domain=0:3/4,surf]{0};
    \addplot3[domain=7/8:15/16,y domain=0:7/8,surf]{0};
    \addplot3[domain=15/16:1,y domain=0:1,surf]{0};
        
  \end{axis}
\end{tikzpicture}\end{center}
\medskip

We claim that \(g\in{\rm UC}_{\rm ord}([0,1];L^0(\mathscr L_1))\). Indeed, it can be readily checked that
\[
{\rm Var}(g;B_{2^{-n}}(\cdot))=\1_{(1-2^{-n+1},1)}^{\mathscr L_1}+\sum_{1\leq i<n}2^{i-n+1}\1_{(1-2^{-i+1},1-2^{-i})}^{\mathscr L_1}
\quad\text{ for every }n\geq 2,
\]
thus \({\rm Var}(g;B_{2^{-2m}}(\cdot))\leq\1_{(1-2^{-m},1)}^{\mathscr L_1}+2^{-m+1}\1_{(0,1-2^{-m})}^{\mathscr L_1}\)
for every \(m\in\N\), which implies that
\[
\bigwedge_{r>0}{\rm Var}(g;B_r(\cdot))\leq\bigwedge_{m\in\N}{\rm Var}(g;B_{2^{-2m}}(\cdot))
\leq\bigwedge_{m\in\N}\big(\1_{(1-2^{-m},1)}^{\mathscr L_1}+2^{-m+1}\1_{(0,1-2^{-m})}^{\mathscr L_1}\big)=0.
\]
This yields \(g\in{\rm UC}_{\rm ord}([0,1];L^0(\mathscr L_1))\). Finally, we claim that
\(g\colon([0,1],\sfd_{\rm e})\to(L^0(\mathscr L_1),\Phi_{L^0(\mathscr L_1)}^c)\) is not uniformly continuous
(in fact, it is not even continuous at \(p=1\)). Indeed, we have that \(g_1=0\) and
\[
g_{1-2^{-2n-1}}=\1_{(1-2^{-n+1},1-2^{-n})}^{\mathscr L_1}\quad\text{ for every }n\in\N,
\]
thus in particular \(\mathscr L_1(\{|g_{1-2^{-2n-1}}-g_1|>2^{-1}\1_{[0,1]}^{\mathscr L_1}\})>0\) for every \(n\in\N\).
This shows that the map \(g\colon([0,1],\sfd_{\rm e})\to(L^0(\mathscr L_1),\Phi_{L^0(\mathscr L_1)}^c)\) is not continuous
at \(1\), as we claimed.
\fr}\end{example}
\begin{lemma}\label{lem:UC_is_L1}
Let \((\Omega,\Phi)\) be a uniform space whose induced topology is Lindel\"{o}f. Then 
\[
{\rm UC}_{{\rm ord},\mu}(\Omega;{\rm M})\subseteq L^\infty_\mu(\Omega;{\rm M})
\quad\text{ for every }\mu\in\mathcal M(\Omega;L^0(\mm)),
\]
where \(L^\infty_\mu(\Omega;{\rm M})\) is defined as in \eqref{defLinfmuOmM}.
In particular, each \(v\in{\rm UC}_{{\rm ord},\mu}(\Omega;{\rm M})\) is \(\mu\)-integrable.
\end{lemma}
\begin{proof}
Fix any \(v\in{\rm UC}_{{\rm ord},\mu}(\Omega;{\rm M})\), with representative \(\bar v\in{\rm UC}_{\rm ord}(\Omega;{\rm M})\). For
any \(n\in\N\), we can find an entourage \(\mathcal U_n\in\Phi\) such that \(\sfd_{L^0(\mm)}({\rm Var}(\bar v;\mathcal U_n),0)\leq 1/n\).
Since the topology induced by \(\Phi\) is Lindel\"{o}f, we can find points \((p^n_i)_{i\in\N}\subseteq\Omega\) such that
\(\Omega=\bigcup_{i\in\N}\mathcal U_n[p^n_i]\), where
\[
\mathcal U[p]\coloneqq\{q\in\Omega\;|\;(p,q)\in\mathcal U\}\quad\text{ for every }\mathcal U\in\Phi\text{ and }p\in\Omega.
\]
We define the Borel partition \((A^n_i)_{i\in\N}\) of \(\Omega\) as
\(A^n_i\coloneqq\mathcal U_n[p^n_i]\setminus\bigcup_{j<i}\mathcal U_n[p^n_j]\). Now, let us define
\[
v^n\coloneqq\sum_{i\in\N}\1_{A^n_i}^\mu\bar v_{p^n_i}\in S_\mu(\Omega;{\rm M}).
\]
Since \(\bar v\colon\Omega\to{\rm M}\) is order bounded, we actually have
\((v^n)_{n\in\N}\subseteq L^\infty_\mu(\Omega;{\rm M})\). Moreover, we have
\[
|v-v^n|^{\infty,\mu}=\bigvee_{i\in\N}\bigvee_{p\in A^n_i}|\bar v_p-\bar v_{p^n_i}|\leq
\bigvee_{i\in\N}\bigvee_{p\in\mathcal U_n[p^n_i]}|\bar v_p-\bar v_{p^n_i}|\leq{\rm Var}(\bar v;\mathcal U_n)\quad\text{ for every }n\in\N,
\]
which gives that \(|v-v^n|^{\infty,\mu}\to 0\) in \(L^0(\mm)\) as \(n\to\infty\), so that accordingly
\(v\in L^\infty_\mu(\Omega;{\rm M})\). The final claim then follows from Proposition \ref{prop:L_infty_in_L_1}.
\end{proof}

In view of Lemma \ref{lem:UC_is_L1}, it makes sense to write
\[
\int v\cdot\d\mu\in{\rm M}\quad\text{ for every }v\in{\rm UC}_{\rm ord}(\Omega;{\rm M}).
\]
Note also that \({\rm UC}_{\rm ord}(\Omega;{\rm M})\ni v\mapsto\int v\cdot\d\mu\in{\rm M}\) is an \(L^0(\mm)\)-linear operator and that
\begin{equation}\label{eq:fund_ineq_int_UC}
\bigg|\int v\cdot\d\mu\bigg|\leq|\mu|_{\rm TV}|v|^\infty\quad\text{ for every }v\in{\rm UC}_{\rm ord}(\Omega;{\rm M}).
\end{equation}

We equip \({\rm UC}_{\rm ord}(\Omega;L^0(\mm))\) with the following partial order: given
\(f,g\in{\rm UC}_{\rm ord}(\Omega;L^0(\mm))\), we declare that
\[
f\leq g\quad\Longleftrightarrow\quad f_p\leq g_p\;\text{ for every }p\in\Omega.
\]
Note that \(f\vee g,f\wedge g\in{\rm UC}_{\rm ord}(\Omega;L^0(\mm))\) for every \(f,g\in{\rm UC}_{\rm ord}(\Omega;L^0(\mm))\),
where we define
\begin{equation}\label{eq:def_sup_UC}
(f\vee g)_p\coloneqq f_p\vee g_p\in L^0(\mm),\qquad(f\wedge g)_p\coloneqq f_p\wedge g_p\in L^0(\mm)\;\quad\text{ for every }p\in\Omega.
\end{equation}
Indeed, letting \(\theta\) be the \(1\)-Lipschitz function \(\R\ni t\mapsto t\vee 0\),
for any \(a\in{\rm UC}_{\rm ord}(\Omega;L^0(\mm))\) we have that
\[
\bigwedge_{\mathcal U\in\Phi}{\rm Var}(a\vee 0;\mathcal U)
=\bigwedge_{\mathcal U\in\Phi}\bigvee_{(p,q)\in\mathcal U}|\theta\circ a_p-\theta\circ a_q|
\leq\bigwedge_{\mathcal U\in\Phi}\bigvee_{(p,q)\in\mathcal U}|a_p-a_q|
=\bigwedge_{\mathcal U\in\Phi}{\rm Var}(a;\mathcal U)=0,
\]
which shows that \(a\vee 0\in{\rm UC}_{\rm ord}(\Omega;L^0(\mm))\), thus also
\(f\vee g=(f-g)\vee 0+g\in{\rm UC}_{\rm ord}(\Omega;L^0(\mm))\) and \(f\wedge g=-(-f)\vee(-g)\in{\rm UC}_{\rm ord}(\Omega;L^0(\mm))\).
It can also be readily checked that
\[
({\rm UC}_{\rm ord}(\Omega;L^0(\mm)),\leq)\quad\text{ is a Riesz space,}
\]
where suprema and infima are given by the formulas in \eqref{eq:def_sup_UC}.
\begin{definition}[Positive operator]\label{def:posit_operators}
Let \((\Omega,\Phi)\) be a uniform space. Then we say that an \(L^0(\mm)\)-linear operator
\(L\colon{\rm UC}_{\rm ord}(\Omega;L^0(\mm))\to L^0(\mm)\) is \textbf{positive} provided it holds that
\[
L(f)\geq 0\quad\text{ for every }f\in{\rm UC}_{\rm ord}(\Omega;L^0(\mm))\text{ such that }f\geq 0.
\]
We denote by \({\rm UC}_{\rm ord}(\Omega;L^0(\mm))^*_+\) the space of positive operators
\(L\colon{\rm UC}_{\rm ord}(\Omega;L^0(\mm))\to L^0(\mm)\).
\end{definition}

The notation \({\rm UC}_{\rm ord}(\Omega;L^0(\mm))^*_+\) is compatible with the fact that all positive operators
are in fact homomorphisms of random normed modules, as we will discuss in the following remark.
\begin{remark}\label{rmk:posit_is_cont}{\rm
Let \(L\colon{\rm UC}_{\rm ord}(\Omega;L^0(\mm))\to L^0(\mm)\) be a positive \(L^0(\mm)\)-linear operator. Then
\[
L\in{\rm UC}_{\rm ord}(\Omega;L^0(\mm))^*,\qquad|L|=L(\1_\Omega\1_\X^\mm).
\]
Indeed, for any \(f\in{\rm UC}_{\rm ord}(\Omega;L^0(\mm))\) we have that \(\pm f+\1_\Omega|f|^\infty\geq 0\) and thus
\[
\pm L(f)+|f|^\infty L(\1_\Omega\1_\X^\mm)=L(\pm f)+L(\1_\Omega|f|^\infty)=L(\pm f+\1_\Omega|f|^\infty)\geq 0,
\]
which can be rearranged as
\[
|L(f)|\leq L(\1_\Omega\1_\X^\mm)|f|^\infty\quad\text{ for every }f\in{\rm UC}_{\rm ord}(\Omega;L^0(\mm)).
\]
Therefore, \(L\in{\rm UC}_{\rm ord}(\Omega;L^0(\mm))^*\) and \(|L|\leq L(\1_\Omega\1_\X^\mm)\).
On the contrary, since \(|\1_\Omega\1_\X^\mm|^\infty=\1_\X^\mm\), we have that
\[
|L|=\bigvee\big\{L(f)\;\big|\;f\in{\rm UC}_{\rm ord}(\Omega;L^0(\mm)),\,|f|^\infty\leq\1_\X^\mm\big\}\geq L(\1_\Omega\1_\X^\mm),
\]
thus showing that \(|L|=L(\1_\Omega\1_\X^\mm)\). This proves the claim.
\fr}\end{remark}

The proof of the next result follows along the lines of \cite[Proposition 12 of Chapter 21]{royden2010real}.
\begin{proposition}[Decomposition into positive operators]\label{prop:decomp_into_posit}
Let \((\Omega,\Phi)\) be a uniform space. Let \(L\in{\rm UC}_{\rm ord}(\Omega;L^0(\mm))^*\) be given. Then there exist
\(L_\pm\in{\rm UC}_{\rm ord}(\Omega;L^0(\mm))^*_+\) such that
\[
L=L_+ -L_-,\qquad|L|=|L_+|+|L_-|.
\]
\end{proposition}
\begin{proof}
Letting \(\mathfrak P\coloneqq\{f\in{\rm UC}_{\rm ord}(\Omega;L^0(\mm)):f\geq 0\}\), we define
\(\tilde L_+\colon\mathfrak P\to L^0_+(\mm)\) as
\[
\tilde L_+(f)\coloneqq\bigvee\big\{L(r)\;\big|\;r\in\mathfrak P,\,r\leq f\big\}\quad\text{ for every }f\in\mathfrak P.
\]
Note that \(0\leq L(r)\leq|L||r|^\infty\leq|L||f|^\infty\) if \(r,f\in\mathfrak P\) and \(r\leq f\), so that
\(0\leq\tilde L_+(f)\leq|L||f|^\infty\), in particular \(\tilde L_+(f)\in L^0_+(\mm)\). Moreover,
\(f+g,h\cdot f\in\mathfrak P\) for every \(f,g\in\mathfrak P\) and \(h\in L^0_+(\mm)\). We claim that
\begin{subequations}\begin{align}
\label{eq:decomp_into_posit_1}
\tilde L_+(f+g)=\tilde L_+(f)+\tilde L_+(g)&\quad\text{ for every }f,g\in\mathfrak P,\\
\label{eq:decomp_into_posit_2}
\tilde L_+(h\cdot f)=h\,\tilde L_+(f)&\quad\text{ for every }f\in\mathfrak P\text{ and }h\in L^0_+(\mm).
\end{align}\end{subequations}
Given \(f,g,r,s\in\mathfrak P\) with \(r\leq f\) and \(s\leq g\), we have that \(r+s\leq f+g\) and thus
\[
L(r)+L(s)=L(r+s)\leq\tilde L_+(f+g).
\]
Taking the supremum over all such \(r\) and \(s\), we deduce that \(\tilde L_+(f)+\tilde L_+(g)\leq\tilde L_+(f+g)\).
Moreover, if \(t\in\mathfrak P\) with \(t\leq f+g\) is given, then \(0\leq t\wedge f\leq f\) and
\(0\leq t-t\wedge f\leq g\), thus
\[
L(t)=L(t-t\wedge f)+L(t\wedge f)\leq\tilde L_+(f)+\tilde L_+(g).
\]
Taking the supremum over all such \(t\), we deduce that \(\tilde L_+(f+g)\leq\tilde L_+(f)+\tilde L_+(g)\).
All in all, \eqref{eq:decomp_into_posit_1} is proved. Next, if \(f\in\mathfrak P\) and \(h\in L^0_+(\mm)\),
then \(\{r\in\mathfrak P:r\leq h\cdot f\}=\{h\cdot s:s\in\mathfrak P,s\leq f\}\) and thus accordingly
\[
\tilde L_+(h\cdot f)=\bigvee\big\{L(h\cdot s)\;\big|\;s\in\mathfrak P,\,s\leq f\big\}
=\bigvee\big\{h\,L(s)\;\big|\;s\in\mathfrak P,\,s\leq f\big\}=h\,\tilde L_+(f),
\]
proving \eqref{eq:decomp_into_posit_2}. Next, observe that if \(f\in{\rm UC}_{\rm ord}(\Omega;L^0(\mm))\)
and \(a,b\in L^0_+(\mm)\) are chosen so that \(f+\1_\Omega a\geq 0\) and \(f+\1_\Omega b\geq 0\), then
it follows from \eqref{eq:decomp_into_posit_1} that
\[\begin{split}
\tilde L_+(f+\1_\Omega a)-\tilde L_+(\1_\Omega a)&=\tilde L_+(f+\1_\Omega a+\1_\Omega b)-\tilde L_+(\1_\Omega b)-\tilde L_+(\1_\Omega a)\\
&=\tilde L_+(f+\1_\Omega b)-\tilde L_+(\1_\Omega b).
\end{split}\]
Therefore, the following definition is well posed: we define \(L_+\colon{\rm UC}_{\rm ord}(\Omega;L^0(\mm))\to L^0(\mm)\) as
\[
L_+(f)\coloneqq\tilde L_+(f+\1_\Omega a)-\tilde L_+(\1_\Omega a)\quad\text{ if }f\in{\rm UC}_{\rm ord}(\Omega;L^0(\mm)),\;
a\in L^0_+(\mm)\text{ and }f+\1_\Omega a\geq 0;
\]
note that, given any \(f\in{\rm UC}_{\rm ord}(\Omega;L^0(\mm))\), the set of \(a\in L^0_+(\mm)\) such that \(f+\1_\Omega a\geq 0\)
is non-empty (as it contains \(|f|^\infty\)). Observe also that
\begin{equation}\label{eq:decomp_into_posit_3}
L_+(f)=\tilde L_+(f)\quad\text{ for every }f\in\mathfrak P.
\end{equation}
Indeed, given \(f\in\mathfrak P\), we can choose \(a=0\), so that \(L_+(f)=\tilde L_+(f+0)-\tilde L_+(f)=\tilde L_+(f)\).

We claim that
\[
L_+\in{\rm UC}_{\rm ord}(\Omega;L^0(\mm))^*_+.
\]
The operator \(L_+\) is positive by \eqref{eq:decomp_into_posit_3}. Let us now show that \(L_+\) is \(L^0(\mm)\)-linear.
Applying \eqref{eq:decomp_into_posit_1} and \eqref{eq:decomp_into_posit_3}, we obtain that
\begin{equation}\label{eq:decomp_into_posit_4}\begin{split}
L_+(f+g)&=\tilde L_+(f+g+\1_\Omega|f|^\infty+\1_\Omega|g|^\infty)-\tilde L_+(\1_\Omega|f|^\infty+\1_\Omega|g|^\infty)\\
&=\tilde L_+(f+\1_\Omega|f|^\infty)+\tilde L_+(g+\1_\Omega|g|^\infty)-\tilde L_+(\1_\Omega|f|^\infty)-\tilde L_+(\1_\Omega|g|^\infty)\\
&=L_+(f)+L_+(g)\quad\text{ for every }f,g\in{\rm UC}_{\rm ord}(\Omega;L^0(\mm)).
\end{split}\end{equation}
In particular, \(L_+(f)+L_+(-f)=L_+(0)=0\) for every \(f\in{\rm UC}_{\rm ord}(\Omega;L^0(\mm))\), or in other words
\begin{equation}\label{eq:decomp_into_posit_5}
L_+(-f)=-L_+(f)\quad\text{ for every }f\in{\rm UC}_{\rm ord}(\Omega;L^0(\mm)).
\end{equation}
Furthermore, it follows from \eqref{eq:decomp_into_posit_2} that
\begin{equation}\label{eq:decomp_into_posit_6}\begin{split}
L_+(h\cdot f)&=\tilde L_+\big(h\cdot(f+\1_\Omega|f|^\infty)\big)-\tilde L_+(\1_\Omega h|f|^\infty)\\
&=h\big(\tilde L_+(f+\1_\Omega|f|^\infty)-\tilde L_+(\1_\Omega|f|^\infty)\big)\\
&=h\,L_+(f)\quad\text{ for every }h\in L^0_+(\mm)\text{ and }f\in{\rm UC}_{\rm ord}(\Omega;L^0(\mm)).
\end{split}\end{equation}
Hence, combining \eqref{eq:decomp_into_posit_4}, \eqref{eq:decomp_into_posit_5} and \eqref{eq:decomp_into_posit_6}, we conclude that
\[\begin{split}
L_+(h\cdot f)&=L_+\big((\1_{\{h\geq 0\}}^\mm h)\cdot f\big)+L_+\big((\1_{\{h<0\}}^\mm h)\cdot f\big)\\
&=L_+\big((\1_{\{h\geq 0\}}^\mm h)\cdot f\big)-L_+\big((-\1_{\{h<0\}}^\mm h)\cdot f\big)\\
&=\1_{\{h\geq 0\}}^\mm h\,L_+(f)-(-\1_{\{h<0\}}^\mm h)L_+(f)\\
&=\1_{\{h\geq 0\}}^\mm h\,L_+(f)+\1_{\{h<0\}}^\mm h\,L_+(f)=h\,L_+(f)
\end{split}\]
for all \(h\in L^0(\mm)\) and \(f\in{\rm UC}_{\rm ord}(\Omega;L^0(\mm))\). All in all, we have shown
that \(L_+\) is \(L^0(\mm)\)-linear. Recalling Remark \ref{rmk:posit_is_cont}, we finally conclude that
\(L_+\in{\rm UC}_{\rm ord}(\Omega;L^0(\mm))^*_+\), proving the claim.

Define \(L_-\coloneqq L_+-L\in{\rm UC}_{\rm ord}(\Omega;L^0(\mm))^*\). Given \(f\in\mathfrak P\), we have
\(L(f)\leq\tilde L_+(f)=L_+(f)\), thus \(L_-(f)=L_+(f)-L(f)\geq 0\). This implies that \(L_+\in{\rm UC}_{\rm ord}(\Omega;L^0(\mm))^*_+\).
Note that \(L=L_+-L_-\), in particular \(|L|\leq|L_+|+|L_-|\). Conversely, given \(f\in\mathfrak P\) with
\(|f|^\infty\leq\1_\X^\mm\), we have \(|2f-\1_\Omega\1_\X^\mm|^\infty\leq\1_\X^\mm\), which (taking also
Remark \ref{rmk:posit_is_cont} into account) implies that
\[\begin{split}
|L_+|+|L_-|&=L_+(\1_\Omega\1_\X^\mm)+L_-(\1_\Omega\1_\X^\mm)=2\,L_+(\1_\Omega\1_\X^\mm)-L(\1_\Omega\1_\X^\mm)\\
&=2\bigvee\big\{L(f)\;\big|\;f\in\mathfrak P,\,|f|^\infty\leq\1_\X^\mm\big\}-L(\1_\Omega\1_\X^\mm)\\
&=\bigvee\big\{L(2f-\1_\Omega\1_\X^\mm)\;\big|\;f\in\mathfrak P,\,|f|^\infty\leq\1_\X^\mm\big\}\leq|L|.
\end{split}\]
This shows that \(|L|=|L_+|+|L_-|\), thus completing the proof of the statement.
\end{proof}
\subsection{Auxiliary results}\label{s:aux_RMK}
We recall that a Hausdorff space \((\Omega,\tau)\) is said to be \textbf{perfectly normal} if any two disjoint closed
subsets of \(\Omega\) can be `precisely separated' by a continuous function, i.e.\ for any \(C_1,C_2\subseteq\Omega\)
closed with \(C_1\cap C_2=\varnothing\) there exists \(\phi\in C(\Omega;[0,1])\) such that \(\{\phi=0\}=C_1\) and
\(\{\phi=1\}=C_2\). For instance, each metrisable space is a perfectly-normal Hausdorff space.
\begin{remark}\label{rmk:cutoff_perf_norm}{\rm
If \((\Omega,\tau)\) is a perfectly-normal Hausdorff space and \(U\subseteq\Omega\) is an open set, then there exists
\((\phi_n)_{n\in\N}\subseteq C(\Omega;[0,1])\) satisfying \(\phi_n(p)\nearrow\1_U(p)\) as \(n\to\infty\) for every \(p\in\Omega\).

To prove it, we first take a function \(\phi\in C(\Omega;[0,1])\) that precisely separates \(\Omega\setminus U\) and \(\varnothing\),
i.e.\ \(\{\phi=0\}=\Omega\setminus U\).  Letting \(C_n\coloneqq\{\phi\geq 1/n\}\) for every \(n\in\N\), we have that each
\(C_n\) is a closed set that does not intersect \(\Omega\setminus U\), and we have \(\bigcup_{n\in\N}C_n=U\) by the
continuity of \(\phi\). Using again the perfect normality of \((\Omega,\tau)\), we find functions
\((\phi_n)_{n\in\N}\subseteq C(\Omega;[0,1])\) satisfying \(\{\phi_n=0\}=\Omega\setminus U\) and \(\{\phi_n=1\}=C_n\)
for every \(n\in\N\). Clearly, the sequence \((\phi_n)_{n\in\N}\) fulfills the sought properties.
\fr}\end{remark}
\begin{lemma}\label{lem:meas_mu_x}
Let \((\Omega,\tau)\) be a perfectly-normal Hausdorff space. Then a given collection of measures
\((\mu_x)_{x\in\X}\subseteq\mathcal M(\Omega)\) satisfies
\begin{equation}\label{eq:meas_mu_x_lemma_hp}
\X\ni x\mapsto\int\phi\,\d\mu_x\in\R\quad\text{ is }\mm\text{-measurable for every }\phi\in C_b(\Omega)\text{ with }\phi\geq 0
\end{equation}
if and only if \((\mu_x)_{x\in\X}\) is an \(\mm\)-measurable collection of measures.
\end{lemma}
\begin{proof}
Assume \eqref{eq:meas_mu_x_lemma_hp} holds. Let \(\mathcal D\) be the family of all those sets
\(A\in\mathscr B(\Omega)\) for which the function \(x\mapsto\mu_x(A)\) is \(\mm\)-measurable. Then:
\begin{itemize}
\item[\(\rm i)\)] Taking \(\phi=\1_\Omega\), we see that \(\Omega\in\mathcal D\).
\item[\(\rm ii)\)] If \(A,B\in\mathcal D\) with \(A\subseteq B\) are given, then \(x\mapsto\mu_x(B\setminus A)=\mu_x(B)-\mu_x(A)\)
is an \(\mm\)-measurable function, thus accordingly \(B\setminus A\in\mathcal D\).
\item[\(\rm iii)\)] Let \((A_n)_{n\in\N}\subseteq\mathcal D\) be such that \(A_n\subseteq A_{n+1}\) for every \(n\in\N\).
Since \(A_{n+1}\setminus A_n\in\mathcal D\) for all \(n\in\N\) by ii), we deduce that the function
\[
\X\ni x\mapsto\mu_x\bigg(\bigcup_{n\in\N}A_n\bigg)=\mu_x(A_1)+\sum_{n\in\N}\mu_x(A_{n+1}\setminus A_n)\in\R
\]
is \(\mm\)-measurable, so that \(\bigcup_{n\in\N}A_n\in\mathcal D\).
\end{itemize}
All in all, we have shown that \(\mathcal D\) is a Dynkin system. Next, we claim that \(\tau\subseteq\mathcal D\).
To prove it, fix any \(U\in\tau\). By virtue of Remark \ref{rmk:cutoff_perf_norm}, we know that there exists a sequence
\((\phi_n)_{n\in\N}\subseteq C(\Omega;[0,1])\) such that \(\phi_n(p)\nearrow\1_U(p)\) as \(n\to\infty\) for every \(p\in\Omega\).
Using \eqref{eq:meas_mu_x_lemma_hp} and the dominated convergence theorem, we conclude that
\[
\X\ni x\mapsto\mu_x(U)=\lim_{n\to\infty}\int\phi_n\,\d\mu_x\quad\text{ is an }\mm\text{-measurable function,}
\]
which yields \(U\in\mathcal D\). Since \(\mathcal D\) is a Dynkin system containing the topology \(\tau\),
it follows from the Sierpi\'{n}ski--Dynkin \(\pi\)-\(\lambda\) theorem that \(\mathcal D=\mathscr B(\Omega)\),
thus \((\mu_x)_{x\in\X}\) is an \(\mm\)-measurable collection.

Conversely, assume \((\mu_x)_{x\in\X}\) is an \(\mm\)-measurable collection of measures. Given \(\phi\in C_b(\Omega)\) with
\(\phi\geq 0\), and letting \(M\geq 0\) denote its supremum norm, we have that the function
\[
\X\ni x\mapsto\int\phi\,\d\mu_x=\lim_{n\to\infty}\int\sum_{i=1}^n\frac{Mi}{n}\1_{A_{n,i}}\,\d\mu_x
=\lim_{n\to\infty}\sum_{i=1}^n\frac{Mi}{n}\mu_x(A_{n,i})
\]
is \(\mm\)-measurable, where we set \(A_{n,i}\coloneqq\big\{\frac{M(i-1)}{n}<\phi\leq\frac{Mi}{n}\big\}\), thus proving \eqref{eq:meas_mu_x_lemma_hp}.
\end{proof}
\begin{corollary}\label{cor:meas_|mu_x|}
Let \((\Omega,\sfd)\) be a compact metric space. Assume \((\mu_x)_{x\in\X}\subseteq\mathfrak M(\Omega)\)
is the foliation of some \(\mu\in\mathcal M(\Omega;L^0(\mm))\). Then \((|\mu_x|)_{x\in\X}\) is an \(\mm\)-measurable
collection of measures.
\end{corollary}
\begin{proof}
As \(C(\Omega)\) is separable, there is a countable dense subset \(\mathcal C\)
of \(\{\phi\in C(\Omega):0\leq\phi\leq\1_U\}\).
The classical Riesz--Markov--Kakutani theorem for real-valued Radon measures ensures that
\[
|\mu_x|(U)=\sup_{\phi\in\mathcal C}\int\phi\,\d\mu_x\quad\text{ for every }x\in\X\text{ and }U\in\tau.
\]
Hence, Lemma \ref{lem:meas_mu_x} implies that \(\X\ni x\mapsto|\mu_x|(U)\) is \(\mm\)-measurable for every \(U\in\tau\),
so that the set \(\mathcal D\) of all \(A\in\mathscr B(\Omega)\) for which \(x\mapsto|\mu_x|(A)\) is \(\mm\)-measurable
contains the topology \(\tau\). Using the fact that each \(|\mu_x|\) is a measure it is easy to check that \(\mathcal D\)
is a Dynkin system, thus the Sierpi\'{n}ski--Dynkin \(\pi\)-\(\lambda\) theorem implies \(\mathcal D=\mathscr B(\Omega)\),
which means that \((|\mu_x|)_{x\in\X}\) is an \(\mm\)-measurable collection.
\end{proof}

We remind that a Hausdorff space is \emph{uniformisable} (i.e.\ it admits a uniform structure compatible with the topology)
if and only if it is completely regular. Moreover, any compact Hausdorff space is a complete uniform space with respect to
a unique uniform structure compatible with the topology.
\begin{lemma}\label{lem:int_mu_ptwse_descript}
Let \((\Omega,\tau)\) be a compact, perfectly-normal Hausdorff space. Let \((\mu_x)_{x\in\X}\)
be a foliation of some \(\mu\in\mathcal M(\Omega;L^0(\mm))\). Then for any \(\phi\in C_b(\Omega)\) and \(\bar f\in L^\infty(\mm)\) it holds that
\[
\bigg(\int\phi\bar f\cdot\,\d\mu\bigg)(x)=\bar f(x)\int\phi\,\d\mu_x\quad\text{ for }\mm\text{-a.e.\ }x\in\X.
\]
\end{lemma}
\begin{proof}
Fix any \(M>\|\phi\|_{C_b(\Omega)}\). For any \(n\in\N\), we define the function \(\phi_n\colon\Omega\to[-M,M]\) as
\[
\phi_n\coloneqq\sum_{i=-n}^{n-1}\frac{Mi}{n}\1_{A^n_i},\quad\text{ where we set }A^n_i\coloneqq\bigg\{\frac{Mi}{n}\leq\phi<\frac{M(i+1)}{n}\bigg\}\in\mathscr B(\Omega).
\]
Then \(\sup_\Omega|\phi-\phi_n|\leq\frac{1}{n}\) and thus \(|\phi\bar f-\phi_n\bar f|^\infty\leq\frac{1}{n}|\bar f|\). Hence,
for \(\mm\)-a.e.\ \(x\in\X\) we have that
\[\begin{split}
&\bigg|\bigg(\int\phi\bar f\cdot\,\d\mu\bigg)(x)-\bar f(x)\int\phi\,\d\mu_x\bigg|\\
\leq\,&\bigg|\bigg(\int(\phi-\phi_n)\bar f\cdot\d\mu\bigg)(x)\bigg|+\bigg|\bigg(\int\phi_n\bar f\cdot\,\d\mu\bigg)(x)-\bar f(x)\int\phi_n\,\d\mu_x\bigg|
+|\bar f|(x)\bigg|\int(\phi_n-\phi)\,\d\mu_x\bigg|\\
\leq\,&|\phi\bar f-\phi_n\bar f|^\infty(x)|\mu|_{\rm TV}(x)+|\bar f|(x)\bigg(\bigg|\sum_{i=-n}^{n-1}\frac{Mi}{n}\big(\mu(A^n_i)(x)-\mu_x(A^n_i)\big)\bigg|+
\sup_\Omega|\phi_n-\phi|\|\mu_x\|_{\rm TV}\bigg)\\
\leq\,&\frac{1}{n}|\bar f|(x)\big(|\mu|_{\rm TV}(x)+\|\mu_x\|_{\rm TV}\big)\to 0\quad\text{ as }n\to\infty,
\end{split}\]
thus proving the statement.
\end{proof}
\begin{lemma}\label{lem:countable_family_for_approx}
Let \((\Omega,\tau)\) be a second-countable Hausdorff space. Let \(\mathscr U\) be a countable base for the topology \(\tau\) that is closed under finite unions.
Then it holds that
\[
\inf_{U\in\mathscr U}\mu(A\Delta U)=0\quad\text{ for every }\mu\in\mathfrak M_+(\Omega)\text{ and }A\in\mathscr B(\Omega).
\]
Moreover, given any compact set \(K\subseteq\Omega\) and any measure \(\mu\in\mathfrak M_+(\Omega)\), we have that
\[
\mu(K)=\inf\big\{\mu(U)\;\big|\;U\in\mathscr U,\,K\subseteq U\big\}.
\]
\end{lemma}
\begin{proof}
Given any \(A\in\mathscr B(\Omega)\), \(\mu\in\mathfrak M_+(\Omega)\) and \(\varepsilon>0\), we take a compact set \(K\subseteq A\)
and an open set \(V\) containing \(A\) such that \(\mu(V\setminus K)\leq\varepsilon\). By the compactness of \(K\), we find
\(U\in\mathscr U\) such that \(K\subseteq U\subseteq V\), so that \(\mu(A\Delta U)\leq\varepsilon\). Taking \(A=K\), we obtain the final claim.
\end{proof}
\begin{remark}\label{rmk:unique_foliation_top}{\rm
If \((\Omega,\tau)\) is a second-countable Hausdorff space and \((\mu_x)_{x\in\X},(\nu_x)_{x\in\X}\subseteq\mathfrak M(\Omega)\)
are foliations of the same \(L^0\)-valued measure \(\mu\in\mathcal M(\Omega;L^0(\mm))\), then
\[
\mu_x=\nu_x\quad\text{ for }\mm\text{-a.e.\ }x\in\X.
\]
Indeed, the second-countability assumption ensures that \(\mathscr B(\Omega)\)
is countably generated, thus the claim follows from Proposition \ref{prop:unique_foliation}. Alternatively,
one can argue in the following way.

Fix some countable base \(\mathscr U\) for \(\tau\). For any \(U\in\mathscr U\), we have \([\mu_\cdot(U)]_\mm=\mu(U)=[\nu_\cdot(U)]_\mm\),
thus there exists an \(\mm\)-null set \(N_U\subseteq\X\) such that \(\mu_x(U)=\nu_x(U)\) for every \(x\in\X\setminus N_U\). Denote by
\(N\subseteq\X\) the \(\mm\)-null set \(\bigcup_{U\in\mathscr U}N_U\). For any \(A\in\mathscr B(\Omega)\), \(x\in\X\setminus N\)
and \(U\in\mathscr U\), we can estimate
\[
|\mu_x(A)-\nu_x(A)|\leq|\mu_x(A)-\mu_x(U)|+|\nu_x(U)-\nu_x(A)|\leq|\mu_x|(A\Delta U)+|\nu_x|(A\Delta U).
\]
Taking the infimum over all \(U\in\mathscr U\) and applying Lemma \ref{lem:countable_family_for_approx}, we conclude that
\[
\mu_x(A)=\nu_x(A)\quad\text{ for every }A\in\mathscr B(\Omega)\text{ and }x\in\X\setminus N,
\]
which gives the claim.
\fr}\end{remark}
\begin{lemma}\label{lem:dense_in_UC_ord}
Let \((\Omega,\tau)\) be a compact Hausdorff space and let \({\rm W}\) be a dense vector subspace of \(C(\Omega)\).
Then it holds that
\[
{\rm V}\coloneqq\bigg\{\sum_{i=1}^n\phi_i\1_{E_i}^\mm\;\bigg|\;n\in\N,\,(\phi_i)_{i=1}^n\subseteq{\rm W},
\,(E_i)_{i=1}^n\subseteq\Sigma\bigg\}
\]
is a dense vector subspace of \({\rm UC}_{\rm ord}(\Omega;L^0(\mm))\).
\end{lemma}
\begin{proof}
It readily follows from \cite[Remark 5.10]{Pas23} that
\[
\bigg\{\sum_{i=1}^n\phi_i\bar f_i\;\bigg|\;n\in\N,\,(\phi_i)_{i=1}^n\subseteq{\rm W},\,(\bar f_i)_{i=1}^n\subseteq L^0(\mm)\bigg\}
\]
is a dense \(L^0(\mm)\)-submodule of \({\rm UC}_{\rm ord}(\Omega;L^0(\mm))\). Thus, it remains to show that for any \(\phi\in{\rm W}\),
\(\bar f\in L^0(\mm)\) and \(\varepsilon>0\) there exists \(g\in{\rm V}\) such that \(\sfd_{L^0(\mm)}(|g-\phi\bar f|^\infty,0)\leq\varepsilon\).
We can find a partition \((E_i)_{i=1}^n\subseteq\Sigma\) of \(\X\) and numbers \((\lambda_i)_{i=1}^n\subseteq\R\) such that
\(\|\phi\|_{C(\Omega)}\sfd_{L^0(\mm)}\big(\bar f,\sum_{i=1}^n\lambda_i\1_{E_i}^\mm\big)\leq\varepsilon\). Then the element
\(g\coloneqq\phi\sum_{i=1}^n\lambda_i\1_{E_i}^\mm\in{\rm V}\) satisfies \(|g-\phi\bar f|^\infty\leq\|\phi\|_{C(\Omega)}\big|\bar f-\sum_{i=1}^n\lambda_i\1_{E_i}^\mm\big|\),
thus accordingly we have that \(\sfd_{L^0(\mm)}(|g-\phi\bar f|^\infty,0)\leq\|\phi\|_{C(\Omega)}\sfd_{L^0(\mm)}\big(\bar f,\sum_{i=1}^n\lambda_i\1_{E_i}^\mm\big)\leq\varepsilon\).
\end{proof}

As a consequence of Lemma \ref{lem:dense_in_UC_ord}, when \((\Omega,\sfd)\) is a compact metric space
there is a natural alternative characterisation of \({\rm UC}_{\rm ord}(\Omega;L^0(\mm))\), which we
are going to discuss. First, we remind that for any separable Banach space \(\mathbb B\) we can consider
the \emph{\(L^0\)-Lebesgue--Bochner space} \(L^0(\mm;\mathbb B)\), which is the complete random normed
module with base \((\X,\Sigma,\mm)\) of all \(\mm\)-measurable maps \(v\colon\X\to\mathbb B\) quotiented
up to \(\mm\)-a.e.\ equality. The \(L^0\)-norm of \(v\in L^0(\mm;\mathbb B)\) is given by
\[
|v|(x)\coloneqq\|v(x)\|_{\mathbb B}\quad\text{ for }\mm\text{-a.e.\ }x\in\X.
\]
See \cite[Definition 2.9]{Pas23}. Recall also that \(C(\Omega)\) is separable if \((\Omega,\sfd)\)
is a compact metric space.
\begin{proposition}\label{prop:charact_UC_ord}
Let \((\Omega,\sfd)\) be a compact metric space. Given any \(f\in L^0(\mm;C(\Omega))\), we fix an
\(\mm\)-measurable representative \(\bar f\colon\X\to C(\Omega)\) of \(f\) and we define
\(\sigma(f)\colon\Omega\to L^0(\mm)\) as
\[
\sigma(f)_p\coloneqq[\bar f(\cdot)(p)]_\mm\quad\text{ for every }p\in\Omega.
\]
Then \(\sigma(f)\in{\rm UC}_{\rm ord}(\Omega;L^0(\mm))\) and the resulting map
\(\sigma\colon L^0(\mm;C(\Omega))\to{\rm UC}_{\rm ord}(\Omega;L^0(\mm))\) is an isometric isomorphism
of random normed modules.
\end{proposition}
\begin{proof}
Clearly, \(\X\times\Omega\ni(x,p)\mapsto\bar f(x)(p)\in\R\) is a Carath\'{e}odory function, thus it is measurable
if its domain is equipped with the product \(\sigma\)-algebra \(\bar\Sigma_\mm\otimes\mathscr B(\Omega)\).
In particular, \(\bar f(\cdot)(p)\) is \(\mm\)-measurable for every \(p\in\Omega\) and thus
\(\sigma(f)_p=[\bar f(\cdot)(p)]_\mm\) is a well-defined element of \(L^0(\mm)\). Let us next check that
\(\sigma(f)\in{\rm UC}_{\rm ord}(\Omega;L^0(\mm))\). Since \(\sup_{p\in\Omega}|\bar f(x)(p)|=\|\bar f(x)\|_{C(\Omega)}<+\infty\)
for every \(x\in\X\), it follows from Remark \ref{rmk:suff_cond_ord-bdd} that \(\sigma(f)\colon\Omega\to L^0(\mm)\)
is order bounded. Letting
\[
v_r(x)\coloneqq\sup\big\{|\bar f(x)(p)-\bar f(x)(q)|\;\big|\;p,q\in\Omega,\,\sfd(p,q)<r\big\}
\quad\text{ for every }r>0\text{ and }x\in\X,
\]
it is easy to show that we obtain an \(\mm\)-measurable function \(v_r\colon\X\to[0,+\infty)\), which satisfies
\(v_r(x)\searrow 0\) as \(r\searrow 0\) for every \(x\in\X\) by the uniform continuity of \(\bar f(x)\colon\Omega\to\R\). Given that
\[
{\rm Var}(\sigma(f);B_r(\cdot))\leq v_r\quad\text{ holds }\mm\text{-a.e.\ on }\X,
\]
we deduce that \({\rm Var}(\sigma(f);B_r(\cdot))\to 0\) in \(L^0(\mm)\) as \(r\searrow 0\), thus proving
that \(\sigma(f)\colon\Omega\to L^0(\mm)\) is uniformly order continuous. All in all, we have shown that
\(\sigma(f)\in{\rm UC}_{\rm ord}(\Omega;L^0(\mm))\).

It is straightforward to check that the resulting map \(\sigma\colon L^0(\mm;C(\Omega))\to{\rm UC}_{\rm ord}(\Omega;L^0(\mm))\)
is \(L^0(\mm)\)-linear. Moreover, for any \(f\in L^0(\mm;C(\Omega))\) we can estimate
\[
|\sigma(f)|^\infty=\bigvee_{p\in\Omega}|\sigma(f)_p|\leq\big[\|\bar f(\cdot)\|_{C(\Omega)}\big]_\mm=|f|
\quad\text{ in the }\mm\text{-a.e.\ sense,}
\]
thus in particular \(\sigma\) is a homomorphism of random normed modules. Given an \(\mm\)-measurable simple
map \(f=\sum_{i=1}^n\1_{E_i}^\mm\phi_i\colon\X\to C(\Omega)\), where \((E_i)_{i=1}^n\subseteq\bar\Sigma_\mm\)
is a partition of the space \(\X\) and \((\phi_i)_{i=1}^n\subseteq C(\Omega)\), we have that
\[
|\sigma(f)|^\infty=\sum_{i=1}^n\1_{E_i}^\mm\|\phi_i\|_{C(\Omega)}=|f|.
\]
Thanks to the density of \(\mm\)-measurable simple maps \(f\colon\X\to C(\Omega)\) in \(L^0(\mm;C(\Omega))\),
we deduce that \(|\sigma(f)|^\infty=|f|\) for every \(f\in L^0(\mm;C(\Omega))\). Therefore, to conclude it remains
to check that \(\sigma\) is surjective. To this aim, note that the image of the set of all \(\mm\)-measurable simple
maps \(f\colon\X\to C(\Omega)\) under the operator \(\sigma\) is given by the set
\[
{\rm V}\coloneqq\bigg\{\sum_{i=1}^n\phi_i\1_{E_i}^\mm\;\bigg|\;n\in\N,\,(\phi_i)_{i=1}^n\subseteq C(\Omega),
\,(E_i)_{i=1}^n\subseteq\bar\Sigma_\mm\text{ partition of }\X\bigg\},
\]
which is dense in \({\rm UC}_{\rm ord}(\Omega;L^0(\mm))\) by Lemma \ref{lem:dense_in_UC_ord}.
We conclude that \(\sigma\) is surjective.
\end{proof}

We remind that a topological space \((K,\sigma)\) is said to be \textbf{extremally disconnected} if the closure of
any open subset of \(K\) is open.
\begin{remark}\label{rmk:properties_ext_disconn}{\rm
We now report on some useful properties of extremally-disconnected spaces:
\begin{itemize}
\item[\(\rm i)\)] If \((K,\sigma)\) is an extremally-disconnected compact Hausdorff space, then the Baire \(\sigma\)-algebra
\(\mathscr B_0(K)\) is generated by the algebra of all \emph{clopen} subsets of \(K\) (i.e.\ of those subsets of
\(K\) that are both open and closed).
\item[\(\rm ii)\)] Under the same assumptions, it holds that
\[
\bigg\{\sum_{i=1}^n\lambda_i\1_{A_i}\;\bigg|\;n\in\N,\,(\lambda_i)_{i=1}^n\subseteq\R,\,(A_i)_{i=1}^n\text{ clopen subsets of }K\bigg\}
\]
is a dense vector subspace of \(C(K)\).
\item[\(\rm iii)\)] Given any compact Hausdorff space \((\Omega,\tau)\), there exist an extremally-disconnected compact Hausdorff
space \((K,\sigma)\) and a surjective continuous map \(\alpha\colon K\to\Omega\).
\end{itemize}
For a proof of i) and ii) we refer to \cite[Lemma 1]{Hartig83}; iii) was observed in \cite[Lemma 3]{Hartig83}.
\fr}\end{remark}

Let \((K,\sigma)\), \((\Omega,\tau)\) be compact Hausdorff spaces and \(\alpha\colon K\to\Omega\) a continuous map.
Then \(\alpha\) is uniformly continuous (by the compactness of \(K\)), whence it easily follows that
\[
f\circ\alpha\in{\rm UC}_{\rm ord}(K;L^0(\mm))\quad\text{ for every }f\in{\rm UC}_{\rm ord}(\Omega;L^0(\mm))
\]
and that \({\rm UC}_{\rm ord}(\Omega;L^0(\mm))\ni f\mapsto f\circ\alpha\in{\rm UC}_{\rm ord}(K;L^0(\mm))\) is a homomorphism
of random normed modules satisfying \(|f\circ\alpha|^\infty\leq|f|^\infty\) for every \(f\in{\rm UC}_{\rm ord}(\Omega;L^0(\mm))\).
If in addition \(\alpha\) is surjective, then we have that \(|f\circ\alpha|^\infty=|f|^\infty\) for every
\(f\in{\rm UC}_{\rm ord}(\Omega;L^0(\mm))\).

\begin{remark}\label{rmk:adj_operator}{\rm
Letting \((K,\sigma)\), \((\Omega,\tau)\) and \(\alpha\) be as above, we can consider the \textbf{adjoint} map
\[
\alpha^\#\colon{\rm UC}_{\rm ord}(K;L^0(\mm))^*\to{\rm UC}_{\rm ord}(\Omega;L^0(\mm))^*
\]
of the homomorphism of random normed modules \(f\mapsto f\circ\alpha\). The map \(\alpha^\#\), characterised by
\begin{equation}\label{eq:def_alpha_sharp}
(\alpha^\#T)(f)=T(f\circ\alpha)\quad\text{ for every }T\in{\rm UC}_{\rm ord}(K;L^0(\mm))^*\text{ and }f\in{\rm UC}_{\rm ord}(\Omega;L^0(\mm)),
\end{equation}
is a homomorphism of random normed modules, and \(|\alpha^\#|\leq\1_\X^\mm\) (see \cite[Proposition 2.4]{Iko:Pas:Sou:22}).

We claim that if \(\alpha\colon K\to\Omega\) is surjective, then
\[
\forall L\in{\rm UC}_{\rm ord}(\Omega;L^0(\mm))^*\quad\exists\,T\in{\rm UC}_{\rm ord}(K;L^0(\mm))^*
:\quad\alpha^\# T=L,\quad|T|=|L|,
\]
thus in particular \(\alpha^\#\) is a \emph{submetry}. To prove it, fix any \(L\in{\rm UC}_{\rm ord}(\Omega;L^0(\mm))^*\). Note that
\[
{\rm UC}_{\rm ord}(\Omega;L^0(\mm))\ni f\mapsto f\circ\alpha\in{\rm F}\coloneqq
\big\{\tilde f\circ\alpha\;\big|\;\tilde f\in{\rm UC}_{\rm ord}(\Omega;L^0(\mm))\big\}\subseteq{\rm UC}_{\rm ord}(K;L^0(\mm))
\]
is an isometric isomorphism of random normed modules. Let \(\psi\colon{\rm F}\to{\rm UC}_{\rm ord}(\Omega;L^0(\mm))\)
denote its inverse.
Letting \(\tilde T\colon{\rm F}\to L^0(\mm)\) be defined as \(\tilde T(g)\coloneqq L(\psi(g))\) for every \(g\in{\rm F}\),
it is easy to check that \(\tilde T\in{\rm F}^*\) and \(|\tilde T|\leq|L|\). By applying the `random version' of the
\emph{Hahn--Banach extension theorem} (which was originally proved in \cite{Guo-1995}, see also \cite[Corollary 1.2.16]{gigli2018nonsmooth}
and \cite[Theorem 3.30]{LP24}), we obtain the existence of an element \(T\in{\rm UC}_{\rm ord}(K;L^0(\mm))^*\) such that \(T|_{\rm F}=\tilde T\)
and \(|T|=|\tilde T|\leq|L|\). Hence, we have that
\[
T(f\circ\alpha)=\tilde T(f\circ\alpha)=(L\circ\psi)(f\circ\alpha)=L(f)\quad\text{ for every }f\in{\rm UC}_{\rm ord}(\Omega;L^0(\mm)),
\]
which gives that \(\alpha^\#T=L\). Since \(|L|=|\alpha^\#T|\leq|\alpha^\#||T|\leq|T|\), we conclude that \(|T|=|L|\),
thus proving the claim.
\fr}\end{remark}
\subsection{Proof of the random Riesz--Markov--Kakutani theorem}\label{s:proof_RMK}
We are now ready to prove the following statement:
\begin{theorem}[Riesz--Markov--Kakutani theorem for \(L^0\)-valued measures]\label{thm:RMK}
Let \((\Omega,\tau)\) be a compact Hausdorff space. Let us define
\({\rm I}={\rm I}_\Omega\colon\mathfrak M(\Omega;L^0(\mm))\to{\rm UC}_{\rm ord}(\Omega;L^0(\mm))^*\) as
\[
{\rm I}(\mu)(f)\coloneqq\int f\,\d\mu\in L^0(\mm)\quad\text{ for every }\mu\in\mathfrak M(\Omega;L^0(\mm))\text{ and }f\in{\rm UC}_{\rm ord}(\Omega;L^0(\mm)).
\]
Then \({\rm I}\) is an isometric isomorphism of complete random normed modules.
\end{theorem}

We will present two different proofs of Theorem \ref{thm:RMK}: the first one works only in the particular case where \((\Omega,\tau)\)
is metrisable, but it has the advantage of showing that in this case every Radon \(L^0\)-valued measure can be foliated (cf.\ Corollary
\ref{corolmainthm}); the second one does not require any additional assumption on \((\Omega,\tau)\)
and is a generalisation of the proof of the classical Riesz--Markov--Kakutani theorem that was presented in \cite{Hartig83}.
\begin{remark}\label{rmk:RMK_easy_ineq}{\rm
The part of Theorem \ref{thm:RMK} that is easy to prove is the following one:
\[
{\rm I}\colon\mathfrak M(\Omega;L^0(\mm))\to{\rm UC}_{\rm ord}(\Omega;L^0(\mm))^*\quad\text{ is a homomorphism of random normed modules}
\]
and \(|{\rm I}|\leq\1_\X^\mm\). Indeed, for any \(\mu\in\mathfrak M(\Omega;L^0(\mm))\) we have that
\({\rm I}(\mu)\colon{\rm UC}_{\rm ord}(\Omega;L^0(\mm))\to L^0(\mm)\) is an element of the random conjugate
space \({\rm UC}_{\rm ord}(\Omega;L^0(\mm))^*\) with \(|{\rm I}(\mu)|\leq|\mu|_{\rm TV}\), as a consequence of
\eqref{eq:fund_ineq_int_UC}. It is then easy to check that \({\rm I}\colon\mathfrak M(\Omega;L^0(\mm))\to{\rm UC}_{\rm ord}(\Omega;L^0(\mm))^*\)
is a homomorphism of random normed modules satisfying \(|{\rm I}|\leq\1_\X^\mm\). Therefore, to complete the
proof of Theorem \ref{thm:RMK}, it remains to show that \({\rm I}\) is surjective and that \(|\mu|_{\rm TV}\leq|{\rm I}(\mu)|\)
for every \(\mu\in\mathfrak M(\Omega;L^0(\mm))\).
\fr}\end{remark}
\begin{proof}[Proof of Theorem \ref{thm:RMK} -- Metrisable case]
Assume in addition that \((\Omega,\tau)\) is metrisable. Fix a distance \(\sfd\) on \(\Omega\) that induces
the topology \(\tau\). In view of Remark \ref{rmk:RMK_easy_ineq}, we only have to show that \({\rm I}\) is surjective
and \(|\mu|_{\rm TV}\leq|{\rm I}(\mu)|\) for every \(\mu\in\mathfrak M(\Omega;L^0(\mm))\).
We subdivide the proof into three steps.
\smallskip

\textbf{Step 1.} Let us prove that \({\rm I}\) is surjective. To this aim, fix any
\(L\in{\rm UC}_{\rm ord}(\Omega;L^0(\mm))^*\). Take a partition \((E_j)_{j\in\N}\subseteq\Sigma\) of \(\X\) such that
\(\1_{E_j}^\mm|L|\in L^\infty(\mm)\) for every \(j\in\N\). Fix a von Neumann lifting \(\ell\) of the measure \(\bar\mm\).
Now, fix any \(j\in\N\) and \(\phi\in C(\Omega)\). Then it is easy to check that the map \(\phi\1_{E_j}^\mm\colon\Omega\to L^0(\mm)\)
belongs to \({\rm UC}_{\rm ord}(\Omega;L^0(\mm))\) and satisfies
\(|\phi\1_{E_j}^\mm|^\infty\leq\|\phi\|_{C(\Omega)}\1_{E_j}^\mm\). It follows that \(|L(\phi\1_{E_j}^\mm)|\leq\|\phi\|_{C(\Omega)}\1_{E_j}^\mm|L|\),
thus in particular \(L(\phi\1_{E_j}^\mm)\in L^\infty(\mm)\). Therefore, given any point \(x\in\X\), it makes sense to define the functional
\(L_x\colon C(\Omega)\to\R\) as
\[
L_x(\phi)\coloneqq\left\{\begin{array}{ll}
\ell(L(\phi\1_{E_j}^\mm))(x)\\
0
\end{array}\quad\begin{array}{ll}
\text{ if }x\in\ell(E_j)\text{ for some }j\in\N,\\
\text{ if }x\in\X\setminus\bigcup_{j\in\N}\ell(E_j).
\end{array}\right.
\]
The properties of von Neumann liftings ensure that \(L_x\) is linear. For any \(x\in\ell(E_j)\), we have
\[
|L_x(\phi)|=\ell(|L(\phi\1_{E_j}^\mm)|)(x)\leq\|\phi\|_{C(\Omega)}\ell(\1_{E_j}^\mm|L|)(x)\leq\|\phi\|_{C(\Omega)}G_L(x)
\quad\text{ for every }\phi\in C(\Omega),
\]
where \(G_L\coloneqq\sum_{j\in\N}\ell(\1_{E_j}^\mm|L|)\in\mathcal L^0_+(\bar\Sigma_\mm)\). This implies
\(L_x\in C(\Omega)^*\) and \(\|L_x\|_{C(\Omega)^*}\leq G_L(x)\) for all \(x\in\X\). By the Riesz--Markov--Kakutani theorem,
there exists a unique \(\mu^L_x\in\mathfrak M(\Omega)\) such that
\begin{equation}\label{eq:def_mu_x}
\int\phi\,\d\mu^L_x=L_x(\phi)\quad\text{ for every }\phi\in C(\Omega)
\end{equation}
and \(\|\mu_x^L\|_{\rm TV}=\|L_x\|_{C(\Omega)^*}\). For any \(\phi\in C(\Omega)\), we have that the function
\[
\X\ni x\mapsto\int\phi\,\d\mu^L_x=\sum_{j\in\N}\1_{\ell(E_j)}(x)\ell(L(\phi\1_{E_j}^\mm))(x)\in\R\quad\text{ is }\mm\text{-measurable,}
\]
so that \((\mu^L_x)_{x\in\X}\) is an \(\mm\)-measurable collection of measures by Lemma \ref{lem:meas_mu_x}. Consequently, we denote by
\(\mu^L\in\mathcal M(\Omega;L^0(\mm))\) the \(L^0\)-valued measure whose foliation is given by \((\mu^L_x)_{x\in\X}\).
Corollary \ref{cor:every_mu_Radon_Polish} ensures that \(\mu^L\in\mathfrak M(\Omega;L^0(\mm))\). We claim that
\begin{equation}\label{eq:I_surj}
{\rm I}(\mu^L)=L,
\end{equation}
which yields the surjectivity of \({\rm I}\). To prove it, fix \(f=\sum_{i=1}^n\phi_i\1_{F_i}^\mm\in{\rm V}\), where \({\rm V}\) is
defined as in Lemma \ref{lem:dense_in_UC_ord}. Hence, for any \(j\in\N\) and \(\mm\)-a.e.\ \(x\in E_j\) we have, by Lemma
\ref{lem:int_mu_ptwse_descript}, that
\[\begin{split}
{\rm I}(\mu^L)(f)(x)&=\sum_{i=1}^n\bigg(\int\phi_i\1_{F_i}^\mm\,\d\mu^L\bigg)(x)
=\sum_{i=1}^n\1_{F_i}^\mm(x)\int\phi_i\,\d\mu^L_x=\sum_{i=1}^n\1_{F_i}^\mm(x)L_x(\phi_i)\\
&=\sum_{i=1}^n\1_{F_i}^\mm(x)\ell(L(\phi_i\1_{E_j}^\mm))(x)=\1_{E_j}^\mm(x)\sum_{i=1}^n L(\phi_i\1_{F_i}^\mm)(x)=L(f)(x).
\end{split}\]
Since \(\bigcup_{j\in\N}E_j=\X\), we deduce that \({\rm I}(\mu^L)(f)=L(f)\) for every \(f\in{\rm V}\). Given that \({\rm V}\) is dense
in \({\rm UC}_{\rm ord}(\Omega;L^0(\mm))\) by Lemma \ref{lem:dense_in_UC_ord}, we conclude that \eqref{eq:I_surj} holds, thus \({\rm I}\)
is surjective.
\smallskip

\textbf{Step 2.} Next, we prove that \({\rm I}\) is injective. To this aim, assume \(\mu,\nu\in\mathfrak M(\Omega;L^0(\mm))\) are chosen so
that \({\rm I}(\mu)={\rm I}(\nu)\). Fix a compact set \(K\subseteq\Omega\). Set \(\eta_n\coloneqq\big(1-n\sfd(\cdot,K)\big)\vee 0\in C(\Omega;[0,1])\)
and \(A_n\coloneqq\{p\in\Omega\setminus K:\sfd(x,K)\leq 1/n\}\in\mathscr B(\Omega)\) for all \(n\in\N\). Since \(|\1_K-\eta_n|\leq\1_{A_n}\), we get
\[\begin{split}
|\mu(K)-\nu(K)|&\leq\bigg|\int(\1_K-\eta_n)\1_\X^\mm\,\d(\mu-\nu)\bigg|+\big|{\rm I}(\mu)(\eta_n\1_\X^\mm)-{\rm I}(\nu)(\eta_n\1_\X^\mm)\big|\\
&\leq\int|\1_K-\eta_n|\1_\X^\mm\,\d|\mu-\nu|\leq\int\1_{A_n}\1_\X^\mm\,\d|\mu-\nu|\\
&=|\mu-\nu|(A_n)\quad\text{ for every }n\in\N.
\end{split}\]
Due to the fact that \(A_{n+1}\subseteq A_n\) for every \(n\in\N\) and \(\bigcap_{n\in\N}A_n=\varnothing\), we know from Remark \ref{rmk:cont_above_L0_meas}
that \(|\mu-\nu|(A_n)\to|\mu-\nu|(\varnothing)=0\) in \(L^0(\mm)\) as \(n\to\infty\). Therefore, \(\mu(K)=\nu(K)\) for every \(K\subseteq\Omega\) compact.
Finally, by using the inner regularity of \(|\mu-\nu|\) (see Lemma \ref{lem:inner_reg_self-improv}) we conclude that
\[\begin{split}
|\mu(B)-\nu(B)|&\leq\bigwedge\bigg\{|(\mu-\nu)(B\setminus K)|+|(\mu-\nu)(K)|\;\bigg|\;K\subseteq B\text{ compact}\bigg\}\\
&\leq\bigwedge\bigg\{|\mu-\nu|(B\setminus K)\;\bigg|\;K\subseteq B\text{ compact}\bigg\}=0\quad\text{ for every }B\in\mathscr B(\Omega),
\end{split}\]
which means that \(\mu=\nu\). Consequently, we have proven that the map \({\rm I}\) is injective.
\smallskip

\textbf{Step 3.} Let us now prove that \(|{\rm I}(\mu)|\geq|\mu|_{\rm TV}\) for every \(\mu\in\mathfrak M(\Omega;L^0(\mm))\), whence it follows
that \({\rm I}\) is an isometric isomorphism of complete random normed modules.
\textbf{Step 1} implies that
\begin{equation}\label{eq:RMK_aux}
|\mu^{{\rm I}(\mu)}_x|(\Omega)=\|\mu^{{\rm I}(\mu)}_x\|_{\rm TV}=\|{\rm I}(\mu)_x\|_{C(\Omega)^*}\leq G_{{\rm I}(\mu)}(x)=|{\rm I}(\mu)|(x)
\quad\text{ for }\mm\text{-a.e.\ }x\in\X.
\end{equation}
Moreover, we know from Corollary \ref{cor:meas_|mu_x|} that \((|\mu^{{\rm I}(\mu)}_x|)_{x\in\X}\) is an \(\mm\)-measurable collection of measures,
and the injectivity of \({\rm I}\) implies that \(\mu^{{\rm I}(\mu)}=\mu\). Hence, Remark \ref{rmk:integr_mu_x} and \eqref{eq:RMK_aux} guarantee that
\[
|\mu|_{\rm TV}(x)=|\mu^{{\rm I}(\mu)}|_{\rm TV}(x)=|\mu^{{\rm I}(\mu)}|(\Omega)(x)\leq|\mu^{{\rm I}(\mu)}_x|(\Omega)\leq|{\rm I}(\mu)|(x)
\quad\text{ for }\mm\text{-a.e.\ }x\in\X,
\]
thus proving that \(|\mu|_{\rm TV}\leq|{\rm I}(\mu)|\). Therefore, the statement is achieved.
\end{proof}
\begin{corollary}\label{corolmainthm}
Let \((\Omega,\sfd)\) be a compact metric space. Then each \(\mu\in\mathfrak M(\Omega;L^0(\mm))\) admits an \(\mm\)-a.e.\ unique foliation
\((\mu_x)_{x\in\X}\subseteq\mathfrak M(\Omega)\). Moreover, \((|\mu_x|)_{x\in\X}\) is the \(\mm\)-a.e.\ unique foliation of \(|\mu|\).
\end{corollary}
\begin{proof}
The fact that each \(\mu\in\mathfrak M(\Omega;L^0(\mm))\) has a foliation \((\mu_x)_{x\in\X}\subseteq\mathfrak M(\Omega)\)
is a byproduct of \textbf{Steps 1} and \textbf{2} of the proof of Theorem \ref{thm:RMK}. Corollary \ref{cor:meas_|mu_x|} ensures that
\((|\mu_x|)_{x\in\X}\) is an \(\mm\)-measurable collection of measures. Moreover, given \(A\in\mathscr B(\Omega)\),
we have \(|\mu|(A)(x)\leq|\mu_x|(A)\) and \(|\mu|(\Omega\setminus A)(x)\leq|\mu_x|(\Omega\setminus A)\) for \(\mm\)-a.e.\ \(x\in\X\)
by Remark \ref{rmk:integr_mu_x}, while in \textbf{Step 3} of the proof of Theorem \ref{thm:RMK} it is shown that
\(|\mu|_{\rm TV}(x)=\|\mu_x\|_{\rm TV}\) for \(\mm\)-a.e.\ \(x\in\X\). Combining these facts together, for \(\mm\)-a.e.\ \(x\in\X\) we obtain that
\[
|\mu|_{\rm TV}(x)=|\mu|(A)(x)+|\mu|(\Omega\setminus A)(x)\leq|\mu_x|(A)+|\mu_x|(\Omega\setminus A)=\|\mu_x\|_{\rm TV}=|\mu|_{\rm TV}(x),
\]
which forces the validity of the \(\mm\)-a.e.\ equality \(|\mu|(A)(x)=|\mu_x|(A)\). This shows that \((|\mu_x|)_{x\in\X}\) is a foliation of
\(|\mu|\). Finally, the \(\mm\)-a.e.\ uniqueness of the foliations of \(\mu\) and \(|\mu|\) is guaranteed by Remark \ref{rmk:unique_foliation_top}.
All in all, the statement is achieved.
\end{proof}

The above result suggests another characterisation of \(\mathfrak M(\Omega;L^0(\mm))\) when \((\Omega,\sfd)\)
is a compact metric space, as we are going to discuss. Given a separable Banach space \(\mathbb B\), we denote
by \(L^0_{w^*}(\mm;\mathbb B^*)\) the space of all those maps \(\omega\colon\X\to\mathbb B^*\) that are
\emph{weakly\(^*\) \(\mm\)-measurable}, i.e.
\[
\X\ni x\mapsto\omega(x)(v)\in\R\quad\text{ is }\mm\text{-measurable for every }v\in\mathbb B,
\]
quotiented up to \(\mm\)-a.e.\ equality. Then \(L^0_{w^*}(\mm;\mathbb B^*)\) is a complete random normed module
with base \((\X,\Sigma,\mm)\), where the \(L^0\)-norm of \(\omega\in L^0_{w^*}(\mm;\mathbb B^*)\) is
given by \(|\omega|(x)\coloneqq\|\omega(x)\|_{\mathbb B^*}\) for \(\mm\)-a.e.\ \(x\in\X\). The \(L^0\)-Lebesgue--Bochner
space \(L^0(\mm;\mathbb B^*)\) can be identified with an \(L^0(\mm)\)-submodule of \(L^0_{w^*}(\mm;\mathbb B^*)\).
Furthermore, it holds that
\begin{equation}\label{eq:char_dual_LebBoch}
L^0_{w^*}(\mm;\mathbb B^*)\text{ and }L^0(\mm;\mathbb B)^*\text{ are isometrically isomorphic.}
\end{equation}
The above statements follow from the results of \cite{LPV22,GuoMuTu} (see also \cite{Guo1996,GLP}).
\begin{theorem}\label{thm:charact_foliation_cpt_metric}
Let \((\Omega,\sfd)\) be a compact metric space. The map sending \(\mu\in\mathfrak M(\Omega;L^0(\mm))\)
to its \(\mm\)-a.e.\ unique foliation \((\mu_x)_{x\in\X}\) is an isometric isomorphism of random normed modules
\[
\mathfrak M(\Omega;L^0(\mm))\to L^0_{w^*}(\mm;\mathfrak M(\Omega)).
\]
\end{theorem}
\begin{proof}
Given \(\mu\in\mathfrak M(\Omega;L^0(\mm))\), we denote by \(\Phi(\mu)\) the \(\mm\)-a.e.\ unique foliation
\(x\mapsto\mu_x\in\mathfrak M(\Omega)\) of \(\mu\) considered up to \(\mm\)-a.e.\ equality. Thanks to Lemma
\ref{lem:meas_mu_x}, the fact that \(x\mapsto\mu_x\) is an \(\mm\)-measurable collection of measures is
equivalent to the requirement that \(\X\ni x\mapsto\int\phi\,\d\mu_x\) is \(\mm\)-measurable for every
\(\phi\in C(\Omega)\). Since \(\mathfrak M(\Omega)\) is the Banach space dual of \(C(\Omega)\), this exactly
means that \(\Phi(\mu)\in L^0_{w^*}(\mm;\mathfrak M(\Omega))\). It is easy to check that
\(\Phi\colon\mathfrak M(\Omega;L^0(\mm))\to L^0_{w^*}(\mm;\mathfrak M(\Omega))\) is a bijective
\(L^0(\mm)\)-linear operator. Finally, the last statement of Corollary \ref{corolmainthm} implies that
\[
|\Phi(\mu)|(x)=|\mu_x|(\Omega)=|\mu|_{\rm TV}(x)\quad\text{ for }\mm\text{-a.e.\ }x\in\X,
\]
thus accordingly \(\Phi\) is an isometric isomorphism of random normed modules.
\end{proof}
\begin{remark}{\rm
It follows from \eqref{eq:char_dual_LebBoch} that \(L^0_{w^*}(\mm;\mathfrak M(\Omega))\) and
\(L^0(\mm;C(\Omega))^*\) are isometrically isomorphic. Therefore, combining Theorems \ref{thm:RMK}
and \ref{thm:charact_foliation_cpt_metric} we obtain that \({\rm UC}_{\rm ord}(\Omega;L^0(\mm))^*\)
and \(L^0(\mm;C(\Omega))^*\) are isometrically isomorphic when \((\Omega,\sfd)\) is a compact
metric space, consistently with Proposition \ref{prop:charact_UC_ord}.
\fr}\end{remark}

We present another proof of Theorem \ref{thm:RMK} (inspired by \cite{Hartig83}) that works unconditionally:
\begin{proof}[Proof of Theorem \ref{thm:RMK} -- General case]
Remark \ref{rmk:properties_ext_disconn} iii) yields an extremally-disconnected compact
Hausdorff space \((K,\sigma)\) and a continuous map \(\alpha\colon K\to\Omega\) such that \(\alpha(K)=\Omega\).
We proceed as follows: first we prove the statement for \((K,\sigma)\), then we deduce it for \((\Omega,\tau)\).
\smallskip

\textbf{Step 1.} We aim to show that \({\rm I}_K\colon\mathfrak M(K;L^0(\mm))\to{\rm UC}_{\rm ord}(K;L^0(\mm))^*\)
is an isometric isomorphism of random normed modules. By Remark \ref{rmk:RMK_easy_ineq}, it suffices
to prove that \({\rm I}_K\) is surjective and \(|\nu|_{\rm TV}\leq|{\rm I}_K(\nu)|\) for all \(\nu\in\mathfrak M(K;L^0(\mm))\).
To this aim, fix \(T\in{\rm UC}_{\rm ord}(K;L^0(\mm))^*\). Proposition \ref{prop:decomp_into_posit} provides us with
two elements \(T_\pm\in{\rm UC}_{\rm ord}(K;L^0(\mm))^*_+\) such that \(T=T_+-T_-\) and \(|T|=|T_+|+|T_-|\).
Let \(\mathcal C\) be the algebra of clopen subsets of \(K\). We set \(\tau\colon\mathcal C\to L^0_+(\mm)\) as
\[
\tau_\pm(A)\coloneqq T_\pm(\1_A\1_\X^\mm)\quad\text{ for every }A\in\mathcal C.
\]
Clearly, we have that \(\tau_\pm(\varnothing)=T_\pm(0)=0\). Let us denote by \(\tau_\pm^*\colon 2^K\to L^0_+(\mm)\) the outer
\(L^0\)-valued measure associated to \(\tau_\pm\) as in Proposition \ref{prop:induced_outer_meas}. We claim that
\begin{equation}\label{eq:RMK_ext_disconn_1}
\tau_\pm^*(A)=\tau_\pm(A)\quad\text{ for every }A\in\mathcal C.
\end{equation}
Indeed, the inequality \(\tau_\pm^*(A)\leq\tau_\pm(A)\) trivially holds, while to show the converse inequality we argue as follows.
Take \((A_n)_{n\in\N}\subseteq\mathcal C\) so that \(A\subseteq\bigcup_{n\in\N}A_n\). Since \(A\) is compact, we can find
\(n_0\in\N\) such that \(A_n=\varnothing\) for every \(n>n_0\). Define \(B_n\coloneqq(A\cap A_n)\setminus\bigcup_{i<n}A_i\subseteq A_n\)
for every \(n=1,\ldots,n_0\). Note that \(B_1,\ldots,B_{n_0}\in\mathcal C\) are pairwise disjoint and \(A=\bigcup_{n=1}^{n_0}B_n\).
Using the linearity and the positivity of \(T_\pm\), we thus obtain that
\[
\sum_{n=1}^\infty\tau_\pm(A_n)=\sum_{n=1}^{n_0}\tau_\pm(A_n)=\sum_{n=1}^{n_0}T_\pm(\1_{A_n}\1_\X^\mm)
\geq\sum_{n=1}^{n_0}T_\pm(\1_{B_n}\1_\X^\mm)=T_\pm(\1_A\1_\X^\mm)=\tau_\pm(A).
\]
Taking the infimum over all such \((A_n)_{n\in\N}\), we conclude that \(\tau_\pm^*(A)\geq\tau_\pm(A)\), yielding \eqref{eq:RMK_ext_disconn_1}.

Next, we claim that \(\mathscr B_0(K)\subseteq\Gamma(\tau_\pm^*)\). Let us first fix \(A\in\mathcal C\). Take an arbitrary set
\(S\in 2^K\) and a sequence \((A_n)_{n\in\N}\subseteq\mathcal C\) such that \(S\subseteq\bigcup_{n\in\N}A_n\). In particular,
\(S\cap A\subseteq\bigcup_{n\in\N}A_n\cap A\) and \(S\setminus A\subseteq\bigcup_{n\in\N}A_n\setminus A\). Since
\(A_n\cap A,A_n\setminus A\in\mathcal C\) for all \(n\in\N\), it follows from \eqref{eq:RMK_ext_disconn_1} that
\[
\tau_\pm^*(S\cap A)+\tau_\pm^*(S\setminus A)\leq\sum_{n\in\N}\tau_\pm(A_n\cap A)+\sum_{n\in\N}\tau_\pm(A_n\setminus A)
=\sum_{n\in\N}\tau_\pm(A_n).
\]
Taking the infimum over all such \((A_n)_{n\in\N}\), we obtain that \(\tau_\pm^*(S\cap A)+\tau_\pm^*(S\setminus A)\leq\tau_\pm^*(S)\),
which means that \(A\in\Gamma(\tau_\pm^*)\). Thanks to Theorem \ref{thm:Carath} and to the fact that the Baire \(\sigma\)-algebra
\(\mathscr B_0(K)\) is generated by \(\mathcal C\) (cf.\ item i) of Remark \ref{rmk:properties_ext_disconn}), we conclude that
\(\mathscr B_0(K)\subseteq\Gamma(\tau_\pm^*)\), as we claimed. Therefore, Theorem \ref{thm:Carath} ensures that
\[
\nu_\pm\coloneqq\tau_\pm^*|_{\mathscr B_0(K)}\in\mathcal M_+(K,\mathscr B_0(K);L^0(\mm)).
\]
We then denote by \(\bar\nu_\pm\in\mathfrak M_+(K;L^0(\mm))\) the unique non-negative Radon \(L^0\)-valued measure that extends
\(\nu_\pm\), whose existence is guaranteed by Theorem \ref{thm:ext_Baire_to_Radon}. Finally, let us define
\[
\nu\coloneqq\bar\nu_+-\bar\nu_-\in\mathfrak M(K;L^0(\mm)).
\]
We claim that
\begin{equation}\label{eq:RMK_ext_disconn_2}
{\rm I}_K(\nu)=T,\qquad|\nu|_{\rm TV}\leq|T|,
\end{equation}
whence it follows that \({\rm I}_K\) is an isometric isomorphism of complete random normed modules. Consider
the vector space \({\rm V}\) of all \(f\in{\rm UC}_{\rm ord}(K;L^0(\mm))\) of the form
\(f=\sum_{i=1}^n\lambda_i\1_{A_i}\1_{E_i}^\mm\), where \(n\in\N\), \((\lambda_i)_{i=1}^n\subseteq\R\),
\((A_i)_{i=1}^n\subseteq\mathcal C\) and \((E_i)_{i=1}^n\subseteq\Sigma\). It follows from Remark \ref{rmk:properties_ext_disconn} ii)
and Lemma \ref{lem:dense_in_UC_ord} that \({\rm V}\) is a dense vector subspace of \({\rm UC}_{\rm ord}(K;L^0(\mm))\).
For any \(f=\sum_{i=1}^n\lambda_i\1_{A_i}\1_{E_i}^\mm\in{\rm V}\), we have that
\[\begin{split}
{\rm I}_K(\nu)(f)&=\sum_{i=1}^n\lambda_i\int\1_{A_i}\1_{E_i}^\mm\,\d\nu=\sum_{i=1}^n\lambda_i\,\nu(A_i)\1_{E_i}^\mm
=\sum_{i=1}^n\lambda_i\big(\nu_+(A_i)-\nu_-(A_i)\big)\1_{E_i}^\mm\\
&=\sum_{i=1}^n\lambda_i\big(T_+(\1_{A_i}\1_\X^\mm)-T_-(\1_{A_i}\1_\X^\mm)\big)\1_{E_i}^\mm
=\sum_{i=1}^n T(\lambda_i\1_{A_i}\1_{E_i}^\mm)=T(f).
\end{split}\]
Given that \({\rm I}_K(\nu)\) and \(T\) are homomorphisms of random normed modules, and the space \({\rm V}\) generates
\({\rm UC}_{\rm ord}(K;L^0(\mm))\), we get that \({\rm I}_K(\nu)=T\). Finally, by Remark \ref{rmk:posit_is_cont}
we can estimate
\[\begin{split}
|\nu|_{\rm TV}&=|\bar\nu_+-\bar\nu_-|_{\rm TV}\leq|\bar\nu_+|_{\rm TV}+|\bar\nu_-|_{\rm TV}
=\nu_+(\Omega)+\nu_-(\Omega)=T_+(\1_\Omega\1_\X^\mm)+T_-(\1_\Omega\1_\X^\mm)\\
&=|T_+|+|T_-|=|T|,
\end{split}\]
thus proving that \eqref{eq:RMK_ext_disconn_2} holds. This completes the proof of the first step.
\smallskip

\textbf{Step 2.} We pass to the verification of the fact that
\({\rm I}_\Omega\colon\mathfrak M(\Omega;L^0(\mm))\to{\rm UC}_{\rm ord}(\Omega;L^0(\mm))^*\) is an isometric isomorphism
of complete random normed modules. Again by Remark \ref{rmk:RMK_easy_ineq}, it suffices to show that \({\rm I}_\Omega\)
is surjective and \(|\mu|_{\rm TV}\leq|{\rm I}_\Omega(\mu)|\) for every \(\mu\in\mathfrak M(\Omega;L^0(\mm))\).
Let \(\alpha^\#\) and \(\alpha_\#\) be as in Remark \ref{rmk:adj_operator} and Definition \ref{def:pushforward_meas},
respectively; recall that \(\alpha_\#\) maps \(\mathfrak M(K;L^0(\mm))\) to \(\mathfrak M(\Omega;L^0(\mm))\), as we
have observed in Remark \ref{rmk:pushforward_Radon}. Applying Lemma \ref{lem:full_formula_pushforward} and noticing that
\(\alpha^\#f\) coincides with \(f\circ\alpha\) when \(f\in{\rm UC}_{\rm ord}(\Omega;L^0(\mm))\), we obtain that
\[
\alpha^\#({\rm I}_K(\nu))(f)={\rm I}_K(\nu)(f\circ\alpha)=\int f\circ\alpha\,\d\nu=\int f\,\d(\alpha_\#\nu)
={\rm I}_\Omega(\alpha_\#\nu)(f)
\]
for every \(\nu\in\mathfrak M(K;L^0(\mm))\) and \(f\in{\rm UC}_{\rm ord}(\Omega;L^0(\mm))\). In other words, we have shown that
\[\begin{tikzcd}
\mathfrak M(K;L^0(\mm)) \arrow[r,"{\rm I}_K"] \arrow[d,swap,"\alpha_\#"] & {\rm UC}_{\rm ord}(K;L^0(\mm))^* \arrow[d,"\alpha^\#"] \\
\mathfrak M(\Omega;L^0(\mm)) \arrow[r,swap,"{\rm I}_\Omega"] & {\rm UC}_{\rm ord}(\Omega;L^0(\mm))^*
\end{tikzcd}\]
is a commutative diagram. Now, fix any \(L\in{\rm UC}_{\rm ord}(\Omega;L^0(\mm))^*\). Remark \ref{rmk:adj_operator} ensures
the existence of an element \(T\in{\rm UC}_{\rm ord}(K;L^0(\mm))^*\) such that \(\alpha^\#T=L\) and \(|T|=|L|\). Moreover,
the first step guarantees that \(T={\rm I}_K(\nu)\) and \(|\nu|_{\rm TV}=|T|\) for some (unique)
\(\nu\in\mathfrak M(K;L^0(\mm))\). Therefore, \(\mu\coloneqq\alpha_\#\nu\in\mathfrak M(\Omega;L^0(\mm))\)
satisfies \({\rm I}_\Omega(\mu)={\rm I}_\Omega(\alpha_\#\nu)=\alpha^\#({\rm I}_K(\nu))=\alpha^\#T=L\) and
\[
|\mu|_{\rm TV}=|\alpha_\#\nu|_{\rm TV}\leq|\nu|_{\rm TV}=|T|=|L|,
\]
where the inequality can be justified by applying \eqref{eq:ineq_pushforward}. This yields the sought conclusion.
\end{proof}
\section{Module-valued measures}\label{s:M-val_meas}
Let us now introduce measures of bounded variation taking values into a random normed module.
In Section \ref{s:M-val_basic} we discuss some basic definitions and properties, which can be easily
generalised from the special case of \(L^0\)-valued measures. A more subtle task is to study the
foliation of a module-valued measure; notice that a priori the elements of a random normed module
are not defined `pointwise', thus it is not even clear where the leaves of a foliation should be defined.
To deal with this issue, we first recall in Section \ref{s:mBb} notions and results from the paper
\cite{DMLP25}, where it is shown that a \emph{countably-generated} complete random normed module \({\rm M}\)
is (isometrically isomorphic to) the section space \(\Gamma(E_{\rm M})\) of a \emph{measurable Banach bundle}
\(E_{\rm M}\) whose fibers \(\X\ni x\mapsto E_{\rm M}(x)\) are separable Banach spaces. We then prove in Theorem
\ref{thm:M-val_foliat_exist} that every \({\rm M}\)-valued measure \(\mu\) defined on a compact metric space \((\Omega,\sfd)\)
can be foliated. A foliation of \(\mu\) is a collection of Banach-space-valued
measures \(\X\ni x\mapsto\mu_x\in\mathfrak M(\Omega;E_{\rm M}(x))\) satisfying suitable conditions (see Definition
\ref{def:M-val_foliat}).
\medskip

We assume that \((\X,\Sigma,\mm)\) is a probability space, \((\Omega,\mathcal A)\) is
a measurable space and $({\rm M},|\cdot|)$ is a complete random normed module with base \((\X,\Sigma,\mm)\).
\subsection{Definitions and basic properties}\label{s:M-val_basic}
In analogy (and consistently) with our definition of \(L^0\)-valued measure, we propose the following notion
of \({\rm M}\)-valued measure:
\begin{definition}[\({\rm M}\)-valued measure]\label{def:M_meas}
We say that \(\mu\colon\mathcal A\to {\rm M}\) is an \textbf{\({\rm M}\)-valued measure} if:
\begin{itemize}
\item[\(\rm i)\)] \(\mu(\varnothing)=0\).
\item[\(\rm ii)\)] If \((A_n)_{n\in\N}\subseteq\mathcal A\) are pairwise disjoint, then \(\{|\mu(A_n)|:n\in\N\}\) is summable and
\[
\mu\bigg(\bigcup_{n\in\N}A_n\bigg)=\sum_{n\in\N}\mu(A_n).
\]
\end{itemize}
\end{definition}

The summability of \(\{\mu(A_n):n\in\N\}\subseteq {\rm M}\) is guaranteed by Remark \ref{rmk:fact_about_summability}
vi). We define the \textbf{\(L^0\)-total variation} of \(\mu\) on a set \(A\in\mathcal A\) as
\[
|\mu|(A)\coloneqq\bigvee\bigg\{\sum_{n=1}^\infty|\mu(A_n)|\;\bigg|\;(A_n)_{n\in\N}\subseteq\mathcal A\text{ partition of }A\bigg\}\in\bar L^0_+(\mm).
\]
We say that the \({\rm M}\)-valued measure \(\mu\) is \textbf{of bounded variation} if
\(|\mu|_{\rm TV}\coloneqq|\mu|(\Omega)\in L^0_+(\mm)\). We denote by \(\mathcal M(\Omega;{\rm M})\)
the space of \({\rm M}\)-valued measures \(\mu\colon\mathcal A\to {\rm M}\) of bounded variation.
\medskip

Many concepts and results concerning Banach-space-valued measures of bounded variation can be extended
to \({\rm M}\)-valued measures of bounded variation. In this regard, we collect below some definitions, along
with some statements that can be achieved either by applying our results on \(L^0\)-valued measures or
by adapting the proof arguments of the latter.
\begin{itemize}
\item \((\mathcal M(\Omega;{\rm M}),|\cdot|_{\rm TV})\) is a complete random normed module with base \((\X,\Sigma,\mm)\).
\item It holds that \(|\mu|\in\mathcal M_+(\Omega;L^0(\mm))\) for every \(\mu\in\mathcal M(\Omega;{\rm M})\).
\item If \((\Omega,\tau)\) is a topological space and \(\mathcal A=\mathscr B(\Omega)\), we call
\textbf{Borel \({\rm M}\)-valued measures} the elements of \(\mathcal M(\Omega;{\rm M})\).
\item If \((\Omega,\tau)\) is Hausdorff, then we say that \(\mu\in\mathcal M(\Omega;{\rm M})\) is \textbf{Radon}
if \(|\mu|\in\mathfrak M_+(\Omega;L^0(\mm))\). We denote by \(\mathfrak M(\Omega;{\rm M})\) the space of all
Radon \({\rm M}\)-valued measures on \(\Omega\). It holds that \((\mathfrak M(\Omega;{\rm M}),|\cdot|_{\rm TV})\)
is a complete random normed module with base \((\X,\Sigma,\mm)\).
\item If \((\Omega,\sfd)\) is a complete separable metric space, then
\(\mathcal M(\Omega;{\rm M})=\mathfrak M(\Omega;{\rm M})\).
\end{itemize}
\subsection{Measurable Banach bundles}\label{s:mBb}
Let us fix a \emph{universal separable Banach space} \(\mathbb U\), i.e.\ a separable Banach space where every
separable Banach space can be embedded linearly and isometrically. Universal separable Banach spaces do exist:
the \emph{Banach--Mazur theorem} (see e.g.\ \cite[Proposition 1.5]{BP75}) ensures that \(C([0,1])\) is a universal
separable Banach space.
\medskip

Consistently with \cite[Definition 4.1]{DMLP25}, we say that a given multivalued map \(E\colon\X\twoheadrightarrow\mathbb U\)
is a \textbf{measurable Banach \(\mathbb U\)-bundle} on \(\X\) provided:
\begin{itemize}
\item[\(\rm i)\)] \(E(x)\) is a closed vector subspace of \(\mathbb U\) for every \(x\in\X\).
\item[\(\rm ii)\)] \(E\) is \emph{weakly measurable}, i.e.
\[
\big\{x\in\X\;\big|\;E(x)\cap U\neq\varnothing\big\}\in\Sigma\quad\text{ for every open set }U\subseteq\mathbb U.
\]
\end{itemize}
We denote by \(\bar\Gamma(E)\) the space of all \textbf{measurable sections} of the bundle \(E\), i.e.\ of
those measurable maps \(v\colon\X\to\mathbb U\) satisfying
\[
v(x)\in E(x)\quad\text{ for every }x\in\X.
\]
We then define the \textbf{section space} \(\Gamma(E)\) of \(E\) as the quotient space of \(\bar\Gamma(E)\)
with respect to the equivalence relation that identifies two measurable sections if they agree in the
\(\mm\)-a.e.\ sense. It is easy to check that \(\Gamma(E)\) is an \(L^0(\mm)\)-module with respect to the
usual pointwise \(\mm\)-a.e.\ operations. Moreover, given any \(v\in\Gamma(E)\), we denote by
\(|v|\in L^0_+(\mm)\) the \(\mm\)-a.e.\ equivalence class of the function \(\X\ni x\mapsto\|\bar v(x)\|_{E(x)}\),
where \(\bar v\in\bar\Gamma(E)\) is an arbitrary representative of \(v\). It is straightforward to check that
\[
(\Gamma(E),|\cdot|)\quad\text{ is a complete random normed module with base }(\X,\Sigma,\mm).
\]
For example, if \(\mathbb V\) is a closed vector subspace of \(\mathbb U\), then the section space of the
constant bundle with fibers \(\mathbb V\) coincides with the `\(L^0\)-Lebesgue--Bochner space'
\(L^0(\mm;\mathbb V)\). In particular, we have that \(L^0(\mm)\) is the section space of the constant
bundle with fibers \(\R\). More generally, \(L^0(\mm)^n\) can be identified with the section space
of the constant bundle with fibers \(\R^n\).
\medskip

We recall that a complete random normed module \({\rm M}\) with base \((\X,\Sigma,\mm)\) is said
to be \textbf{countably generated} if it has a countable subset \(C\) such that
the \(L^0(\mm)\)-submodule of \({\rm M}\) generated by \(C\) is dense in \({\rm M}\). For example,
the random normed module \(L^0(\mm)\) is countably generated (since it is generated, as an algebraic
\(L^0(\mm)\)-module, by \(\1_\X^\mm\)). Many random normed modules that are considered in nonsmooth
differential geometry are countably generated, cf.\ \cite[Appendix B]{DMLP25}. We also point out that
if \(L^0(\mm)\) is separable (which is the case e.g.\ if \(\mm\) is a Radon measure over some Hausdorff space
\(\X\)), then \({\rm M}\) is countably generated if and only if it is separable (while sufficiency
holds unconditionally).
\medskip

Given any measurable Banach \(\mathbb U\)-bundle \(E\), it holds that \(\Gamma(E)\) is countably generated.
The main result of \cite{DMLP25} states that the converse holds as well: given any countably-generated
complete random normed module \({\rm M}\), there exists a measurable Banach \(\mathbb U\)-bundle \(E_{\rm M}\)
on \(\X\) such that \(\Gamma(E_{\rm M})\) is isometrically isomorphic as a random normed module to \({\rm M}\).
Moreover, the bundle \(E_{\rm M}\) associated to \({\rm M}\) is \emph{essentially} unique,
cf.\ \cite[Section 4.3]{DMLP25}. Let us then define
\begin{equation}\label{eq:fibers_of_M}
({\rm M}_x,\|\cdot\|_x)\coloneqq(E_{\rm M}(x),\|\cdot\|_{E_{\rm M}(x)})\quad\text{ for every }x\in\X.
\end{equation}
\subsection{Foliation of a module-valued measure}\label{s:M-val_foliat}
In view of the results that we have reported in Section \ref{s:mBb}, we can now formulate a definition of
foliation for module-valued measures whose target is a countably-generated complete random normed module:
\begin{definition}[Foliation of an \({\rm M}\)-valued measure]\label{def:M-val_foliat}
Let \({\rm M}\) be a countably-generated complete random normed module with base \((\X,\Sigma,\mm)\).
Let \(\mu\in\mathcal M(\Omega;{\rm M})\) be given. Then we say that a family \((\mu_x)_{x\in\X}\) is a
\textbf{foliation} of \(\mu\) provided the following conditions hold:
\begin{itemize}
\item[\(\rm i)\)] \(\mu_x\in\mathcal M(\Omega;{\rm M}_x)\) for every \(x\in\X\),
i.e.\ each \(\mu_x\colon\mathscr B(\Omega)\to{\rm M}_x\) is a vector measure of bounded variation. 
\item[\(\rm ii)\)] \(x\mapsto\mu_x\) is \emph{\(\mm\)-measurable}, i.e.\ \(\X\ni x\mapsto\mu_x(A)\in\mathbb U\)
is \(\mm\)-measurable for every \(A\in\mathcal A\).
\item[\(\rm iii)\)] Given any \(A\in\mathcal A\), the element \(\mu(A)\in{\rm M}\cong\Gamma(E_{\rm M})\)
is the \(\mm\)-a.e.\ equivalence class of the section \(\X\ni x\mapsto\mu_x(A)\in{\rm M}_x=E_{\rm M}(x)\)
of the measurable Banach \(\mathbb U\)-bundle \(E_{\rm M}\).
\end{itemize}
\end{definition}

Under suitable assumptions, we can prove existence of a foliation of an \({\rm M}\)-valued measure:
\begin{theorem}\label{thm:M-val_foliat_exist}
Let \((\Omega,\sfd)\) be a compact metric space. Let \({\rm M}\) be a countably-generated complete random
normed module with base \((\X,\Sigma,\mm)\). Let \(\mu\in\mathfrak M(\Omega;{\rm M})\) be given. Then
\(\mu\) admits a foliation \((\mu_x)_{x\in\X}\). Moreover, letting \((|\mu|_x)_{x\in\X}\) be the
\(\mm\)-a.e.\ unique foliation of \(|\mu|\in\mathfrak M_+(\Omega;L^0(\mm))\) (whose existence and
essential uniqueness are guaranteed by Corollary \ref{corolmainthm}), it holds that
\begin{equation}\label{eq:foliat_M-valued}
\|\mu_x\|_x\leq|\mu|_x\quad\text{ for }\mm\text{-a.e.\ }x\in\X,
\end{equation}
where \(\|\mu_x\|_x\in\mathfrak M_+(\Omega)\) denotes the total variation measure of \(\mu_x\).
\end{theorem}
\begin{proof}
Fix a countable base \(\mathcal C\) for the topology of \(\Omega\) that is closed under finite unions
and contains \(\varnothing\). For any \(U\in\mathcal C\), we fix a representative
\(\theta_U\in\bar\Gamma(E_{\rm M})\) of \(\mu(U)\in{\rm M}\cong\Gamma(E_{\rm M})\).
Given that \(|\mu(U)-\mu(V)|\leq|\mu|(U\Delta V)\) and \(|\mu(U\cup V)-\mu(U)-\mu(V)|\leq|\mu|(U\cap V)\),
we can find a set \(N\in\Sigma\) with \(\mm(N)=0\) such that for any \(x\in\X\setminus N\)
the following conditions hold:
\begin{subequations}\begin{align}
\label{eq:foliat_vect_meas_hp1}
\theta_\varnothing(x)=0\in{\rm M}_x&,\\
\label{eq:foliat_vect_meas_hp2}
\|\theta_U(x)\|_x\leq|\mu|_x(U)&\quad\text{ for every }U\in\mathcal C,\\
\label{eq:foliat_vect_meas_hp3}
\|\theta_U(x)-\theta_V(x)\|_x\leq|\mu|_x(U\Delta V)&\quad\text{ for every }U,V\in\mathcal C,\\
\label{eq:foliat_vect_meas_hp4}
\|\theta_{U\cup V}(x)-\theta_U(x)-\theta_V(x)\|_x\leq|\mu|_x(U\cap V)&\quad\text{ for every }U,V\in\mathcal C.
\end{align}\end{subequations}
Let \(x\in\X\setminus N\) be fixed. For any set \(A\in\mathscr B(\Omega)\), we define
\[
\mathcal F_{A,x}\coloneqq\big\{(U_n)_{n\in\N}\subseteq\mathcal C\;\big|\;|\mu|_x(A\Delta U_n)\to 0\text{ as }n\to\infty\big\}.
\]
We claim that
\begin{subequations}\begin{align}
\label{eq:foliat_vect_meas_aux1}
\mathcal F_{A,x}\neq\varnothing&,\\
\label{eq:foliat_vect_meas_aux2}
(\theta_{U_n}(x))_{n\in\N}\subseteq{\rm M}_x&\quad\text{ is Cauchy for every }(U_n)_{n\in\N}\in\mathcal F_{A,x},\\
\label{eq:foliat_vect_meas_aux3}
\lim_{n\to\infty}\theta_{U_n}(x)=\lim_{n\to\infty}\theta_{V_n}(x)&\quad\text{ for every }
(U_n)_{n\in\N},(V_n)_{n\in\N}\in\mathcal F_{A,x}.
\end{align}\end{subequations}
To prove \eqref{eq:foliat_vect_meas_aux1}, note that the inner regularity of \(|\mu|_x\) yields the existence of
a sequence \((K_n)_{n\in\N}\) of compact subsets of \(A\) such that \(|\mu|_x(A\setminus K_n)\to 0\). By the outer
regularity of \(|\mu|_x\) and the compactness of \(K_n\), for any \(n\in\N\) we can find \(U_n\in\mathcal C\)
such that \(K_n\subseteq U_n\) and \(|\mu|_x(U_n\setminus K_n)\to 0\). In particular, we have
\(|\mu|_x(A\Delta U_n)\to 0\), which shows that \((U_n)_{n\in\N}\in\mathcal F_{A,x}\) and thus
\(\mathcal F_{A,x}\neq\varnothing\). Next, given any \((U_n)_{n\in\N}\in\mathcal F_{A,x}\), it follows from
\eqref{eq:foliat_vect_meas_hp3} that
\[
\|\theta_{U_n}(x)-\theta_{U_m}(x)\|_x\leq|\mu|_x(U_n\Delta U_m)\leq|\mu|_x(A\Delta U_n)+|\mu|_x(A\Delta U_m)\to 0
\quad\text{ as }n,m\to\infty,
\]
proving \eqref{eq:foliat_vect_meas_aux2}. Moreover, if \((U_n)_{n\in\N},(V_n)_{n\in\N}\in\mathcal F_{A,x}\)
are given, then the sequence \((W_k)_{k\in\N}\) defined as \(U_1,V_1,U_2,V_2,\ldots\) clearly belongs to
\(\mathcal F_{A,x}\), so that accordingly \eqref{eq:foliat_vect_meas_aux2} ensures that
\[
\lim_{n\to\infty}\theta_{U_n}(x)=\lim_{n\to\infty}\theta_{W_{2n+1}}(x)=\lim_{k\to\infty}\theta_{W_k}(x)
=\lim_{n\to\infty}\theta_{W_{2n}}(x)=\lim_{n\to\infty}\theta_{V_n}(x),
\]
proving \eqref{eq:foliat_vect_meas_aux3}. In view of \eqref{eq:foliat_vect_meas_aux1}, \eqref{eq:foliat_vect_meas_aux2},
\eqref{eq:foliat_vect_meas_aux3}, it makes sense to define \(\mu_x\colon\mathscr B(\Omega)\to{\rm M}_x\) as
\[
\mu_x(A)\coloneqq\lim_{n\to\infty}\theta_{U_n}(x)\quad\text{ whenever }A\in\mathscr B(\Omega)\text{ and }
(U_n)_{n\in\N}\in\mathcal F_{A,x}.
\]
Note that \(\mu_x(U)=\theta_U(x)\) for all \(U\in\mathcal C\). To conclude, it remains to prove the following claims:
\begin{itemize}
\item[\(\rm a)\)] \(\mu_x\in\mathfrak M(\Omega;{\rm M}_x)\) for every \(x\in\X\setminus N\) and \eqref{eq:foliat_M-valued} holds.
\item[\(\rm b)\)] Letting \(\mu_x\coloneqq 0\in\mathfrak M(\Omega;{\rm M}_x)\) for every \(x\in N\),
we have that \(x\mapsto\mu_x\) is \(\mm\)-measurable.
\item[\(\rm c)\)] \((\mu_x)_{x\in\X}\) is a foliation of \(\mu\).
\end{itemize}
Let us pass to the verification of the above claims:
\smallskip

{\bf a)} It follows from \eqref{eq:foliat_vect_meas_hp1} that \(\mu_x(\varnothing)=\theta_\varnothing(x)=0\).
Given \(A,B\in\mathscr B(\Omega)\) with \(A\cap B=\varnothing\), \((U_n)_{n\in\N}\in\mathcal F_{A,x}\)
and \((V_n)_{n\in\N}\in\mathcal F_{B,x}\), we have \((U_n\cup V_n)_{n\in\N}\in\mathcal F_{A\cup B,x}\),
and \eqref{eq:foliat_vect_meas_hp4} ensures that
\[
\|\theta_{U_n\cup V_n}(x)-\theta_{U_n}(x)-\theta_{V_n}(x)\|_x\leq|\mu|_x(U_n\cap V_n)
\leq|\mu|_x(U_n\Delta A)+|\mu|_x(V_n\Delta B)\to 0
\]
as \(n\to\infty\), which implies that
\[
\mu_x(A\cup B)=\lim_{n\to\infty}\theta_{U_n\cup V_n}(x)
=\lim_{n\to\infty}\theta_{U_n}(x)+\lim_{n\to\infty}\theta_{V_n}(x)=\mu_x(A)+\mu_x(B).
\]
This shows that \(\mu_x\) is finitely additive. Moreover, in view of \eqref{eq:foliat_vect_meas_hp2} we can estimate
\[
\|\mu_x(A)\|_x=\lim_{n\to\infty}\|\theta_{U_n}(x)\|_x\leq\lim_{n\to\infty}|\mu|_x(U_n)=|\mu|_x(A).
\]
In particular, if \((A_n)_{n\in\N}\subseteq\mathscr B(\Omega)\) are pairwise disjoint, then we have that
\[
\sum_{n=1}^\infty\|\mu_x(A_n)\|_x\leq\sum_{n=1}^\infty|\mu|_x(A_n)=|\mu|_x\bigg(\bigcup_{n\in\N}A_n\bigg)<+\infty.
\]
In combination with the finite additivity of \(\mu_x\), this implies that \(\mu_x\) is an \({\rm M}_x\)-valued measure.
The above estimates show also that the measure \(\mu_x\) is of bounded variation and \(\|\mu_x\|_x\leq|\mu|_x\), so that
\(\mu_x\in\mathfrak M(\Omega;{\rm M}_x)\). This completes the proof of a).
\smallskip

{\bf b)} Fix \(A\in\mathscr B(\Omega)\). For any \(n\in\N\), it follows from the \(\mm\)-measurability of
\(x\mapsto|\mu|_x\) that there exists a partition \((E^U_n)_{U\in\mathcal C}\subseteq\bar\Sigma_\mm\) of \(\X\) such that
\[
|\mu|_x(U\Delta A)\leq\frac{1}{n}\quad\text{ for every }U\in\mathcal C\text{ and }x\in E^U_n.
\]
Letting \(v_n(x)\coloneqq\sum_{U\in\mathcal C}\1_{E^U_n}(x)\theta_U(x)\in{\rm M}_x\) for every \(x\in\X\), we thus
have that \(v_n\colon\X\to\mathbb U\) is \(\mm\)-measurable, so that also
\(\X\ni x\mapsto\mu_x(A)=\lim_n v_n(x)\in\mathbb U\) is \(\mm\)-measurable, proving b).
\smallskip

{\bf c)} Fix \(A\in\mathscr B(\Omega)\). Taking \(v_n\) as in the proof of b), we have
\(v_n(x)=\sum_{U\in\mathcal C}\1_{E^U_n}^\mm(x)\mu(U)(x)\) and
\(\big|\mu(A)-\sum_{U\in\mathcal C}\1_{E^U_n}^\mm\mu(U)\big|(x)\leq\sum_{U\in\mathcal C}\1_{E^U_n}^\mm(x)|\mu|_x(U\Delta A)\leq 1/n\) for \(\mm\)-a.e.\ \(x\in\X\). Hence,
\[
\mu_x(A)=\lim_{n\to\infty}v_n(x)=\lim_{n\to\infty}\sum_{U\in\mathcal C}\1_{E^U_n}^\mm(x)\mu(U)(x)
=\mu(A)(x)\quad\text{ for }\mm\text{-a.e.\ }x\in\X.
\]
This shows that \(x\mapsto\mu_x(A)\) is a representative of \(\mu(A)\), thus proving c).
\end{proof}
\section{Applications}\label{s:appl}
\subsection{Module-valued martingales}\label{s:random_martingale}
By making use of the machinery that we have developed in this paper, we now propose a notion of
discrete-time martingale taking values into a complete random normed module, see Definition \ref{def:RNM_martingale};
we refer e.g.\ to the monograph \cite{Pis16} for a detailed account of the theory of Banach-space-valued martingales
(see also \cite{sousi2013advanced}).
In order to introduce our definition of martingale, we first need to prove the existence of the conditional expectation
for maps taking values into a complete random normed module (Proposition \ref{prop:E_L0} and Theorem \ref{thm:E_M}).
\medskip

As usual, we fix a probability space \((\X,\Sigma,\mm)\). Let us also fix a measurable space \((\Omega,\mathcal A)\) together
with a \(\sigma\)-subalgebra \(\mathcal B\) of \(\mathcal A\). Since we are considering two different
\(\sigma\)-algebras on the same set \(\Omega\), in this section we adopt the notations \(\mathcal P(\Omega,\mathcal A;L^0(\mm))\),
\(\mathcal P(\Omega,\mathcal B;L^0(\mm))\) and so on.
As \(\mathcal B\subseteq\mathcal A\), for any \(\mu\in\mathcal P(\Omega,\mathcal A;L^0(\mm))\)
we have that \(S_{\mu,f}(\Omega,\mathcal B;L^0(\mm))\subseteq S_{\mu,f}(\Omega,\mathcal A;L^0(\mm))\),
thus accordingly \(L^2_\mu(\Omega,\mathcal B;L^0(\mm))\) can be canonically identified with a closed
\(L^0(\mm)\)-submodule of \(L^2_\mu(\Omega,\mathcal A;L^0(\mm))\). We denote by
\[
{\rm p}_{\mathcal B}\colon L^2_\mu(\Omega,\mathcal A;L^0(\mm))\to L^2_\mu(\Omega,\mathcal B;L^0(\mm))
\]
the orthogonal projection map. Likewise, \(L^1_\mu(\Omega,\mathcal B;L^0(\mm))\) can be canonically identified with a closed
\(L^0(\mm)\)-submodule of \(L^1_\mu(\Omega,\mathcal A;L^0(\mm))\).
\medskip

First of all, let us prove the existence of the conditional expectation for \(L^0(\mm)\)-valued maps:
\begin{proposition}[Conditional expectation for \(L^0\)-valued maps]\label{prop:E_L0}
Fix \(\mu\in\mathcal P(\Omega,\mathcal A;L^0(\mm))\). Then there exists a unique homomorphism of random normed modules
\begin{equation}\label{eq:def_E_L0}
\mathbb E(\,\cdot\,|\mathcal B)\colon L^1_\mu(\Omega,\mathcal A;L^0(\mm))\to L^1_\mu(\Omega,\mathcal B;L^0(\mm))
\end{equation}
such that \(\mathbb E(\1_A^\mu\1_\X^\mm|\mathcal B)={\rm p}_{\mathcal B}(\1_A^\mu\1_\X^\mm)\) for every \(A\in\mathcal A\),
and we have that
\begin{equation}\label{eq:ineq_E_L0}
|\mathbb E(g|\mathcal B)|^{1,\mu}\leq|g|^{1,\mu}\quad\text{ for every }g\in L^1_\mu(\Omega,\mathcal A;L^0(\mm)).
\end{equation}
Moreover, it holds that
\begin{equation}\label{eq:mult_E_L0}
\mathbb E(fg|\mathcal B)=f\,\mathbb E(g|\mathcal B)\quad\text{ for every }f\in L^\infty_\mu(\Omega,\mathcal B;L^0(\mm))
\text{ and }g\in L^1_\mu(\Omega,\mathcal A;L^0(\mm)).
\end{equation}
We say that \(\mathbb E(\,\cdot\,|\mathcal B)\) is the \textbf{\(L^0(\mm)\)-valued conditional expectation} induced by \(\mathcal B\).
\end{proposition}
\begin{proof}
Let \(\theta_2\colon L^2_\mu(\Omega,\mathcal B;L^0(\mm))\to L^1_\mu(\Omega,\mathcal B;L^0(\mm))\)
be given by Remark \ref{rmk:def_theta_p}. We then define the operator
\(\varphi\colon S_{\mu,f}(\Omega,\mathcal A;L^0(\mm))\to L^1_\mu(\Omega,\mathcal B;L^0(\mm))\) as
\[
\varphi(g)\coloneqq(\theta_2\circ{\rm p}_{\mathcal B})(g)\quad\text{ for every }g\in S_{\mu,f}(\Omega,\mathcal A;L^0(\mm)).
\]
Applying Lemma \ref{lem:formula_norm_1_mu}, \eqref{eq:compat_mult_Lp}, \eqref{eq:char_orth_proj}
and Proposition \ref{prop:mult_Linfty_L1}, we obtain that
\[\begin{split}
|\varphi(g)|^{1,\mu}&=
\bigvee\bigg\{\int f\,\varphi(g)\,\d\mu\;\bigg|\;f\in S_{\mu,f}(\Omega,\mathcal B;L^0(\mm)),\,|f|^{\infty,\mu}\leq\1_\X^\mm\bigg\}\\
&=\bigvee\big\{\langle f,{\rm p}_{\mathcal B}(g)\rangle\;\big|\;f\in S_{\mu,f}(\Omega,\mathcal B;L^0(\mm)),\,|f|^{\infty,\mu}\leq\1_\X^\mm\big\}\\
&=\bigvee\big\{\langle f,g\rangle\;\big|\;f\in S_{\mu,f}(\Omega,\mathcal B;L^0(\mm)),\,|f|^{\infty,\mu}\leq\1_\X^\mm\big\}\\
&=\bigvee\bigg\{\int fg\,\d\mu\;\bigg|\;f\in S_{\mu,f}(\Omega,\mathcal B;L^0(\mm)),\,|f|^{\infty,\mu}\leq\1_\X^\mm\bigg\}\\
&\leq\bigvee\big\{|f|^{\infty,\mu}|g|^{1,\mu}\;\big|\;f\in S_{\mu,f}(\Omega,\mathcal B;L^0(\mm)),\,|f|^{\infty,\mu}\leq\1_\X^\mm\big\}
\leq|g|^{1,\mu},
\end{split}\]
thus in particular \(\varphi\) is well defined. Note also that \(\varphi\) is a homomorphism of random normed modules
satisfying \(|\varphi(g)|^{1,\mu}\leq|g|^{1,\mu}\) for every \(g\in S_{\mu,f}(\Omega,\mathcal A;L^0(\mm))\). Moreover, we have that
\[
\varphi\bigg(\sum_{i\in F}\1_{A_i}^\mu\bar g^i\bigg)=\sum_{i\in F}\bar g^i\,\varphi(\1_{A_i}^\mu\1_\X^\mm)
=\sum_{i\in F}\bar g^i\,{\rm p}_{\mathcal B}(\1_{A_i}^\mu\1_\X^\mm)
\quad\text{ for all }\sum_{i\in F}\1_{A_i}^\mu\bar g^i\in S_{\mu,f}(\Omega,\mathcal A;L^0(\mm)),
\]
which shows that \(\varphi\) is the unique \(L^0(\mm)\)-linear map from \(S_{\mu,f}(\Omega,\mathcal A;L^0(\mm))\)
to \(L^1_\mu(\Omega,\mathcal B;L^0(\mm))\) such that \(\varphi(\1_A^\mu\1_\X^\mm)={\rm p}_{\mathcal B}(\1_A^\mu\1_\X^\mm)\)
for every \(A\in\mathcal A\). Given that \(S_{\mu,f}(\Omega,\mathcal A;L^0(\mm))\) is dense in \(L^1_\mu(\Omega,\mathcal A;L^0(\mm))\),
the map \(\varphi\) can be uniquely extended to a homomorphism of random normed modules as in \eqref{eq:def_E_L0}
satisfying \eqref{eq:ineq_E_L0}. Finally, \({\rm p}_{\mathcal B}(\1_A^\mu\1_B^\mu\1_\X^\mm)
=(\1_B^\mu\1_\X^\mm)\,{\rm p}_{\mathcal B}(\1_A^\mu\1_\X^\mm)\) for every \(A\in\mathcal A\) and \(B\in\mathcal B\),
whence it follows that
\[\begin{split}
\mathbb E(fg|\mathcal B)&=\varphi(fg)=\sum_{(i,j)\in F\times G}\bar f^i\bar g^j\varphi(\1_{B_i\cap A_j}^\mu\1_\X^\mm)
=\sum_{(i,j)\in F\times G}(\1_{B_i}^\mu\bar f^i)\big(\bar g^j\varphi(\1_{A_j}^\mm\1_\X^\mm)\big)=f\,\mathbb E(g|\mathcal B)
\end{split}\]
for every \(f=\sum_{i\in F}\1_{B_i}^\mu\bar f^i\in S_{\mu,f}(\Omega,\mathcal B;L^0(\mm))\)
and \(g=\sum_{j\in G}\1_{A_j}^\mu\bar g^j\in S_{\mu,f}(\Omega,\mathcal A;L^0(\mm))\). By an approximation argument,
one can thus conclude that \eqref{eq:mult_E_L0} holds.
\end{proof}

As a corollary of Proposition \ref{prop:E_L0}, we can prove the existence of the conditional expectation
for random-normed-module-valued maps:
\begin{theorem}[Conditional expectation for \({\rm M}\)-valued maps]\label{thm:E_M}
Let \(({\rm M},|\cdot|)\) be a complete random normed module with base \((\X,\Sigma,\mm)\).
Let \(\mu\in\mathcal P(\Omega,\mathcal A;L^0(\mm))\) be given. Then there exists a unique homomorphism of random normed modules
\begin{equation}\label{eq:def_E_M}
\mathbb E(\,\cdot\,|\mathcal B)\colon L^1_\mu(\Omega,\mathcal A;{\rm M})\to L^1_\mu(\Omega,\mathcal B;{\rm M})
\end{equation}
such that \(|\mathbb E(v|\mathcal B)|^{1,\mu}\leq|v|^{1,\mu}\) for every \(v\in L^1_\mu(\Omega,\mathcal A;{\rm M})\) and
\begin{equation}\label{eq:char_E_M}
\mathbb E(f\cdot v|\mathcal B)=f\cdot\mathbb E(v|\mathcal B)\quad
\text{ for every }f\in L^\infty_\mu(\Omega,\mathcal B;L^0(\mm))\text{ and }v\in L^1_\mu(\Omega,\mathcal A;{\rm M}).
\end{equation}
We say that \(\mathbb E(\,\cdot\,|\mathcal B)\) is the \textbf{\({\rm M}\)-valued conditional expectation} induced by \(\mathcal B\).
\end{theorem}
\begin{proof}
Let \(\varphi\coloneqq\mathbb E(\,\cdot\,|\mathcal B)\colon L^1_\mu(\Omega,\mathcal A;L^0(\mm))\to L^1_\mu(\Omega,\mathcal B;L^0(\mm))\)
be given by Proposition \ref{prop:E_L0}.
Let \(J_{\mathcal A}\colon L^1_\mu(\Omega,\mathcal A;L^0(\mm))\hat\otimes_\pi{\rm M}\to L^1_\mu(\Omega,\mathcal A;{\rm M})\)
and \(J_{\mathcal B}\colon L^1_\mu(\Omega,\mathcal B;L^0(\mm))\hat\otimes_\pi{\rm M}\to L^1_\mu(\Omega,\mathcal B;{\rm M})\)
be the isometric isomorphisms of random normed modules given by Theorem \ref{thm:L1_M_as_tens_prod}. We define
\[
\mathbb E(\,\cdot\,|\mathcal B)\coloneqq J_{\mathcal B}\circ(\varphi\otimes_\pi{\rm id}_{\rm M})\circ J_{\mathcal A}^{-1}
\colon L^1_\mu(\Omega,\mathcal A;{\rm M})\to L^1_\mu(\Omega,\mathcal B;{\rm M}).
\]
One can check that for any \(f\in L^\infty_\mu(\Omega,\mathcal B;L^0(\mm))\), \(A\in\mathcal A\) and \(\bar v\in{\rm M}\) it holds that
\[\begin{split}
\mathbb E(f\cdot(\1_A^\mu\bar v)|\mathcal B)&=(J_{\mathcal B}\circ(\varphi\otimes_\pi{\rm id}_{\rm M}))((f\1_A^\mu)\otimes\bar v)
=J_{\mathcal B}(\varphi(f\1_A^\mu)\otimes\bar v)=\varphi(f\1_A^\mu)\cdot\bar v\\
&=f\cdot(\varphi(\1_A^\mu)\cdot\bar v)=f\cdot J_{\mathcal B}(\varphi(\1_A^\mu)\otimes\bar v)
=f\cdot(J_{\mathcal B}\circ(\varphi\otimes_\pi{\rm id}_{\rm M}))(\1_A^\mu\otimes\bar v)\\
&=f\cdot\mathbb E(\1_A^\mu\bar v|\mathcal B),
\end{split}\]
which implies that \(\mathbb E(\,\cdot\,|\mathcal B)\) is the unique homomorphism satisfying \eqref{eq:char_E_M}.
Moreover, we have
\[
|\mathbb E(v|\mathcal B)|^{1,\mu}\leq\big|(\varphi\otimes_\pi{\rm id}_{\rm M})(J_{\mathcal A}^{-1}(v))\big|_\pi
\leq|\varphi||{\rm id}_{\rm M}||J_{\mathcal A}^{-1}(v)|_\pi\leq|v|^{1,\mu},
\]
thus accordingly the proof is complete.
\end{proof}

Having Theorem \ref{thm:E_M} at our disposal, we can now formulate our definition of a random-normed-module-valued martingale:
\begin{definition}[Random-normed-module-valued martingale]\label{def:RNM_martingale}
Let \(({\rm M},|\cdot|)\) be a complete random normed module with base \((\X,\Sigma,\mm)\).
Let \(\mu\in\mathcal P(\Omega,\mathcal A;L^0(\mm))\) be given. Let \((\mathcal A_n)_{n\in\N}\) be a given \textbf{filtration}
of \(\mathcal A\), i.e.\ each \(\mathcal A_n\) is a \(\sigma\)-subalgebra of \(\mathcal A\) and it holds that
\(\mathcal A_1\subseteq\mathcal A_2\subseteq\ldots\). Then we say that a sequence
\((v_n)_{n\in\N}\subseteq L^1_\mu(\Omega,\mathcal A;{\rm M})\) is an \textbf{\({\rm M}\)-valued martingale}
provided it holds that \(v_n\in L^1_\mu(\Omega,\mathcal A_n;L^0(\mm))\) for every \(n\in\N\) and
\[
v_n=\mathbb E(v_{n+1}|\mathcal A_n)\quad\text{ for every }n\in\N.
\]
\end{definition}

We expect that many results for Banach-space-valued martingales can be extended to the setting of
random-normed-module-valued martingales, but in this paper we do not investigate further in that direction.
\subsection{The random Radon--Nikod\'{y}m property}\label{s:rRNP}
As usual, \((\X,\Sigma,\mm)\) is a given probability space. In light of Theorem \ref{thm:RN},
we propose the following notion of Radon--Nikod\'{y}m property for complete random normed modules
with base \((\X,\Sigma,\mm)\):
\begin{definition}[Random Radon--Nikod\'{y}m property]\label{def:RNP}
Let \(({\rm M},|\cdot|)\) be a complete random normed module with base \((\X,\Sigma,\mm)\). Then
we say that \({\rm M}\) has the \textbf{random Radon--Nikod\'{y}m property} if the following holds:
given \(\mu\in\mathcal M([0,1];{\rm M})\) such that \(\nu\coloneqq|\mu|\in\mathcal M_+([0,1];L^0(\mm))\)
satisfies \(\hat\nu\ll\mathscr L_1\otimes\mm\), there exists \(\delta\in L^1_{\mathscr L_1^c}([0,1];{\rm M})\)
such that
\[
\mu(A)=\int_A \delta\cdot\d\mathscr L_1^c\quad\text{ for every }A\in\mathscr B([0,1]),
\]
where \(\mathscr L_1^c\) is the \(L^0\)-valued measure defined as in Remark \ref{rmk:lambda^c}.
\end{definition}

We chose this axiomatisation of random Radon--Nikod\'{y}m property due to the next result:
\begin{proposition}
The space \(L^0(\mm)\) has the random Radon--Nikod\'{y}m property.
\end{proposition}
\begin{proof}
Fix any \(\mu\in\mathcal M([0,1];L^0(\mm))\) such that \(\nu\coloneqq|\mu|\) satisfies \(\hat\nu\ll\mathscr L_1\otimes\mm\).
Define
\[
\xi_\pm\coloneqq\frac{\1_{\{\mu^\pm([0,1])>0\}}^\mm}{\mu^\pm([0,1])}\cdot\mu^\pm\in\mathcal P([0,1];L^0(\mm))
\quad\text{ if }\mu^\pm\neq 0,
\]
or \(\xi_\pm\coloneqq 0\in\mathcal M_+([0,1];L^0(\mm))\) otherwise. Since \(\mu^\pm\leq\nu\), we have that
\(\hat{\mu^\pm}\leq\hat\nu\ll\mathscr L_1\otimes\mm\) and thus \(\xi_\pm\ll\mathscr L_1\otimes\mm\). Given
that \(\mathscr L_1\otimes\mm=\hat{\mathscr L_1^c}\) by \eqref{eq:hat_lambda^c} and \(\xi_\pm\) has a foliation
in \(\mathcal P([0,1])\) by Corollary \ref{corolmainthm} (if \(\mu^\pm\neq 0\), or the null foliation otherwise),
it follows from Theorem \ref{thm:RN} that there exists a (unique) element
\(\delta_\pm\in L^1_{\mathscr L_1^c}([0,1];L^0_+(\mm))\) such that
\(\xi_\pm(A)=\int_A\delta_\pm\,\d\mathscr L_1^c\) for all \(A\in\mathscr B([0,1])\). Therefore,
\(\delta\coloneqq\mu^+([0,1])\delta_+ -\mu^-([0,1])\delta_-\in L^1_{\mathscr L_1^c}([0,1];L^0(\mm))\) satisfies
\[\begin{split}
\mu(A)&=\mu^+([0,1])\cdot\xi_+(A)-\mu^-([0,1])\cdot\xi_-(A)\\
&=\mu^+([0,1])\int_A\delta_+(A)\,\d\mathscr L_1^c-\mu^-([0,1])\int_A\delta_-(A)\,\d\mathscr L_1^c
=\int_A\delta\,\d\mathscr L_1^c
\end{split}\]
for all \(A\in\mathscr B([0,1])\). This proves that \(L^0(\mm)\) has the random Radon--Nikod\'{y}m property.
\end{proof}

In view of the fundamental importance of the Radon--Nikod\'{y}m property for Banach spaces, we believe that
it would be very interesting to conduct an in-depth investigation of the random Radon--Nikod\'{y}m property,
but we leave it for future research. It would interesting to understand whether the class of complete random
normed modules having the random Radon--Nikod\'{y}m property includes, for example, those modules that are
`random reflexive', or those that are countably generated and admits an isometric predual. Furthermore, it is
well known that the Radon--Nikod\'{y}m property for Banach spaces can be characterised in many different ways.
To mention a few, for a Banach space \(\mathbb B\) the following are equivalent:
\begin{itemize}
\item[\(\rm i)\)] \(\mathbb B\) has the Radon--Nikod\'{y}m property.
\item[\(\rm ii)\)] Every \(\mathbb B\)-valued martingale bounded in \(L^1([0,1];\mathbb B)\) converges
almost surely.
\item[\(\rm iii)\)] Every Lipschitz curve \(\gamma\colon[0,1]\to\mathbb B\) is almost everywhere differentiable.
\end{itemize}
It would be interesting to understand if the above equivalences can be generalised in some way to the setting
of random normed modules. The equivalence \({\rm i)}\Longleftrightarrow{\rm ii)}\) might be useful for studying
further the theory of random-normed-module-valued martingales that we have introduced in Section
\ref{s:random_martingale}, whereas the equivalence \({\rm i)}\Longleftrightarrow{\rm iii)}\) is likely to be
useful in nonsmooth differential geometry. Conjecturally, the random Radon--Nikod\`{y}m property that we have
proposed in Definition \ref{def:RNP} (or some variant of it) is equivalent to the almost everywhere
diffentiability of \emph{random Lipschitz curves} taking values into the random normed module under consideration.
Namely, given a complete random normed module \(({\rm M},|\cdot|)\) with base \((\X,\Sigma,\mm)\), we say that
a map \(\gamma\colon[0,1]\to{\rm M}\) is a random Lipschitz curve if there is \(L\in L^0_+(\mm)\) such that
\[
|\gamma_t-\gamma_s|\leq|t-s|L\quad\text{ for every }s,t\in[0,1].
\]
By saying that such a curve \(\gamma\colon[0,1]\to{\rm M}\) is a.e.\ differentiable, we mean that
\[
\exists\,\dot\gamma_t\coloneqq\lim_{[0,1]\ni s\to t}\frac{\gamma_s-\gamma_t}{s-t}\in{\rm M}\quad\text{ for a.e.\ }t\in[0,1].
\]
\subsection{Random sets of finite perimeter}\label{s:random_FP_sets}
In this section, we introduce a notion of \emph{random set of finite perimeter} in \(\R^N\). For an in-depth
treatment of the classical theory of sets of finite perimeter in the Euclidean space, we refer to
\cite{Amb:Fus:Pal:00,maggi2012sets}. Let us fix a probability space \((\X,\Sigma,\mm)\) and a bounded
open subset \(\Omega\neq\varnothing\) of \(\R^N\). For brevity, we denote
\[
\lambda_\Omega\coloneqq(\mathscr L^N\mrestr\Omega)^c\in\mathfrak M_+(\Omega;L^0(\mm)),
\]
where \(\mathscr L^N\mrestr\Omega\) stands for the restriction of the Lebesgue measure \(\mathscr L^N\)
to \(\Omega\), while \((\mathscr L^N\mrestr\Omega)^c\) is defined as in Remark \ref{rmk:lambda^c}.
We recall from \eqref{eq:hat_lambda^c} that \(\hat{\lambda_\Omega}=(\mathscr L^N\mrestr\Omega)\otimes\mm\).
\begin{definition}[Random set of finite perimeter]\label{def:random_set_FP}
Let \(\nchi\in L^1_{\lambda_\Omega}(\Omega;L^0(\mm))\) be given. Then we say that \(\nchi\) is a
\textbf{random set of finite perimeter} provided the following conditions hold:
\begin{itemize}
\item[\(\rm i)\)] \(I(\nchi)(p,x)\in\{0,1\}\) for \((\mathscr L^N\mrestr\Omega)\otimes\mm\)-a.e.\ \((p,x)\in\Omega\times\X\),
where the isometric isomorphism \(I\colon L^1_{\lambda_\Omega}(\Omega;L^0(\mm))\to\tilde L^1_{\lambda_\Omega}(\Omega;L^0(\mm))\)
is given by Theorem \ref{thm:L1_tilde_L1}.
\item[\(\rm ii)\)] There exists a Radon \(L^0(\mm)^N\)-valued measure
\(\mu^\Omega_\nchi\in\mathfrak M(\Omega;L^0(\mm)^N)\) such that
\begin{equation}\label{eq:def_random_FP}
\int\nchi\,{\rm div}(T)\,\d\lambda_\Omega=\int\1_\X^\mm T\cdot\d\mu^\Omega_\nchi
\quad\text{ for every }T\in C^1_c(\Omega;\R^N).
\end{equation}
\end{itemize}
Moreover, we define the \textbf{random perimeter} of \(\nchi\) as
\[
P(\nchi)\coloneqq|\mu^\Omega_\nchi|_{\rm TV}\in L^0_+(\mm).
\]
\end{definition}

We consider random sets of finite perimeter inside a bounded domain \(\Omega\), and not on
the whole \(\R^N\) or on an unbounded domain, for the sake of presentation. Such assumption
is made only to guarantee that \(\mathscr L^N\mrestr\Omega\) is a finite measure and thus
\(\lambda_\Omega\) in an \(L^0(\mm)\)-valued measure. However, it would be easy to adapt the
definitions and the related results to arbitrary Euclidean domains. Likewise, one could define
random sets of locally finite perimeter.
\begin{lemma}\label{lem:slices_random_FP}
Let \(\nchi\in L^1_{\lambda_\Omega}(\Omega;L^0(\mm))\) be a random set of finite perimeter.
Then for \(\mm\)-a.e.\ \(x\in\X\) there exists a Borel set \(E_x\subseteq\Omega\) having finite
perimeter in \(\Omega\) such that \(I(\nchi)(\cdot,x)=\1_{E_x}\) holds in the \((\mathscr L^N\mrestr\Omega)\)-a.e.\ sense.
\end{lemma}
\begin{proof}
Fubini's theorem ensures that \(I(\nchi)(\cdot,x)\in\{0,1\}\) holds \((\mathscr L^N\mrestr\Omega)\)-a.e.\ for
\(\mm\)-a.e.\ \(x\in\X\), thus accordingly for \(\mm\)-a.e.\ \(x\in\X\) there exists a Borel set \(E_x\subseteq\Omega\)
such that \(I(\nchi)(\cdot,x)=\1_{E_x}\) in the \((\mathscr L^N\mrestr\Omega)\)-a.e.\ sense. It readily follows
from Theorem \ref{thm:M-val_foliat_exist} via an easy `patching argument' that \(\mu^\Omega_\nchi\) has a foliation
\(\big((\mu^\Omega_\nchi)_x\big)_{x\in\X}\subseteq\mathfrak M(\Omega;\R^N)\).
Given that \(C_c(\Omega;\R^N)\times C_c(\Omega)\) is separable with respect to the supremum norm, we can find
a countable set \(\mathcal C\subseteq C^1_c(\Omega;\R^N)\) such that \(\{(T,{\rm div}(T)):T\in\mathcal C\}\)
is dense in \(\{(T,{\rm div}(T)):T\in C^1_c(\Omega;\R^N)\}\subseteq C_c(\Omega;\R^N)\times C_c(\Omega)\).
Therefore, Proposition \ref{prop:ptwse_formula_int_of_tilde_L1} and \eqref{eq:def_random_FP} imply that
for \(\mm\)-a.e.\ \(x\in\X\) it holds that
\[
\int_{E_x}{\rm div}(T)\,\d\mathcal L^N=\int T\cdot\d(\mu^\Omega_\nchi)_x\quad\text{ for every }T\in\mathcal C.
\]
By approximation it follows that \(\int_{E_x}{\rm div}(T)\,\d\mathcal L^N=\int T\cdot\d(\mu^\Omega_\nchi)_x\)
for every \(T\in C^1_c(\Omega;\R^N)\), which guarantees that \(E_x\) has finite perimeter in \(\Omega\) for
\(\mm\)-a.e.\ \(x\in\X\), as desired.
\end{proof}

Many well-known properties of standard sets of finite perimeter can be extended to random sets of finite
perimeter, e.g.\ the ensuing lower semicontinuity property of random perimeters:
\begin{proposition}
Let \((\nchi_n)_{n\in\N}\subseteq L^1_{\lambda_\Omega}(\Omega;L^0(\mm))\) and
\(\nchi\in L^1_{\lambda_\Omega}(\Omega;L^0(\mm))\) be random sets of finite perimeter satisfying
\(\nchi_n\to\nchi\) in \(L^1_{\lambda_\Omega}(\Omega;L^0(\mm))\) as \(n\to\infty\).
Then we can extract a subsequence \((\nchi_{n_k})_{k\in\N}\) such that
\begin{equation}\label{eq:lsc_random_FP}
P(\nchi)\leq\varliminf_{k\to\infty}P(\nchi_{n_k})\quad\text{ holds }\mm\text{-a.e.\ on }\X.
\end{equation}
\end{proposition}
\begin{proof}
Let \((E^n_x)_{x\in\X}\) and \((E_x)_{x\in\X}\) be the sets of finite perimeter associated to
\(\nchi_n\) and \(\nchi\), respectively, as in Lemma \ref{lem:slices_random_FP}. Proposition
\ref{prop:ptwse_formula_int_of_tilde_L1} guarantees that we can find a subsequence
\((\nchi_{n_k})_{k\in\N}\) such that \(\nchi_{E^{n_k}_x}\to\nchi_{E_x}\) in \(L^1(\mathscr L^N\mrestr\Omega)\)
as \(k\to\infty\) for \(\mm\)-a.e.\ \(x\in\X\). The lower semicontinuity property of standard perimeters
then implies that \eqref{eq:lsc_random_FP} holds.
\end{proof}

We believe that the introduction of a notion of random set of finite perimeter makes it possible to develop,
as in the classical deterministic setting, an extensive theory including compactness results and related properties. 
Moreover, we expect that this framework can be employed to implement a minimising-movements scheme (see \cite{atw})
in order to model the stochastic mean curvature flow, in the spirit of the work \cite{yip1998stochastic}.
\section{Open problems}
In this conclusive section we collect several specific questions that are left open after this paper.
In order to develop further the theory that we have presented, it would be interesting to prove or disprove the following statements:
\begin{itemize}
\item[1)] \emph{Every \(L^0\)-valued measure \(\mu\colon\mathcal A\to L^0(\mm)\) is of bounded variation.}

Originally we were expecting that the above claim is true, by analogy with the classical Banach space case:
every vector measure whose target is a finite-dimensional Banach space is of bounded variation, and \(L^0(\mm)\)
has dimension 1 as an \(L^0(\mm)\)-module (since it is generated by \(\1_\X^\mm\)). However, our attempts to
prove the above claim failed, due to the fact that \(L^0\)-valued measures do not admit a Hahn-type decomposition.
Therefore, it could well be that the above claim is in fact false.
\item[2)] \emph{There exist \(L^0\)-valued measures \(\mu\in\mathcal M(\Omega;L^0(\mm))\) that cannot be foliated.}

We expect there there do exist, possibly even for \(\mu\) Radon \(L^0\)-valued measure over some
compact Hausdorff space \((\Omega,\tau)\) that is not metrisable (since every Radon \(L^0\)-valued measure over
a compact metric space can be foliated, cf.\ Corollary \ref{corolmainthm}).

\item[3)] \emph{The elements of \(L^1_\mu(\Omega;{\rm M})\) are \(\mu\)-a.e.\ equivalence classes of maps
from \(\Omega\) to \({\rm M}\).}

See the discussion preceding Lemma \ref{lem:full_formula_pushforward}.
\item[4)] \emph{The random Radon--Nikod\'{y}m theorem (i.e.\ Theorem \ref{thm:RN}) holds assuming that \(\mu\ll\nu\).}

Note that the answer is `yes' if the next statement is correct.
\item[5)] \emph{The conditions \(\mu\ll\nu\) and \(\hat\mu\ll\hat\nu\) are equivalent.}

The fact that \(\hat\mu\ll\hat\nu\) implies \(\mu\ll\nu\) is known to be true; recall Remark \ref{rmk:implication_ac}.
\item[6)] \emph{The uniform-order-continuity is exactly the uniform continuity with respect to some uniform
structure on complete random normed modules.}

See the paragraph after Proposition \ref{prop:impl_UC}.
\end{itemize}
\section*{Acknowledgements}
The authors are grateful to Hannah Geiss, Stefan Geiss and Timo Schultz for the useful discussions on
the topics of this paper, as well as to Tiexin Guo for the frequent academic exchanges.
A.\ Kubin was supported by the Academy of Finland grant 347550. 
E.\ Pasqualetto was supported by the Research Council of Finland grant 362898.
%
%
%
%

\begin{thebibliography}{10}

\bibitem{atw}
{\sc F.~Almgren, J.~E. Taylor, and L.~Wang}, {\em Curvature-driven flows: a variational approach}, SIAM Journal on Control and Optimization, 31 (1993), pp.~387--438.

\bibitem{Amb:Fus:Pal:00}
{\sc L.~Ambrosio, N.~Fusco, and D.~Pallara}, {\em Functions of bounded variation and free discontinuity problems}, Clarendon Press, Oxford New York, 2000.

\bibitem{Applebaum}
{\sc D.~Applebaum}, {\em Martingale-Valued Measures, {O}rnstein-{U}hlenbeck Processes with Jumps and Operator Self-Decomposability in {H}ilbert Space}, Springer Berlin Heidelberg, Berlin, Heidelberg, 2006, pp.~171--196.

\bibitem{BP75}
{\sc C.~Bessaga and A.~Pe{\l}czy\'{n}ski}, {\em Selected topics in infinite-dimensional topology}, Monografie matematyczne, Polska Akademia Nauk. Instytut Matematyczny, 1975.

\bibitem{Bog:07}
{\sc V.~I. Bogachev}, {\em Measure theory. {V}ol. {I}, {II}}, Springer-Verlag, Berlin, 2007.

\bibitem{BreGig23}
{\sc C.~Brena and N.~Gigli}, {\em Calculus and fine properties of functions of bounded variation on {R}{C}{D} spaces}, The Journal of Geometric Analysis, 34 (2023), p.~11.

\bibitem{BreGig}
{\sc C.~Brena and N.~Gigli}, {\em Local vector measures}, Journal of Functional Analysis, 286 (2024), p.~110202.

\bibitem{BruePasqualettoSemola22}
{\sc E.~Bru\`{e}, E.~Pasqualetto, and D.~Semola}, {\em Rectifiability of the reduced boundary for sets of finite perimeter over ${{\mathrm{RCD}}(K,N)}$ spaces}, J. Eur. Math. Soc., 25 (2022), pp.~413--465.

\bibitem{CLP}
{\sc M.~Cakovi{\'c}, D.~Lu{\v{c}}i{\'c}, and E.~Pasqualetto}, {\em Tensor products of measurable {B}anach bundles}, Ann. Funct. Anal., 17 (2026).

\bibitem{CGP}
{\sc E.~Caputo, N.~Gigli, and E.~Pasqualetto}, {\em Parallel transport on non-collapsed ${{\mathrm{RCD}}(K,N)}$ spaces}, J. Reine Angew. Math., 2025 (2025), pp.~135--204.

\bibitem{CerreiaVioglioKupperMaccheroniMarinacciVogelpoth16}
{\sc S.~Cerreia-Vioglio, M.~Kupper, F.~Maccheroni, M.~Marinacci, and N.~Vogelpoth}, {\em Conditional ${L}^p$-spaces and the duality of modules over $f$-algebras}, J. Math. Anal. Appl., 444 (2016), pp.~1045--1070.

\bibitem{CerreiaVioglioMaccheroniMarinacci17}
{\sc S.~Cerreia-Vioglio, F.~Maccheroni, and M.~Marinacci}, {\em Hilbert ${A}$-modules}, J. Math. Anal. Appl., 446 (2017), pp.~970--1017.

\bibitem{DMLP25}
{\sc S.~Di~Marino, D.~Lu\v{c}i\'{c}, and E.~Pasqualetto}, {\em Representation theorems for normed modules}, Rev. Real Acad. Cienc. Exactas Fis. Nat. Ser. A-Mat., 119 (2025).

\bibitem{DiestelUhl}
{\sc J.~Diestel and J.~Uhl}, {\em Vector measures}, vol.~15, (AMS) Mathematical Surveys, 1977.

\bibitem{evans2015measure}
{\sc L.~Evans and R.~Gariepy}, {\em Measure Theory and Fine Properties of Functions, Revised Edition}, Textbooks in Mathematics, CRC Press, 2015.

\bibitem{FilKupVog09}
{\sc D.~Filipovi{\'c}, M.~Kupper, and N.~Vogelpoth}, {\em Separation and duality in locally ${L}^0$-convex modules}, J. Funct. Anal., 256 (2009), pp.~3996--4029.

\bibitem{fre200}
{\sc D.~H. Fremlin}, {\em Measure theory}, vol.~4, Torres Fremlin, 2000.

\bibitem{gigli2017}
{\sc N.~Gigli}, {\em Lecture notes on differential calculus on {R}{C}{D} spaces}, Publ. RIMS Kyoto Univ., 54 (2018), pp.~855--918.

\bibitem{gigli2018nonsmooth}
{\sc N.~Gigli}, {\em Nonsmooth differential geometry--an approach tailored for spaces with {R}icci curvature bounded from below}, Memoirs of the American Mathematical Society, 251 (2018).

\bibitem{GLP}
{\sc N.~Gigli, D.~Lu\v{c}i\'{c}, and E.~Pasqualetto}, {\em Duals and pullbacks of normed modules}, Isr. J. Math., 267 (2025), pp.~821--866.

\bibitem{guo1989theory}
{\sc T.~Guo}, {\em The theory of probabilistic metric spaces with applications to random functional analysis}, Master's thesis, Xi'an Jiaotong University (China), 1989.

\bibitem{guo1992}
\leavevmode\vrule height 2pt depth -1.6pt width 23pt, {\em Random metric theory and its applications}, PhD thesis, Xi'an Jiaotong University (China), 1992.

\bibitem{guo1993}
\leavevmode\vrule height 2pt depth -1.6pt width 23pt, {\em A new approach to probabilistic functional analysis (in {C}hinese)}, in Proceedings of the first China Postdoctoral Academic Conference, The China National Defense and Industry Press, Beijing, 1993, pp.~1150--1154.

\bibitem{Guo-1995}
\leavevmode\vrule height 2pt depth -1.6pt width 23pt, {\em Extension theorems of continuous random linear operators on random domains}, Journal of Mathematical Analysis and Applications, 193 (1995), pp.~15--27.

\bibitem{Guo1996}
\leavevmode\vrule height 2pt depth -1.6pt width 23pt, {\em The {R}adon-{N}ikodym property of conjugate {B}anach spaces and \(w^*\)-equivalence theorems of \(w^*\)-measurable functions}, Sci. China Ser. A, 39 (1996), pp.~1034--1041.

\bibitem{guo1999}
\leavevmode\vrule height 2pt depth -1.6pt width 23pt, {\em Some basic theories of random normed linear spaces and random inner product spaces}, Acta Anal. Funct. Appl., 1 (1999), pp.~160--184.

\bibitem{Guo2000}
\leavevmode\vrule height 2pt depth -1.6pt width 23pt, {\em Representation theorems of the dual of {L}ebesgue-{B}ochner function spaces}, Sci. China Ser. A, 43 (2000), pp.~234--243.

\bibitem{Guo2008}
\leavevmode\vrule height 2pt depth -1.6pt width 23pt, {\em The relation of {B}anach--{A}laoglu theorem and {B}anach-{B}ourbaki-{K}akutani-{\v S}mulian theorem in complete random normed modules to stratification structure}, Sci. China Ser. A, 51 (2008), pp.~1651--1663.

\bibitem{Guo10}
\leavevmode\vrule height 2pt depth -1.6pt width 23pt, {\em Relations between some basic results derived form two kinds of topologies for a random locally convex module}, J. Funct. Anal., 258 (2010), pp.~3024--3047.

\bibitem{Guo-2011}
\leavevmode\vrule height 2pt depth -1.6pt width 23pt, {\em Recent progress in random metric theory and its applications to conditional risk measures}, Science China Mathematics, 54 (2011), pp.~633--660.

\bibitem{guo2013}
\leavevmode\vrule height 2pt depth -1.6pt width 23pt, {\em On some basic theorems of continuous module homomorphisms between random normed modules}, J. Func. Spaces,  (2013), p.~989102.

\bibitem{Guo24}
\leavevmode\vrule height 2pt depth -1.6pt width 23pt, {\em Optimization of conditional convex risk measures}, in Proceedings of the 8th International Congress of Chinese Mathematicians (Beijing, 2019), vol.~1, International Press of Boston, Somerville, 2024, pp.~347--371.

\bibitem{GuoLi05}
{\sc T.~Guo and S.~Li}, {\em The {J}ames theorem in complete random normed modules}, J. Math. Anal. Appl., 308 (2005), pp.~257--265.

\bibitem{GuoMuTu}
{\sc T.~Guo, X.~Mu, and Q.~Tu}, {\em Relations among the notions of various kinds of stability and applications}, Banach Journal of Mathematical Analysis, 18 (2024).

\bibitem{GuoTuMuSun26}
{\sc T.~Guo, Q.~Tu, X.~Mu, and Y.~Sun}, {\em Survey of {M}etric fixed point theory in random functional analysis}.
\newblock Preprint arXiv:2603.27963, 2026.

\bibitem{GuoWangTang23}
{\sc T.~Guo, Y.~Wang, and Y.~Tang}, {\em The {K}rein-{M}ilman theorem in random locally convex modules and its applications (in {C}hinese)}, Sci. Sinica Math, 53 (2023), pp.~1667--1684.

\bibitem{GuoWangXuYuanChen25}
{\sc T.~Guo, Y.~Wang, H.~Xu, G.~Yuan, and G.~Chen}, {\em A noncompact {S}chauder fixed point theorem in random normed modules and its applications}, Math. Ann., 391 (2025), pp.~3863--3911.

\bibitem{GuoWangYangZhang20}
{\sc T.~Guo, W.~Y.C., B.~Yang, and E.~Zhang}, {\em On $d$-$\sigma$-stability in random metric spaces and its applications}, J. Nonlinear Convex Anal., 21 (2020), pp.~1297--1316.

\bibitem{GuoYou96}
{\sc T.~Guo and Z.~You}, {\em The {R}iesz's representation theorem in complete random inner product modules and its applications}, Chin. Ann. Math. Ser. A, 17 (1997), pp.~361--364.

\bibitem{GuoZeng10}
{\sc T.~Guo and X.~Zeng}, {\em Random strict convexity and random uniform convexity in random normed modules}, Nonlinear Anal., 73 (2010), pp.~1239--1263.

\bibitem{GuoZhangWuYangYuanZeng17}
{\sc T.~Guo, E.~Zhang, W.~M.Z., B.~Yang, G.~Yuan, and X.~Zeng}, {\em On random convex analysis}, J. Nonlinear Convex Anal., 18 (2017), pp.~1967--1996.

\bibitem{GuoZhangWangGuo20}
{\sc T.~Guo, E.~Zhang, Y.~Wang, and Z.~Guo}, {\em Two fixed point theorems in complete random normed modules and their applications to backward stochastic equations}, J. Math. Anal. Appl., 483 (2020), p.~123644.

\bibitem{GuoZhaoZeng14}
{\sc T.~Guo, S.~Zhao, and X.~Zeng}, {\em The relations among the three kinds of conditional risk measures}, Sci. China Math., 57 (2014), pp.~1753--1764.

\bibitem{GuoZhaoZeng15}
\leavevmode\vrule height 2pt depth -1.6pt width 23pt, {\em Random convex analysis ({I}): separation and {F}enchel-{M}oreau duality in random locally convex modules (in {C}hinese)}, Sci. Sinica Math., 45 (2015), pp.~1961--1980.

\bibitem{Hartig83}
{\sc D.~G. Hartig}, {\em The {R}iesz {R}epresentation {T}heorem {R}evisited}, The American Mathematical Monthly, 90 (1983), pp.~277--280.

\bibitem{HLR91}
{\sc R.~Haydon, M.~Levy, and Y.~Raynaud}, {\em Randomly normed spaces}, vol.~41 of Travaux en Cours [Works in Progress], Hermann, Paris, 1991.

\bibitem{HKST:15}
{\sc J.~Heinonen, P.~Koskela, N.~Shanmugalingam, and J.~T. Tyson}, {\em Sobolev spaces on metric measure spaces. An approach based on upper gradients}, vol.~27 of New Mathematical Monographs, Cambridge University Press, Cambridge, 2015.

\bibitem{HinzRocknerTeplyaev13}
{\sc M.~Hinz, M.~R\"{o}ckner, and A.~Teplyaev}, {\em Vector analysis for {D}irichlet forms and quasilinear {P}{D}{E} and {S}{P}{D}{E} on metric measure spaces}, Stochastic Processes and their Applications, 123 (2013), pp.~4373--4406.

\bibitem{Iko:Pas:Sou:22}
{\sc T.~Ikonen, E.~Pasqualetto, and E.~Soultanis}, {\em Abstract and concrete tangent modules on {L}ipschitz differentiability spaces}, Proc. Amer. Math. Soc., 150 (2022), pp.~327--343.

\bibitem{IonescuRogersTeplayev12}
{\sc M.~Ionescu, L.~G. Rogers, and A.~Teplyaev}, {\em Derivations and {D}irichlet forms on fractals}, Journal of Functional Analysis, 263 (2012), pp.~2141--2169.

\bibitem{luvcic2024axiomatic}
{\sc D.~Lu{\v{c}}i{\'c} and E.~Pasqualetto}, {\em An axiomatic theory of normed modules via riesz spaces}, The Quarterly Journal of Mathematics, 75 (2024), pp.~1429--1479.

\bibitem{LP19}
{\sc D.~Lu\v{c}i\'{c} and E.~Pasqualetto}, {\em The {S}erre-{S}wan theorem for normed modules}, Rend. Circ. Mat. Palermo (2), 68 (2019), pp.~385--404.

\bibitem{LP24}
\leavevmode\vrule height 2pt depth -1.6pt width 23pt, {\em An axiomatic theory of normed modules via {R}iesz spaces}, The Quarterly Journal of Mathematics, 75 (2024), pp.~1429--1479.

\bibitem{LPV22}
{\sc M.~Lu\v{c}i\'{c}, E.~Pasqualetto, and I.~Vojnovi\'{c}}, {\em On the reflexivity properties of {B}anach bundles and {B}anach modules}, Banach J. Math. Anal., 18 (2024).

\bibitem{maggi2012sets}
{\sc F.~Maggi}, {\em Sets of finite perimeter and geometric variational problems: an introduction to Geometric Measure Theory}, vol.~135, Cambridge University Press, 2012.

\bibitem{Pas23}
{\sc E.~Pasqualetto}, {\em Projective and injective tensor products of {B}anach ${L}^0$-modules}, Ann. Funct. Anal., 15 (2024).

\bibitem{Pas24}
\leavevmode\vrule height 2pt depth -1.6pt width 23pt, {\em Limits and colimits in the category of {B}anach ${L}^0$-modules}, Rend. Semin. Mat. Univ. Padova, 154 (2025), pp.~105--142.

\bibitem{Phelps03}
{\sc R.~R. Phelps}, {\em Lectures on {C}hoquet's {T}heorem}, Lecture Notes in Mathematics, Springer Berlin Heidelberg, 2003.

\bibitem{Pis16}
{\sc G.~Pisier}, {\em Martingales in {B}anach {S}paces}, Cambridge Studies in Advanced Mathematics, Cambridge University Press, 2016.

\bibitem{royden2010real}
{\sc H.~Royden and P.~Fitzpatrick}, {\em Real {A}nalysis}, Prentice Hall, 2010.

\bibitem{Ryan02}
{\sc R.~A. Ryan}, {\em Introduction to {T}ensor {P}roducts of {B}anach {S}paces}, Springer Monographs in Mathematics, Springer London, 2002.

\bibitem{SS1983}
{\sc B.~Schweizer and A.~Sklar}, {\em Probabilistic Metric Spaces}, Elsevier/North-Holland, New York,; reissued by Dover Publications, New York (2005), 1983.

\bibitem{sousi2013advanced}
{\sc P.~Sousi}, {\em Advanced probability}, University of Cambridge, Cambridge,  (2013).

\bibitem{SunGuoTu25}
{\sc Y.~Sun, T.~Guo, and Q.~Tu}, {\em A fixed point theorem for random asymptotically nonexpansive mappings}, New York J. Math., 31 (2025), pp.~182--194.

\bibitem{TuMuGuo24}
{\sc Q.~Tu, X.~Mu, and T.~Guo}, {\em The random {M}arkov-{K}akutani fixed point theorem in a random locally convex module}, New York J. Math., 30 (2024), pp.~1196--1219.

\bibitem{TuMuGuoYangSun25}
{\sc Q.~Tu, X.~Mu, T.~Guo, G.~Yang, and Y.~Sun}, {\em The random {K}akutani fixed point theorem in random normed modules}, New York J. Math., 31 (2025), pp.~1543--1564.

\bibitem{Weaver99}
{\sc N.~Weaver}, {\em Lipschitz algebras}, World Scientific Publishing Co., Inc., River Edge, NJ, 1999.

\bibitem{WuGuo15}
{\sc M.~Wu and T.~Guo}, {\em A counterexample shows that not every locally ${L}^{0}$-convex topology is necessarily induced by a family of ${L}^0$-seminorms}.
\newblock Preprint arXiv:1501.04400, 2015.

\bibitem{yip1998stochastic}
{\sc N.~K. Yip}, {\em Stochastic motion by mean curvature}, Archive for rational mechanics and analysis, 144 (1998), pp.~313--355.

\bibitem{Zapata17}
{\sc J.~Zapata}, {\em On the characterization of locally ${L}^0$-convex topologies induced by a family of ${L}^0$-seminorms}, J. Convex Anal., 24 (2017), pp.~383--391.

\end{thebibliography}
%
%

%
%
%
%
\end{document}